\documentclass[12pt]{amsart}

\makeatletter
\def\section{\@startsection{section}{1}%
  \z@{1.2\linespacing\@plus\linespacing}{.5\linespacing}%
  {\normalfont\scshape\centering}}
\def\subsection{\@startsection{subsection}{2}%
  \z@{0.9\linespacing\@plus.7\linespacing}{-.5em}%
  {\normalfont\bfseries}}
\renewcommand\thesubsubsection{\@alph\c@subsubsection}
\renewcommand\paragraph{\@startsection{paragraph}{4}{\z@}%
                                    {0.33\linespacing \@plus1ex \@minus.2ex}%
                                    {-0.3em}%
                                    {\normalfont\normalsize\itshape $\bullet$\ \ }}
\makeatother

\usepackage{color}

\definecolor{gr}{rgb}   {0.,   0.69,   0.23 }
\definecolor{yl}{rgb}   {0.5,   0.5,   0.1 }
\definecolor{bl}{rgb}   {0.,    0.,   1. }
\definecolor{mg}{rgb}   {0.85,  0.,    0.85}
\definecolor{gy}{rgb}   {0.8,  0.8,   0.8}

\definecolor{marin}{rgb}   {0.,   0.65,   0.25}
\definecolor{rouge}{rgb}   {0.8,   0.,   0.}

\usepackage{amsmath}
\usepackage{amsfonts}
\usepackage{amsbsy}
\usepackage{amssymb}
\usepackage{mathrsfs}
\usepackage{times}
\usepackage{exscale}
\usepackage[colorlinks,citecolor=marin,linkcolor=rouge,
            bookmarksopen,
            bookmarksnumbered
           ]{hyperref}
\usepackage{fancyhdr}
\def\Bullet{\raise-1.15ex\hbox{$\bullet$}}
\ifx\figforTeXisloaded\relax \else\global\let\figforTeXisloaded=\relax\fi
\catcode`\@=11
\ifx\ctr@ln@m\undefined\else%
    \immediate\write16{*** Fig4TeX WARNING : \string\ctr@ln@m\space already defined.}\fi
\def\ctr@ln@m#1{\ifx#1\undefined\else%
    \immediate\write16{*** Fig4TeX WARNING : \string#1 already defined.}\fi}
\ctr@ln@m\ctr@ld@f
\def\ctr@ld@f#1#2{\ctr@ln@m#2#1#2}
\ctr@ld@f\def\ctr@ln@w#1#2{\ctr@ln@m#2\csname#1\endcsname#2}
{\catcode`\/=0 \catcode`/\=12 /ctr@ld@f/gdef/BS@{\}}
\ctr@ld@f\def\ctr@lcsn@m#1{\expandafter\ifx\csname#1\endcsname\relax\else%
    \immediate\write16{*** Fig4TeX WARNING : \BS@\expandafter\string#1\space already defined.}\fi}
\ctr@ld@f\edef\colonc@tcode{\the\catcode`\:}
\ctr@ld@f\edef\semicolonc@tcode{\the\catcode`\;}
\ctr@ld@f\def\t@stc@tcodech@nge{{\let\c@tcodech@nged=\z@%
    \ifnum\colonc@tcode=\the\catcode`\:\else\let\c@tcodech@nged=\@ne\fi%
    \ifnum\semicolonc@tcode=\the\catcode`\;\else\let\c@tcodech@nged=\@ne\fi%
    \ifx\c@tcodech@nged\@ne%
    \immediate\write16{}
    \immediate\write16{!!!=============================================================!!!}
    \immediate\write16{ Fig4TeX WARNING :}
    \immediate\write16{ The category code of some characters has been changed, which will}
    \immediate\write16{ result in an error (message "Runaway argument?").}
    \immediate\write16{ This probably comes from another package that changed the category}
    \immediate\write16{ code after Fig4TeX was loaded. If that proves to be exact, the}
    \immediate\write16{ solution is to exchange the loading commands on top of your file}
    \immediate\write16{ so that Fig4TeX is loaded last. For example, in LaTeX, we should}
    \immediate\write16{ say :}
    \immediate\write16{\BS@ usepackage[french]{babel}}
    \immediate\write16{\BS@ usepackage{fig4tex}}
    \immediate\write16{!!!=============================================================!!!}
    \immediate\write16{}
    \fi}}
\ctr@ld@f\def\FigforTeX{F\kern-.05em i\kern-.05em g\kern-.1em\raise-.14em\hbox{4}\kern-.19em\TeX}
\ctr@ln@w{newdimen}\epsil@n\epsil@n=0.00005pt
\ctr@ln@w{newdimen}\Cepsil@n\Cepsil@n=0.005pt
\ctr@ln@w{newdimen}\dcq@\dcq@=254pt
\ctr@ln@w{newdimen}\PI@\PI@=3.141592pt
\ctr@ln@w{newdimen}\DemiPI@deg\DemiPI@deg=90pt
\ctr@ln@w{newdimen}\PI@deg\PI@deg=180pt
\ctr@ln@w{newdimen}\DePI@deg\DePI@deg=360pt
\ctr@ld@f\chardef\t@n=10
\ctr@ld@f\chardef\c@nt=100
\ctr@ld@f\chardef\@lxxiv=74
\ctr@ld@f\chardef\@xci=91
\ctr@ld@f\mathchardef\@nMnCQn=9949
\ctr@ld@f\chardef\@vi=6
\ctr@ld@f\chardef\@xxx=30
\ctr@ld@f\chardef\@lvi=56
\ctr@ld@f\chardef\@@lxxi=71
\ctr@ld@f\chardef\@lxxxv=85
\ctr@ld@f\mathchardef\@@mmmmlxviii=4068
\ctr@ld@f\mathchardef\@ccclx=360
\ctr@ld@f\mathchardef\@dccxx=720
\ctr@ln@w{newcount}\p@rtent \ctr@ln@w{newcount}\f@ctech \ctr@ln@w{newcount}\result@tent
\ctr@ln@w{newdimen}\v@lmin \ctr@ln@w{newdimen}\v@lmax \ctr@ln@w{newdimen}\v@leur
\ctr@ln@w{newdimen}\result@t\ctr@ln@w{newdimen}\result@@t
\ctr@ln@w{newdimen}\mili@u \ctr@ln@w{newdimen}\c@rre \ctr@ln@w{newdimen}\delt@
\ctr@ld@f\def\degT@rd{0.017453 }  % pi/180
\ctr@ld@f\def\rdT@deg{57.295779 } % 180/pi
\ctr@ln@m\v@leurseule
{\catcode`p=12 \catcode`t=12 \gdef\v@leurseule#1pt{#1}}
\ctr@ld@f\def\repdecn@mb#1{\expandafter\v@leurseule\the#1\space}
\ctr@ld@f\def\arct@n#1(#2,#3){{\v@lmin=#2\v@lmax=#3%
    \maxim@m{\mili@u}{-\v@lmin}{\v@lmin}\maxim@m{\c@rre}{-\v@lmax}{\v@lmax}%
    \delt@=\mili@u\m@ech\mili@u%
    \ifdim\c@rre>\@nMnCQn\mili@u\divide\v@lmax\tw@\c@lATAN\v@leur(\z@,\v@lmax)% DY > 9949 DX
    \else%
    \maxim@m{\mili@u}{-\v@lmin}{\v@lmin}\maxim@m{\c@rre}{-\v@lmax}{\v@lmax}%
    \m@ech\c@rre%
    \ifdim\mili@u>\@nMnCQn\c@rre\divide\v@lmin\tw@% DX > 9949 DY
    \maxim@m{\mili@u}{-\v@lmin}{\v@lmin}\c@lATAN\v@leur(\mili@u,\z@)%
    \else\c@lATAN\v@leur(\delt@,\v@lmax)\fi\fi%
    \ifdim\v@lmin<\z@\v@leur=-\v@leur\ifdim\v@lmax<\z@\advance\v@leur-\PI@%
    \else\advance\v@leur\PI@\fi\fi%
    \global\result@t=\v@leur}#1=\result@t}
\ctr@ld@f\def\m@ech#1{\ifdim#1>1.646pt\divide\mili@u\t@n\divide\c@rre\t@n\m@ech#1\fi}
\ctr@ld@f\def\c@lATAN#1(#2,#3){{\v@lmin=#2\v@lmax=#3\v@leur=\z@\delt@=\tw@ pt%
    \un@iter{0.785398}{\v@lmax<}%
    \un@iter{0.463648}{\v@lmax<}%
    \un@iter{0.244979}{\v@lmax<}%
    \un@iter{0.124355}{\v@lmax<}%
    \un@iter{0.062419}{\v@lmax<}%
    \un@iter{0.031240}{\v@lmax<}%
    \un@iter{0.015624}{\v@lmax<}%
    \un@iter{0.007812}{\v@lmax<}%
    \un@iter{0.003906}{\v@lmax<}%
    \un@iter{0.001953}{\v@lmax<}%
    \un@iter{0.000976}{\v@lmax<}%
    \un@iter{0.000488}{\v@lmax<}%
    \un@iter{0.000244}{\v@lmax<}%
    \un@iter{0.000122}{\v@lmax<}%
    \un@iter{0.000061}{\v@lmax<}%
    \un@iter{0.000030}{\v@lmax<}%
    \un@iter{0.000015}{\v@lmax<}%
    \global\result@t=\v@leur}#1=\result@t}
\ctr@ld@f\def\un@iter#1#2{%
    \divide\delt@\tw@\edef\dpmn@{\repdecn@mb{\delt@}}%
    \mili@u=\v@lmin%
    \ifdim#2\z@%
      \advance\v@lmin-\dpmn@\v@lmax\advance\v@lmax\dpmn@\mili@u%
      \advance\v@leur-#1pt%
    \else%
      \advance\v@lmin\dpmn@\v@lmax\advance\v@lmax-\dpmn@\mili@u%
      \advance\v@leur#1pt%
    \fi}
\ctr@ld@f\def\c@ssin#1#2#3{\expandafter\ifx\csname COS@\number#3\endcsname\relax\c@lCS{#3pt}%
    \expandafter\xdef\csname COS@\number#3\endcsname{\repdecn@mb\result@t}%
    \expandafter\xdef\csname SIN@\number#3\endcsname{\repdecn@mb\result@@t}\fi%
    \edef#1{\csname COS@\number#3\endcsname}\edef#2{\csname SIN@\number#3\endcsname}}
\ctr@ld@f\def\c@lCS#1{{\mili@u=#1\p@rtent=\@ne%
    \relax\ifdim\mili@u<\z@\red@ng<-\else\red@ng>+\fi\f@ctech=\p@rtent%
    \relax\ifdim\mili@u<\z@\mili@u=-\mili@u\f@ctech=-\f@ctech\fi\c@@lCS}}
\ctr@ld@f\def\c@@lCS{\v@lmin=\mili@u\c@rre=-\mili@u\advance\c@rre\DemiPI@deg\v@lmax=\c@rre%
    \mili@u\@@lxxi\mili@u\divide\mili@u\@@mmmmlxviii%
    \edef\v@larg{\repdecn@mb{\mili@u}}\mili@u=-\v@larg\mili@u%
    \edef\v@lmxde{\repdecn@mb{\mili@u}}%
    \c@rre\@@lxxi\c@rre\divide\c@rre\@@mmmmlxviii%
    \edef\v@largC{\repdecn@mb{\c@rre}}\c@rre=-\v@largC\c@rre%
    \edef\v@lmxdeC{\repdecn@mb{\c@rre}}%
    \fctc@s\mili@u\v@lmin\global\result@t\p@rtent\v@leur%
    \let\t@mp=\v@larg\let\v@larg=\v@largC\let\v@largC=\t@mp%
    \let\t@mp=\v@lmxde\let\v@lmxde=\v@lmxdeC\let\v@lmxdeC=\t@mp%
    \fctc@s\c@rre\v@lmax\global\result@@t\f@ctech\v@leur}
\ctr@ld@f\def\fctc@s#1#2{\v@leur=#1\relax\ifdim#2<\@lxxxv\p@\cosser@h\else\sinser@t\fi}
\ctr@ld@f\def\cosser@h{\advance\v@leur\@lvi\p@\divide\v@leur\@lvi%
    \v@leur=\v@lmxde\v@leur\advance\v@leur\@xxx\p@%
    \v@leur=\v@lmxde\v@leur\advance\v@leur\@ccclx\p@%
    \v@leur=\v@lmxde\v@leur\advance\v@leur\@dccxx\p@\divide\v@leur\@dccxx}
\ctr@ld@f\def\sinser@t{\v@leur=\v@lmxdeC\p@\advance\v@leur\@vi\p@%
    \v@leur=\v@largC\v@leur\divide\v@leur\@vi}
\ctr@ld@f\def\red@ng#1#2{\relax\ifdim\mili@u#1#2\DemiPI@deg\advance\mili@u#2-\PI@deg%
    \p@rtent=-\p@rtent\red@ng#1#2\fi}
\ctr@ld@f\def\pr@c@lCS#1#2#3{\ctr@lcsn@m{COS@\number#3 }%
    \expandafter\xdef\csname COS@\number#3\endcsname{#1}%
    \expandafter\xdef\csname SIN@\number#3\endcsname{#2}}
\pr@c@lCS{1}{0}{0}
\pr@c@lCS{0.7071}{0.7071}{45}\pr@c@lCS{0.7071}{-0.7071}{-45}
\pr@c@lCS{0}{1}{90}          \pr@c@lCS{0}{-1}{-90}
\pr@c@lCS{-1}{0}{180}        \pr@c@lCS{-1}{0}{-180}
\pr@c@lCS{0}{-1}{270}        \pr@c@lCS{0}{1}{-270}
\ctr@ld@f\def\invers@#1#2{{\v@leur=#2\maxim@m{\v@lmax}{-\v@leur}{\v@leur}%
    \f@ctech=\@ne\m@inv@rs%
    \multiply\v@leur\f@ctech\edef\v@lv@leur{\repdecn@mb{\v@leur}}%
    \p@rtentiere{\p@rtent}{\v@leur}\v@lmin=\p@\divide\v@lmin\p@rtent%
    \inv@rs@\multiply\v@lmax\f@ctech\global\result@t=\v@lmax}#1=\result@t}
\ctr@ld@f\def\m@inv@rs{\ifdim\v@lmax<\p@\multiply\v@lmax\t@n\multiply\f@ctech\t@n\m@inv@rs\fi}
\ctr@ld@f\def\inv@rs@{\v@lmax=-\v@lmin\v@lmax=\v@lv@leur\v@lmax%
    \advance\v@lmax\tw@ pt\v@lmax=\repdecn@mb{\v@lmin}\v@lmax%
    \delt@=\v@lmax\advance\delt@-\v@lmin\ifdim\delt@<\z@\delt@=-\delt@\fi%
    \ifdim\delt@>\epsil@n\v@lmin=\v@lmax\inv@rs@\fi}
\ctr@ld@f\def\minim@m#1#2#3{\relax\ifdim#2<#3#1=#2\else#1=#3\fi}
\ctr@ld@f\def\maxim@m#1#2#3{\relax\ifdim#2>#3#1=#2\else#1=#3\fi}
\ctr@ld@f\def\p@rtentiere#1#2{#1=#2\divide#1by65536 }
\ctr@ld@f\def\r@undint#1#2{{\v@leur=#2\divide\v@leur\t@n\p@rtentiere{\p@rtent}{\v@leur}%
    \v@leur=\p@rtent pt\global\result@t=\t@n\v@leur}#1=\result@t}
\ctr@ld@f\def\sqrt@#1#2{{\v@leur=#2%
    \minim@m{\v@lmin}{\p@}{\v@leur}\maxim@m{\v@lmax}{\p@}{\v@leur}%
    \f@ctech=\@ne\m@sqrt@\sqrt@@%
    \mili@u=\v@lmin\advance\mili@u\v@lmax\divide\mili@u\tw@\multiply\mili@u\f@ctech%
    \global\result@t=\mili@u}#1=\result@t}
\ctr@ld@f\def\m@sqrt@{\ifdim\v@leur>\dcq@\divide\v@leur\c@nt\v@lmax=\v@leur%
    \multiply\f@ctech\t@n\m@sqrt@\fi}
\ctr@ld@f\def\sqrt@@{\mili@u=\v@lmin\advance\mili@u\v@lmax\divide\mili@u\tw@%
    \c@rre=\repdecn@mb{\mili@u}\mili@u%
    \ifdim\c@rre<\v@leur\v@lmin=\mili@u\else\v@lmax=\mili@u\fi%
    \delt@=\v@lmax\advance\delt@-\v@lmin\ifdim\delt@>\epsil@n\sqrt@@\fi}
\ctr@ld@f\def\extrairelepremi@r#1\de#2{\expandafter\lepremi@r#2@#1#2}
\ctr@ld@f\def\lepremi@r#1,#2@#3#4{\def#3{#1}\def#4{#2}\ignorespaces}
\ctr@ld@f\def\@cfor#1:=#2\do#3{%
  \edef\@fortemp{#2}%
  \ifx\@fortemp\empty\else\@cforloop#2,\@nil,\@nil\@@#1{#3}\fi}
\ctr@ln@m\@nextwhile
\ctr@ld@f\def\@cforloop#1,#2\@@#3#4{%
  \def#3{#1}%
  \ifx#3\Fig@nnil\let\@nextwhile=\Fig@fornoop\else#4\relax\let\@nextwhile=\@cforloop\fi%
  \@nextwhile#2\@@#3{#4}}

\ctr@ld@f\def\@ecfor#1:=#2\do#3{%
  \def\@@cfor{\@cfor#1:=}%
  \edef\@@@cfor{#2}%
  \expandafter\@@cfor\@@@cfor\do{#3}}
\ctr@ld@f\def\Fig@nnil{\@nil}
\ctr@ld@f\def\Fig@fornoop#1\@@#2#3{}
\ctr@ln@m\list@@rg
\ctr@ld@f\def\trtlis@rg#1#2{\def\list@@rg{#1}%
    \@ecfor\p@rv@l:=\list@@rg\do{\expandafter#2\p@rv@l|}}
\ctr@ln@w{newbox}\b@xvisu
\ctr@ln@w{newtoks}\let@xte
\ctr@ln@w{newif}\ifitis@K
\ctr@ln@w{newcount}\s@mme
\ctr@ln@w{newcount}\l@mbd@un \ctr@ln@w{newcount}\l@mbd@de
\ctr@ln@w{newcount}\superc@ntr@l\superc@ntr@l=\@ne        % Controle impose
\ctr@ln@w{newcount}\typec@ntr@l\typec@ntr@l=\superc@ntr@l % Controle souhaite
\ctr@ln@w{newdimen}\v@lX  \ctr@ln@w{newdimen}\v@lY  \ctr@ln@w{newdimen}\v@lZ
\ctr@ln@w{newdimen}\v@lXa \ctr@ln@w{newdimen}\v@lYa \ctr@ln@w{newdimen}\v@lZa
\ctr@ln@w{newdimen}\unit@\unit@=\p@ % Initialisation a la valeur par defaut.
\ctr@ld@f\def\unit@util{pt}
\ctr@ld@f\def\ptT@ptps{0.996264 }
\ctr@ld@f\def\ptpsT@pt{1.00375 }
\ctr@ld@f\def\ptT@unit@{1} % Initialisation correspondant a la valeur par defaut de \unit@
\ctr@ld@f\def\setunit@#1{\def\unit@util{#1}\setunit@@#1:\invers@{\result@t}{\unit@}%
    \edef\ptT@unit@{\repdecn@mb\result@t}}
\ctr@ld@f\def\setunit@@#1#2:{\ifcat#1a\unit@=\@ne#1#2\else\unit@=#1#2\fi}
\ctr@ld@f\def\d@fm@cdim#1#2{{\v@leur=#2\v@leur=\ptT@unit@\v@leur\xdef#1{\repdecn@mb\v@leur}}}
\ctr@ln@w{newif}\ifBdingB@x\BdingB@xtrue
\ctr@ln@w{newdimen}\c@@rdXmin \ctr@ln@w{newdimen}\c@@rdYmin  % Dimensions de la BoundingBox
\ctr@ln@w{newdimen}\c@@rdXmax \ctr@ln@w{newdimen}\c@@rdYmax
\ctr@ld@f\def\b@undb@x#1#2{\ifBdingB@x%
    \relax\ifdim#1<\c@@rdXmin\global\c@@rdXmin=#1\fi%
    \relax\ifdim#2<\c@@rdYmin\global\c@@rdYmin=#2\fi%
    \relax\ifdim#1>\c@@rdXmax\global\c@@rdXmax=#1\fi%
    \relax\ifdim#2>\c@@rdYmax\global\c@@rdYmax=#2\fi\fi}
\ctr@ld@f\def\b@undb@xP#1{{\Figg@tXY{#1}\b@undb@x{\v@lX}{\v@lY}}}
\ctr@ld@f\def\ellBB@x#1;#2,#3(#4,#5,#6){{\s@uvc@ntr@l\et@tellBB@x%
    \setc@ntr@l{2}\figptell-2::#1;#2,#3(#4,#6)\b@undb@xP{-2}%
    \figptell-2::#1;#2,#3(#5,#6)\b@undb@xP{-2}%
    \c@ssin{\C@}{\S@}{#6}\v@lmin=\C@ pt\v@lmax=\S@ pt%
    \mili@u=#3\v@lmin\delt@=#2\v@lmax\arct@n\v@leur(\delt@,\mili@u)%
    \mili@u=-#3\v@lmax\delt@=#2\v@lmin\arct@n\c@rre(\delt@,\mili@u)%
    \v@leur=\rdT@deg\v@leur\advance\v@leur-\DePI@deg%
    \c@rre=\rdT@deg\c@rre\advance\c@rre-\DePI@deg%
    \v@lmin=#4pt\v@lmax=#5pt%
    \loop\ifdim\v@leur<\v@lmax\ifdim\v@leur>\v@lmin%
    \edef\@ngle{\repdecn@mb\v@leur}\figptell-2::#1;#2,#3(\@ngle,#6)%
    \b@undb@xP{-2}\fi\advance\v@leur\PI@deg\repeat%
    \loop\ifdim\c@rre<\v@lmax\ifdim\c@rre>\v@lmin%
    \edef\@ngle{\repdecn@mb\c@rre}\figptell-2::#1;#2,#3(\@ngle,#6)%
    \b@undb@xP{-2}\fi\advance\c@rre\PI@deg\repeat%
    \resetc@ntr@l\et@tellBB@x}\ignorespaces}
\ctr@ld@f\def\initb@undb@x{\c@@rdXmin=\maxdimen\c@@rdYmin=\maxdimen%
    \c@@rdXmax=-\maxdimen\c@@rdYmax=-\maxdimen}
\ctr@ld@f\def\c@ntr@lnum#1{%
    \relax\ifnum\typec@ntr@l=\@ne%
    \ifnum#1<\z@%
    \immediate\write16{*** Forbidden point number (#1). Abort.}\end\fi\fi%
    \set@bjc@de{#1}}
\ctr@ln@m\objc@de
\ctr@ld@f\def\set@bjc@de#1{\edef\objc@de{@BJ\ifnum#1<\z@ M\romannumeral-#1\else\romannumeral#1\fi}}
\s@mme=\m@ne\loop\ifnum\s@mme>-19
  \set@bjc@de{\s@mme}\ctr@lcsn@m\objc@de\ctr@lcsn@m{\objc@de T}
\advance\s@mme\m@ne\repeat
\s@mme=\@ne\loop\ifnum\s@mme<6
  \set@bjc@de{\s@mme}\ctr@lcsn@m\objc@de\ctr@lcsn@m{\objc@de T}
\advance\s@mme\@ne\repeat
\ctr@ld@f\def\setc@ntr@l#1{\ifnum\superc@ntr@l>#1\typec@ntr@l=\superc@ntr@l%
    \else\typec@ntr@l=#1\fi}
\ctr@ld@f\def\resetc@ntr@l#1{\global\superc@ntr@l=#1\setc@ntr@l{#1}}
\ctr@ld@f\def\s@uvc@ntr@l#1{\edef#1{\the\superc@ntr@l}}
\ctr@ln@m\c@lproscal
\ctr@ld@f\def\c@lproscalDD#1[#2,#3]{{\Figg@tXY{#2}%
    \edef\Xu@{\repdecn@mb{\v@lX}}\edef\Yu@{\repdecn@mb{\v@lY}}\Figg@tXY{#3}%
    \global\result@t=\Xu@\v@lX\global\advance\result@t\Yu@\v@lY}#1=\result@t}
\ctr@ld@f\def\c@lproscalTD#1[#2,#3]{{\Figg@tXY{#2}\edef\Xu@{\repdecn@mb{\v@lX}}%
    \edef\Yu@{\repdecn@mb{\v@lY}}\edef\Zu@{\repdecn@mb{\v@lZ}}%
    \Figg@tXY{#3}\global\result@t=\Xu@\v@lX\global\advance\result@t\Yu@\v@lY%
    \global\advance\result@t\Zu@\v@lZ}#1=\result@t}
\ctr@ld@f\def\c@lprovec#1{%
    \det@rmC\v@lZa(\v@lX,\v@lY,\v@lmin,\v@lmax)%
    \det@rmC\v@lXa(\v@lY,\v@lZ,\v@lmax,\v@leur)%
    \det@rmC\v@lYa(\v@lZ,\v@lX,\v@leur,\v@lmin)%
    \Figv@ctCreg#1(\v@lXa,\v@lYa,\v@lZa)}
\ctr@ld@f\def\det@rm#1[#2,#3]{{\Figg@tXY{#2}\Figg@tXYa{#3}%
    \delt@=\repdecn@mb{\v@lX}\v@lYa\advance\delt@-\repdecn@mb{\v@lY}\v@lXa%
    \global\result@t=\delt@}#1=\result@t}
\ctr@ld@f\def\det@rmC#1(#2,#3,#4,#5){{\global\result@t=\repdecn@mb{#2}#5%
    \global\advance\result@t-\repdecn@mb{#3}#4}#1=\result@t}
\ctr@ld@f\def\getredf@ctDD#1(#2,#3){{\maxim@m{\v@lXa}{-#2}{#2}\maxim@m{\v@lYa}{-#3}{#3}%
    \maxim@m{\v@lXa}{\v@lXa}{\v@lYa}% \v@lXa = ||X||inf
    \ifdim\v@lXa>\@xci pt\divide\v@lXa\@xci%
    \p@rtentiere{\p@rtent}{\v@lXa}\advance\p@rtent\@ne\else\p@rtent=\@ne\fi%
    \global\result@tent=\p@rtent}#1=\result@tent\ignorespaces}
\ctr@ld@f\def\getredf@ctTD#1(#2,#3,#4){{\maxim@m{\v@lXa}{-#2}{#2}\maxim@m{\v@lYa}{-#3}{#3}%
    \maxim@m{\v@lZa}{-#4}{#4}\maxim@m{\v@lXa}{\v@lXa}{\v@lYa}%
    \maxim@m{\v@lXa}{\v@lXa}{\v@lZa}% \v@lXa = ||X||inf
    \ifdim\v@lXa>\@lxxiv pt\divide\v@lXa\@lxxiv%
    \p@rtentiere{\p@rtent}{\v@lXa}\advance\p@rtent\@ne\else\p@rtent=\@ne\fi%
    \global\result@tent=\p@rtent}#1=\result@tent\ignorespaces}
\ctr@ld@f\def\FigptintercircB@zDD#1:#2:#3,#4[#5,#6,#7,#8]{{\s@uvc@ntr@l\et@tfigptintercircB@zDD%
    \setc@ntr@l{2}\figvectPDD-1[#5,#8]\Figg@tXY{-1}\getredf@ctDD\f@ctech(\v@lX,\v@lY)%
    \mili@u=#4\unit@\divide\mili@u\f@ctech\c@rre=\repdecn@mb{\mili@u}\mili@u%
    \figptBezierDD-5::#3[#5,#6,#7,#8]%
    \v@lmin=#3\p@\v@lmax=\v@lmin\advance\v@lmax0.1\p@%
    \loop\edef\T@{\repdecn@mb{\v@lmax}}\figptBezierDD-2::\T@[#5,#6,#7,#8]%
    \figvectPDD-1[-5,-2]\n@rmeucCDD{\delt@}{-1}\ifdim\delt@<\c@rre\v@lmin=\v@lmax%
    \advance\v@lmax0.1\p@\repeat%
    \loop\mili@u=\v@lmin\advance\mili@u\v@lmax%
    \divide\mili@u\tw@\edef\T@{\repdecn@mb{\mili@u}}\figptBezierDD-2::\T@[#5,#6,#7,#8]%
    \figvectPDD-1[-5,-2]\n@rmeucCDD{\delt@}{-1}\ifdim\delt@>\c@rre\v@lmax=\mili@u%
    \else\v@lmin=\mili@u\fi\v@leur=\v@lmax\advance\v@leur-\v@lmin%
    \ifdim\v@leur>\epsil@n\repeat\figptcopyDD#1:#2/-2/%
    \resetc@ntr@l\et@tfigptintercircB@zDD}\ignorespaces}
\ctr@ln@m\figptinterlines
\ctr@ld@f\def\inters@cDD#1:#2[#3,#4;#5,#6]{{\s@uvc@ntr@l\et@tinters@cDD%
    \setc@ntr@l{2}\vecunit@{-1}{#4}\vecunit@{-2}{#6}%
    \Figg@tXY{-1}\setc@ntr@l{1}\Figg@tXYa{#3}%
    \edef\A@{\repdecn@mb{\v@lX}}\edef\B@{\repdecn@mb{\v@lY}}%
    \v@lmin=\B@\v@lXa\advance\v@lmin-\A@\v@lYa%
    \Figg@tXYa{#5}\setc@ntr@l{2}\Figg@tXY{-2}%
    \edef\C@{\repdecn@mb{\v@lX}}\edef\D@{\repdecn@mb{\v@lY}}%
    \v@lmax=\D@\v@lXa\advance\v@lmax-\C@\v@lYa%
    \delt@=\A@\v@lY\advance\delt@-\B@\v@lX%
    \invers@{\v@leur}{\delt@}\edef\v@ldelta{\repdecn@mb{\v@leur}}%
    \v@lXa=\A@\v@lmax\advance\v@lXa-\C@\v@lmin%
    \v@lYa=\B@\v@lmax\advance\v@lYa-\D@\v@lmin%
    \v@lXa=\v@ldelta\v@lXa\v@lYa=\v@ldelta\v@lYa%
    \setc@ntr@l{1}\Figp@intregDD#1:{#2}(\v@lXa,\v@lYa)%
    \resetc@ntr@l\et@tinters@cDD}\ignorespaces}
\ctr@ld@f\def\inters@cTD#1:#2[#3,#4;#5,#6]{{\s@uvc@ntr@l\et@tinters@cTD%
    \setc@ntr@l{2}\figvectNVTD-1[#4,#6]\figvectNVTD-2[#6,-1]\figvectPTD-1[#3,#5]%
    \r@pPSTD\v@leur[-2,-1,#4]\edef\v@lcoef{\repdecn@mb{\v@leur}}%
    \figpttraTD#1:{#2}=#3/\v@lcoef,#4/\resetc@ntr@l\et@tinters@cTD}\ignorespaces}
\ctr@ld@f\def\r@pPSTD#1[#2,#3,#4]{{\Figg@tXY{#2}\edef\Xu@{\repdecn@mb{\v@lX}}%
    \edef\Yu@{\repdecn@mb{\v@lY}}\edef\Zu@{\repdecn@mb{\v@lZ}}%
    \Figg@tXY{#3}\v@lmin=\Xu@\v@lX\advance\v@lmin\Yu@\v@lY\advance\v@lmin\Zu@\v@lZ%
    \Figg@tXY{#4}\v@lmax=\Xu@\v@lX\advance\v@lmax\Yu@\v@lY\advance\v@lmax\Zu@\v@lZ%
    \invers@{\v@leur}{\v@lmax}\global\result@t=\repdecn@mb{\v@leur}\v@lmin}%
    #1=\result@t}
\ctr@ln@m\n@rminf
\ctr@ld@f\def\n@rminfDD#1#2{{\Figg@tXY{#2}\maxim@m{\v@lX}{\v@lX}{-\v@lX}%
    \maxim@m{\v@lY}{\v@lY}{-\v@lY}\maxim@m{\global\result@t}{\v@lX}{\v@lY}}%
    #1=\result@t}
\ctr@ld@f\def\n@rminfTD#1#2{{\Figg@tXY{#2}\maxim@m{\v@lX}{\v@lX}{-\v@lX}%
    \maxim@m{\v@lY}{\v@lY}{-\v@lY}\maxim@m{\v@lZ}{\v@lZ}{-\v@lZ}%
    \maxim@m{\v@lX}{\v@lX}{\v@lY}\maxim@m{\global\result@t}{\v@lX}{\v@lZ}}%
    #1=\result@t}
\ctr@ld@f\def\n@rmeucCDD#1#2{\Figg@tXY{#2}\divide\v@lX\f@ctech\divide\v@lY\f@ctech%
    #1=\repdecn@mb{\v@lX}\v@lX\v@lX=\repdecn@mb{\v@lY}\v@lY\advance#1\v@lX}
\ctr@ld@f\def\n@rmeucCTD#1#2{\Figg@tXY{#2}%
    \divide\v@lX\f@ctech\divide\v@lY\f@ctech\divide\v@lZ\f@ctech%
    #1=\repdecn@mb{\v@lX}\v@lX\v@lX=\repdecn@mb{\v@lY}\v@lY\advance#1\v@lX%
    \v@lX=\repdecn@mb{\v@lZ}\v@lZ\advance#1\v@lX}
\ctr@ln@m\n@rmeucSV
\ctr@ld@f\def\n@rmeucSVDD#1#2{{\Figg@tXY{#2}%
    \v@lXa=\repdecn@mb{\v@lX}\v@lX\v@lYa=\repdecn@mb{\v@lY}\v@lY%
    \advance\v@lXa\v@lYa\sqrt@{\global\result@t}{\v@lXa}}#1=\result@t}
\ctr@ld@f\def\n@rmeucSVTD#1#2{{\Figg@tXY{#2}\v@lXa=\repdecn@mb{\v@lX}\v@lX%
    \v@lYa=\repdecn@mb{\v@lY}\v@lY\v@lZa=\repdecn@mb{\v@lZ}\v@lZ%
    \advance\v@lXa\v@lYa\advance\v@lXa\v@lZa\sqrt@{\global\result@t}{\v@lXa}}#1=\result@t}
\ctr@ln@m\n@rmeuc
\ctr@ld@f\def\n@rmeucDD#1#2{{\Figg@tXY{#2}\getredf@ctDD\f@ctech(\v@lX,\v@lY)%
    \divide\v@lX\f@ctech\divide\v@lY\f@ctech%
    \v@lXa=\repdecn@mb{\v@lX}\v@lX\v@lYa=\repdecn@mb{\v@lY}\v@lY%
    \advance\v@lXa\v@lYa\sqrt@{\global\result@t}{\v@lXa}%
    \global\multiply\result@t\f@ctech}#1=\result@t}
\ctr@ld@f\def\n@rmeucTD#1#2{{\Figg@tXY{#2}\getredf@ctTD\f@ctech(\v@lX,\v@lY,\v@lZ)%
    \divide\v@lX\f@ctech\divide\v@lY\f@ctech\divide\v@lZ\f@ctech%
    \v@lXa=\repdecn@mb{\v@lX}\v@lX%
    \v@lYa=\repdecn@mb{\v@lY}\v@lY\v@lZa=\repdecn@mb{\v@lZ}\v@lZ%
    \advance\v@lXa\v@lYa\advance\v@lXa\v@lZa\sqrt@{\global\result@t}{\v@lXa}%
    \global\multiply\result@t\f@ctech}#1=\result@t}
\ctr@ln@m\vecunit@
\ctr@ld@f\def\vecunit@DD#1#2{{\Figg@tXY{#2}\getredf@ctDD\f@ctech(\v@lX,\v@lY)%
    \divide\v@lX\f@ctech\divide\v@lY\f@ctech%
    \Figv@ctCreg#1(\v@lX,\v@lY)\n@rmeucSV{\v@lYa}{#1}%
    \invers@{\v@lXa}{\v@lYa}\edef\v@lv@lXa{\repdecn@mb{\v@lXa}}%
    \v@lX=\v@lv@lXa\v@lX\v@lY=\v@lv@lXa\v@lY%
    \Figv@ctCreg#1(\v@lX,\v@lY)\multiply\v@lYa\f@ctech\global\result@t=\v@lYa}}
\ctr@ld@f\def\vecunit@TD#1#2{{\Figg@tXY{#2}\getredf@ctTD\f@ctech(\v@lX,\v@lY,\v@lZ)%
    \divide\v@lX\f@ctech\divide\v@lY\f@ctech\divide\v@lZ\f@ctech%
    \Figv@ctCreg#1(\v@lX,\v@lY,\v@lZ)\n@rmeucSV{\v@lYa}{#1}%
    \invers@{\v@lXa}{\v@lYa}\edef\v@lv@lXa{\repdecn@mb{\v@lXa}}%
    \v@lX=\v@lv@lXa\v@lX\v@lY=\v@lv@lXa\v@lY\v@lZ=\v@lv@lXa\v@lZ%
    \Figv@ctCreg#1(\v@lX,\v@lY,\v@lZ)\multiply\v@lYa\f@ctech\global\result@t=\v@lYa}}
\ctr@ld@f\def\vecunitC@TD[#1,#2]{\Figg@tXYa{#1}\Figg@tXY{#2}%
    \advance\v@lX-\v@lXa\advance\v@lY-\v@lYa\advance\v@lZ-\v@lZa\c@lvecunitTD}
\ctr@ld@f\def\vecunitCV@TD#1{\Figg@tXY{#1}\c@lvecunitTD}
\ctr@ld@f\def\c@lvecunitTD{\getredf@ctTD\f@ctech(\v@lX,\v@lY,\v@lZ)%
    \divide\v@lX\f@ctech\divide\v@lY\f@ctech\divide\v@lZ\f@ctech%
    \v@lXa=\repdecn@mb{\v@lX}\v@lX%
    \v@lYa=\repdecn@mb{\v@lY}\v@lY\v@lZa=\repdecn@mb{\v@lZ}\v@lZ%
    \advance\v@lXa\v@lYa\advance\v@lXa\v@lZa\sqrt@{\v@lYa}{\v@lXa}%
    \invers@{\v@lXa}{\v@lYa}\edef\v@lv@lXa{\repdecn@mb{\v@lXa}}%
    \v@lX=\v@lv@lXa\v@lX\v@lY=\v@lv@lXa\v@lY\v@lZ=\v@lv@lXa\v@lZ}
\ctr@ln@m\figgetangle
\ctr@ld@f\def\figgetangleDD#1[#2,#3,#4]{\ifps@cri{\s@uvc@ntr@l\et@tfiggetangleDD\setc@ntr@l{2}%
    \figvectPDD-1[#2,#3]\figvectPDD-2[#2,#4]\vecunit@{-1}{-1}%
    \c@lproscalDD\delt@[-2,-1]\figvectNVDD-1[-1]\c@lproscalDD\v@leur[-2,-1]%
    \arct@n\v@lmax(\delt@,\v@leur)\v@lmax=\rdT@deg\v@lmax%
    \ifdim\v@lmax<\z@\advance\v@lmax\DePI@deg\fi\xdef#1{\repdecn@mb{\v@lmax}}%
    \resetc@ntr@l\et@tfiggetangleDD}\ignorespaces\fi}
\ctr@ld@f\def\figgetangleTD#1[#2,#3,#4,#5]{\ifps@cri{\s@uvc@ntr@l\et@tfiggetangleTD\setc@ntr@l{2}%
    \figvectPTD-1[#2,#3]\figvectPTD-2[#2,#5]\figvectNVTD-3[-1,-2]%
    \figvectPTD-2[#2,#4]\figvectNVTD-4[-3,-1]%
    \vecunit@{-1}{-1}\c@lproscalTD\delt@[-2,-1]\c@lproscalTD\v@leur[-2,-4]%
    \arct@n\v@lmax(\delt@,\v@leur)\v@lmax=\rdT@deg\v@lmax%
    \ifdim\v@lmax<\z@\advance\v@lmax\DePI@deg\fi\xdef#1{\repdecn@mb{\v@lmax}}%
    \resetc@ntr@l\et@tfiggetangleTD}\ignorespaces\fi}    
\ctr@ld@f\def\figgetdist#1[#2,#3]{\ifps@cri{\s@uvc@ntr@l\et@tfiggetdist\setc@ntr@l{2}%
    \figvectP-1[#2,#3]\n@rmeuc{\v@lX}{-1}\v@lX=\ptT@unit@\v@lX\xdef#1{\repdecn@mb{\v@lX}}%
    \resetc@ntr@l\et@tfiggetdist}\ignorespaces\fi}
\ctr@ld@f\def\Figg@tT#1{\c@ntr@lnum{#1}%
    {\expandafter\expandafter\expandafter\extr@ctT\csname\objc@de\endcsname:%
     \ifnum\B@@ltxt=\z@\ptn@me{#1}\else\csname\objc@de T\endcsname\fi}}
\ctr@ld@f\def\extr@ctT#1,#2,#3/#4:{\def\B@@ltxt{#3}}
\ctr@ld@f\def\Figg@tXY#1{\c@ntr@lnum{#1}%
    \expandafter\expandafter\expandafter\extr@ctC\csname\objc@de\endcsname:}
\ctr@ln@m\extr@ctC
\ctr@ld@f\def\extr@ctCDD#1/#2,#3,#4:{\v@lX=#2\v@lY=#3}
\ctr@ld@f\def\extr@ctCTD#1/#2,#3,#4:{\v@lX=#2\v@lY=#3\v@lZ=#4}
\ctr@ld@f\def\Figg@tXYa#1{\c@ntr@lnum{#1}%
    \expandafter\expandafter\expandafter\extr@ctCa\csname\objc@de\endcsname:}
\ctr@ln@m\extr@ctCa
\ctr@ld@f\def\extr@ctCaDD#1/#2,#3,#4:{\v@lXa=#2\v@lYa=#3}
\ctr@ld@f\def\extr@ctCaTD#1/#2,#3,#4:{\v@lXa=#2\v@lYa=#3\v@lZa=#4}
\ctr@ln@m\t@xt@
\ctr@ld@f\def\figinit#1{\t@stc@tcodech@nge\initpr@lim\Figinit@#1,:\initpss@ttings\ignorespaces}
\ctr@ld@f\def\Figinit@#1,#2:{\setunit@{#1}\def\t@xt@{#2}\ifx\t@xt@\empty\else\Figinit@@#2:\fi}
\ctr@ld@f\def\Figinit@@#1#2:{\if#12 \else\Figs@tproj{#1}\initTD@\fi}
\ctr@ln@w{newif}\ifTr@isDim
\ctr@ld@f\def\UnD@fined{UNDEFINED}
\ctr@ld@f\def\ifundefined#1{\expandafter\ifx\csname#1\endcsname\relax}
\ctr@ln@m\@utoFN
\ctr@ln@m\@utoFInDone
\ctr@ln@m\disob@unit
\ctr@ld@f\def\initpr@lim{\initb@undb@x\figsetmark{}\figsetptname{$A_{##1}$}\def\Sc@leFact{1}%
    \initDD@\figsetroundcoord{yes}\ps@critrue\expandafter\setupd@te\defaultupdate:%
    \edef\disob@unit{\UnD@fined}\edef\t@rgetpt{\UnD@fined}\gdef\@utoFInDone{1}\gdef\@utoFN{0}}
\ctr@ld@f\def\initDD@{\Tr@isDimfalse%
    \ifPDFm@ke%
     \let\Ps@rcerc=\Ps@rcercBz%
     \let\Ps@rell=\Ps@rellBz%
    \fi
    \let\c@lDCUn=\c@lDCUnDD%
    \let\c@lDCDeux=\c@lDCDeuxDD%
    \let\c@ldefproj=\relax%
    \let\c@lproscal=\c@lproscalDD%
    \let\c@lprojSP=\relax%
    \let\extr@ctC=\extr@ctCDD%
    \let\extr@ctCa=\extr@ctCaDD%
    \let\extr@ctCF=\extr@ctCFDD%
    \let\Figp@intreg=\Figp@intregDD%
    \let\Figpts@xes=\Figpts@xesDD%
    \let\n@rmeucSV=\n@rmeucSVDD\let\n@rmeuc=\n@rmeucDD\let\n@rminf=\n@rminfDD%
    \let\pr@dMatV=\pr@dMatVDD%
    \let\ps@xes=\ps@xesDD%
    \let\vecunit@=\vecunit@DD%
    \let\figcoord=\figcoordDD%
    \let\figgetangle=\figgetangleDD%
    \let\figpt=\figptDD%
    \let\figptBezier=\figptBezierDD%
    \let\figptbary=\figptbaryDD%
    \let\figptcirc=\figptcircDD%
    \let\figptcircumcenter=\figptcircumcenterDD%
    \let\figptcopy=\figptcopyDD%
    \let\figptcurvcenter=\figptcurvcenterDD%
    \let\figptell=\figptellDD%
    \let\figptendnormal=\figptendnormalDD%
    \let\figptinterlineplane=\figptinterlineplaneDD%
    \let\figptinterlines=\inters@cDD%
    \let\figptorthocenter=\figptorthocenterDD%
    \let\figptorthoprojline=\figptorthoprojlineDD%
    \let\figptorthoprojplane=\figptorthoprojplaneDD%
    \let\figptrot=\figptrotDD%
    \let\figptscontrol=\figptscontrolDD%
    \let\figptsintercirc=\figptsintercircDD%
    \let\figptsinterlinell=\figptsinterlinellDD%
    \let\figptsorthoprojline=\figptsorthoprojlineDD%
    \let\figptorthoprojplane=\figptorthoprojplaneDD%
    \let\figptsrot=\figptsrotDD%
    \let\figptssym=\figptssymDD%
    \let\figptstra=\figptstraDD%
    \let\figptsym=\figptsymDD%
    \let\figpttraC=\figpttraCDD%
    \let\figpttra=\figpttraDD%
    \let\figptvisilimSL=\figptvisilimSLDD%
    \let\figsetobdist=\figsetobdistDD%
    \let\figsettarget=\figsettargetDD%
    \let\figsetview=\figsetviewDD%
    \let\figvectDBezier=\figvectDBezierDD%
    \let\figvectN=\figvectNDD%
    \let\figvectNV=\figvectNVDD%
    \let\figvectP=\figvectPDD%
    \let\figvectU=\figvectUDD%
    \let\psarccircP=\psarccircPDD%
    \let\psarccirc=\psarccircDD%
    \let\psarcell=\psarcellDD%
    \let\psarcellPA=\psarcellPADD%
    \let\psarrowBezier=\psarrowBezierDD%
    \let\psarrowcircP=\psarrowcircPDD%
    \let\psarrowcirc=\psarrowcircDD%
    \let\psarrowhead=\psarrowheadDD%
    \let\psarrow=\psarrowDD%
    \let\psBezier=\psBezierDD%
    \let\pscirc=\pscircDD%
    \let\pscurve=\pscurveDD%
    \let\psnormal=\psnormalDD%
    }
\ctr@ld@f\def\initTD@{\Tr@isDimtrue\initb@undb@xTD\newt@rgetptfalse\newdis@bfalse%
    \let\c@lDCUn=\c@lDCUnTD%
    \let\c@lDCDeux=\c@lDCDeuxTD%
    \let\c@ldefproj=\c@ldefprojTD%
    \let\c@lproscal=\c@lproscalTD%
    \let\extr@ctC=\extr@ctCTD%
    \let\extr@ctCa=\extr@ctCaTD%
    \let\extr@ctCF=\extr@ctCFTD%
    \let\Figp@intreg=\Figp@intregTD%
    \let\Figpts@xes=\Figpts@xesTD%
    \let\n@rmeucSV=\n@rmeucSVTD\let\n@rmeuc=\n@rmeucTD\let\n@rminf=\n@rminfTD%
    \let\pr@dMatV=\pr@dMatVTD%
    \let\ps@xes=\ps@xesTD%
    \let\vecunit@=\vecunit@TD%
    \let\figcoord=\figcoordTD%
    \let\figgetangle=\figgetangleTD%
    \let\figpt=\figptTD%
    \let\figptBezier=\figptBezierTD%
    \let\figptbary=\figptbaryTD%
    \let\figptcirc=\figptcircTD%
    \let\figptcircumcenter=\figptcircumcenterTD%
    \let\figptcopy=\figptcopyTD%
    \let\figptcurvcenter=\figptcurvcenterTD%
    \let\figptinterlineplane=\figptinterlineplaneTD%
    \let\figptinterlines=\inters@cTD%
    \let\figptorthocenter=\figptorthocenterTD%
    \let\figptorthoprojline=\figptorthoprojlineTD%
    \let\figptorthoprojplane=\figptorthoprojplaneTD%
    \let\figptrot=\figptrotTD%
    \let\figptscontrol=\figptscontrolTD%
    \let\figptsintercirc=\figptsintercircTD%
    \let\figptsorthoprojline=\figptsorthoprojlineTD%
    \let\figptsorthoprojplane=\figptsorthoprojplaneTD%
    \let\figptsrot=\figptsrotTD%
    \let\figptssym=\figptssymTD%
    \let\figptstra=\figptstraTD%
    \let\figptsym=\figptsymTD%
    \let\figpttraC=\figpttraCTD%
    \let\figpttra=\figpttraTD%
    \let\figptvisilimSL=\figptvisilimSLTD%
    \let\figsetobdist=\figsetobdistTD%
    \let\figsettarget=\figsettargetTD%
    \let\figsetview=\figsetviewTD%
    \let\figvectDBezier=\figvectDBezierTD%
    \let\figvectN=\figvectNTD%
    \let\figvectNV=\figvectNVTD%
    \let\figvectP=\figvectPTD%
    \let\figvectU=\figvectUTD%
    \let\psarccircP=\psarccircPTD%
    \let\psarccirc=\psarccircTD%
    \let\psarcell=\psarcellTD%
    \let\psarcellPA=\psarcellPATD%
    \let\psarrowBezier=\psarrowBezierTD%
    \let\psarrowcircP=\psarrowcircPTD%
    \let\psarrowcirc=\psarrowcircTD%
    \let\psarrowhead=\psarrowheadTD%
    \let\psarrow=\psarrowTD%
    \let\psBezier=\psBezierTD%
    \let\pscirc=\pscircTD%
    \let\pscurve=\pscurveTD%
    }
\ctr@ld@f\def\un@v@ilable#1{\immediate\write16{*** The macro #1 is not available in the current context.}}
\ctr@ld@f\def\figinsert#1{{\def\t@xt@{#1}\relax%
    \ifx\t@xt@\empty\ifnum\@utoFInDone>\z@\Figinsert@\DefGIfilen@me,:\fi%
    \else\expandafter\FiginsertNu@#1 :\fi}\ignorespaces}
\ctr@ld@f\def\FiginsertNu@#1 #2:{\def\t@xt@{#1}\relax\ifx\t@xt@\empty\def\t@xt@{#2}%
    \ifx\t@xt@\empty\ifnum\@utoFInDone>\z@\Figinsert@\DefGIfilen@me,:\fi%
    \else\FiginsertNu@#2:\fi\else\expandafter\FiginsertNd@#1 #2:\fi}
\ctr@ld@f\def\FiginsertNd@#1#2:{\ifcat#1a\Figinsert@#1#2,:\else%
    \ifnum\@utoFInDone>\z@\Figinsert@\DefGIfilen@me,#1#2,:\fi\fi}
\ctr@ln@m\Sc@leFact
\ctr@ld@f\def\Figinsert@#1,#2:{\def\t@xt@{#2}\ifx\t@xt@\empty\xdef\Sc@leFact{1}\else%
    \X@rgdeux@#2\xdef\Sc@leFact{\@rgdeux}\fi%
    \Figdisc@rdLTS{#1}{\t@xt@}\@psfgetbb{\t@xt@}%
    \v@lX=\@psfllx\p@\v@lX=\ptpsT@pt\v@lX\v@lX=\Sc@leFact\v@lX%
    \v@lY=\@psflly\p@\v@lY=\ptpsT@pt\v@lY\v@lY=\Sc@leFact\v@lY%
    \b@undb@x{\v@lX}{\v@lY}%
    \v@lX=\@psfurx\p@\v@lX=\ptpsT@pt\v@lX\v@lX=\Sc@leFact\v@lX%
    \v@lY=\@psfury\p@\v@lY=\ptpsT@pt\v@lY\v@lY=\Sc@leFact\v@lY%
    \b@undb@x{\v@lX}{\v@lY}%
    \ifPDFm@ke\Figinclud@PDF{\t@xt@}{\Sc@leFact}\else%
    \v@lX=\c@nt pt\v@lX=\Sc@leFact\v@lX\edef\F@ct{\repdecn@mb{\v@lX}}%
    \ifx\TeXturesonMacOSltX\special{postscriptfile #1 vscale=\F@ct\space hscale=\F@ct}%
    \else\includegraphics{#1}\fi\fi%
    \message{[\t@xt@]}\ignorespaces}
\ctr@ld@f\def\Figdisc@rdLTS#1#2{\expandafter\Figdisc@rdLTS@#1 :#2}
\ctr@ld@f\def\Figdisc@rdLTS@#1 #2:#3{\def#3{#1}\relax\ifx#3\empty\expandafter\Figdisc@rdLTS@#2:#3\fi}
\ctr@ld@f\def\figinsertE#1{\FiginsertE@#1,:\ignorespaces}
\ctr@ld@f\def\FiginsertE@#1,#2:{{\def\t@xt@{#2}\ifx\t@xt@\empty\xdef\Sc@leFact{1}\else%
    \X@rgdeux@#2\xdef\Sc@leFact{\@rgdeux}\fi%
    \Figdisc@rdLTS{#1}{\t@xt@}\pdfximage{\t@xt@}%
    \setbox\Gb@x=\hbox{\pdfrefximage\pdflastximage}%
    \v@lX=\z@\v@lY=-\Sc@leFact\dp\Gb@x\b@undb@x{\v@lX}{\v@lY}%
    \advance\v@lX\Sc@leFact\wd\Gb@x\advance\v@lY\Sc@leFact\dp\Gb@x%
    \advance\v@lY\Sc@leFact\ht\Gb@x\b@undb@x{\v@lX}{\v@lY}%
    \v@lX=\Sc@leFact\wd\Gb@x\pdfximage width \v@lX {\t@xt@}%
    \rlap{\pdfrefximage\pdflastximage}\message{[\t@xt@]}}\ignorespaces}
\ctr@ld@f\def\X@rgdeux@#1,{\edef\@rgdeux{#1}}
\ctr@ln@m\figpt
\ctr@ld@f\def\figptDD#1:#2(#3,#4){\ifps@cri\c@ntr@lnum{#1}%
    {\v@lX=#3\unit@\v@lY=#4\unit@\Fig@dmpt{#2}{\z@}}\ignorespaces\fi}
\ctr@ld@f\def\Fig@dmpt#1#2{\def\t@xt@{#1}\ifx\t@xt@\empty\def\B@@ltxt{\z@}%
    \else\expandafter\gdef\csname\objc@de T\endcsname{#1}\def\B@@ltxt{\@ne}\fi%
    \expandafter\xdef\csname\objc@de\endcsname{\ifitis@vect@r\C@dCl@svect%
    \else\C@dCl@spt\fi,\z@,\B@@ltxt/\the\v@lX,\the\v@lY,#2}}
\ctr@ld@f\def\C@dCl@spt{P}
\ctr@ld@f\def\C@dCl@svect{V}
\ctr@ln@m\c@@rdYZ
\ctr@ln@m\c@@rdY
\ctr@ld@f\def\figptTD#1:#2(#3,#4){\ifps@cri\c@ntr@lnum{#1}%
    \def\c@@rdYZ{#4,0,0}\extrairelepremi@r\c@@rdY\de\c@@rdYZ%
    \extrairelepremi@r\c@@rdZ\de\c@@rdYZ%
    {\v@lX=#3\unit@\v@lY=\c@@rdY\unit@\v@lZ=\c@@rdZ\unit@\Fig@dmpt{#2}{\the\v@lZ}%
    \b@undb@xTD{\v@lX}{\v@lY}{\v@lZ}}\ignorespaces\fi}
\ctr@ln@m\Figp@intreg
\ctr@ld@f\def\Figp@intregDD#1:#2(#3,#4){\c@ntr@lnum{#1}%
    {\result@t=#4\v@lX=#3\v@lY=\result@t\Fig@dmpt{#2}{\z@}}\ignorespaces}
\ctr@ld@f\def\Figp@intregTD#1:#2(#3,#4){\c@ntr@lnum{#1}%
    \def\c@@rdYZ{#4,\z@,\z@}\extrairelepremi@r\c@@rdY\de\c@@rdYZ%
    \extrairelepremi@r\c@@rdZ\de\c@@rdYZ%
    {\v@lX=#3\v@lY=\c@@rdY\v@lZ=\c@@rdZ\Fig@dmpt{#2}{\the\v@lZ}%
    \b@undb@xTD{\v@lX}{\v@lY}{\v@lZ}}\ignorespaces}
\ctr@ln@m\figptBezier
\ctr@ld@f\def\figptBezierDD#1:#2:#3[#4,#5,#6,#7]{\ifps@cri{\s@uvc@ntr@l\et@tfigptBezierDD%
    \FigptBezier@#3[#4,#5,#6,#7]\Figp@intregDD#1:{#2}(\v@lX,\v@lY)%
    \resetc@ntr@l\et@tfigptBezierDD}\ignorespaces\fi}
\ctr@ld@f\def\figptBezierTD#1:#2:#3[#4,#5,#6,#7]{\ifps@cri{\s@uvc@ntr@l\et@tfigptBezierTD%
    \FigptBezier@#3[#4,#5,#6,#7]\Figp@intregTD#1:{#2}(\v@lX,\v@lY,\v@lZ)%
    \resetc@ntr@l\et@tfigptBezierTD}\ignorespaces\fi}
\ctr@ld@f\def\FigptBezier@#1[#2,#3,#4,#5]{\setc@ntr@l{2}%
    \edef\T@{#1}\v@leur=\p@\advance\v@leur-#1pt\edef\UNmT@{\repdecn@mb{\v@leur}}%
    \figptcopy-4:/#2/\figptcopy-3:/#3/\figptcopy-2:/#4/\figptcopy-1:/#5/%
    \l@mbd@un=-4 \l@mbd@de=-\thr@@\p@rtent=\m@ne\c@lDecast%
    \l@mbd@un=-4 \l@mbd@de=-\thr@@\p@rtent=-\tw@\c@lDecast%
    \l@mbd@un=-4 \l@mbd@de=-\thr@@\p@rtent=-\thr@@\c@lDecast\Figg@tXY{-4}}
\ctr@ln@m\c@lDCUn
\ctr@ld@f\def\c@lDCUnDD#1#2{\Figg@tXY{#1}\v@lX=\UNmT@\v@lX\v@lY=\UNmT@\v@lY%
    \Figg@tXYa{#2}\advance\v@lX\T@\v@lXa\advance\v@lY\T@\v@lYa%
    \Figp@intregDD#1:(\v@lX,\v@lY)}
\ctr@ld@f\def\c@lDCUnTD#1#2{\Figg@tXY{#1}\v@lX=\UNmT@\v@lX\v@lY=\UNmT@\v@lY\v@lZ=\UNmT@\v@lZ%
    \Figg@tXYa{#2}\advance\v@lX\T@\v@lXa\advance\v@lY\T@\v@lYa\advance\v@lZ\T@\v@lZa%
    \Figp@intregTD#1:(\v@lX,\v@lY,\v@lZ)}
\ctr@ld@f\def\c@lDecast{\relax\ifnum\l@mbd@un<\p@rtent\c@lDCUn{\l@mbd@un}{\l@mbd@de}%
    \advance\l@mbd@un\@ne\advance\l@mbd@de\@ne\c@lDecast\fi}
\ctr@ld@f\def\figptmap#1:#2=#3/#4/#5/{\ifps@cri{\s@uvc@ntr@l\et@tfigptmap%
    \setc@ntr@l{2}\figvectP-1[#4,#3]\Figg@tXY{-1}%
    \pr@dMatV/#5/\figpttra#1:{#2}=#4/1,-1/%
    \resetc@ntr@l\et@tfigptmap}\ignorespaces\fi}
\ctr@ln@m\pr@dMatV
\ctr@ld@f\def\pr@dMatVDD/#1,#2;#3,#4/{\v@lXa=#1\v@lX\advance\v@lXa#2\v@lY%
    \v@lYa=#3\v@lX\advance\v@lYa#4\v@lY\Figv@ctCreg-1(\v@lXa,\v@lYa)}
\ctr@ld@f\def\pr@dMatVTD/#1,#2,#3;#4,#5,#6;#7,#8,#9/{%
    \v@lXa=#1\v@lX\advance\v@lXa#2\v@lY\advance\v@lXa#3\v@lZ%
    \v@lYa=#4\v@lX\advance\v@lYa#5\v@lY\advance\v@lYa#6\v@lZ%
    \v@lZa=#7\v@lX\advance\v@lZa#8\v@lY\advance\v@lZa#9\v@lZ%
    \Figv@ctCreg-1(\v@lXa,\v@lYa,\v@lZa)}
\ctr@ln@m\figptbary
\ctr@ld@f\def\figptbaryDD#1:#2[#3;#4]{\ifps@cri{\edef\list@num{#3}\extrairelepremi@r\p@int\de\list@num%
    \s@mme=\z@\@ecfor\c@ef:=#4\do{\advance\s@mme\c@ef}%
    \edef\listec@ef{#4,0}\extrairelepremi@r\c@ef\de\listec@ef%
    \Figg@tXY{\p@int}\divide\v@lX\s@mme\divide\v@lY\s@mme%
    \multiply\v@lX\c@ef\multiply\v@lY\c@ef%
    \@ecfor\p@int:=\list@num\do{\extrairelepremi@r\c@ef\de\listec@ef%
           \Figg@tXYa{\p@int}\divide\v@lXa\s@mme\divide\v@lYa\s@mme%
           \multiply\v@lXa\c@ef\multiply\v@lYa\c@ef%
           \advance\v@lX\v@lXa\advance\v@lY\v@lYa}%
    \Figp@intregDD#1:{#2}(\v@lX,\v@lY)}\ignorespaces\fi}
\ctr@ld@f\def\figptbaryTD#1:#2[#3;#4]{\ifps@cri{\edef\list@num{#3}\extrairelepremi@r\p@int\de\list@num%
    \s@mme=\z@\@ecfor\c@ef:=#4\do{\advance\s@mme\c@ef}%
    \edef\listec@ef{#4,0}\extrairelepremi@r\c@ef\de\listec@ef%
    \Figg@tXY{\p@int}\divide\v@lX\s@mme\divide\v@lY\s@mme\divide\v@lZ\s@mme%
    \multiply\v@lX\c@ef\multiply\v@lY\c@ef\multiply\v@lZ\c@ef%
    \@ecfor\p@int:=\list@num\do{\extrairelepremi@r\c@ef\de\listec@ef%
           \Figg@tXYa{\p@int}\divide\v@lXa\s@mme\divide\v@lYa\s@mme\divide\v@lZa\s@mme%
           \multiply\v@lXa\c@ef\multiply\v@lYa\c@ef\multiply\v@lZa\c@ef%
           \advance\v@lX\v@lXa\advance\v@lY\v@lYa\advance\v@lZ\v@lZa}%
    \Figp@intregTD#1:{#2}(\v@lX,\v@lY,\v@lZ)}\ignorespaces\fi}
\ctr@ld@f\def\figptbaryR#1:#2[#3;#4]{\ifps@cri{%
    \v@leur=\z@\@ecfor\c@ef:=#4\do{\maxim@m{\v@lmax}{\c@ef pt}{-\c@ef pt}%
    \ifdim\v@lmax>\v@leur\v@leur=\v@lmax\fi}%
    \ifdim\v@leur<\p@\f@ctech=\@M\else\ifdim\v@leur<\t@n\p@\f@ctech=\@m\else%
    \ifdim\v@leur<\c@nt\p@\f@ctech=\c@nt\else\ifdim\v@leur<\@m\p@\f@ctech=\t@n\else%
    \f@ctech=\@ne\fi\fi\fi\fi%
    \def\listec@ef{0}%
    \@ecfor\c@ef:=#4\do{\sc@lec@nvRI{\c@ef pt}\edef\listec@ef{\listec@ef,\the\s@mme}}%
    \extrairelepremi@r\c@ef\de\listec@ef\figptbary#1:#2[#3;\listec@ef]}\ignorespaces\fi}
\ctr@ld@f\def\sc@lec@nvRI#1{\v@leur=#1\p@rtentiere{\s@mme}{\v@leur}\advance\v@leur-\s@mme\p@%
    \multiply\v@leur\f@ctech\p@rtentiere{\p@rtent}{\v@leur}%
    \multiply\s@mme\f@ctech\advance\s@mme\p@rtent}
\ctr@ln@m\figptcirc
\ctr@ld@f\def\figptcircDD#1:#2:#3;#4(#5){\ifps@cri{\s@uvc@ntr@l\et@tfigptcircDD%
    \c@lptellDD#1:{#2}:#3;#4,#4(#5)\resetc@ntr@l\et@tfigptcircDD}\ignorespaces\fi}
\ctr@ld@f\def\figptcircTD#1:#2:#3,#4,#5;#6(#7){\ifps@cri{\s@uvc@ntr@l\et@tfigptcircTD%
    \setc@ntr@l{2}\c@lExtAxes#3,#4,#5(#6)\figptellP#1:{#2}:#3,-4,-5(#7)%
    \resetc@ntr@l\et@tfigptcircTD}\ignorespaces\fi}
\ctr@ln@m\figptcircumcenter
\ctr@ld@f\def\figptcircumcenterDD#1:#2[#3,#4,#5]{\ifps@cri{\s@uvc@ntr@l\et@tfigptcircumcenterDD%
    \setc@ntr@l{2}\figvectNDD-5[#3,#4]\figptbaryDD-3:[#3,#4;1,1]%
                  \figvectNDD-6[#4,#5]\figptbaryDD-4:[#4,#5;1,1]%
    \resetc@ntr@l{2}\inters@cDD#1:{#2}[-3,-5;-4,-6]%
    \resetc@ntr@l\et@tfigptcircumcenterDD}\ignorespaces\fi}
\ctr@ld@f\def\figptcircumcenterTD#1:#2[#3,#4,#5]{\ifps@cri{\s@uvc@ntr@l\et@tfigptcircumcenterTD%
    \setc@ntr@l{2}\figvectNTD-1[#3,#4,#5]%
    \figvectPTD-3[#3,#4]\figvectNVTD-5[-1,-3]\figptbaryTD-3:[#3,#4;1,1]%
    \figvectPTD-4[#4,#5]\figvectNVTD-6[-1,-4]\figptbaryTD-4:[#4,#5;1,1]%
    \resetc@ntr@l{2}\inters@cTD#1:{#2}[-3,-5;-4,-6]%
    \resetc@ntr@l\et@tfigptcircumcenterTD}\ignorespaces\fi}
\ctr@ln@m\figptcopy
\ctr@ld@f\def\figptcopyDD#1:#2/#3/{\ifps@cri{\Figg@tXY{#3}%
    \Figp@intregDD#1:{#2}(\v@lX,\v@lY)}\ignorespaces\fi}
\ctr@ld@f\def\figptcopyTD#1:#2/#3/{\ifps@cri{\Figg@tXY{#3}%
    \Figp@intregTD#1:{#2}(\v@lX,\v@lY,\v@lZ)}\ignorespaces\fi}
\ctr@ln@m\figptcurvcenter
\ctr@ld@f\def\figptcurvcenterDD#1:#2:#3[#4,#5,#6,#7]{\ifps@cri{\s@uvc@ntr@l\et@tfigptcurvcenterDD%
    \setc@ntr@l{2}\c@lcurvradDD#3[#4,#5,#6,#7]\edef\Sprim@{\repdecn@mb{\result@t}}%
    \figptBezierDD-1::#3[#4,#5,#6,#7]\figpttraDD#1:{#2}=-1/\Sprim@,-5/%
    \resetc@ntr@l\et@tfigptcurvcenterDD}\ignorespaces\fi}
\ctr@ld@f\def\figptcurvcenterTD#1:#2:#3[#4,#5,#6,#7]{\ifps@cri{\s@uvc@ntr@l\et@tfigptcurvcenterTD%
    \setc@ntr@l{2}\figvectDBezierTD -5:1,#3[#4,#5,#6,#7]%
    \figvectDBezierTD -6:2,#3[#4,#5,#6,#7]\vecunit@TD{-5}{-5}%
    \edef\Sprim@{\repdecn@mb{\result@t}}\figvectNVTD-1[-6,-5]%
    \figvectNVTD-5[-5,-1]\c@lproscalTD\v@leur[-6,-5]%
    \invers@{\v@leur}{\v@leur}\v@leur=\Sprim@\v@leur\v@leur=\Sprim@\v@leur%
    \figptBezierTD-1::#3[#4,#5,#6,#7]\edef\Sprim@{\repdecn@mb{\v@leur}}%
    \figpttraTD#1:{#2}=-1/\Sprim@,-5/\resetc@ntr@l\et@tfigptcurvcenterTD}\ignorespaces\fi}
\ctr@ld@f\def\c@lcurvradDD#1[#2,#3,#4,#5]{{\figvectDBezierDD -5:1,#1[#2,#3,#4,#5]%
    \figvectDBezierDD -6:2,#1[#2,#3,#4,#5]\vecunit@DD{-5}{-5}%
    \edef\Sprim@{\repdecn@mb{\result@t}}\figvectNVDD-5[-5]\c@lproscalDD\v@leur[-6,-5]%
    \invers@{\v@leur}{\v@leur}\v@leur=\Sprim@\v@leur\v@leur=\Sprim@\v@leur%
    \global\result@t=\v@leur}}
\ctr@ln@m\figptell
\ctr@ld@f\def\figptellDD#1:#2:#3;#4,#5(#6,#7){\ifps@cri{\s@uvc@ntr@l\et@tfigptell%
    \c@lptellDD#1::#3;#4,#5(#6)\figptrotDD#1:{#2}=#1/#3,#7/%
    \resetc@ntr@l\et@tfigptell}\ignorespaces\fi}
\ctr@ld@f\def\c@lptellDD#1:#2:#3;#4,#5(#6){\c@ssin{\C@}{\S@}{#6}\v@lmin=\C@ pt\v@lmax=\S@ pt%
    \v@lmin=#4\v@lmin\v@lmax=#5\v@lmax%
    \edef\Xc@mp{\repdecn@mb{\v@lmin}}\edef\Yc@mp{\repdecn@mb{\v@lmax}}%
    \setc@ntr@l{2}\figvectC-1(\Xc@mp,\Yc@mp)\figpttraDD#1:{#2}=#3/1,-1/}
\ctr@ld@f\def\figptellP#1:#2:#3,#4,#5(#6){\ifps@cri{\s@uvc@ntr@l\et@tfigptellP%
    \setc@ntr@l{2}\figvectP-1[#3,#4]\figvectP-2[#3,#5]%
    \v@leur=#6pt\c@lptellP{#3}{-1}{-2}\figptcopy#1:{#2}/-3/%
    \resetc@ntr@l\et@tfigptellP}\ignorespaces\fi}
\ctr@ln@m\@ngle
\ctr@ld@f\def\c@lptellP#1#2#3{\edef\@ngle{\repdecn@mb\v@leur}\c@ssin{\C@}{\S@}{\@ngle}%
    \figpttra-3:=#1/\C@,#2/\figpttra-3:=-3/\S@,#3/}
\ctr@ln@m\figptendnormal
\ctr@ld@f\def\figptendnormalDD#1:#2:#3,#4[#5,#6]{\ifps@cri{\s@uvc@ntr@l\et@tfigptendnormal%
    \Figg@tXYa{#5}\Figg@tXY{#6}%
    \advance\v@lX-\v@lXa\advance\v@lY-\v@lYa%
    \setc@ntr@l{2}\Figv@ctCreg-1(\v@lX,\v@lY)\vecunit@{-1}{-1}\Figg@tXY{-1}%
    \delt@=#3\unit@\maxim@m{\delt@}{\delt@}{-\delt@}\edef\l@ngueur{\repdecn@mb{\delt@}}%
    \v@lX=\l@ngueur\v@lX\v@lY=\l@ngueur\v@lY%
    \delt@=\p@\advance\delt@-#4pt\edef\l@ngueur{\repdecn@mb{\delt@}}%
    \figptbaryR-1:[#5,#6;#4,\l@ngueur]\Figg@tXYa{-1}%
    \advance\v@lXa\v@lY\advance\v@lYa-\v@lX%
    \setc@ntr@l{1}\Figp@intregDD#1:{#2}(\v@lXa,\v@lYa)\resetc@ntr@l\et@tfigptendnormal}%
    \ignorespaces\fi}
\ctr@ld@f\def\figptexcenter#1:#2[#3,#4,#5]{\ifps@cri{\let@xte={-}%
    \Figptexinsc@nter#1:#2[#3,#4,#5]}\ignorespaces\fi}
\ctr@ld@f\def\figptincenter#1:#2[#3,#4,#5]{\ifps@cri{\let@xte={}%
    \Figptexinsc@nter#1:#2[#3,#4,#5]}\ignorespaces\fi}
\ctr@ld@f% pour compatibilite avec anciennes versions
\ctr@ld@f\def\Figptexinsc@nter#1:#2[#3,#4,#5]{%
    \figgetdist\LA@[#4,#5]\figgetdist\LB@[#3,#5]\figgetdist\LC@[#3,#4]%
    \figptbaryR#1:{#2}[#3,#4,#5;\the\let@xte\LA@,\LB@,\LC@]}
\ctr@ln@m\figptinterlineplane
\ctr@ld@f\def\figptinterlineplaneDD{\un@v@ilable{figptinterlineplane}}
\ctr@ld@f\def\figptinterlineplaneTD#1:#2[#3,#4;#5,#6]{\ifps@cri{\s@uvc@ntr@l\et@tfigptinterlineplane%
    \setc@ntr@l{2}\figvectPTD-1[#3,#5]\vecunit@TD{-2}{#6}%
    \r@pPSTD\v@leur[-2,-1,#4]\edef\v@lcoef{\repdecn@mb{\v@leur}}%
    \figpttraTD#1:{#2}=#3/\v@lcoef,#4/\resetc@ntr@l\et@tfigptinterlineplane}\ignorespaces\fi}
\ctr@ln@m\figptorthocenter
\ctr@ld@f\def\figptorthocenterDD#1:#2[#3,#4,#5]{\ifps@cri{\s@uvc@ntr@l\et@tfigptorthocenterDD%
    \setc@ntr@l{2}\figvectNDD-3[#3,#4]\figvectNDD-4[#4,#5]%
    \resetc@ntr@l{2}\inters@cDD#1:{#2}[#5,-3;#3,-4]%
    \resetc@ntr@l\et@tfigptorthocenterDD}\ignorespaces\fi}
\ctr@ld@f\def\figptorthocenterTD#1:#2[#3,#4,#5]{\ifps@cri{\s@uvc@ntr@l\et@tfigptorthocenterTD%
    \setc@ntr@l{2}\figvectNTD-1[#3,#4,#5]%
    \figvectPTD-2[#3,#4]\figvectNVTD-3[-1,-2]%
    \figvectPTD-2[#4,#5]\figvectNVTD-4[-1,-2]%
    \resetc@ntr@l{2}\inters@cTD#1:{#2}[#5,-3;#3,-4]%
    \resetc@ntr@l\et@tfigptorthocenterTD}\ignorespaces\fi}
\ctr@ln@m\figptorthoprojline
\ctr@ld@f\def\figptorthoprojlineDD#1:#2=#3/#4,#5/{\ifps@cri{\s@uvc@ntr@l\et@tfigptorthoprojlineDD%
    \setc@ntr@l{2}\figvectPDD-3[#4,#5]\figvectNVDD-4[-3]\resetc@ntr@l{2}%
    \inters@cDD#1:{#2}[#3,-4;#4,-3]\resetc@ntr@l\et@tfigptorthoprojlineDD}\ignorespaces\fi}
\ctr@ld@f\def\figptorthoprojlineTD#1:#2=#3/#4,#5/{\ifps@cri{\s@uvc@ntr@l\et@tfigptorthoprojlineTD%
    \setc@ntr@l{2}\figvectPTD-1[#4,#3]\figvectPTD-2[#4,#5]\vecunit@TD{-2}{-2}%
    \c@lproscalTD\v@leur[-1,-2]\edef\v@lcoef{\repdecn@mb{\v@leur}}%
    \figpttraTD#1:{#2}=#4/\v@lcoef,-2/\resetc@ntr@l\et@tfigptorthoprojlineTD}\ignorespaces\fi}
\ctr@ln@m\figptorthoprojplane
\ctr@ld@f\def\figptorthoprojplaneDD{\un@v@ilable{figptorthoprojplane}}
\ctr@ld@f\def\figptorthoprojplaneTD#1:#2=#3/#4,#5/{\ifps@cri{\s@uvc@ntr@l\et@tfigptorthoprojplane%
    \setc@ntr@l{2}\figvectPTD-1[#3,#4]\vecunit@TD{-2}{#5}%
    \c@lproscalTD\v@leur[-1,-2]\edef\v@lcoef{\repdecn@mb{\v@leur}}%
    \figpttraTD#1:{#2}=#3/\v@lcoef,-2/\resetc@ntr@l\et@tfigptorthoprojplane}\ignorespaces\fi}
\ctr@ld@f\def\figpthom#1:#2=#3/#4,#5/{\ifps@cri{\s@uvc@ntr@l\et@tfigpthom%
    \setc@ntr@l{2}\figvectP-1[#4,#3]\figpttra#1:{#2}=#4/#5,-1/%
    \resetc@ntr@l\et@tfigpthom}\ignorespaces\fi}
\ctr@ln@m\figptrot
\ctr@ld@f\def\figptrotDD#1:#2=#3/#4,#5/{\ifps@cri{\s@uvc@ntr@l\et@tfigptrotDD%
    \c@ssin{\C@}{\S@}{#5}\setc@ntr@l{2}\figvectPDD-1[#4,#3]\Figg@tXY{-1}%
    \v@lXa=\C@\v@lX\advance\v@lXa-\S@\v@lY%
    \v@lYa=\S@\v@lX\advance\v@lYa\C@\v@lY%
    \Figv@ctCreg-1(\v@lXa,\v@lYa)\figpttraDD#1:{#2}=#4/1,-1/%
    \resetc@ntr@l\et@tfigptrotDD}\ignorespaces\fi}
\ctr@ld@f\def\figptrotTD#1:#2=#3/#4,#5,#6/{\ifps@cri{\s@uvc@ntr@l\et@tfigptrotTD%
    \c@ssin{\C@}{\S@}{#5}%
    \setc@ntr@l{2}\figptorthoprojplaneTD-3:=#4/#3,#6/\figvectPTD-2[-3,#3]%
    \n@rmeucTD\v@leur{-2}\ifdim\v@leur<\Cepsil@n\Figg@tXYa{#3}\else%
    \edef\v@lcoef{\repdecn@mb{\v@leur}}\figvectNVTD-1[#6,-2]%
    \Figg@tXYa{-1}\v@lXa=\v@lcoef\v@lXa\v@lYa=\v@lcoef\v@lYa\v@lZa=\v@lcoef\v@lZa%
    \v@lXa=\S@\v@lXa\v@lYa=\S@\v@lYa\v@lZa=\S@\v@lZa\Figg@tXY{-2}%
    \advance\v@lXa\C@\v@lX\advance\v@lYa\C@\v@lY\advance\v@lZa\C@\v@lZ%
    \Figg@tXY{-3}\advance\v@lXa\v@lX\advance\v@lYa\v@lY\advance\v@lZa\v@lZ\fi%
    \Figp@intregTD#1:{#2}(\v@lXa,\v@lYa,\v@lZa)\resetc@ntr@l\et@tfigptrotTD}\ignorespaces\fi}
\ctr@ln@m\figptsym
\ctr@ld@f\def\figptsymDD#1:#2=#3/#4,#5/{\ifps@cri{\s@uvc@ntr@l\et@tfigptsymDD%
    \resetc@ntr@l{2}\figptorthoprojlineDD-5:=#3/#4,#5/\figvectPDD-2[#3,-5]%
    \figpttraDD#1:{#2}=#3/2,-2/\resetc@ntr@l\et@tfigptsymDD}\ignorespaces\fi}
\ctr@ld@f\def\figptsymTD#1:#2=#3/#4,#5/{\ifps@cri{\s@uvc@ntr@l\et@tfigptsymTD%
    \resetc@ntr@l{2}\figptorthoprojplaneTD-3:=#3/#4,#5/\figvectPTD-2[#3,-3]%
    \figpttraTD#1:{#2}=#3/2,-2/\resetc@ntr@l\et@tfigptsymTD}\ignorespaces\fi}
\ctr@ln@m\figpttra
\ctr@ld@f\def\figpttraDD#1:#2=#3/#4,#5/{\ifps@cri{\Figg@tXYa{#5}\v@lXa=#4\v@lXa\v@lYa=#4\v@lYa%
    \Figg@tXY{#3}\advance\v@lX\v@lXa\advance\v@lY\v@lYa%
    \Figp@intregDD#1:{#2}(\v@lX,\v@lY)}\ignorespaces\fi}
\ctr@ld@f\def\figpttraTD#1:#2=#3/#4,#5/{\ifps@cri{\Figg@tXYa{#5}\v@lXa=#4\v@lXa\v@lYa=#4\v@lYa%
    \v@lZa=#4\v@lZa\Figg@tXY{#3}\advance\v@lX\v@lXa\advance\v@lY\v@lYa%
    \advance\v@lZ\v@lZa\Figp@intregTD#1:{#2}(\v@lX,\v@lY,\v@lZ)}\ignorespaces\fi}
\ctr@ln@m\figpttraC
\ctr@ld@f\def\figpttraCDD#1:#2=#3/#4,#5/{\ifps@cri{\v@lXa=#4\unit@\v@lYa=#5\unit@%
    \Figg@tXY{#3}\advance\v@lX\v@lXa\advance\v@lY\v@lYa%
    \Figp@intregDD#1:{#2}(\v@lX,\v@lY)}\ignorespaces\fi}
\ctr@ld@f\def\figpttraCTD#1:#2=#3/#4,#5,#6/{\ifps@cri{\v@lXa=#4\unit@\v@lYa=#5\unit@\v@lZa=#6\unit@%
    \Figg@tXY{#3}\advance\v@lX\v@lXa\advance\v@lY\v@lYa\advance\v@lZ\v@lZa%
    \Figp@intregTD#1:{#2}(\v@lX,\v@lY,\v@lZ)}\ignorespaces\fi}
\ctr@ld@f\def\figptsaxes#1:#2(#3){\ifps@cri{\an@lys@xes#3,:\ifx\t@xt@\empty%
    \ifTr@isDim\Figpts@xes#1:#2(0,#3,0,#3,0,#3)\else\Figpts@xes#1:#2(0,#3,0,#3)\fi%
    \else\Figpts@xes#1:#2(#3)\fi}\ignorespaces\fi}
\ctr@ln@m\Figpts@xes
\ctr@ld@f\def\Figpts@xesDD#1:#2(#3,#4,#5,#6){%
    \s@mme=#1\figpttraC\the\s@mme:$x$=#2/#4,0/%
    \advance\s@mme\@ne\figpttraC\the\s@mme:$y$=#2/0,#6/}
\ctr@ld@f\def\Figpts@xesTD#1:#2(#3,#4,#5,#6,#7,#8){%
    \s@mme=#1\figpttraC\the\s@mme:$x$=#2/#4,0,0/%
    \advance\s@mme\@ne\figpttraC\the\s@mme:$y$=#2/0,#6,0/%
    \advance\s@mme\@ne\figpttraC\the\s@mme:$z$=#2/0,0,#8/}
\ctr@ld@f\def\figptsmap#1=#2/#3/#4/{\ifps@cri{\s@uvc@ntr@l\et@tfigptsmap%
    \setc@ntr@l{2}\def\list@num{#2}\s@mme=#1%
    \@ecfor\p@int:=\list@num\do{\figvectP-1[#3,\p@int]\Figg@tXY{-1}%
    \pr@dMatV/#4/\figpttra\the\s@mme:=#3/1,-1/\advance\s@mme\@ne}%
    \resetc@ntr@l\et@tfigptsmap}\ignorespaces\fi}
\ctr@ln@m\figptscontrol
\ctr@ld@f\def\figptscontrolDD#1[#2,#3,#4,#5]{\ifps@cri{\s@uvc@ntr@l\et@tfigptscontrolDD\setc@ntr@l{2}%
    \v@lX=\z@\v@lY=\z@\Figtr@nptDD{-5}{#2}\Figtr@nptDD{2}{#5}%
    \divide\v@lX\@vi\divide\v@lY\@vi%
    \Figtr@nptDD{3}{#3}\Figtr@nptDD{-1.5}{#4}\Figp@intregDD-1:(\v@lX,\v@lY)%
    \v@lX=\z@\v@lY=\z@\Figtr@nptDD{2}{#2}\Figtr@nptDD{-5}{#5}%
    \divide\v@lX\@vi\divide\v@lY\@vi\Figtr@nptDD{-1.5}{#3}\Figtr@nptDD{3}{#4}%
    \s@mme=#1\advance\s@mme\@ne\Figp@intregDD\the\s@mme:(\v@lX,\v@lY)%
    \figptcopyDD#1:/-1/\resetc@ntr@l\et@tfigptscontrolDD}\ignorespaces\fi}
\ctr@ld@f\def\figptscontrolTD#1[#2,#3,#4,#5]{\ifps@cri{\s@uvc@ntr@l\et@tfigptscontrolTD\setc@ntr@l{2}%
    \v@lX=\z@\v@lY=\z@\v@lZ=\z@\Figtr@nptTD{-5}{#2}\Figtr@nptTD{2}{#5}%
    \divide\v@lX\@vi\divide\v@lY\@vi\divide\v@lZ\@vi%
    \Figtr@nptTD{3}{#3}\Figtr@nptTD{-1.5}{#4}\Figp@intregTD-1:(\v@lX,\v@lY,\v@lZ)%
    \v@lX=\z@\v@lY=\z@\v@lZ=\z@\Figtr@nptTD{2}{#2}\Figtr@nptTD{-5}{#5}%
    \divide\v@lX\@vi\divide\v@lY\@vi\divide\v@lZ\@vi\Figtr@nptTD{-1.5}{#3}\Figtr@nptTD{3}{#4}%
    \s@mme=#1\advance\s@mme\@ne\Figp@intregTD\the\s@mme:(\v@lX,\v@lY,\v@lZ)%
    \figptcopyTD#1:/-1/\resetc@ntr@l\et@tfigptscontrolTD}\ignorespaces\fi}
\ctr@ld@f\def\Figtr@nptDD#1#2{\Figg@tXYa{#2}\v@lXa=#1\v@lXa\v@lYa=#1\v@lYa%
    \advance\v@lX\v@lXa\advance\v@lY\v@lYa}
\ctr@ld@f\def\Figtr@nptTD#1#2{\Figg@tXYa{#2}\v@lXa=#1\v@lXa\v@lYa=#1\v@lYa\v@lZa=#1\v@lZa%
    \advance\v@lX\v@lXa\advance\v@lY\v@lYa\advance\v@lZ\v@lZa}
\ctr@ld@f\def\figptscontrolcurve#1,#2[#3]{\ifps@cri{\s@uvc@ntr@l\et@tfigptscontrolcurve%
    \def\list@num{#3}\extrairelepremi@r\Ak@\de\list@num%
    \extrairelepremi@r\Ai@\de\list@num\extrairelepremi@r\Aj@\de\list@num%
    \s@mme=#1\figptcopy\the\s@mme:/\Ai@/%
    \setc@ntr@l{2}\figvectP -1[\Ak@,\Aj@]%
    \@ecfor\Ak@:=\list@num\do{\advance\s@mme\@ne\figpttra\the\s@mme:=\Ai@/\curv@roundness,-1/%
       \figvectP -1[\Ai@,\Ak@]\advance\s@mme\@ne\figpttra\the\s@mme:=\Aj@/-\curv@roundness,-1/%
       \advance\s@mme\@ne\figptcopy\the\s@mme:/\Aj@/%
       \edef\Ai@{\Aj@}\edef\Aj@{\Ak@}}\advance\s@mme-#1\divide\s@mme\thr@@%
       \xdef#2{\the\s@mme}%
    \resetc@ntr@l\et@tfigptscontrolcurve}\ignorespaces\fi}
\ctr@ln@m\figptsintercirc
\ctr@ld@f\def\figptsintercircDD#1[#2,#3;#4,#5]{\ifps@cri{\s@uvc@ntr@l\et@tfigptsintercircDD%
    \setc@ntr@l{2}\let\c@lNVintc=\c@lNVintcDD\Figptsintercirc@#1[#2,#3;#4,#5]%    
    \resetc@ntr@l\et@tfigptsintercircDD}\ignorespaces\fi}
\ctr@ld@f\def\figptsintercircTD#1[#2,#3;#4,#5;#6]{\ifps@cri{\s@uvc@ntr@l\et@tfigptsintercircTD%
    \setc@ntr@l{2}\let\c@lNVintc=\c@lNVintcTD\vecunitC@TD[#2,#6]%
    \Figv@ctCreg-3(\v@lX,\v@lY,\v@lZ)\Figptsintercirc@#1[#2,#3;#4,#5]%
    \resetc@ntr@l\et@tfigptsintercircTD}\ignorespaces\fi}
\ctr@ld@f\def\Figptsintercirc@#1[#2,#3;#4,#5]{\figvectP-1[#2,#4]%
    \vecunit@{-1}{-1}\delt@=\result@t\f@ctech=\result@tent%
    \s@mme=#1\advance\s@mme\@ne\figptcopy#1:/#2/\figptcopy\the\s@mme:/#4/%
    \ifdim\delt@=\z@\else%
    \v@lmin=#3\unit@\v@lmax=#5\unit@\v@leur=\v@lmin\advance\v@leur\v@lmax%
    \ifdim\v@leur>\delt@%
    \v@leur=\v@lmin\advance\v@leur-\v@lmax\maxim@m{\v@leur}{\v@leur}{-\v@leur}%
    \ifdim\v@leur<\delt@%
    \divide\v@lmin\f@ctech\divide\v@lmax\f@ctech\divide\delt@\f@ctech%
    \v@lmin=\repdecn@mb{\v@lmin}\v@lmin\v@lmax=\repdecn@mb{\v@lmax}\v@lmax%
    \invers@{\v@leur}{\delt@}\advance\v@lmax-\v@lmin%
    \v@lmax=-\repdecn@mb{\v@leur}\v@lmax\advance\delt@\v@lmax\delt@=.5\delt@%
    \v@lmax=\delt@\multiply\v@lmax\f@ctech%
    \edef\t@ille{\repdecn@mb{\v@lmax}}\figpttra-2:=#2/\t@ille,-1/%
    \delt@=\repdecn@mb{\delt@}\delt@\advance\v@lmin-\delt@%
    \sqrt@{\v@leur}{\v@lmin}\multiply\v@leur\f@ctech\edef\t@ille{\repdecn@mb{\v@leur}}%
    \c@lNVintc\figpttra#1:=-2/-\t@ille,-1/\figpttra\the\s@mme:=-2/\t@ille,-1/\fi\fi\fi}
\ctr@ld@f\def\c@lNVintcDD{\Figg@tXY{-1}\Figv@ctCreg-1(-\v@lY,\v@lX)} % <=> \figvectNVDD-1[-1]
\ctr@ld@f\def\c@lNVintcTD{{\Figg@tXY{-3}\v@lmin=\v@lX\v@lmax=\v@lY\v@leur=\v@lZ%
    \Figg@tXY{-1}\c@lprovec{-3}\vecunit@{-3}{-3}% <=> \figvectNVTD-3[-1,-3]\vecunit@{-3}{-3}
    \Figg@tXY{-1}\v@lmin=\v@lX\v@lmax=\v@lY%
    \v@leur=\v@lZ\Figg@tXY{-3}\c@lprovec{-1}}} % <=> \figvectNVTD-1[-3,-1]
\ctr@ln@m\figptsinterlinell
\ctr@ld@f\def\figptsinterlinellDD#1[#2,#3,#4,#5;#6,#7]{\ifps@cri{\s@uvc@ntr@l\et@tfigptsinterlinellDD%
    \figptcopy#1:/#6/\s@mme=#1\advance\s@mme\@ne\figptcopy\the\s@mme:/#7/%
    \v@lmin=#3\unit@\v@lmax=#4\unit@% a, b
    \setc@ntr@l{2}\figptbaryDD-4:[#6,#7;1,1]\figptsrotDD-3=-4,#7/#2,-#5/% D et rotation
    \Figg@tXY{-3}\Figg@tXYa{#2}\advance\v@lX-\v@lXa\advance\v@lY-\v@lYa% alpha, beta
    \figvectP-1[-3,-2]\Figg@tXYa{-1}\figvectP-3[-4,#7]\Figptsint@rLE{#1}% u1, u2
    \resetc@ntr@l\et@tfigptsinterlinellDD}\ignorespaces\fi}
\ctr@ld@f\def\figptsinterlinellP#1[#2,#3,#4;#5,#6]{\ifps@cri{\s@uvc@ntr@l\et@tfigptsinterlinellP%
    \figptcopy#1:/#5/\s@mme=#1\advance\s@mme\@ne\figptcopy\the\s@mme:/#6/\setc@ntr@l{2}%
    \figvectP-1[#2,#3]\vecunit@{-1}{-1}\v@lmin=\result@t% a
    \figvectP-2[#2,#4]\vecunit@{-2}{-2}\v@lmax=\result@t% b
    \figptbary-4:[#5,#6;1,1]% D
    \figvectP-3[#2,-4]\c@lproscal\v@lX[-3,-1]\c@lproscal\v@lY[-3,-2]% alpha, beta
    \figvectP-3[-4,#6]\c@lproscal\v@lXa[-3,-1]\c@lproscal\v@lYa[-3,-2]% u1, u2
    \Figptsint@rLE{#1}\resetc@ntr@l\et@tfigptsinterlinellP}\ignorespaces\fi}
\ctr@ld@f\def\Figptsint@rLE#1{%
    \getredf@ctDD\f@ctech(\v@lmin,\v@lmax)%
    \getredf@ctDD\p@rtent(\v@lX,\v@lY)\ifnum\p@rtent>\f@ctech\f@ctech=\p@rtent\fi%
    \getredf@ctDD\p@rtent(\v@lXa,\v@lYa)\ifnum\p@rtent>\f@ctech\f@ctech=\p@rtent\fi%
    \divide\v@lmin\f@ctech\divide\v@lmax\f@ctech\divide\v@lX\f@ctech\divide\v@lY\f@ctech%
    \divide\v@lXa\f@ctech\divide\v@lYa\f@ctech%
    \c@rre=\repdecn@mb\v@lXa\v@lmax\mili@u=\repdecn@mb\v@lYa\v@lmin%
    \getredf@ctDD\f@ctech(\c@rre,\mili@u)%
    \c@rre=\repdecn@mb\v@lX\v@lmax\mili@u=\repdecn@mb\v@lY\v@lmin%
    \getredf@ctDD\p@rtent(\c@rre,\mili@u)\ifnum\p@rtent>\f@ctech\f@ctech=\p@rtent\fi%
    \divide\v@lmin\f@ctech\divide\v@lmax\f@ctech\divide\v@lX\f@ctech\divide\v@lY\f@ctech%
    \divide\v@lXa\f@ctech\divide\v@lYa\f@ctech%
    \v@lmin=\repdecn@mb{\v@lmin}\v@lmin\v@lmax=\repdecn@mb{\v@lmax}\v@lmax%
    \edef\G@xde{\repdecn@mb\v@lmin}\edef\P@xde{\repdecn@mb\v@lmax}%
    \c@rre=-\v@lmax\v@leur=\repdecn@mb\v@lY\v@lY\advance\c@rre\v@leur\c@rre=\G@xde\c@rre%
    \v@leur=\repdecn@mb\v@lX\v@lX\v@leur=\P@xde\v@leur\advance\c@rre\v@leur% C
    \v@lmin=\repdecn@mb\v@lYa\v@lmin\v@lmax=\repdecn@mb\v@lXa\v@lmax%
    \mili@u=\repdecn@mb\v@lX\v@lmax\advance\mili@u\repdecn@mb\v@lY\v@lmin% B
    \v@lmax=\repdecn@mb\v@lXa\v@lmax\advance\v@lmax\repdecn@mb\v@lYa\v@lmin% A
    \ifdim\v@lmax>\epsil@n%
    \maxim@m{\v@leur}{\c@rre}{-\c@rre}\maxim@m{\v@lmin}{\mili@u}{-\mili@u}%
    \maxim@m{\v@leur}{\v@leur}{\v@lmin}\maxim@m{\v@lmin}{\v@lmax}{-\v@lmax}%
    \maxim@m{\v@leur}{\v@leur}{\v@lmin}\p@rtentiere{\p@rtent}{\v@leur}\advance\p@rtent\@ne%
    \divide\c@rre\p@rtent\divide\mili@u\p@rtent\divide\v@lmax\p@rtent%
    \delt@=\repdecn@mb{\mili@u}\mili@u\v@leur=\repdecn@mb{\v@lmax}\c@rre%
    \advance\delt@-\v@leur\ifdim\delt@<\z@\else\sqrt@\delt@\delt@%
    \invers@\v@lmax\v@lmax\edef\Uns@rAp{\repdecn@mb\v@lmax}%
    \v@leur=-\mili@u\advance\v@leur-\delt@\v@leur=\Uns@rAp\v@leur%
    \edef\t@ille{\repdecn@mb\v@leur}\figpttra#1:=-4/\t@ille,-3/\s@mme=#1\advance\s@mme\@ne%
    \v@leur=-\mili@u\advance\v@leur\delt@\v@leur=\Uns@rAp\v@leur%
    \edef\t@ille{\repdecn@mb\v@leur}\figpttra\the\s@mme:=-4/\t@ille,-3/\fi\fi}
\ctr@ln@m\figptsorthoprojline
\ctr@ld@f\def\figptsorthoprojlineDD#1=#2/#3,#4/{\ifps@cri{\s@uvc@ntr@l\et@tfigptsorthoprojlineDD%
    \setc@ntr@l{2}\figvectPDD-3[#3,#4]\figvectNVDD-4[-3]\resetc@ntr@l{2}%
    \def\list@num{#2}\s@mme=#1\@ecfor\p@int:=\list@num\do{%
    \inters@cDD\the\s@mme:[\p@int,-4;#3,-3]\advance\s@mme\@ne}%
    \resetc@ntr@l\et@tfigptsorthoprojlineDD}\ignorespaces\fi}
\ctr@ld@f\def\figptsorthoprojlineTD#1=#2/#3,#4/{\ifps@cri{\s@uvc@ntr@l\et@tfigptsorthoprojlineTD%
    \setc@ntr@l{2}\figvectPTD-2[#3,#4]\vecunit@TD{-2}{-2}%
    \def\list@num{#2}\s@mme=#1\@ecfor\p@int:=\list@num\do{%
    \figvectPTD-1[#3,\p@int]\c@lproscalTD\v@leur[-1,-2]%
    \edef\v@lcoef{\repdecn@mb{\v@leur}}\figpttraTD\the\s@mme:=#3/\v@lcoef,-2/%
    \advance\s@mme\@ne}\resetc@ntr@l\et@tfigptsorthoprojlineTD}\ignorespaces\fi}
\ctr@ln@m\figptsorthoprojplane
\ctr@ld@f\def\figptsorthoprojplaneDD{\un@v@ilable{figptsorthoprojplane}}
\ctr@ld@f\def\figptsorthoprojplaneTD#1=#2/#3,#4/{\ifps@cri{\s@uvc@ntr@l\et@tfigptsorthoprojplane%
    \setc@ntr@l{2}\vecunit@TD{-2}{#4}%
    \def\list@num{#2}\s@mme=#1\@ecfor\p@int:=\list@num\do{\figvectPTD-1[\p@int,#3]%
    \c@lproscalTD\v@leur[-1,-2]\edef\v@lcoef{\repdecn@mb{\v@leur}}%
    \figpttraTD\the\s@mme:=\p@int/\v@lcoef,-2/\advance\s@mme\@ne}%
    \resetc@ntr@l\et@tfigptsorthoprojplane}\ignorespaces\fi}
\ctr@ld@f\def\figptshom#1=#2/#3,#4/{\ifps@cri{\s@uvc@ntr@l\et@tfigptshom%
    \setc@ntr@l{2}\def\list@num{#2}\s@mme=#1%
    \@ecfor\p@int:=\list@num\do{\figvectP-1[#3,\p@int]%
    \figpttra\the\s@mme:=#3/#4,-1/\advance\s@mme\@ne}%
    \resetc@ntr@l\et@tfigptshom}\ignorespaces\fi}
\ctr@ln@m\figptsrot
\ctr@ld@f\def\figptsrotDD#1=#2/#3,#4/{\ifps@cri{\s@uvc@ntr@l\et@tfigptsrotDD%
    \c@ssin{\C@}{\S@}{#4}\setc@ntr@l{2}\def\list@num{#2}\s@mme=#1%
    \@ecfor\p@int:=\list@num\do{\figvectPDD-1[#3,\p@int]\Figg@tXY{-1}%
    \v@lXa=\C@\v@lX\advance\v@lXa-\S@\v@lY%
    \v@lYa=\S@\v@lX\advance\v@lYa\C@\v@lY%
    \Figv@ctCreg-1(\v@lXa,\v@lYa)\figpttraDD\the\s@mme:=#3/1,-1/\advance\s@mme\@ne}%
    \resetc@ntr@l\et@tfigptsrotDD}\ignorespaces\fi}
\ctr@ld@f\def\figptsrotTD#1=#2/#3,#4,#5/{\ifps@cri{\s@uvc@ntr@l\et@tfigptsrotTD%
    \c@ssin{\C@}{\S@}{#4}%
    \setc@ntr@l{2}\def\list@num{#2}\s@mme=#1%
    \@ecfor\p@int:=\list@num\do{\figptorthoprojplaneTD-3:=#3/\p@int,#5/%
    \figvectPTD-2[-3,\p@int]%
    \figvectNVTD-1[#5,-2]\n@rmeucTD\v@leur{-2}\edef\v@lcoef{\repdecn@mb{\v@leur}}%
    \Figg@tXYa{-1}\v@lXa=\v@lcoef\v@lXa\v@lYa=\v@lcoef\v@lYa\v@lZa=\v@lcoef\v@lZa%
    \v@lXa=\S@\v@lXa\v@lYa=\S@\v@lYa\v@lZa=\S@\v@lZa\Figg@tXY{-2}%
    \advance\v@lXa\C@\v@lX\advance\v@lYa\C@\v@lY\advance\v@lZa\C@\v@lZ%
    \Figg@tXY{-3}\advance\v@lXa\v@lX\advance\v@lYa\v@lY\advance\v@lZa\v@lZ%
    \Figp@intregTD\the\s@mme:(\v@lXa,\v@lYa,\v@lZa)\advance\s@mme\@ne}%
    \resetc@ntr@l\et@tfigptsrotTD}\ignorespaces\fi}
\ctr@ln@m\figptssym
\ctr@ld@f\def\figptssymDD#1=#2/#3,#4/{\ifps@cri{\s@uvc@ntr@l\et@tfigptssymDD%
    \setc@ntr@l{2}\figvectPDD-3[#3,#4]\Figg@tXY{-3}\Figv@ctCreg-4(-\v@lY,\v@lX)%
    \resetc@ntr@l{2}\def\list@num{#2}\s@mme=#1%
    \@ecfor\p@int:=\list@num\do{\inters@cDD-5:[#3,-3;\p@int,-4]\figvectPDD-2[\p@int,-5]%
    \figpttraDD\the\s@mme:=\p@int/2,-2/\advance\s@mme\@ne}%
    \resetc@ntr@l\et@tfigptssymDD}\ignorespaces\fi}
\ctr@ld@f\def\figptssymTD#1=#2/#3,#4/{\ifps@cri{\s@uvc@ntr@l\et@tfigptssymTD%
    \setc@ntr@l{2}\vecunit@TD{-2}{#4}\def\list@num{#2}\s@mme=#1%
    \@ecfor\p@int:=\list@num\do{\figvectPTD-1[\p@int,#3]%
    \c@lproscalTD\v@leur[-1,-2]\v@leur=2\v@leur\edef\v@lcoef{\repdecn@mb{\v@leur}}%
    \figpttraTD\the\s@mme:=\p@int/\v@lcoef,-2/\advance\s@mme\@ne}%
    \resetc@ntr@l\et@tfigptssymTD}\ignorespaces\fi}
\ctr@ln@m\figptstra
\ctr@ld@f\def\figptstraDD#1=#2/#3,#4/{\ifps@cri{\Figg@tXYa{#4}\v@lXa=#3\v@lXa\v@lYa=#3\v@lYa%
    \def\list@num{#2}\s@mme=#1\@ecfor\p@int:=\list@num\do{\Figg@tXY{\p@int}%
    \advance\v@lX\v@lXa\advance\v@lY\v@lYa%
    \Figp@intregDD\the\s@mme:(\v@lX,\v@lY)\advance\s@mme\@ne}}\ignorespaces\fi}
\ctr@ld@f\def\figptstraTD#1=#2/#3,#4/{\ifps@cri{\Figg@tXYa{#4}\v@lXa=#3\v@lXa\v@lYa=#3\v@lYa%
    \v@lZa=#3\v@lZa\def\list@num{#2}\s@mme=#1\@ecfor\p@int:=\list@num\do{\Figg@tXY{\p@int}%
    \advance\v@lX\v@lXa\advance\v@lY\v@lYa\advance\v@lZ\v@lZa%
    \Figp@intregTD\the\s@mme:(\v@lX,\v@lY,\v@lZ)\advance\s@mme\@ne}}\ignorespaces\fi}
\ctr@ln@m\figptvisilimSL
\ctr@ld@f\def\figptvisilimSLDD{\un@v@ilable{figptvisilimSL}}
\ctr@ld@f\def\figptvisilimSLTD#1:#2[#3,#4;#5,#6]{\ifps@cri{\s@uvc@ntr@l\et@tfigptvisilimSLTD%
    \setc@ntr@l{2}\figvectP-1[#3,#4]\n@rminf{\delt@}{-1}%
    \ifcase\curr@ntproj\v@lX=\cxa@\p@\v@lY=-\p@\v@lZ=\cxb@\p@% Proj cav
    \Figv@ctCreg-2(\v@lX,\v@lY,\v@lZ)\figvectP-3[#5,#6]\figvectNV-1[-2,-3]%
    \or\figvectP-1[#5,#6]\vecunitCV@TD{-1}\v@lmin=\v@lX\v@lmax=\v@lY% Proj ortho
    \v@leur=\v@lZ\v@lX=\cza@\p@\v@lY=\czb@\p@\v@lZ=\czc@\p@\c@lprovec{-1}%
    \or\c@ley@pt{-2}\figvectN-1[#5,#6,-2]\fi% Proj rea
    \edef\Ai@{#3}\edef\Aj@{#4}\figvectP-2[#5,\Ai@]\c@lproscal\v@leur[-1,-2]%
    \ifdim\v@leur>\z@\p@rtent=\@ne\else\p@rtent=\m@ne\fi%
    \figvectP-2[#5,\Aj@]\c@lproscal\v@leur[-1,-2]%
    \ifdim\p@rtent\v@leur>\z@\figptcopy#1:#2/#3/%
    \message{*** \BS@ figptvisilimSL: points are on the same side.}\else%
    \figptcopy-3:/#3/\figptcopy-4:/#4/%
    \loop\figptbary-5:[-3,-4;1,1]\figvectP-2[#5,-5]\c@lproscal\v@leur[-1,-2]%
    \ifdim\p@rtent\v@leur>\z@\figptcopy-3:/-5/\else\figptcopy-4:/-5/\fi%
    \divide\delt@\tw@\ifdim\delt@>\epsil@n\repeat%
    \figptbary#1:#2[-3,-4;1,1]\fi\resetc@ntr@l\et@tfigptvisilimSLTD}\ignorespaces\fi}
\ctr@ld@f\def\c@ley@pt#1{\t@stp@r\ifitis@K\v@lX=\cza@\p@\v@lY=\czb@\p@\v@lZ=\czc@\p@%
    \Figv@ctCreg-1(\v@lX,\v@lY,\v@lZ)\Figp@intreg-2:(\wd\Bt@rget,\ht\Bt@rget,\dp\Bt@rget)%
    \figpttra#1:=-2/-\disob@intern,-1/\else\end\fi}
\ctr@ld@f\def\t@stp@r{\itis@Ktrue\ifnewt@rgetpt\else\itis@Kfalse%
    \message{*** \BS@ figptvisilimXX: target point undefined.}\fi\ifnewdis@b\else%
    \itis@Kfalse\message{*** \BS@ figptvisilimXX: observation distance undefined.}\fi%
    \ifitis@K\else\message{*** This macro must be called after \BS@ psbeginfig or after
    having set the missing parameter(s) with \BS@ figset proj()}\fi}
\ctr@ld@f\def\figscan#1(#2,#3){{\s@uvc@ntr@l\et@tfigscan\@psfgetbb{#1}\if@psfbbfound\else%
    \def\@psfllx{0}\def\@psflly{20}\def\@psfurx{540}\def\@psfury{640}\fi\figscan@{#2}{#3}%
    \resetc@ntr@l\et@tfigscan}\ignorespaces}
\ctr@ld@f\def\figscan@#1#2{%
    \unit@=\@ne bp\setc@ntr@l{2}\figsetmark{}%
    \def\minst@p{20pt}%
    \v@lX=\@psfllx\p@\v@lX=\Sc@leFact\v@lX\r@undint\v@lX\v@lX%
    \v@lY=\@psflly\p@\v@lY=\Sc@leFact\v@lY\ifdim\v@lY>\z@\r@undint\v@lY\v@lY\fi%
    \delt@=\@psfury\p@\delt@=\Sc@leFact\delt@%
    \advance\delt@-\v@lY\v@lXa=\@psfurx\p@\v@lXa=\Sc@leFact\v@lXa\v@leur=\minst@p%
    \edef\valv@lY{\repdecn@mb{\v@lY}}\edef\LgTr@it{\the\delt@}%
    \loop\ifdim\v@lX<\v@lXa\edef\valv@lX{\repdecn@mb{\v@lX}}%
    \figptDD -1:(\valv@lX,\valv@lY)\figwriten -1:\hbox{\vrule height\LgTr@it}(0)%
    \ifdim\v@leur<\minst@p\else\figsetmark{\raise-8bp\hbox{$\scriptscriptstyle\triangle$}}%
    \figwrites -1:\@ffichnb{0}{\valv@lX}(6)\v@leur=\z@\figsetmark{}\fi%
    \advance\v@leur#1pt\advance\v@lX#1pt\repeat%
    \def\minst@p{10pt}%
    \v@lX=\@psfllx\p@\v@lX=\Sc@leFact\v@lX\ifdim\v@lX>\z@\r@undint\v@lX\v@lX\fi%
    \v@lY=\@psflly\p@\v@lY=\Sc@leFact\v@lY\r@undint\v@lY\v@lY%
    \delt@=\@psfurx\p@\delt@=\Sc@leFact\delt@%
    \advance\delt@-\v@lX\v@lYa=\@psfury\p@\v@lYa=\Sc@leFact\v@lYa\v@leur=\minst@p%
    \edef\valv@lX{\repdecn@mb{\v@lX}}\edef\LgTr@it{\the\delt@}%
    \loop\ifdim\v@lY<\v@lYa\edef\valv@lY{\repdecn@mb{\v@lY}}%
    \figptDD -1:(\valv@lX,\valv@lY)\figwritee -1:\vbox{\hrule width\LgTr@it}(0)%
    \ifdim\v@leur<\minst@p\else\figsetmark{$\triangleright$\kern4bp}%
    \figwritew -1:\@ffichnb{0}{\valv@lY}(6)\v@leur=\z@\figsetmark{}\fi%
    \advance\v@leur#2pt\advance\v@lY#2pt\repeat}
\ctr@ld@f
\ctr@ld@f\def\figscan@E#1(#2,#3){{\s@uvc@ntr@l\et@tfigscan@E%
    \Figdisc@rdLTS{#1}{\t@xt@}\pdfximage{\t@xt@}%
    \setbox\Gb@x=\hbox{\pdfrefximage\pdflastximage}%
    \edef\@psfllx{0}\v@lY=-\dp\Gb@x\edef\@psflly{\repdecn@mb{\v@lY}}%
    \edef\@psfurx{\repdecn@mb{\wd\Gb@x}}%
    \v@lY=\dp\Gb@x\advance\v@lY\ht\Gb@x\edef\@psfury{\repdecn@mb{\v@lY}}%
    \figscan@{#2}{#3}\resetc@ntr@l\et@tfigscan@E}\ignorespaces}
\ctr@ld@f\def\figshowpts[#1,#2]{{\figsetmark{$\bullet$}\figsetptname{\bf ##1}%
    \p@rtent=#2\relax\ifnum\p@rtent<\z@\p@rtent=\z@\fi%
    \s@mme=#1\relax\ifnum\s@mme<\z@\s@mme=\z@\fi%
    \loop\ifnum\s@mme<\p@rtent\pt@rvect{\s@mme}%
    \ifitis@K\figwriten{\the\s@mme}:(4pt)\fi\advance\s@mme\@ne\repeat%
    \pt@rvect{\s@mme}\ifitis@K\figwriten{\the\s@mme}:(4pt)\fi}\ignorespaces}
\ctr@ld@f\def\pt@rvect#1{\set@bjc@de{#1}%
    \expandafter\expandafter\expandafter\inqpt@rvec\csname\objc@de\endcsname:}
\ctr@ld@f\def\inqpt@rvec#1#2:{\if#1\C@dCl@spt\itis@Ktrue\else\itis@Kfalse\fi}
\ctr@ld@f\def\figshowsettings{{%
    \immediate\write16{====================================================================}%
    \immediate\write16{ Current settings about:}%
    \immediate\write16{ --- GENERAL ---}%
    \immediate\write16{Scale factor and Unit = \unit@util\space (\the\unit@)
     \space -> \BS@ figinit{ScaleFactorUnit}}%
    \immediate\write16{Update mode = \ifpsupdatem@de yes\else no\fi
     \space-> \BS@ psset(update=yes/no) or \BS@ pssetdefault(update=yes/no)}%
    \immediate\write16{ --- PRINTING ---}%
    \immediate\write16{Implicit point name = \ptn@me{i} \space-> \BS@ figsetptname{Name}}%
    \immediate\write16{Point marker = \the\c@nsymb \space -> \BS@ figsetmark{Mark}}%
    \immediate\write16{Print rounded coordinates = \ifr@undcoord yes\else no\fi
     \space-> \BS@ figsetroundcoord{yes/no}}%
    \immediate\write16{ --- GRAPHICAL (general) ---}%
    \immediate\write16{First-level (or primary) settings:}%
    \immediate\write16{ Color = \curr@ntcolor \space-> \BS@ psset(color=ColorDefinition)}%
    \immediate\write16{ Filling mode = \iffillm@de yes\else no\fi
     \space-> \BS@ psset(fillmode=yes/no)}%
    \immediate\write16{ Line join = \curr@ntjoin \space-> \BS@ psset(join=miter/round/bevel)}%
    \immediate\write16{ Line style = \curr@ntdash \space-> \BS@ psset(dash=Index/Pattern)}%
    \immediate\write16{ Line width = \curr@ntwidth
     \space-> \BS@ psset(width=real in PostScript units)}%
    \immediate\write16{Second-level (or secondary) settings:}%
    \immediate\write16{ Color = \sec@ndcolor \space-> \BS@ psset second(color=ColorDefinition)}%
    \immediate\write16{ Line style = \curr@ntseconddash
     \space-> \BS@ psset second(dash=Index/Pattern)}%
    \immediate\write16{ Line width = \curr@ntsecondwidth
     \space-> \BS@ psset second(width=real in PostScript units)}%
    \immediate\write16{Third-level (or ternary) settings:}%
    \immediate\write16{ Color = \th@rdcolor \space-> \BS@ psset third(color=ColorDefinition)}%
    \immediate\write16{ --- GRAPHICAL (specific) ---}%
    \immediate\write16{Arrow-head:}%
    \immediate\write16{ (half-)Angle = \@rrowheadangle
     \space-> \BS@ psset arrowhead(angle=real in degrees)}%
    \immediate\write16{ Filling mode = \if@rrowhfill yes\else no\fi
     \space-> \BS@ psset arrowhead(fillmode=yes/no)}%
    \immediate\write16{ "Outside" = \if@rrowhout yes\else no\fi
     \space-> \BS@ psset arrowhead(out=yes/no)}%
    \immediate\write16{ Length = \@rrowheadlength
     \if@rrowratio\space(not active)\else\space(active)\fi
     \space-> \BS@ psset arrowhead(length=real in user coord.)}%
    \immediate\write16{ Ratio = \@rrowheadratio
     \if@rrowratio\space(active)\else\space(not active)\fi
     \space-> \BS@ psset arrowhead(ratio=real in [0,1])}%
    \immediate\write16{Curve: Roundness = \curv@roundness
     \space-> \BS@ psset curve(roundness=real in [0,0.5])}%
    \immediate\write16{Mesh: Diagonal = \c@ntrolmesh
     \space-> \BS@ psset mesh(diag=integer in {-1,0,1})}%
    \immediate\write16{Flow chart:}%
    \immediate\write16{ Arrow position = \@rrowp@s
     \space-> \BS@ psset flowchart(arrowposition=real in [0,1])}%
    \immediate\write16{ Arrow reference point = \ifcase\@rrowr@fpt start\else end\fi
     \space-> \BS@ psset flowchart(arrowrefpt = start/end)}%     
    \immediate\write16{ Line type = \ifcase\fclin@typ@ curve\else polygon\fi
     \space-> \BS@ psset flowchart(line=polygon/curve)}%
    \immediate\write16{ Padding = (\Xp@dd, \Yp@dd)
     \space-> \BS@ psset flowchart(padding = real in user coord.)}%
    \immediate\write16{\space\space\space\space(or
     \BS@ psset flowchart(xpadding=real, ypadding=real) )}%
    \immediate\write16{ Radius = \fclin@r@d
     \space-> \BS@ psset flowchart(radius=positive real in user coord.)}%
    \immediate\write16{ Shape = \fcsh@pe
     \space-> \BS@ psset flowchart(shape = rectangle, ellipse or lozenge)}%
    \immediate\write16{ Thickness = \thickn@ss
     \space-> \BS@ psset flowchart(thickness = real in user coord.)}%
    \ifTr@isDim%
    \immediate\write16{ --- 3D to 2D PROJECTION ---}%
    \immediate\write16{Projection : \typ@proj \space-> \BS@ figinit{ScaleFactorUnit, ProjType}}%
    \immediate\write16{Longitude (psi) = \v@lPsi \space-> \BS@ figset proj(psi=real in degrees)}%
    \ifcase\curr@ntproj\immediate\write16{Depth coeff. (Lambda)
     \space = \v@lTheta \space-> \BS@ figset proj(lambda=real in [0,1])}%
    \else\immediate\write16{Latitude (theta)
     \space = \v@lTheta \space-> \BS@ figset proj(theta=real in degrees)}%
    \fi%
    \ifnum\curr@ntproj=\tw@%
    \immediate\write16{Observation distance = \disob@unit
     \space-> \BS@ figset proj(dist=real in user coord.)}%
    \immediate\write16{Target point = \t@rgetpt \space-> \BS@ figset proj(targetpt=pt number)}%
     \v@lX=\ptT@unit@\wd\Bt@rget\v@lY=\ptT@unit@\ht\Bt@rget\v@lZ=\ptT@unit@\dp\Bt@rget%
    \immediate\write16{ Its coordinates are
     (\repdecn@mb{\v@lX}, \repdecn@mb{\v@lY}, \repdecn@mb{\v@lZ})}%
    \fi%
    \fi%
    \immediate\write16{====================================================================}%
    \ignorespaces}}
\ctr@ln@w{newif}\ifitis@vect@r
\ctr@ld@f\def\figvectC#1(#2,#3){{\itis@vect@rtrue\figpt#1:(#2,#3)}\ignorespaces}
\ctr@ld@f\def\Figv@ctCreg#1(#2,#3){{\itis@vect@rtrue\Figp@intreg#1:(#2,#3)}\ignorespaces}
\ctr@ln@m\figvectDBezier
\ctr@ld@f\def\figvectDBezierDD#1:#2,#3[#4,#5,#6,#7]{\ifps@cri{\s@uvc@ntr@l\et@tfigvectDBezierDD%
    \FigvectDBezier@#2,#3[#4,#5,#6,#7]\v@lX=\c@ef\v@lX\v@lY=\c@ef\v@lY%
    \Figv@ctCreg#1(\v@lX,\v@lY)\resetc@ntr@l\et@tfigvectDBezierDD}\ignorespaces\fi}
\ctr@ld@f\def\figvectDBezierTD#1:#2,#3[#4,#5,#6,#7]{\ifps@cri{\s@uvc@ntr@l\et@tfigvectDBezierTD%
    \FigvectDBezier@#2,#3[#4,#5,#6,#7]\v@lX=\c@ef\v@lX\v@lY=\c@ef\v@lY\v@lZ=\c@ef\v@lZ%
    \Figv@ctCreg#1(\v@lX,\v@lY,\v@lZ)\resetc@ntr@l\et@tfigvectDBezierTD}\ignorespaces\fi}
\ctr@ld@f\def\FigvectDBezier@#1,#2[#3,#4,#5,#6]{\setc@ntr@l{2}%
    \edef\T@{#2}\v@leur=\p@\advance\v@leur-#2pt\edef\UNmT@{\repdecn@mb{\v@leur}}%
    \ifnum#1=\tw@\def\c@ef{6}\else\def\c@ef{3}\fi%
    \figptcopy-4:/#3/\figptcopy-3:/#4/\figptcopy-2:/#5/\figptcopy-1:/#6/%
    \l@mbd@un=-4 \l@mbd@de=-\thr@@\p@rtent=\m@ne\c@lDecast%
    \ifnum#1=\tw@\c@lDCDeux{-4}{-3}\c@lDCDeux{-3}{-2}\c@lDCDeux{-4}{-3}\else%
    \l@mbd@un=-4 \l@mbd@de=-\thr@@\p@rtent=-\tw@\c@lDecast%
    \c@lDCDeux{-4}{-3}\fi\Figg@tXY{-4}}
\ctr@ln@m\c@lDCDeux
\ctr@ld@f\def\c@lDCDeuxDD#1#2{\Figg@tXY{#2}\Figg@tXYa{#1}%
    \advance\v@lX-\v@lXa\advance\v@lY-\v@lYa\Figp@intregDD#1:(\v@lX,\v@lY)}
\ctr@ld@f\def\c@lDCDeuxTD#1#2{\Figg@tXY{#2}\Figg@tXYa{#1}\advance\v@lX-\v@lXa%
    \advance\v@lY-\v@lYa\advance\v@lZ-\v@lZa\Figp@intregTD#1:(\v@lX,\v@lY,\v@lZ)}
\ctr@ln@m\figvectN
\ctr@ld@f\def\figvectNDD#1[#2,#3]{\ifps@cri{\Figg@tXYa{#2}\Figg@tXY{#3}%
    \advance\v@lX-\v@lXa\advance\v@lY-\v@lYa%
    \Figv@ctCreg#1(-\v@lY,\v@lX)}\ignorespaces\fi}
\ctr@ld@f\def\figvectNTD#1[#2,#3,#4]{\ifps@cri{\vecunitC@TD[#2,#4]\v@lmin=\v@lX\v@lmax=\v@lY%
    \v@leur=\v@lZ\vecunitC@TD[#2,#3]\c@lprovec{#1}}\ignorespaces\fi}
\ctr@ln@m\figvectNV
\ctr@ld@f\def\figvectNVDD#1[#2]{\ifps@cri{\Figg@tXY{#2}\Figv@ctCreg#1(-\v@lY,\v@lX)}\ignorespaces\fi}
\ctr@ld@f\def\figvectNVTD#1[#2,#3]{\ifps@cri{\vecunitCV@TD{#3}\v@lmin=\v@lX\v@lmax=\v@lY%
    \v@leur=\v@lZ\vecunitCV@TD{#2}\c@lprovec{#1}}\ignorespaces\fi}
\ctr@ln@m\figvectP
\ctr@ld@f\def\figvectPDD#1[#2,#3]{\ifps@cri{\Figg@tXYa{#2}\Figg@tXY{#3}%
    \advance\v@lX-\v@lXa\advance\v@lY-\v@lYa%
    \Figv@ctCreg#1(\v@lX,\v@lY)}\ignorespaces\fi}
\ctr@ld@f\def\figvectPTD#1[#2,#3]{\ifps@cri{\Figg@tXYa{#2}\Figg@tXY{#3}%
    \advance\v@lX-\v@lXa\advance\v@lY-\v@lYa\advance\v@lZ-\v@lZa%
    \Figv@ctCreg#1(\v@lX,\v@lY,\v@lZ)}\ignorespaces\fi}
\ctr@ln@m\figvectU
\ctr@ld@f\def\figvectUDD#1[#2]{\ifps@cri{\n@rmeuc\v@leur{#2}\invers@\v@leur\v@leur%
    \delt@=\repdecn@mb{\v@leur}\unit@\edef\v@ldelt@{\repdecn@mb{\delt@}}%
    \Figg@tXY{#2}\v@lX=\v@ldelt@\v@lX\v@lY=\v@ldelt@\v@lY%
    \Figv@ctCreg#1(\v@lX,\v@lY)}\ignorespaces\fi}
\ctr@ld@f\def\figvectUTD#1[#2]{\ifps@cri{\n@rmeuc\v@leur{#2}\invers@\v@leur\v@leur%
    \delt@=\repdecn@mb{\v@leur}\unit@\edef\v@ldelt@{\repdecn@mb{\delt@}}%
    \Figg@tXY{#2}\v@lX=\v@ldelt@\v@lX\v@lY=\v@ldelt@\v@lY\v@lZ=\v@ldelt@\v@lZ%
    \Figv@ctCreg#1(\v@lX,\v@lY,\v@lZ)}\ignorespaces\fi}
\ctr@ld@f\def\figvisu#1#2#3{\c@ldefproj\initb@undb@x\xdef\figforTeXFigno{\figforTeXnextFigno}%
    \s@mme=\figforTeXnextFigno\advance\s@mme\@ne\xdef\figforTeXnextFigno{\number\s@mme}%
    \setbox\b@xvisu=\hbox{\ifnum\@utoFN>\z@\figinsert{}\gdef\@utoFInDone{0}\fi\ignorespaces#3}%
    \gdef\@utoFInDone{1}\gdef\@utoFN{0}%
    \v@lXa=-\c@@rdYmin\v@lYa=\c@@rdYmax\advance\v@lYa-\c@@rdYmin%
    \v@lX=\c@@rdXmax\advance\v@lX-\c@@rdXmin%
    \setbox#1=\hbox{#2}\v@lY=-\v@lX\maxim@m{\v@lX}{\v@lX}{\wd#1}%
    \advance\v@lY\v@lX\divide\v@lY\tw@\advance\v@lY-\c@@rdXmin%
    \setbox#1=\vbox{\parindent0mm\hsize=\v@lX\vskip\v@lYa%
    \rlap{\hskip\v@lY\smash{\raise\v@lXa\box\b@xvisu}}%
    \def\t@xt@{#2}\ifx\t@xt@\empty\else\medskip\centerline{#2}\fi}\wd#1=\v@lX}
\ctr@ld@f\def\figDecrementFigno{{\xdef\figforTeXnextFigno{\figforTeXFigno}%
    \s@mme=\figforTeXFigno\advance\s@mme\m@ne\xdef\figforTeXFigno{\number\s@mme}}}
\ctr@ln@w{newbox}\Bt@rget\setbox\Bt@rget=\null
\ctr@ln@w{newbox}\BminTD@\setbox\BminTD@=\null
\ctr@ln@w{newbox}\BmaxTD@\setbox\BmaxTD@=\null
\ctr@ln@w{newif}\ifnewt@rgetpt\ctr@ln@w{newif}\ifnewdis@b
\ctr@ld@f\def\b@undb@xTD#1#2#3{%
    \relax\ifdim#1<\wd\BminTD@\global\wd\BminTD@=#1\fi%
    \relax\ifdim#2<\ht\BminTD@\global\ht\BminTD@=#2\fi%
    \relax\ifdim#3<\dp\BminTD@\global\dp\BminTD@=#3\fi%
    \relax\ifdim#1>\wd\BmaxTD@\global\wd\BmaxTD@=#1\fi%
    \relax\ifdim#2>\ht\BmaxTD@\global\ht\BmaxTD@=#2\fi%
    \relax\ifdim#3>\dp\BmaxTD@\global\dp\BmaxTD@=#3\fi}
\ctr@ld@f\def\c@ldefdisob{{\ifdim\wd\BminTD@<\maxdimen\v@leur=\wd\BmaxTD@\advance\v@leur-\wd\BminTD@%
    \delt@=\ht\BmaxTD@\advance\delt@-\ht\BminTD@\maxim@m{\v@leur}{\v@leur}{\delt@}%
    \delt@=\dp\BmaxTD@\advance\delt@-\dp\BminTD@\maxim@m{\v@leur}{\v@leur}{\delt@}%
    \v@leur=5\v@leur\else\v@leur=800pt\fi\c@ldefdisob@{\v@leur}}}
\ctr@ln@m\disob@intern
\ctr@ln@m\disob@
\ctr@ln@m\divf@ctproj
\ctr@ld@f\def\c@ldefdisob@#1{{\v@leur=#1\ifdim\v@leur<\p@\v@leur=800pt\fi%
    \xdef\disob@intern{\repdecn@mb{\v@leur}}%
    \delt@=\ptT@unit@\v@leur\xdef\disob@unit{\repdecn@mb{\delt@}}%
    \f@ctech=\@ne\loop\ifdim\v@leur>\t@n pt\divide\v@leur\t@n\multiply\f@ctech\t@n\repeat%
    \xdef\disob@{\repdecn@mb{\v@leur}}\xdef\divf@ctproj{\the\f@ctech}}%
    \global\newdis@btrue}
\ctr@ln@m\t@rgetpt
\ctr@ld@f\def\c@ldeft@rgetpt{\newt@rgetpttrue\def\t@rgetpt{CenterBoundBox}{%
    \delt@=\wd\BmaxTD@\advance\delt@-\wd\BminTD@\divide\delt@\tw@%
    \v@leur=\wd\BminTD@\advance\v@leur\delt@\global\wd\Bt@rget=\v@leur%
    \delt@=\ht\BmaxTD@\advance\delt@-\ht\BminTD@\divide\delt@\tw@%
    \v@leur=\ht\BminTD@\advance\v@leur\delt@\global\ht\Bt@rget=\v@leur%
    \delt@=\dp\BmaxTD@\advance\delt@-\dp\BminTD@\divide\delt@\tw@%
    \v@leur=\dp\BminTD@\advance\v@leur\delt@\global\dp\Bt@rget=\v@leur}}
\ctr@ln@m\c@ldefproj
\ctr@ld@f\def\c@ldefprojTD{\ifnewt@rgetpt\else\c@ldeft@rgetpt\fi\ifnewdis@b\else\c@ldefdisob\fi}
\ctr@ld@f\def\c@lprojcav{% Projection cavaliere : X = x + y L cos t, Y = z + y L sin t
    \v@lZa=\cxa@\v@lY\advance\v@lX\v@lZa%
    \v@lZa=\cxb@\v@lY\v@lY=\v@lZ\advance\v@lY\v@lZa\ignorespaces}
\ctr@ln@m\v@lcoef
\ctr@ld@f\def\c@lprojrea{% Projection realiste
    \advance\v@lX-\wd\Bt@rget\advance\v@lY-\ht\Bt@rget\advance\v@lZ-\dp\Bt@rget%
    \v@lZa=\cza@\v@lX\advance\v@lZa\czb@\v@lY\advance\v@lZa\czc@\v@lZ%
    \divide\v@lZa\divf@ctproj\advance\v@lZa\disob@ pt\invers@{\v@lZa}{\v@lZa}%
    \v@lZa=\disob@\v@lZa\edef\v@lcoef{\repdecn@mb{\v@lZa}}%
    \v@lXa=\cxa@\v@lX\advance\v@lXa\cxb@\v@lY\v@lXa=\v@lcoef\v@lXa%
    \v@lY=\cyb@\v@lY\advance\v@lY\cya@\v@lX\advance\v@lY\cyc@\v@lZ%
    \v@lY=\v@lcoef\v@lY\v@lX=\v@lXa\ignorespaces}
\ctr@ld@f\def\c@lprojort{% Projection orthogonale
    \v@lXa=\cxa@\v@lX\advance\v@lXa\cxb@\v@lY%
    \v@lY=\cyb@\v@lY\advance\v@lY\cya@\v@lX\advance\v@lY\cyc@\v@lZ%
    \v@lX=\v@lXa\ignorespaces}
\ctr@ld@f\def\Figptpr@j#1:#2/#3/{{\Figg@tXY{#3}\superc@lprojSP%
    \Figp@intregDD#1:{#2}(\v@lX,\v@lY)}\ignorespaces}
\ctr@ln@m\figsetobdist
\ctr@ld@f\def\figsetobdistDD{\un@v@ilable{figsetobdist}}
\ctr@ld@f\def\figsetobdistTD(#1){{\ifcurr@ntPS%
    \immediate\write16{*** \BS@ figsetobdist is ignored inside a
     \BS@ psbeginfig-\BS@ psendfig block.}%
    \else\v@leur=#1\unit@\c@ldefdisob@{\v@leur}\fi}\ignorespaces}
\ctr@ln@m\c@lprojSP
\ctr@ln@m\curr@ntproj
\ctr@ln@m\typ@proj
\ctr@ln@m\superc@lprojSP
\ctr@ld@f\def\Figs@tproj#1{%
    \if#13 \d@faultproj\else\if#1c\d@faultproj%
    \else\if#1o\xdef\curr@ntproj{1}\xdef\typ@proj{orthogonal}%
         \figsetviewTD(\def@ultpsi,\def@ulttheta)%
         \global\let\c@lprojSP=\c@lprojort\global\let\superc@lprojSP=\c@lprojort%
    \else\if#1r\xdef\curr@ntproj{2}\xdef\typ@proj{realistic}%
         \figsetviewTD(\def@ultpsi,\def@ulttheta)%
         \global\let\c@lprojSP=\c@lprojrea\global\let\superc@lprojSP=\c@lprojrea%
    \else\d@faultproj\message{*** Unknown projection. Cavalier projection assumed.}%
    \fi\fi\fi\fi}
\ctr@ld@f\def\d@faultproj{\xdef\curr@ntproj{0}\xdef\typ@proj{cavalier}\figsetviewTD(\def@ultpsi,0.5)%
         \global\let\c@lprojSP=\c@lprojcav\global\let\superc@lprojSP=\c@lprojcav}
\ctr@ln@m\figsettarget
\ctr@ld@f\def\figsettargetDD{\un@v@ilable{figsettarget}}
\ctr@ld@f\def\figsettargetTD[#1]{{\ifcurr@ntPS%
    \immediate\write16{*** \BS@ figsettarget is ignored inside a
     \BS@ psbeginfig-\BS@ psendfig block.}%
    \else\global\newt@rgetpttrue\xdef\t@rgetpt{#1}\Figg@tXY{#1}\global\wd\Bt@rget=\v@lX%
    \global\ht\Bt@rget=\v@lY\global\dp\Bt@rget=\v@lZ\fi}\ignorespaces}
\ctr@ln@m\figsetview
\ctr@ld@f\def\figsetviewDD{\un@v@ilable{figsetview}}
\ctr@ld@f\def\figsetviewTD(#1){\ifcurr@ntPS%
     \immediate\write16{*** \BS@ figsetview is ignored inside a
     \BS@ psbeginfig-\BS@ psendfig block.}\else\Figsetview@#1,:\fi\ignorespaces}
\ctr@ld@f\def\Figsetview@#1,#2:{{\xdef\v@lPsi{#1}\def\t@xt@{#2}%
    \ifx\t@xt@\empty\def\@rgdeux{\v@lTheta}\else\X@rgdeux@#2\fi%
    \c@ssin{\costhet@}{\sinthet@}{#1}\v@lmin=\costhet@ pt\v@lmax=\sinthet@ pt%
    \ifcase\curr@ntproj%
    \v@leur=\@rgdeux\v@lmin\xdef\cxa@{\repdecn@mb{\v@leur}}%
    \v@leur=\@rgdeux\v@lmax\xdef\cxb@{\repdecn@mb{\v@leur}}\v@leur=\@rgdeux pt%
    \relax\ifdim\v@leur>\p@\message{*** Lambda too large ! See \BS@ figset proj() !}\fi%
    \else%
    \v@lmax=-\v@lmax\xdef\cxa@{\repdecn@mb{\v@lmax}}\xdef\cxb@{\costhet@}%
    \ifx\t@xt@\empty\edef\@rgdeux{\def@ulttheta}\fi\c@ssin{\C@}{\S@}{\@rgdeux}%
    \v@lmax=-\S@ pt%
    \v@leur=\v@lmax\v@leur=\costhet@\v@leur\xdef\cya@{\repdecn@mb{\v@leur}}%
    \v@leur=\v@lmax\v@leur=\sinthet@\v@leur\xdef\cyb@{\repdecn@mb{\v@leur}}%
    \xdef\cyc@{\C@}\v@lmin=-\C@ pt%
    \v@leur=\v@lmin\v@leur=\costhet@\v@leur\xdef\cza@{\repdecn@mb{\v@leur}}%
    \v@leur=\v@lmin\v@leur=\sinthet@\v@leur\xdef\czb@{\repdecn@mb{\v@leur}}%
    \xdef\czc@{\repdecn@mb{\v@lmax}}\fi%
    \xdef\v@lTheta{\@rgdeux}}}
\ctr@ld@f\def\def@ultpsi{40}
\ctr@ld@f\def\def@ulttheta{25}
\ctr@ln@m\l@debut
\ctr@ln@m\n@mref
\ctr@ld@f\def\figset#1(#2){\def\t@xt@{#1}\ifx\t@xt@\empty\trtlis@rg{#2}{\Figsetwr@te}% write
    \else\keln@mde#1|%
    \def\n@mref{pr}\ifx\l@debut\n@mref\ifcurr@ntPS% projection
     \immediate\write16{*** \BS@ figset proj(...) is ignored inside a
     \BS@ psbeginfig-\BS@ psendfig block.}\else\trtlis@rg{#2}{\Figsetpr@j}\fi\else%
    \def\n@mref{wr}\ifx\l@debut\n@mref\trtlis@rg{#2}{\Figsetwr@te}\else% write
    \immediate\write16{*** Unknown keyword: \BS@ figset #1(...)}%
    \fi\fi\fi\ignorespaces}
\ctr@ld@f\def\Figsetpr@j#1=#2|{\keln@mtr#1|%
    \def\n@mref{dep}\ifx\l@debut\n@mref\Figsetd@p{#2}\else% depth (lambda)
    \def\n@mref{dis}\ifx\l@debut\n@mref%
     \ifnum\curr@ntproj=\tw@\figsetobdist(#2)\else\Figset@rr\fi\else% dist
    \def\n@mref{lam}\ifx\l@debut\n@mref\Figsetd@p{#2}\else% depth (lambda)
    \def\n@mref{lat}\ifx\l@debut\n@mref\Figsetth@{#2}\else% latitude (theta)
    \def\n@mref{lon}\ifx\l@debut\n@mref\figsetview(#2)\else% longitude (psi)
    \def\n@mref{psi}\ifx\l@debut\n@mref\figsetview(#2)\else% longitude (psi)
    \def\n@mref{tar}\ifx\l@debut\n@mref%
     \ifnum\curr@ntproj=\tw@\figsettarget[#2]\else\Figset@rr\fi\else% target point
    \def\n@mref{the}\ifx\l@debut\n@mref\Figsetth@{#2}\else% latitude (theta)
    \immediate\write16{*** Unknown attribute: \BS@ figset proj(..., #1=...).}%
    \fi\fi\fi\fi\fi\fi\fi\fi}
\ctr@ld@f\def\Figsetd@p#1{\ifnum\curr@ntproj=\z@\figsetview(\v@lPsi,#1)\else\Figset@rr\fi}
\ctr@ld@f\def\Figsetth@#1{\ifnum\curr@ntproj=\z@\Figset@rr\else\figsetview(\v@lPsi,#1)\fi}
\ctr@ld@f\def\Figset@rr{\message{*** \BS@ figset proj(): Attribute "\n@mref" ignored, incompatible
    with current projection}}
\ctr@ld@f\def\initb@undb@xTD{\wd\BminTD@=\maxdimen\ht\BminTD@=\maxdimen\dp\BminTD@=\maxdimen%
    \wd\BmaxTD@=-\maxdimen\ht\BmaxTD@=-\maxdimen\dp\BmaxTD@=-\maxdimen}
\ctr@ln@w{newbox}\Gb@x      % boite a tout faire
\ctr@ln@w{newbox}\Gb@xSC    % boite qui contient le point marker
\ctr@ln@w{newtoks}\c@nsymb  % the point marker
\ctr@ln@w{newif}\ifr@undcoord\ctr@ln@w{newif}\ifunitpr@sent
\ctr@ld@f\def\unssqrttw@{0.707106 }
\ctr@ld@f\def\figAst{\raise-1.15ex\hbox{$\ast$}}
\ctr@ld@f\def\figBullet{\raise-1.15ex\hbox{$\bullet$}}
\ctr@ld@f\def\figCirc{\raise-1.15ex\hbox{$\circ$}}
\ctr@ld@f\def\figDiamond{\raise-1.15ex\hbox{$\diamond$}}%
\ctr@ld@f\def\boxit#1#2{\leavevmode\hbox{\vrule\vbox{\hrule\vglue#1%
    \vtop{\hbox{\kern#1{#2}\kern#1}\vglue#1\hrule}}\vrule}}
\ctr@ld@f
\ctr@ld@f
\ctr@ld@f\def\c@nterpt{\ignorespaces%
    \kern-.5\wd\Gb@xSC%
    \raise-.5\ht\Gb@xSC\rlap{\hbox{\raise.5\dp\Gb@xSC\hbox{\copy\Gb@xSC}}}%
    \kern .5\wd\Gb@xSC\ignorespaces}
\ctr@ld@f\def\b@undb@xSC#1#2{{\v@lXa=#1\v@lYa=#2%
    \v@leur=\ht\Gb@xSC\advance\v@leur\dp\Gb@xSC%
    \advance\v@lXa-.5\wd\Gb@xSC\advance\v@lYa-.5\v@leur\b@undb@x{\v@lXa}{\v@lYa}%
    \advance\v@lXa\wd\Gb@xSC\advance\v@lYa\v@leur\b@undb@x{\v@lXa}{\v@lYa}}}
\ctr@ln@m\Dist@n
\ctr@ln@m\l@suite
\ctr@ld@f\def\@keldist#1#2{\edef\Dist@n{#2}\y@tiunit{\Dist@n}%
    \ifunitpr@sent#1=\Dist@n\else#1=\Dist@n\unit@\fi}
\ctr@ld@f\def\y@tiunit#1{\unitpr@sentfalse\expandafter\y@tiunit@#1:}
\ctr@ld@f\def\y@tiunit@#1#2:{\ifcat#1a\unitpr@senttrue\else\def\l@suite{#2}%
    \ifx\l@suite\empty\else\y@tiunit@#2:\fi\fi}
\ctr@ln@m\figcoord
\ctr@ld@f\def\figcoordDD#1{{\v@lX=\ptT@unit@\v@lX\v@lY=\ptT@unit@\v@lY%
    \ifr@undcoord\ifcase#1\v@leur=0.5pt\or\v@leur=0.05pt\or\v@leur=0.005pt%
    \or\v@leur=0.0005pt\else\v@leur=\z@\fi%
    \ifdim\v@lX<\z@\advance\v@lX-\v@leur\else\advance\v@lX\v@leur\fi%
    \ifdim\v@lY<\z@\advance\v@lY-\v@leur\else\advance\v@lY\v@leur\fi\fi%
    (\@ffichnb{#1}{\repdecn@mb{\v@lX}},\ifmmode\else\thinspace\fi%
    \@ffichnb{#1}{\repdecn@mb{\v@lY}})}}
\ctr@ld@f\def\@ffichnb#1#2{{\def\@@ffich{\@ffich#1(}\edef\n@mbre{#2}%
    \expandafter\@@ffich\n@mbre)}}
\ctr@ld@f\def\@ffich#1(#2.#3){{#2\ifnum#1>\z@.\fi\def\dig@ts{#3}\s@mme=\z@%
    \loop\ifnum\s@mme<#1\expandafter\@ffichdec\dig@ts:\advance\s@mme\@ne\repeat}}
\ctr@ld@f\def\@ffichdec#1#2:{\relax#1\def\dig@ts{#20}}
\ctr@ld@f\def\figcoordTD#1{{\v@lX=\ptT@unit@\v@lX\v@lY=\ptT@unit@\v@lY\v@lZ=\ptT@unit@\v@lZ%
    \ifr@undcoord\ifcase#1\v@leur=0.5pt\or\v@leur=0.05pt\or\v@leur=0.005pt%
    \or\v@leur=0.0005pt\else\v@leur=\z@\fi%
    \ifdim\v@lX<\z@\advance\v@lX-\v@leur\else\advance\v@lX\v@leur\fi%
    \ifdim\v@lY<\z@\advance\v@lY-\v@leur\else\advance\v@lY\v@leur\fi%
    \ifdim\v@lZ<\z@\advance\v@lZ-\v@leur\else\advance\v@lZ\v@leur\fi\fi%
    (\@ffichnb{#1}{\repdecn@mb{\v@lX}},\ifmmode\else\thinspace\fi%
     \@ffichnb{#1}{\repdecn@mb{\v@lY}},\ifmmode\else\thinspace\fi%
     \@ffichnb{#1}{\repdecn@mb{\v@lZ}})}}
\ctr@ld@f\def\figsetroundcoord#1{\expandafter\Figsetr@undcoord#1:\ignorespaces}
\ctr@ld@f\def\Figsetr@undcoord#1#2:{\if#1n\r@undcoordfalse\else\r@undcoordtrue\fi}
\ctr@ld@f\def\Figsetwr@te#1=#2|{\keln@mun#1|%
    \def\n@mref{m}\ifx\l@debut\n@mref\figsetmark{#2}\else% mark
    \immediate\write16{*** Unknown attribute: \BS@ figset (..., #1=...)}%
    \fi}
\ctr@ld@f\def\figsetmark#1{\c@nsymb={#1}\setbox\Gb@xSC=\hbox{\the\c@nsymb}\ignorespaces}
\ctr@ln@m\ptn@me
\ctr@ld@f\def\figsetptname#1{\def\ptn@me##1{#1}\ignorespaces}
\ctr@ld@f\def\FigWrit@L#1:#2(#3,#4){\ignorespaces\@keldist\v@leur{#3}\@keldist\delt@{#4}%
    \C@rp@r@m\def\list@num{#1}\@ecfor\p@int:=\list@num\do{\FigWrit@pt{\p@int}{#2}}}
\ctr@ld@f\def\FigWrit@pt#1#2{\FigWp@r@m{#1}{#2}\Vc@rrect\figWp@si%
    \ifdim\wd\Gb@xSC>\z@\b@undb@xSC{\v@lX}{\v@lY}\fi\figWBB@x}
\ctr@ld@f\def\FigWp@r@m#1#2{\Figg@tXY{#1}%
    \setbox\Gb@x=\hbox{\def\t@xt@{#2}\ifx\t@xt@\empty\Figg@tT{#1}\else#2\fi}\c@lprojSP}
\ctr@ld@f\let\Vc@rrect=\relax
\ctr@ld@f\let\C@rp@r@m=\relax
\ctr@ld@f\def\figwrite[#1]#2{{\ignorespaces\def\list@num{#1}\@ecfor\p@int:=\list@num\do{%
    \setbox\Gb@x=\hbox{\def\t@xt@{#2}\ifx\t@xt@\empty\Figg@tT{\p@int}\else#2\fi}%
    \Figwrit@{\p@int}}}\ignorespaces}
\ctr@ld@f\def\Figwrit@#1{\Figg@tXY{#1}\c@lprojSP%
    \rlap{\kern\v@lX\raise\v@lY\hbox{\unhcopy\Gb@x}}\v@leur=\v@lY%
    \advance\v@lY\ht\Gb@x\b@undb@x{\v@lX}{\v@lY}\advance\v@lX\wd\Gb@x%
    \v@lY=\v@leur\advance\v@lY-\dp\Gb@x\b@undb@x{\v@lX}{\v@lY}}
\ctr@ld@f\def\figwritec[#1]#2{{\ignorespaces\def\list@num{#1}%
    \@ecfor\p@int:=\list@num\do{\Figwrit@c{\p@int}{#2}}}\ignorespaces}
\ctr@ld@f\def\Figwrit@c#1#2{\FigWp@r@m{#1}{#2}%
    \rlap{\kern\v@lX\raise\v@lY\hbox{\rlap{\kern-.5\wd\Gb@x%
    \raise-.5\ht\Gb@x\hbox{\raise.5\dp\Gb@x\hbox{\unhcopy\Gb@x}}}}}%
    \v@leur=\ht\Gb@x\advance\v@leur\dp\Gb@x%
    \advance\v@lX-.5\wd\Gb@x\advance\v@lY-.5\v@leur\b@undb@x{\v@lX}{\v@lY}%
    \advance\v@lX\wd\Gb@x\advance\v@lY\v@leur\b@undb@x{\v@lX}{\v@lY}}
\ctr@ld@f\def\figwritep[#1]{{\ignorespaces\def\list@num{#1}\setbox\Gb@x=\hbox{\c@nterpt}%
    \@ecfor\p@int:=\list@num\do{\Figwrit@{\p@int}}}\ignorespaces}
\ctr@ld@f\def\figwritew#1:#2(#3){\figwritegcw#1:{#2}(#3,0pt)}
\ctr@ld@f\def\figwritee#1:#2(#3){\figwritegce#1:{#2}(#3,0pt)}
\ctr@ld@f\def\figwriten#1:#2(#3){{\def\Vc@rrect{\v@lZ=\v@leur\advance\v@lZ\dp\Gb@x}%
    \Figwrit@NS#1:{#2}(#3)}\ignorespaces}
\ctr@ld@f\def\figwrites#1:#2(#3){{\def\Vc@rrect{\v@lZ=-\v@leur\advance\v@lZ-\ht\Gb@x}%
    \Figwrit@NS#1:{#2}(#3)}\ignorespaces}
\ctr@ld@f\def\Figwrit@NS#1:#2(#3){\let\figWp@si=\FigWp@siNS\let\figWBB@x=\FigWBB@xNS%
    \FigWrit@L#1:{#2}(#3,0pt)}
\ctr@ld@f\def\FigWp@siNS{\rlap{\kern\v@lX\raise\v@lY\hbox{\rlap{\kern-.5\wd\Gb@x%
    \raise\v@lZ\hbox{\unhcopy\Gb@x}}\c@nterpt}}}
\ctr@ld@f\def\FigWBB@xNS{\advance\v@lY\v@lZ%
    \advance\v@lY-\dp\Gb@x\advance\v@lX-.5\wd\Gb@x\b@undb@x{\v@lX}{\v@lY}%
    \advance\v@lY\ht\Gb@x\advance\v@lY\dp\Gb@x%
    \advance\v@lX\wd\Gb@x\b@undb@x{\v@lX}{\v@lY}}
\ctr@ld@f\def\figwritenw#1:#2(#3){{\let\figWp@si=\FigWp@sigW\let\figWBB@x=\FigWBB@xgWE%
    \def\C@rp@r@m{\v@leur=\unssqrttw@\v@leur\delt@=\v@leur%
    \ifdim\delt@=\z@\delt@=\epsil@n\fi}\let@xte={-}\FigWrit@L#1:{#2}(#3,0pt)}\ignorespaces}
\ctr@ld@f\def\figwritesw#1:#2(#3){{\let\figWp@si=\FigWp@sigW\let\figWBB@x=\FigWBB@xgWE%
    \def\C@rp@r@m{\v@leur=\unssqrttw@\v@leur\delt@=-\v@leur%
    \ifdim\delt@=\z@\delt@=-\epsil@n\fi}\let@xte={-}\FigWrit@L#1:{#2}(#3,0pt)}\ignorespaces}
\ctr@ld@f\def\figwritene#1:#2(#3){{\let\figWp@si=\FigWp@sigE\let\figWBB@x=\FigWBB@xgWE%
    \def\C@rp@r@m{\v@leur=\unssqrttw@\v@leur\delt@=\v@leur%
    \ifdim\delt@=\z@\delt@=\epsil@n\fi}\let@xte={}\FigWrit@L#1:{#2}(#3,0pt)}\ignorespaces}
\ctr@ld@f\def\figwritese#1:#2(#3){{\let\figWp@si=\FigWp@sigE\let\figWBB@x=\FigWBB@xgWE%
    \def\C@rp@r@m{\v@leur=\unssqrttw@\v@leur\delt@=-\v@leur%
    \ifdim\delt@=\z@\delt@=-\epsil@n\fi}\let@xte={}\FigWrit@L#1:{#2}(#3,0pt)}\ignorespaces}
\ctr@ld@f\def\figwritegw#1:#2(#3,#4){{\let\figWp@si=\FigWp@sigW\let\figWBB@x=\FigWBB@xgWE%
    \let@xte={-}\FigWrit@L#1:{#2}(#3,#4)}\ignorespaces}
\ctr@ld@f\def\figwritege#1:#2(#3,#4){{\let\figWp@si=\FigWp@sigE\let\figWBB@x=\FigWBB@xgWE%
    \let@xte={}\FigWrit@L#1:{#2}(#3,#4)}\ignorespaces}
\ctr@ld@f\def\FigWp@sigW{\v@lXa=\z@\v@lYa=\ht\Gb@x\advance\v@lYa\dp\Gb@x%
    \ifdim\delt@>\z@\relax%
    \rlap{\kern\v@lX\raise\v@lY\hbox{\rlap{\kern-\wd\Gb@x\kern-\v@leur%
          \raise\delt@\hbox{\raise\dp\Gb@x\hbox{\unhcopy\Gb@x}}}\c@nterpt}}%
    \else\ifdim\delt@<\z@\relax\v@lYa=-\v@lYa%
    \rlap{\kern\v@lX\raise\v@lY\hbox{\rlap{\kern-\wd\Gb@x\kern-\v@leur%
          \raise\delt@\hbox{\raise-\ht\Gb@x\hbox{\unhcopy\Gb@x}}}\c@nterpt}}%
    \else\v@lXa=-.5\v@lYa%
    \rlap{\kern\v@lX\raise\v@lY\hbox{\rlap{\kern-\wd\Gb@x\kern-\v@leur%
          \raise-.5\ht\Gb@x\hbox{\raise.5\dp\Gb@x\hbox{\unhcopy\Gb@x}}}\c@nterpt}}%
    \fi\fi}
\ctr@ld@f\def\FigWp@sigE{\v@lXa=\z@\v@lYa=\ht\Gb@x\advance\v@lYa\dp\Gb@x%
    \ifdim\delt@>\z@\relax%
    \rlap{\kern\v@lX\raise\v@lY\hbox{\c@nterpt\kern\v@leur%
          \raise\delt@\hbox{\raise\dp\Gb@x\hbox{\unhcopy\Gb@x}}}}%
    \else\ifdim\delt@<\z@\relax\v@lYa=-\v@lYa%
    \rlap{\kern\v@lX\raise\v@lY\hbox{\c@nterpt\kern\v@leur%
          \raise\delt@\hbox{\raise-\ht\Gb@x\hbox{\unhcopy\Gb@x}}}}%
    \else\v@lXa=-.5\v@lYa%
    \rlap{\kern\v@lX\raise\v@lY\hbox{\c@nterpt\kern\v@leur%
          \raise-.5\ht\Gb@x\hbox{\raise.5\dp\Gb@x\hbox{\unhcopy\Gb@x}}}}%
    \fi\fi}
\ctr@ld@f\def\FigWBB@xgWE{\advance\v@lY\delt@%
    \advance\v@lX\the\let@xte\v@leur\advance\v@lY\v@lXa\b@undb@x{\v@lX}{\v@lY}%
    \advance\v@lX\the\let@xte\wd\Gb@x\advance\v@lY\v@lYa\b@undb@x{\v@lX}{\v@lY}}
\ctr@ld@f\def\figwritegcw#1:#2(#3,#4){{\let\figWp@si=\FigWp@sigcW\let\figWBB@x=\FigWBB@xgcWE%
    \let@xte={-}\FigWrit@L#1:{#2}(#3,#4)}\ignorespaces}
\ctr@ld@f\def\figwritegce#1:#2(#3,#4){{\let\figWp@si=\FigWp@sigcE\let\figWBB@x=\FigWBB@xgcWE%
    \let@xte={}\FigWrit@L#1:{#2}(#3,#4)}\ignorespaces}
\ctr@ld@f\def\FigWp@sigcW{\rlap{\kern\v@lX\raise\v@lY\hbox{\rlap{\kern-\wd\Gb@x\kern-\v@leur%
     \raise-.5\ht\Gb@x\hbox{\raise\delt@\hbox{\raise.5\dp\Gb@x\hbox{\unhcopy\Gb@x}}}}%
     \c@nterpt}}}
\ctr@ld@f\def\FigWp@sigcE{\rlap{\kern\v@lX\raise\v@lY\hbox{\c@nterpt\kern\v@leur%
    \raise-.5\ht\Gb@x\hbox{\raise\delt@\hbox{\raise.5\dp\Gb@x\hbox{\unhcopy\Gb@x}}}}}}
\ctr@ld@f\def\FigWBB@xgcWE{\v@lZ=\ht\Gb@x\advance\v@lZ\dp\Gb@x%
    \advance\v@lX\the\let@xte\v@leur\advance\v@lY\delt@\advance\v@lY.5\v@lZ%
    \b@undb@x{\v@lX}{\v@lY}%
    \advance\v@lX\the\let@xte\wd\Gb@x\advance\v@lY-\v@lZ\b@undb@x{\v@lX}{\v@lY}}
\ctr@ld@f\def\figwritebn#1:#2(#3){{\def\Vc@rrect{\v@lZ=\v@leur}\Figwrit@NS#1:{#2}(#3)}\ignorespaces}
\ctr@ld@f\def\figwritebs#1:#2(#3){{\def\Vc@rrect{\v@lZ=-\v@leur}\Figwrit@NS#1:{#2}(#3)}\ignorespaces}
\ctr@ld@f\def\figwritebw#1:#2(#3){{\let\figWp@si=\FigWp@sibW\let\figWBB@x=\FigWBB@xbWE%
    \let@xte={-}\FigWrit@L#1:{#2}(#3,0pt)}\ignorespaces}
\ctr@ld@f\def\figwritebe#1:#2(#3){{\let\figWp@si=\FigWp@sibE\let\figWBB@x=\FigWBB@xbWE%
    \let@xte={}\FigWrit@L#1:{#2}(#3,0pt)}\ignorespaces}
\ctr@ld@f\def\FigWp@sibW{\rlap{\kern\v@lX\raise\v@lY\hbox{\rlap{\kern-\wd\Gb@x\kern-\v@leur%
          \hbox{\unhcopy\Gb@x}}\c@nterpt}}}
\ctr@ld@f\def\FigWp@sibE{\rlap{\kern\v@lX\raise\v@lY\hbox{\c@nterpt\kern\v@leur%
          \hbox{\unhcopy\Gb@x}}}}
\ctr@ld@f\def\FigWBB@xbWE{\v@lZ=\ht\Gb@x\advance\v@lZ\dp\Gb@x%
    \advance\v@lX\the\let@xte\v@leur\advance\v@lY\ht\Gb@x\b@undb@x{\v@lX}{\v@lY}%
    \advance\v@lX\the\let@xte\wd\Gb@x\advance\v@lY-\v@lZ\b@undb@x{\v@lX}{\v@lY}}
\ctr@ln@w{newread}\frf@g  \ctr@ln@w{newwrite}\fwf@g
\ctr@ln@w{newif}\ifcurr@ntPS
\ctr@ln@w{newif}\ifps@cri
\ctr@ln@w{newif}\ifUse@llipse
\ctr@ln@w{newif}\ifpsdebugmode \psdebugmodefalse 
\ctr@ln@w{newif}\ifPDFm@ke
\ifx\pdfliteral\undefined\else\ifnum\pdfoutput>\z@\PDFm@ketrue\fi\fi
\ctr@ld@f\def\initPDF@rDVI{%
\ifPDFm@ke
 \let\figscan=\figscan@E
 \let\newGr@FN=\newGr@FNPDF
 \ctr@ld@f\def\c@mcurveto{c}
 \ctr@ld@f\def\c@mfill{f}
 \ctr@ld@f\def\c@mgsave{q}
 \ctr@ld@f\def\c@mgrestore{Q}
 \ctr@ld@f\def\c@mlineto{l}
 \ctr@ld@f\def\c@mmoveto{m}
 \ctr@ld@f\def\c@msetgray{g}     \ctr@ld@f\def\c@msetgrayStroke{G}
 \ctr@ld@f\def\c@msetcmykcolor{k}\ctr@ld@f\def\c@msetcmykcolorStroke{K}
 \ctr@ld@f\def\c@msetrgbcolor{rg}\ctr@ld@f\def\c@msetrgbcolorStroke{RG}
 \ctr@ld@f\def\d@fprimarC@lor{\curr@ntcolor\space\curr@ntcolorc@md%
               \space\curr@ntcolor\space\curr@ntcolorc@mdStroke}
 \ctr@ld@f\def\d@fsecondC@lor{\sec@ndcolor\space\sec@ndcolorc@md%
               \space\sec@ndcolor\space\sec@ndcolorc@mdStroke}
 \ctr@ld@f\def\d@fthirdC@lor{\th@rdcolor\space\th@rdcolorc@md%
              \space\th@rdcolor\space\th@rdcolorc@mdStroke}
 \ctr@ld@f\def\c@msetdash{d}
 \ctr@ld@f\def\c@msetlinejoin{j}
 \ctr@ld@f\def\c@msetlinewidth{w}
 \ctr@ld@f\def\f@gclosestroke{\immediate\write\fwf@g{s}}
 \ctr@ld@f\def\f@gfill{\immediate\write\fwf@g{\fillc@md}}% Voir la def de \fillc@md ****
 \ctr@ld@f\def\f@gnewpath{}
 \ctr@ld@f\def\f@gstroke{\immediate\write\fwf@g{S}}
\else
 \let\figinsertE=\figinsert
 \let\newGr@FN=\newGr@FNDVI
 \ctr@ld@f\def\c@mcurveto{curveto}
 \ctr@ld@f\def\c@mfill{fill}
 \ctr@ld@f\def\c@mgsave{gsave}
 \ctr@ld@f\def\c@mgrestore{grestore}
 \ctr@ld@f\def\c@mlineto{lineto}
 \ctr@ld@f\def\c@mmoveto{moveto}
 \ctr@ld@f\def\c@msetgray{setgray}          \ctr@ld@f\def\c@msetgrayStroke{}
 \ctr@ld@f\def\c@msetcmykcolor{setcmykcolor}\ctr@ld@f\def\c@msetcmykcolorStroke{}
 \ctr@ld@f\def\c@msetrgbcolor{setrgbcolor}  \ctr@ld@f\def\c@msetrgbcolorStroke{}
 \ctr@ld@f\def\d@fprimarC@lor{\curr@ntcolor\space\curr@ntcolorc@md}
 \ctr@ld@f\def\d@fsecondC@lor{\sec@ndcolor\space\sec@ndcolorc@md}
 \ctr@ld@f\def\d@fthirdC@lor{\th@rdcolor\space\th@rdcolorc@md}
 \ctr@ld@f\def\c@msetdash{setdash}
 \ctr@ld@f\def\c@msetlinejoin{setlinejoin}
 \ctr@ld@f\def\c@msetlinewidth{setlinewidth}
 \ctr@ld@f\def\f@gclosestroke{\immediate\write\fwf@g{closepath\space stroke}}
 \ctr@ld@f\def\f@gfill{\immediate\write\fwf@g{\fillc@md}}
 \ctr@ld@f\def\f@gnewpath{\immediate\write\fwf@g{newpath}}
 \ctr@ld@f\def\f@gstroke{\immediate\write\fwf@g{stroke}}
\fi}
\ctr@ld@f\def\c@pypsfile#1#2{\c@pyfil@{\immediate\write#1}{#2}}
\ctr@ld@f\def\Figinclud@PDF#1#2{\openin\frf@g=#1\pdfliteral{q #2 0 0 #2 0 0 cm}%
    \c@pyfil@{\pdfliteral}{\frf@g}\pdfliteral{Q}\closein\frf@g}
\ctr@ln@w{newif}\ifmored@ta
\ctr@ln@m\bl@nkline
\ctr@ld@f\def\c@pyfil@#1#2{\def\bl@nkline{\par}{\catcode`\%=12
    \loop\ifeof#2\mored@tafalse\else\mored@tatrue\immediate\read#2 to\tr@c
    \ifx\tr@c\bl@nkline\else#1{\tr@c}\fi\fi\ifmored@ta\repeat}}
\ctr@ld@f\def\keln@mun#1#2|{\def\l@debut{#1}\def\l@suite{#2}}
\ctr@ld@f\def\keln@mde#1#2#3|{\def\l@debut{#1#2}\def\l@suite{#3}}
\ctr@ld@f\def\keln@mtr#1#2#3#4|{\def\l@debut{#1#2#3}\def\l@suite{#4}}
\ctr@ld@f\def\keln@mqu#1#2#3#4#5|{\def\l@debut{#1#2#3#4}\def\l@suite{#5}}
\ctr@ld@f\let\@psffilein=\frf@g % file to \read
\ctr@ln@w{newif}\if@psffileok    % continue looking for the bounding box?
\ctr@ln@w{newif}\if@psfbbfound   % success?
\ctr@ln@w{newif}\if@psfverbose   % report what you're making?
\@psfverbosetrue
\ctr@ln@m\@psfllx \ctr@ln@m\@psflly
\ctr@ln@m\@psfurx \ctr@ln@m\@psfury
\ctr@ln@m\resetcolonc@tcode
\ctr@ld@f\def\@psfgetbb#1{\global\@psfbbfoundfalse%
\global\def\@psfllx{0}\global\def\@psflly{0}%
\global\def\@psfurx{30}\global\def\@psfury{30}%
\openin\@psffilein=#1\relax
\ifeof\@psffilein\errmessage{I couldn't open #1, will ignore it}\else
   \edef\resetcolonc@tcode{\catcode`\noexpand\:\the\catcode`\:\relax}%
   {\@psffileoktrue \chardef\other=12
    \def\do##1{\catcode`##1=\other}\dospecials \catcode`\ =10 \resetcolonc@tcode
    \loop
       \read\@psffilein to \@psffileline
       \ifeof\@psffilein\@psffileokfalse\else
          \expandafter\@psfaux\@psffileline:. \\%
       \fi
   \if@psffileok\repeat
   \if@psfbbfound\else
    \if@psfverbose\message{No bounding box comment in #1; using defaults}\fi\fi
   }\closein\@psffilein\fi}%
\ctr@ln@m\@psfbblit
\ctr@ln@m\@psfpercent
{\catcode`\%=12 \global\let\@psfpercent=%\global\def\@psfbblit{%BoundingBox}}%
\ctr@ln@m\@psfaux
\long\def\@psfaux#1#2:#3\\{\ifx#1\@psfpercent
   \def\testit{#2}\ifx\testit\@psfbblit
      \@psfgrab #3 . . . \\%
      \@psffileokfalse
      \global\@psfbbfoundtrue
   \fi\else\ifx#1\par\else\@psffileokfalse\fi\fi}%
\ctr@ld@f\def\@psfempty{}%
\ctr@ld@f\def\@psfgrab #1 #2 #3 #4 #5\\{%
\global\def\@psfllx{#1}\ifx\@psfllx\@psfempty
      \@psfgrab #2 #3 #4 #5 .\\\else
   \global\def\@psflly{#2}%
   \global\def\@psfurx{#3}\global\def\@psfury{#4}\fi}%
\ctr@ld@f\def\PSwrit@cmd#1#2#3{{\Figg@tXY{#1}\c@lprojSP\b@undb@x{\v@lX}{\v@lY}%
    \v@lX=\ptT@ptps\v@lX\v@lY=\ptT@ptps\v@lY%
    \immediate\write#3{\repdecn@mb{\v@lX}\space\repdecn@mb{\v@lY}\space#2}}}
\ctr@ld@f\def\PSwrit@cmdS#1#2#3#4#5{{\Figg@tXY{#1}\c@lprojSP\b@undb@x{\v@lX}{\v@lY}%
    \global\result@t=\v@lX\global\result@@t=\v@lY%
    \v@lX=\ptT@ptps\v@lX\v@lY=\ptT@ptps\v@lY%
    \immediate\write#3{\repdecn@mb{\v@lX}\space\repdecn@mb{\v@lY}\space#2}}%
    \edef#4{\the\result@t}\edef#5{\the\result@@t}}
\ctr@ld@f\def\psaltitude#1[#2,#3,#4]{{\ifcurr@ntPS\ifps@cri%
    \PSc@mment{psaltitude Square Dim=#1, Triangle=[#2 / #3,#4]}%
    \s@uvc@ntr@l\et@tpsaltitude\resetc@ntr@l{2}\figptorthoprojline-5:=#2/#3,#4/%
    \figvectP -1[#3,#4]\n@rminf{\v@leur}{-1}\vecunit@{-3}{-1}%
    \figvectP -1[-5,#3]\n@rminf{\v@lmin}{-1}\figvectP -2[-5,#4]\n@rminf{\v@lmax}{-2}%
    \ifdim\v@lmin<\v@lmax\s@mme=#3\else\v@lmax=\v@lmin\s@mme=#4\fi%
    \figvectP -4[-5,#2]\vecunit@{-4}{-4}\delt@=#1\unit@%
    \edef\t@ille{\repdecn@mb{\delt@}}\figpttra-1:=-5/\t@ille,-3/%
    \figptstra-3=-5,-1/\t@ille,-4/\psline[#2,-5]\s@uvdash{\typ@dash}%
    \pssetdash{\defaultdash}\psline[-1,-2,-3]\pssetdash{\typ@dash}%
    \ifdim\v@leur<\v@lmax\Pss@tsecondSt\psline[-5,\the\s@mme]\Psrest@reSt\fi%
    \PSc@mment{End psaltitude}\resetc@ntr@l\et@tpsaltitude\fi\fi}}
\ctr@ld@f\def\Ps@rcerc#1;#2(#3,#4){\ellBB@x#1;#2,#2(#3,#4,0)%
    \f@gnewpath{\delt@=#2\unit@\delt@=\ptT@ptps\delt@%
    \BdingB@xfalse%
    \PSwrit@cmd{#1}{\repdecn@mb{\delt@}\space #3\space #4\space arc}{\fwf@g}}}
\ctr@ln@m\psarccirc
\ctr@ld@f\def\psarccircDD#1;#2(#3,#4){\ifcurr@ntPS\ifps@cri%
    \PSc@mment{psarccircDD Center=#1 ; Radius=#2 (Ang1=#3, Ang2=#4)}%
    \iffillm@de\Ps@rcerc#1;#2(#3,#4)%
    \f@gfill%
    \else\Ps@rcerc#1;#2(#3,#4)\f@gstroke\fi%
    \PSc@mment{End psarccircDD}\fi\fi}
\ctr@ld@f\def\psarccircTD#1,#2,#3;#4(#5,#6){{\ifcurr@ntPS\ifps@cri\s@uvc@ntr@l\et@tpsarccircTD%
    \PSc@mment{psarccircTD Center=#1,P1=#2,P2=#3 ; Radius=#4 (Ang1=#5, Ang2=#6)}%
    \setc@ntr@l{2}\c@lExtAxes#1,#2,#3(#4)\psarcellPATD#1,-4,-5(#5,#6)%
    \PSc@mment{End psarccircTD}\resetc@ntr@l\et@tpsarccircTD\fi\fi}}
\ctr@ld@f\def\c@lExtAxes#1,#2,#3(#4){%
    \figvectPTD-5[#1,#2]\vecunit@{-5}{-5}\figvectNTD-4[#1,#2,#3]\vecunit@{-4}{-4}%
    \figvectNVTD-3[-4,-5]\delt@=#4\unit@\edef\r@yon{\repdecn@mb{\delt@}}%
    \figpttra-4:=#1/\r@yon,-5/\figpttra-5:=#1/\r@yon,-3/}
\ctr@ln@m\psarccircP
\ctr@ld@f\def\psarccircPDD#1;#2[#3,#4]{{\ifcurr@ntPS\ifps@cri\s@uvc@ntr@l\et@tpsarccircPDD%
    \PSc@mment{psarccircPDD Center=#1; Radius=#2, [P1=#3, P2=#4]}%
    \Ps@ngleparam#1;#2[#3,#4]\ifdim\v@lmin>\v@lmax\advance\v@lmax\DePI@deg\fi%
    \edef\@ngdeb{\repdecn@mb{\v@lmin}}\edef\@ngfin{\repdecn@mb{\v@lmax}}%
    \psarccirc#1;\r@dius(\@ngdeb,\@ngfin)%
    \PSc@mment{End psarccircPDD}\resetc@ntr@l\et@tpsarccircPDD\fi\fi}}
\ctr@ld@f\def\psarccircPTD#1;#2[#3,#4,#5]{{\ifcurr@ntPS\ifps@cri\s@uvc@ntr@l\et@tpsarccircPTD%
    \PSc@mment{psarccircPTD Center=#1; Radius=#2, [P1=#3, P2=#4, P3=#5]}%
    \setc@ntr@l{2}\c@lExtAxes#1,#3,#5(#2)\psarcellPP#1,-4,-5[#3,#4]%
    \PSc@mment{End psarccircPTD}\resetc@ntr@l\et@tpsarccircPTD\fi\fi}}
\ctr@ld@f\def\Ps@ngleparam#1;#2[#3,#4]{\setc@ntr@l{2}%
    \figvectPDD-1[#1,#3]\vecunit@{-1}{-1}\Figg@tXY{-1}\arct@n\v@lmin(\v@lX,\v@lY)%
    \figvectPDD-2[#1,#4]\vecunit@{-2}{-2}\Figg@tXY{-2}\arct@n\v@lmax(\v@lX,\v@lY)%
    \v@lmin=\rdT@deg\v@lmin\v@lmax=\rdT@deg\v@lmax%
    \v@leur=#2pt\maxim@m{\mili@u}{-\v@leur}{\v@leur}%
    \edef\r@dius{\repdecn@mb{\mili@u}}}
\ctr@ld@f\def\Ps@rcercBz#1;#2(#3,#4){\Ps@rellBz#1;#2,#2(#3,#4,0)}
\ctr@ld@f\def\Ps@rellBz#1;#2,#3(#4,#5,#6){%
    \ellBB@x#1;#2,#3(#4,#5,#6)\BdingB@xfalse%
    \c@lNbarcs{#4}{#5}\v@leur=#4pt\setc@ntr@l{2}\figptell-13::#1;#2,#3(#4,#6)%
    \f@gnewpath\PSwrit@cmd{-13}{\c@mmoveto}{\fwf@g}%
    \s@mme=\z@\bcl@rellBz#1;#2,#3(#6)\BdingB@xtrue}
\ctr@ld@f\def\bcl@rellBz#1;#2,#3(#4){\relax%
    \ifnum\s@mme<\p@rtent\advance\s@mme\@ne%
    \advance\v@leur\delt@\edef\@ngle{\repdecn@mb\v@leur}\figptell-14::#1;#2,#3(\@ngle,#4)%
    \advance\v@leur\delt@\edef\@ngle{\repdecn@mb\v@leur}\figptell-15::#1;#2,#3(\@ngle,#4)%
    \advance\v@leur\delt@\edef\@ngle{\repdecn@mb\v@leur}\figptell-16::#1;#2,#3(\@ngle,#4)%
    \figptscontrolDD-18[-13,-14,-15,-16]%
    \PSwrit@cmd{-18}{}{\fwf@g}\PSwrit@cmd{-17}{}{\fwf@g}%
    \PSwrit@cmd{-16}{\c@mcurveto}{\fwf@g}%
    \figptcopyDD-13:/-16/\bcl@rellBz#1;#2,#3(#4)\fi}
\ctr@ld@f\def\Ps@rell#1;#2,#3(#4,#5,#6){\ellBB@x#1;#2,#3(#4,#5,#6)%
    \f@gnewpath{\v@lmin=#2\unit@\v@lmin=\ptT@ptps\v@lmin%
    \v@lmax=#3\unit@\v@lmax=\ptT@ptps\v@lmax\BdingB@xfalse%
    \PSwrit@cmd{#1}%
    {#6\space\repdecn@mb{\v@lmin}\space\repdecn@mb{\v@lmax}\space #4\space #5\space ellipse}{\fwf@g}}%
    \global\Use@llipsetrue}
\ctr@ln@m\psarcell
\ctr@ld@f\def\psarcellDD#1;#2,#3(#4,#5,#6){{\ifcurr@ntPS\ifps@cri%
    \PSc@mment{psarcellDD Center=#1 ; XRad=#2, YRad=#3 (Ang1=#4, Ang2=#5, Inclination=#6)}%
    \iffillm@de\Ps@rell#1;#2,#3(#4,#5,#6)%
    \f@gfill%
    \else\Ps@rell#1;#2,#3(#4,#5,#6)\f@gstroke\fi%
    \PSc@mment{End psarcellDD}\fi\fi}}
\ctr@ld@f\def\psarcellTD#1;#2,#3(#4,#5,#6){{\ifcurr@ntPS\ifps@cri\s@uvc@ntr@l\et@tpsarcellTD%
    \PSc@mment{psarcellTD Center=#1 ; XRad=#2, YRad=#3 (Ang1=#4, Ang2=#5, Inclination=#6)}%
    \setc@ntr@l{2}\figpttraC -8:=#1/#2,0,0/\figpttraC -7:=#1/0,#3,0/%
    \figvectC -4(0,0,1)\figptsrot -8=-8,-7/#1,#6,-4/\psarcellPATD#1,-8,-7(#4,#5)%
    \PSc@mment{End psarcellTD}\resetc@ntr@l\et@tpsarcellTD\fi\fi}}
\ctr@ln@m\psarcellPA
\ctr@ld@f\def\psarcellPADD#1,#2,#3(#4,#5){{\ifcurr@ntPS\ifps@cri\s@uvc@ntr@l\et@tpsarcellPADD%
    \PSc@mment{psarcellPADD Center=#1,PtAxis1=#2,PtAxis2=#3 (Ang1=#4, Ang2=#5)}%
    \setc@ntr@l{2}\figvectPDD-1[#1,#2]\vecunit@DD{-1}{-1}\v@lX=\ptT@unit@\result@t%
    \edef\XR@d{\repdecn@mb{\v@lX}}\Figg@tXY{-1}\arct@n\v@lmin(\v@lX,\v@lY)%
    \v@lmin=\rdT@deg\v@lmin\edef\Inclin@{\repdecn@mb{\v@lmin}}%
    \figgetdist\YR@d[#1,#3]\psarcellDD#1;\XR@d,\YR@d(#4,#5,\Inclin@)%
    \PSc@mment{End psarcellPADD}\resetc@ntr@l\et@tpsarcellPADD\fi\fi}}
\ctr@ld@f\def\psarcellPATD#1,#2,#3(#4,#5){{\ifcurr@ntPS\ifps@cri\s@uvc@ntr@l\et@tpsarcellPATD%
    \PSc@mment{psarcellPATD Center=#1,PtAxis1=#2,PtAxis2=#3 (Ang1=#4, Ang2=#5)}%
    \iffillm@de\Ps@rellPATD#1,#2,#3(#4,#5)%
    \f@gfill%
    \else\Ps@rellPATD#1,#2,#3(#4,#5)\f@gstroke\fi%
    \PSc@mment{End psarcellPATD}\resetc@ntr@l\et@tpsarcellPATD\fi\fi}}
\ctr@ld@f\def\Ps@rellPATD#1,#2,#3(#4,#5){\let\c@lprojSP=\relax%
    \setc@ntr@l{2}\figvectPTD-1[#1,#2]\figvectPTD-2[#1,#3]\c@lNbarcs{#4}{#5}%
    \v@leur=#4pt\c@lptellP{#1}{-1}{-2}\Figptpr@j-5:/-3/%
    \f@gnewpath\PSwrit@cmdS{-5}{\c@mmoveto}{\fwf@g}{\X@un}{\Y@un}%
    \edef\C@nt@r{#1}\s@mme=\z@\bcl@rellPATD}
\ctr@ld@f\def\bcl@rellPATD{\relax%
    \ifnum\s@mme<\p@rtent\advance\s@mme\@ne%
    \advance\v@leur\delt@\c@lptellP{\C@nt@r}{-1}{-2}\Figptpr@j-4:/-3/%
    \advance\v@leur\delt@\c@lptellP{\C@nt@r}{-1}{-2}\Figptpr@j-6:/-3/%
    \advance\v@leur\delt@\c@lptellP{\C@nt@r}{-1}{-2}\Figptpr@j-3:/-3/%
    \v@lX=\z@\v@lY=\z@\Figtr@nptDD{-5}{-5}\Figtr@nptDD{2}{-3}%
    \divide\v@lX\@vi\divide\v@lY\@vi%
    \Figtr@nptDD{3}{-4}\Figtr@nptDD{-1.5}{-6}\v@lmin=\v@lX\v@lmax=\v@lY%
    \v@lX=\z@\v@lY=\z@\Figtr@nptDD{2}{-5}\Figtr@nptDD{-5}{-3}%
    \divide\v@lX\@vi\divide\v@lY\@vi\Figtr@nptDD{-1.5}{-4}\Figtr@nptDD{3}{-6}%
    \BdingB@xfalse%
    \Figp@intregDD-4:(\v@lmin,\v@lmax)\PSwrit@cmdS{-4}{}{\fwf@g}{\X@de}{\Y@de}%
    \Figp@intregDD-4:(\v@lX,\v@lY)\PSwrit@cmdS{-4}{}{\fwf@g}{\X@tr}{\Y@tr}%
    \BdingB@xtrue\PSwrit@cmdS{-3}{\c@mcurveto}{\fwf@g}{\X@qu}{\Y@qu}%
    \B@zierBB@x{1}{\Y@un}(\X@un,\X@de,\X@tr,\X@qu)%
    \B@zierBB@x{2}{\X@un}(\Y@un,\Y@de,\Y@tr,\Y@qu)%
    \edef\X@un{\X@qu}\edef\Y@un{\Y@qu}\figptcopyDD-5:/-3/\bcl@rellPATD\fi}
\ctr@ld@f\def\c@lNbarcs#1#2{%
    \delt@=#2pt\advance\delt@-#1pt\maxim@m{\v@lmax}{\delt@}{-\delt@}%
    \v@leur=\v@lmax\divide\v@leur45 \p@rtentiere{\p@rtent}{\v@leur}\advance\p@rtent\@ne%
    \s@mme=\p@rtent\multiply\s@mme\thr@@\divide\delt@\s@mme}
\ctr@ld@f\def\psarcellPP#1,#2,#3[#4,#5]{{\ifcurr@ntPS\ifps@cri\s@uvc@ntr@l\et@tpsarcellPP%
    \PSc@mment{psarcellPP Center=#1,PtAxis1=#2,PtAxis2=#3 [Point1=#4, Point2=#5]}%
    \setc@ntr@l{2}\figvectP-2[#1,#3]\vecunit@{-2}{-2}\v@lmin=\result@t%
    \invers@{\v@lmax}{\v@lmin}%
    \figvectP-1[#1,#2]\vecunit@{-1}{-1}\v@leur=\result@t%
    \v@leur=\repdecn@mb{\v@lmax}\v@leur\edef\AsB@{\repdecn@mb{\v@leur}}% a/b
    \c@lAngle{#1}{#4}{\v@lmin}\edef\@ngdeb{\repdecn@mb{\v@lmin}}%
    \c@lAngle{#1}{#5}{\v@lmax}\ifdim\v@lmin>\v@lmax\advance\v@lmax\DePI@deg\fi%
    \edef\@ngfin{\repdecn@mb{\v@lmax}}\psarcellPA#1,#2,#3(\@ngdeb,\@ngfin)%
    \PSc@mment{End psarcellPP}\resetc@ntr@l\et@tpsarcellPP\fi\fi}}
\ctr@ld@f\def\c@lAngle#1#2#3{\figvectP-3[#1,#2]%
    \c@lproscal\delt@[-3,-1]\c@lproscal\v@leur[-3,-2]%
    \v@leur=\AsB@\v@leur\arct@n#3(\delt@,\v@leur)#3=\rdT@deg#3}
\ctr@ln@w{newif}\if@rrowratio\@rrowratiotrue
\ctr@ln@w{newif}\if@rrowhfill
\ctr@ln@w{newif}\if@rrowhout
\ctr@ld@f\def\Psset@rrowhe@d#1=#2|{\keln@mun#1|%
    \def\n@mref{a}\ifx\l@debut\n@mref\pssetarrowheadangle{#2}\else% angle
    \def\n@mref{f}\ifx\l@debut\n@mref\pssetarrowheadfill{#2}\else% fillmode
    \def\n@mref{l}\ifx\l@debut\n@mref\pssetarrowheadlength{#2}\else% length
    \def\n@mref{o}\ifx\l@debut\n@mref\pssetarrowheadout{#2}\else% out
    \def\n@mref{r}\ifx\l@debut\n@mref\pssetarrowheadratio{#2}\else% ratio
    \immediate\write16{*** Unknown attribute: \BS@ psset arrowhead(..., #1=...)}%
    \fi\fi\fi\fi\fi}
\ctr@ln@m\@rrowheadangle
\ctr@ln@m\C@AHANG \ctr@ln@m\S@AHANG \ctr@ln@m\UNSS@N
\ctr@ld@f\def\pssetarrowheadangle#1{\edef\@rrowheadangle{#1}{\c@ssin{\C@}{\S@}{#1}%
    \xdef\C@AHANG{\C@}\xdef\S@AHANG{\S@}\v@lmax=\S@ pt%
    \invers@{\v@leur}{\v@lmax}\maxim@m{\v@leur}{\v@leur}{-\v@leur}%
    \xdef\UNSS@N{\the\v@leur}}}
\ctr@ld@f\def\pssetarrowheadfill#1{\expandafter\set@rrowhfill#1:}
\ctr@ld@f\def\set@rrowhfill#1#2:{\if#1n\@rrowhfillfalse\else\@rrowhfilltrue\fi}
\ctr@ld@f\def\pssetarrowheadout#1{\expandafter\set@rrowhout#1:}
\ctr@ld@f\def\set@rrowhout#1#2:{\if#1n\@rrowhoutfalse\else\@rrowhouttrue\fi}
\ctr@ln@m\@rrowheadlength
\ctr@ld@f\def\pssetarrowheadlength#1{\edef\@rrowheadlength{#1}\@rrowratiofalse}
\ctr@ln@m\@rrowheadratio
\ctr@ld@f\def\pssetarrowheadratio#1{\edef\@rrowheadratio{#1}\@rrowratiotrue}
\ctr@ln@m\defaultarrowheadlength
\ctr@ld@f\def\psresetarrowhead{%
    \pssetarrowheadangle{\defaultarrowheadangle}%
    \pssetarrowheadfill{\defaultarrowheadfill}%
    \pssetarrowheadout{\defaultarrowheadout}%
    \pssetarrowheadratio{\defaultarrowheadratio}%
    \d@fm@cdim\defaultarrowheadlength{\defaulth@rdahlength}% Valeur par defaut...
    \pssetarrowheadlength{\defaultarrowheadlength}}
\ctr@ld@f\def\defaultarrowheadratio{0.1}
\ctr@ld@f\def\defaultarrowheadangle{20}
\ctr@ld@f\def\defaultarrowheadfill{no}
\ctr@ld@f\def\defaultarrowheadout{no}
\ctr@ld@f\def\defaulth@rdahlength{8pt}
\ctr@ln@m\psarrow
\ctr@ld@f\def\psarrowDD[#1,#2]{{\ifcurr@ntPS\ifps@cri\s@uvc@ntr@l\et@tpsarrow%
    \PSc@mment{psarrowDD [Pt1,Pt2]=[#1,#2]}\pssetfillmode{no}%
    \psarrowheadDD[#1,#2]\setc@ntr@l{2}\psline[#1,-3]%
    \PSc@mment{End psarrowDD}\resetc@ntr@l\et@tpsarrow\fi\fi}}
\ctr@ld@f\def\psarrowTD[#1,#2]{{\ifcurr@ntPS\ifps@cri\s@uvc@ntr@l\et@tpsarrowTD%
    \PSc@mment{psarrowTD [Pt1,Pt2]=[#1,#2]}\resetc@ntr@l{2}%
    \Figptpr@j-5:/#1/\Figptpr@j-6:/#2/\let\c@lprojSP=\relax\psarrowDD[-5,-6]%
    \PSc@mment{End psarrowTD}\resetc@ntr@l\et@tpsarrowTD\fi\fi}}
\ctr@ln@m\psarrowhead
\ctr@ld@f\def\psarrowheadDD[#1,#2]{{\ifcurr@ntPS\ifps@cri\s@uvc@ntr@l\et@tpsarrowheadDD%
    \if@rrowhfill\def\@hangle{-\@rrowheadangle}\else\def\@hangle{\@rrowheadangle}\fi%
    \if@rrowratio%
    \if@rrowhout\def\@hratio{-\@rrowheadratio}\else\def\@hratio{\@rrowheadratio}\fi%
    \PSc@mment{psarrowheadDD Ratio=\@hratio, Angle=\@hangle, [Pt1,Pt2]=[#1,#2]}%
    \Ps@rrowhead\@hratio,\@hangle[#1,#2]%
    \else%
    \if@rrowhout\def\@hlength{-\@rrowheadlength}\else\def\@hlength{\@rrowheadlength}\fi%
    \PSc@mment{psarrowheadDD Length=\@hlength, Angle=\@hangle, [Pt1,Pt2]=[#1,#2]}%
    \Ps@rrowheadfd\@hlength,\@hangle[#1,#2]%
    \fi%
    \PSc@mment{End psarrowheadDD}\resetc@ntr@l\et@tpsarrowheadDD\fi\fi}}
\ctr@ld@f\def\psarrowheadTD[#1,#2]{{\ifcurr@ntPS\ifps@cri\s@uvc@ntr@l\et@tpsarrowheadTD%
    \PSc@mment{psarrowheadTD [Pt1,Pt2]=[#1,#2]}\resetc@ntr@l{2}%
    \Figptpr@j-5:/#1/\Figptpr@j-6:/#2/\let\c@lprojSP=\relax\psarrowheadDD[-5,-6]%
    \PSc@mment{End psarrowheadTD}\resetc@ntr@l\et@tpsarrowheadTD\fi\fi}}
\ctr@ld@f\def\Ps@rrowhead#1,#2[#3,#4]{\v@leur=#1\p@\maxim@m{\v@leur}{\v@leur}{-\v@leur}%
    \ifdim\v@leur>\Cepsil@n{% Arrow is not degenerated
    \PSc@mment{ps@rrowhead Ratio=#1, Angle=#2, [Pt1,Pt2]=[#3,#4]}\v@leur=\UNSS@N%
    \v@leur=\curr@ntwidth\v@leur\v@leur=\ptpsT@pt\v@leur\delt@=.5\v@leur% = width / (2 sin(Angle))
    \setc@ntr@l{2}\figvectPDD-3[#4,#3]%
    \Figg@tXY{-3}\v@lX=#1\v@lX\v@lY=#1\v@lY\Figv@ctCreg-3(\v@lX,\v@lY)%
    \vecunit@{-4}{-3}\mili@u=\result@t%
    \ifdim#2pt>\z@\v@lXa=-\C@AHANG\delt@%
     \edef\c@ef{\repdecn@mb{\v@lXa}}\figpttraDD-3:=-3/\c@ef,-4/\fi%
    \edef\c@ef{\repdecn@mb{\delt@}}%
    \v@lXa=\mili@u\v@lXa=\C@AHANG\v@lXa%
    \v@lYa=\ptpsT@pt\p@\v@lYa=\curr@ntwidth\v@lYa\v@lYa=\sDcc@ngle\v@lYa%
    \advance\v@lXa-\v@lYa\gdef\sDcc@ngle{0}%
    \ifdim\v@lXa>\v@leur\edef\c@efendpt{\repdecn@mb{\v@leur}}%
    \else\edef\c@efendpt{\repdecn@mb{\v@lXa}}\fi%
    \Figg@tXY{-3}\v@lmin=\v@lX\v@lmax=\v@lY%
    \v@lXa=\C@AHANG\v@lmin\v@lYa=\S@AHANG\v@lmax\advance\v@lXa\v@lYa%
    \v@lYa=-\S@AHANG\v@lmin\v@lX=\C@AHANG\v@lmax\advance\v@lYa\v@lX%
    \setc@ntr@l{1}\Figg@tXY{#4}\advance\v@lX\v@lXa\advance\v@lY\v@lYa%
    \setc@ntr@l{2}\Figp@intregDD-2:(\v@lX,\v@lY)%
    \v@lXa=\C@AHANG\v@lmin\v@lYa=-\S@AHANG\v@lmax\advance\v@lXa\v@lYa%
    \v@lYa=\S@AHANG\v@lmin\v@lX=\C@AHANG\v@lmax\advance\v@lYa\v@lX%
    \setc@ntr@l{1}\Figg@tXY{#4}\advance\v@lX\v@lXa\advance\v@lY\v@lYa%
    \setc@ntr@l{2}\Figp@intregDD-1:(\v@lX,\v@lY)%
    \ifdim#2pt<\z@\fillm@detrue\psline[-2,#4,-1]% fill
    \else\figptstraDD-3=#4,-2,-1/\c@ef,-4/\psline[-2,-3,-1]\fi% no fill
    \ifdim#1pt>\z@\figpttraDD-3:=#4/\c@efendpt,-4/\else\figptcopyDD-3:/#4/\fi%
    \PSc@mment{End ps@rrowhead}}\fi}
\ctr@ld@f\def\sDcc@ngle{0}% Initialisation
\ctr@ld@f\def\Ps@rrowheadfd#1,#2[#3,#4]{{%
    \PSc@mment{ps@rrowheadfd Length=#1, Angle=#2, [Pt1,Pt2]=[#3,#4]}%
    \setc@ntr@l{2}\figvectPDD-1[#3,#4]\n@rmeucDD{\v@leur}{-1}\v@leur=\ptT@unit@\v@leur%
    \invers@{\v@leur}{\v@leur}\v@leur=#1\v@leur\edef\R@tio{\repdecn@mb{\v@leur}}%
    \Ps@rrowhead\R@tio,#2[#3,#4]\PSc@mment{End ps@rrowheadfd}}}
\ctr@ln@m\psarrowBezier
\ctr@ld@f\def\psarrowBezierDD[#1,#2,#3,#4]{{\ifcurr@ntPS\ifps@cri\s@uvc@ntr@l\et@tpsarrowBezierDD%
    \PSc@mment{psarrowBezierDD Control points=#1,#2,#3,#4}\setc@ntr@l{2}%
    \if@rrowratio\c@larclengthDD\v@leur,10[#1,#2,#3,#4]\else\v@leur=\z@\fi%
    \Ps@rrowB@zDD\v@leur[#1,#2,#3,#4]%
    \PSc@mment{End psarrowBezierDD}\resetc@ntr@l\et@tpsarrowBezierDD\fi\fi}}
\ctr@ld@f\def\psarrowBezierTD[#1,#2,#3,#4]{{\ifcurr@ntPS\ifps@cri\s@uvc@ntr@l\et@tpsarrowBezierTD%
    \PSc@mment{psarrowBezierTD Control points=#1,#2,#3,#4}\resetc@ntr@l{2}%
    \Figptpr@j-7:/#1/\Figptpr@j-8:/#2/\Figptpr@j-9:/#3/\Figptpr@j-10:/#4/%
    \let\c@lprojSP=\relax\ifnum\curr@ntproj<\tw@\psarrowBezierDD[-7,-8,-9,-10]%
    \else\f@gnewpath\PSwrit@cmd{-7}{\c@mmoveto}{\fwf@g}%
    \if@rrowratio\c@larclengthDD\mili@u,10[-7,-8,-9,-10]\else\mili@u=\z@\fi%
    \p@rtent=\NBz@rcs\advance\p@rtent\m@ne\subB@zierTD\p@rtent[#1,#2,#3,#4]%
    \f@gstroke%
    \advance\v@lmin\p@rtent\delt@% Initialized in \subB@zierTD
    \v@leur=\v@lmin\advance\v@leur0.33333 \delt@\edef\unti@rs{\repdecn@mb{\v@leur}}%
    \v@leur=\v@lmin\advance\v@leur0.66666 \delt@\edef\deti@rs{\repdecn@mb{\v@leur}}%
    \figptcopyDD-8:/-10/\c@lsubBzarc\unti@rs,\deti@rs[#1,#2,#3,#4]%
    \figptcopyDD-8:/-4/\figptcopyDD-9:/-3/\Ps@rrowB@zDD\mili@u[-7,-8,-9,-10]\fi%
    \PSc@mment{End psarrowBezierTD}\resetc@ntr@l\et@tpsarrowBezierTD\fi\fi}}
\ctr@ld@f\def\c@larclengthDD#1,#2[#3,#4,#5,#6]{{\p@rtent=#2\figptcopyDD-5:/#3/%
    \delt@=\p@\divide\delt@\p@rtent\c@rre=\z@\v@leur=\z@\s@mme=\z@%
    \loop\ifnum\s@mme<\p@rtent\advance\s@mme\@ne\advance\v@leur\delt@%
    \edef\T@{\repdecn@mb{\v@leur}}\figptBezierDD-6::\T@[#3,#4,#5,#6]%
    \figvectPDD-1[-5,-6]\n@rmeucDD{\mili@u}{-1}\advance\c@rre\mili@u%
    \figptcopyDD-5:/-6/\repeat\global\result@t=\ptT@unit@\c@rre}#1=\result@t}
\ctr@ld@f\def\Ps@rrowB@zDD#1[#2,#3,#4,#5]{{\pssetfillmode{no}%
    \if@rrowratio\delt@=\@rrowheadratio#1\else\delt@=\@rrowheadlength pt\fi%
    \v@leur=\C@AHANG\delt@\edef\R@dius{\repdecn@mb{\v@leur}}%
    \FigptintercircB@zDD-5::0,\R@dius[#5,#4,#3,#2]%
    \pssetarrowheadlength{\repdecn@mb{\delt@}}\psarrowheadDD[-5,#5]%
    \let\n@rmeuc=\n@rmeucDD\figgetdist\R@dius[#5,-3]%
    \FigptintercircB@zDD-6::0,\R@dius[#5,#4,#3,#2]%
    \figptBezierDD-5::0.33333[#5,#4,#3,#2]\figptBezierDD-3::0.66666[#5,#4,#3,#2]%
    \figptscontrolDD-5[-6,-5,-3,#2]\psBezierDD1[-6,-5,-4,#2]}}
\ctr@ln@m\psarrowcirc
\ctr@ld@f\def\psarrowcircDD#1;#2(#3,#4){{\ifcurr@ntPS\ifps@cri\s@uvc@ntr@l\et@tpsarrowcircDD%
    \PSc@mment{psarrowcircDD Center=#1 ; Radius=#2 (Ang1=#3,Ang2=#4)}%
    \pssetfillmode{no}\Pscirc@rrowhead#1;#2(#3,#4)%
    \setc@ntr@l{2}\figvectPDD -4[#1,-3]\vecunit@{-4}{-4}%
    \Figg@tXY{-4}\arct@n\v@lmin(\v@lX,\v@lY)%
    \v@lmin=\rdT@deg\v@lmin\v@leur=#4pt\advance\v@leur-\v@lmin%
    \maxim@m{\v@leur}{\v@leur}{-\v@leur}%
    \ifdim\v@leur>\DemiPI@deg\relax\ifdim\v@lmin<#4pt\advance\v@lmin\DePI@deg%
    \else\advance\v@lmin-\DePI@deg\fi\fi\edef\ar@ngle{\repdecn@mb{\v@lmin}}%
    \ifdim#3pt<#4pt\psarccirc#1;#2(#3,\ar@ngle)\else\psarccirc#1;#2(\ar@ngle,#3)\fi%
    \PSc@mment{End psarrowcircDD}\resetc@ntr@l\et@tpsarrowcircDD\fi\fi}}
\ctr@ld@f\def\psarrowcircTD#1,#2,#3;#4(#5,#6){{\ifcurr@ntPS\ifps@cri\s@uvc@ntr@l\et@tpsarrowcircTD%
    \PSc@mment{psarrowcircTD Center=#1,P1=#2,P2=#3 ; Radius=#4 (Ang1=#5, Ang2=#6)}%
    \resetc@ntr@l{2}\c@lExtAxes#1,#2,#3(#4)\let\c@lprojSP=\relax%
    \figvectPTD-11[#1,-4]\figvectPTD-12[#1,-5]\c@lNbarcs{#5}{#6}%
    \if@rrowratio\v@lmax=\degT@rd\v@lmax\edef\D@lpha{\repdecn@mb{\v@lmax}}\fi%
    \advance\p@rtent\m@ne\mili@u=\z@%
    \v@leur=#5pt\c@lptellP{#1}{-11}{-12}\Figptpr@j-9:/-3/%
    \f@gnewpath\PSwrit@cmdS{-9}{\c@mmoveto}{\fwf@g}{\X@un}{\Y@un}%
    \edef\C@nt@r{#1}\s@mme=\z@\bcl@rcircTD\f@gstroke%
    \advance\v@leur\delt@\c@lptellP{#1}{-11}{-12}\Figptpr@j-5:/-3/%
    \advance\v@leur\delt@\c@lptellP{#1}{-11}{-12}\Figptpr@j-6:/-3/%
    \advance\v@leur\delt@\c@lptellP{#1}{-11}{-12}\Figptpr@j-10:/-3/%
    \figptscontrolDD-8[-9,-5,-6,-10]%
    \if@rrowratio\c@lcurvradDD0.5[-9,-8,-7,-10]\advance\mili@u\result@t%
    \maxim@m{\mili@u}{\mili@u}{-\mili@u}\mili@u=\ptT@unit@\mili@u%
    \mili@u=\D@lpha\mili@u\advance\p@rtent\@ne\divide\mili@u\p@rtent\fi%
    \Ps@rrowB@zDD\mili@u[-9,-8,-7,-10]%
    \PSc@mment{End psarrowcircTD}\resetc@ntr@l\et@tpsarrowcircTD\fi\fi}}
\ctr@ld@f\def\bcl@rcircTD{\relax%
    \ifnum\s@mme<\p@rtent\advance\s@mme\@ne%
    \advance\v@leur\delt@\c@lptellP{\C@nt@r}{-11}{-12}\Figptpr@j-5:/-3/%
    \advance\v@leur\delt@\c@lptellP{\C@nt@r}{-11}{-12}\Figptpr@j-6:/-3/%
    \advance\v@leur\delt@\c@lptellP{\C@nt@r}{-11}{-12}\Figptpr@j-10:/-3/%
    \figptscontrolDD-8[-9,-5,-6,-10]\BdingB@xfalse%
    \PSwrit@cmdS{-8}{}{\fwf@g}{\X@de}{\Y@de}\PSwrit@cmdS{-7}{}{\fwf@g}{\X@tr}{\Y@tr}%
    \BdingB@xtrue\PSwrit@cmdS{-10}{\c@mcurveto}{\fwf@g}{\X@qu}{\Y@qu}%
    \if@rrowratio\c@lcurvradDD0.5[-9,-8,-7,-10]\advance\mili@u\result@t\fi%
    \B@zierBB@x{1}{\Y@un}(\X@un,\X@de,\X@tr,\X@qu)%
    \B@zierBB@x{2}{\X@un}(\Y@un,\Y@de,\Y@tr,\Y@qu)%
    \edef\X@un{\X@qu}\edef\Y@un{\Y@qu}\figptcopyDD-9:/-10/\bcl@rcircTD\fi}
\ctr@ld@f\def\Pscirc@rrowhead#1;#2(#3,#4){{%
    \PSc@mment{pscirc@rrowhead Center=#1 ; Radius=#2 (Ang1=#3,Ang2=#4)}%
    \v@leur=#2\unit@\edef\s@glen{\repdecn@mb{\v@leur}}\v@lY=\z@\v@lX=\v@leur%
    \resetc@ntr@l{2}\Figv@ctCreg-3(\v@lX,\v@lY)\figpttraDD-5:=#1/1,-3/%
    \figptrotDD-5:=-5/#1,#4/%
    \figvectPDD-3[#1,-5]\Figg@tXY{-3}\v@leur=\v@lX%
    \ifdim#3pt<#4pt\v@lX=\v@lY\v@lY=-\v@leur\else\v@lX=-\v@lY\v@lY=\v@leur\fi%
    \Figv@ctCreg-3(\v@lX,\v@lY)\vecunit@{-3}{-3}%
    \if@rrowratio\v@leur=#4pt\advance\v@leur-#3pt\maxim@m{\mili@u}{-\v@leur}{\v@leur}%
    \mili@u=\degT@rd\mili@u\v@leur=\s@glen\mili@u\edef\s@glen{\repdecn@mb{\v@leur}}%
    \mili@u=#2\mili@u\mili@u=\@rrowheadratio\mili@u\else\mili@u=\@rrowheadlength pt\fi%
    \figpttraDD-6:=-5/\s@glen,-3/\v@leur=#2pt\v@leur=2\v@leur%
    \invers@{\v@leur}{\v@leur}\c@rre=\repdecn@mb{\v@leur}\mili@u% = sin = L/(2R)
    \mili@u=\c@rre\mili@u=\repdecn@mb{\c@rre}\mili@u%
    \v@leur=\p@\advance\v@leur-\mili@u% \v@leur = cos*cos
    \invers@{\mili@u}{2\v@leur}\delt@=\c@rre\delt@=\repdecn@mb{\mili@u}\delt@%
    \xdef\sDcc@ngle{\repdecn@mb{\delt@}}% sin/(2*cos*cos) used in \Ps@rrowhead
    \sqrt@{\mili@u}{\v@leur}\arct@n\v@leur(\mili@u,\c@rre)%
    \v@leur=\rdT@deg\v@leur% \cor@ngle = atan(L/sqrt(4R*R-L*L))
    \ifdim#3pt<#4pt\v@leur=-\v@leur\fi%
    \if@rrowhout\v@leur=-\v@leur\fi\edef\cor@ngle{\repdecn@mb{\v@leur}}%
    \figptrotDD-6:=-6/-5,\cor@ngle/\psarrowheadDD[-6,-5]%
    \PSc@mment{End pscirc@rrowhead}}}
\ctr@ln@m\psarrowcircP
\ctr@ld@f\def\psarrowcircPDD#1;#2[#3,#4]{{\ifcurr@ntPS\ifps@cri%
    \PSc@mment{psarrowcircPDD Center=#1; Radius=#2, [P1=#3,P2=#4]}%
    \s@uvc@ntr@l\et@tpsarrowcircPDD\Ps@ngleparam#1;#2[#3,#4]%
    \ifdim\v@leur>\z@\ifdim\v@lmin>\v@lmax\advance\v@lmax\DePI@deg\fi%
    \else\ifdim\v@lmin<\v@lmax\advance\v@lmin\DePI@deg\fi\fi%
    \edef\@ngdeb{\repdecn@mb{\v@lmin}}\edef\@ngfin{\repdecn@mb{\v@lmax}}%
    \psarrowcirc#1;\r@dius(\@ngdeb,\@ngfin)%
    \PSc@mment{End psarrowcircPDD}\resetc@ntr@l\et@tpsarrowcircPDD\fi\fi}}
\ctr@ld@f\def\psarrowcircPTD#1;#2[#3,#4,#5]{{\ifcurr@ntPS\ifps@cri\s@uvc@ntr@l\et@tpsarrowcircPTD%
    \PSc@mment{psarrowcircPTD Center=#1; Radius=#2, [P1=#3,P2=#4,P3=#5]}%
    \figgetangleTD\@ngfin[#1,#3,#4,#5]\v@leur=#2pt%
    \maxim@m{\mili@u}{-\v@leur}{\v@leur}\edef\r@dius{\repdecn@mb{\mili@u}}%
    \ifdim\v@leur<\z@\v@lmax=\@ngfin pt\advance\v@lmax-\DePI@deg%
    \edef\@ngfin{\repdecn@mb{\v@lmax}}\fi\psarrowcircTD#1,#3,#5;\r@dius(0,\@ngfin)%
    \PSc@mment{End psarrowcircPTD}\resetc@ntr@l\et@tpsarrowcircPTD\fi\fi}}
\ctr@ld@f\def\psaxes#1(#2){{\ifcurr@ntPS\ifps@cri\s@uvc@ntr@l\et@tpsaxes%
    \PSc@mment{psaxes Origin=#1 Range=(#2)}\an@lys@xes#2,:\resetc@ntr@l{2}%
    \ifx\t@xt@\empty\ifTr@isDim\ps@xes#1(0,#2,0,#2,0,#2)\else\ps@xes#1(0,#2,0,#2)\fi%
    \else\ps@xes#1(#2)\fi\PSc@mment{End psaxes}\resetc@ntr@l\et@tpsaxes\fi\fi}}
\ctr@ld@f\def\an@lys@xes#1,#2:{\def\t@xt@{#2}}
\ctr@ln@m\ps@xes
\ctr@ld@f\def\ps@xesDD#1(#2,#3,#4,#5){%
    \figpttraC-5:=#1/#2,0/\figpttraC-6:=#1/#3,0/\psarrowDD[-5,-6]%
    \figpttraC-5:=#1/0,#4/\figpttraC-6:=#1/0,#5/\psarrowDD[-5,-6]}
\ctr@ld@f\def\ps@xesTD#1(#2,#3,#4,#5,#6,#7){%
    \figpttraC-7:=#1/#2,0,0/\figpttraC-8:=#1/#3,0,0/\psarrowTD[-7,-8]%
    \figpttraC-7:=#1/0,#4,0/\figpttraC-8:=#1/0,#5,0/\psarrowTD[-7,-8]%
    \figpttraC-7:=#1/0,0,#6/\figpttraC-8:=#1/0,0,#7/\psarrowTD[-7,-8]}
\ctr@ln@m\newGr@FN
\ctr@ld@f\def\newGr@FNPDF#1{\s@mme=\Gr@FNb\advance\s@mme\@ne\xdef\Gr@FNb{\number\s@mme}}
\ctr@ld@f\def\newGr@FNDVI#1{\newGr@FNPDF{}\xdef#1{\jobname GI\Gr@FNb.anx}}
\ctr@ld@f\def\psbeginfig#1{\newGr@FN\DefGIfilen@me\gdef\@utoFN{0}%
    \def\t@xt@{#1}\relax\ifx\t@xt@\empty\psupdatem@detrue%
    \gdef\@utoFN{1}\Psb@ginfig\DefGIfilen@me\else\expandafter\Psb@ginfigNu@#1 :\fi}
\ctr@ld@f\def\Psb@ginfigNu@#1 #2:{\def\t@xt@{#1}\relax\ifx\t@xt@\empty\def\t@xt@{#2}%
    \ifx\t@xt@\empty\psupdatem@detrue\gdef\@utoFN{1}\Psb@ginfig\DefGIfilen@me%
    \else\Psb@ginfigNu@#2:\fi\else\Psb@ginfig{#1}\fi}
\ctr@ln@m\PSfilen@me \ctr@ln@m\auxfilen@me
\ctr@ld@f\def\Psb@ginfig#1{\ifcurr@ntPS\else%
    \edef\PSfilen@me{#1}\edef\auxfilen@me{\jobname.anx}%
    \ifpsupdatem@de\ps@critrue\else\openin\frf@g=\PSfilen@me\relax%
    \ifeof\frf@g\ps@critrue\else\ps@crifalse\fi\closein\frf@g\fi%
    \curr@ntPStrue\c@ldefproj\expandafter\setupd@te\defaultupdate:%
    \ifps@cri\initb@undb@x%
    \immediate\openout\fwf@g=\auxfilen@me\initpss@ttings\fi%
    \fi}
\ctr@ld@f\def\Gr@FNb{0}
\ctr@ld@f\def\figforTeXFileno{\Gr@FNb}
\ctr@ld@f\def\figforTeXFigno{0 }
\ctr@ld@f\def\figforTeXnextFigno{1 }
\ctr@ld@f\edef\DefGIfilen@me{\jobname GI.anx}
\ctr@ld@f\def\initpss@ttings{\psreset{arrowhead,curve,first,flowchart,mesh,second,third}%
    \Use@llipsefalse}
\ctr@ld@f\def\B@zierBB@x#1#2(#3,#4,#5,#6){{\c@rre=\t@n\epsil@n% Do not reduce this value
    \v@lmax=#4\advance\v@lmax-#5\v@lmax=\thr@@\v@lmax\advance\v@lmax#6\advance\v@lmax-#3%
    \mili@u=#4\mili@u=-\tw@\mili@u\advance\mili@u#3\advance\mili@u#5%
    \v@lmin=#4\advance\v@lmin-#3\maxim@m{\v@leur}{-\v@lmax}{\v@lmax}%
    \maxim@m{\delt@}{-\mili@u}{\mili@u}\maxim@m{\v@leur}{\v@leur}{\delt@}%
    \maxim@m{\delt@}{-\v@lmin}{\v@lmin}\maxim@m{\v@leur}{\v@leur}{\delt@}%
    \ifdim\v@leur>\c@rre\invers@{\v@leur}{\v@leur}\edef\Uns@rM@x{\repdecn@mb{\v@leur}}%
    \v@lmax=\Uns@rM@x\v@lmax\mili@u=\Uns@rM@x\mili@u\v@lmin=\Uns@rM@x\v@lmin%
    \maxim@m{\v@leur}{-\v@lmax}{\v@lmax}\ifdim\v@leur<\c@rre%
    \maxim@m{\v@leur}{-\mili@u}{\mili@u}\ifdim\v@leur<\c@rre\else%
    \invers@{\mili@u}{\mili@u}\v@leur=-0.5\v@lmin%
    \v@leur=\repdecn@mb{\mili@u}\v@leur\m@jBBB@x{\v@leur}{#1}{#2}(#3,#4,#5,#6)\fi%
    \else\delt@=\repdecn@mb{\mili@u}\mili@u\v@leur=\repdecn@mb{\v@lmax}\v@lmin%
    \advance\delt@-\v@leur\ifdim\delt@<\z@\else\invers@{\v@lmax}{\v@lmax}%
    \edef\Uns@rAp{\repdecn@mb{\v@lmax}}\sqrt@{\delt@}{\delt@}%
    \v@leur=-\mili@u\advance\v@leur\delt@\v@leur=\Uns@rAp\v@leur%
    \m@jBBB@x{\v@leur}{#1}{#2}(#3,#4,#5,#6)%
    \v@leur=-\mili@u\advance\v@leur-\delt@\v@leur=\Uns@rAp\v@leur%
    \m@jBBB@x{\v@leur}{#1}{#2}(#3,#4,#5,#6)\fi\fi\fi}}
\ctr@ld@f\def\m@jBBB@x#1#2#3(#4,#5,#6,#7){{\relax\ifdim#1>\z@\ifdim#1<\p@%
    \edef\T@{\repdecn@mb{#1}}\v@lX=\p@\advance\v@lX-#1\edef\UNmT@{\repdecn@mb{\v@lX}}%
    \v@lX=#4\v@lY=#5\v@lZ=#6\v@lXa=#7\v@lX=\UNmT@\v@lX\advance\v@lX\T@\v@lY%
    \v@lY=\UNmT@\v@lY\advance\v@lY\T@\v@lZ\v@lZ=\UNmT@\v@lZ\advance\v@lZ\T@\v@lXa%
    \v@lX=\UNmT@\v@lX\advance\v@lX\T@\v@lY\v@lY=\UNmT@\v@lY\advance\v@lY\T@\v@lZ%
    \v@lX=\UNmT@\v@lX\advance\v@lX\T@\v@lY%
    \ifcase#2\or\v@lY=#3\or\v@lY=\v@lX\v@lX=#3\fi\b@undb@x{\v@lX}{\v@lY}\fi\fi}}
\ctr@ld@f\def\PsB@zier#1[#2]{{\f@gnewpath%
    \s@mme=\z@\def\list@num{#2,0}\extrairelepremi@r\p@int\de\list@num%
    \PSwrit@cmdS{\p@int}{\c@mmoveto}{\fwf@g}{\X@un}{\Y@un}\p@rtent=#1\bclB@zier}}
\ctr@ld@f\def\bclB@zier{\relax%
    \ifnum\s@mme<\p@rtent\advance\s@mme\@ne\BdingB@xfalse%
    \extrairelepremi@r\p@int\de\list@num\PSwrit@cmdS{\p@int}{}{\fwf@g}{\X@de}{\Y@de}%
    \extrairelepremi@r\p@int\de\list@num\PSwrit@cmdS{\p@int}{}{\fwf@g}{\X@tr}{\Y@tr}%
    \BdingB@xtrue%
    \extrairelepremi@r\p@int\de\list@num\PSwrit@cmdS{\p@int}{\c@mcurveto}{\fwf@g}{\X@qu}{\Y@qu}%
    \B@zierBB@x{1}{\Y@un}(\X@un,\X@de,\X@tr,\X@qu)%
    \B@zierBB@x{2}{\X@un}(\Y@un,\Y@de,\Y@tr,\Y@qu)%
    \edef\X@un{\X@qu}\edef\Y@un{\Y@qu}\bclB@zier\fi}
\ctr@ln@m\psBezier
\ctr@ld@f\def\psBezierDD#1[#2]{\ifcurr@ntPS\ifps@cri%
    \PSc@mment{psBezierDD N arcs=#1, Control points=#2}%
    \iffillm@de\PsB@zier#1[#2]%
    \f@gfill%
    \else\PsB@zier#1[#2]\f@gstroke\fi%
    \PSc@mment{End psBezierDD}\fi\fi}
\ctr@ln@m\et@tpsBezierTD% ou doubler les {}
\ctr@ld@f\def\psBezierTD#1[#2]{\ifcurr@ntPS\ifps@cri\s@uvc@ntr@l\et@tpsBezierTD%
    \PSc@mment{psBezierTD N arcs=#1, Control points=#2}%
    \iffillm@de\PsB@zierTD#1[#2]%
    \f@gfill%
    \else\PsB@zierTD#1[#2]\f@gstroke\fi%
    \PSc@mment{End psBezierTD}\resetc@ntr@l\et@tpsBezierTD\fi\fi}
\ctr@ld@f\def\PsB@zierTD#1[#2]{\ifnum\curr@ntproj<\tw@\PsB@zier#1[#2]\else\PsB@zier@TD#1[#2]\fi}
\ctr@ld@f\def\PsB@zier@TD#1[#2]{{\f@gnewpath%
    \s@mme=\z@\def\list@num{#2,0}\extrairelepremi@r\p@int\de\list@num%
    \let\c@lprojSP=\relax\setc@ntr@l{2}\Figptpr@j-7:/\p@int/%
    \PSwrit@cmd{-7}{\c@mmoveto}{\fwf@g}%
    \loop\ifnum\s@mme<#1\advance\s@mme\@ne\extrairelepremi@r\p@intun\de\list@num%
    \extrairelepremi@r\p@intde\de\list@num\extrairelepremi@r\p@inttr\de\list@num%
    \subB@zierTD\NBz@rcs[\p@int,\p@intun,\p@intde,\p@inttr]\edef\p@int{\p@inttr}\repeat}}
\ctr@ld@f\def\subB@zierTD#1[#2,#3,#4,#5]{\delt@=\p@\divide\delt@\NBz@rcs\v@lmin=\z@%
    {\Figg@tXY{-7}\edef\X@un{\the\v@lX}\edef\Y@un{\the\v@lY}%
    \s@mme=\z@\loop\ifnum\s@mme<#1\advance\s@mme\@ne%
    \v@leur=\v@lmin\advance\v@leur0.33333 \delt@\edef\unti@rs{\repdecn@mb{\v@leur}}%
    \v@leur=\v@lmin\advance\v@leur0.66666 \delt@\edef\deti@rs{\repdecn@mb{\v@leur}}%
    \advance\v@lmin\delt@\edef\trti@rs{\repdecn@mb{\v@lmin}}%
    \figptBezierTD-8::\trti@rs[#2,#3,#4,#5]\Figptpr@j-8:/-8/%
    \c@lsubBzarc\unti@rs,\deti@rs[#2,#3,#4,#5]\BdingB@xfalse%
    \PSwrit@cmdS{-4}{}{\fwf@g}{\X@de}{\Y@de}\PSwrit@cmdS{-3}{}{\fwf@g}{\X@tr}{\Y@tr}%
    \BdingB@xtrue\PSwrit@cmdS{-8}{\c@mcurveto}{\fwf@g}{\X@qu}{\Y@qu}%
    \B@zierBB@x{1}{\Y@un}(\X@un,\X@de,\X@tr,\X@qu)%
    \B@zierBB@x{2}{\X@un}(\Y@un,\Y@de,\Y@tr,\Y@qu)%
    \edef\X@un{\X@qu}\edef\Y@un{\Y@qu}\figptcopyDD-7:/-8/\repeat}}
\ctr@ld@f\def\NBz@rcs{2}
\ctr@ld@f\def\c@lsubBzarc#1,#2[#3,#4,#5,#6]{\figptBezierTD-5::#1[#3,#4,#5,#6]%
    \figptBezierTD-6::#2[#3,#4,#5,#6]\Figptpr@j-4:/-5/\Figptpr@j-5:/-6/%
    \figptscontrolDD-4[-7,-4,-5,-8]}
\ctr@ln@m\pscirc
\ctr@ld@f\def\pscircDD#1(#2){\ifcurr@ntPS\ifps@cri\PSc@mment{pscircDD Center=#1 (Radius=#2)}%
    \psarccircDD#1;#2(0,360)\PSc@mment{End pscircDD}\fi\fi}
\ctr@ld@f\def\pscircTD#1,#2,#3(#4){\ifcurr@ntPS\ifps@cri%
    \PSc@mment{pscircTD Center=#1,P1=#2,P2=#3 (Radius=#4)}%
    \psarccircTD#1,#2,#3;#4(0,360)\PSc@mment{End pscircTD}\fi\fi}
\ctr@ln@m\p@urcent
{\catcode`\%=12\gdef\p@urcent{%}}
\ctr@ld@f\def\PSc@mment#1{\ifpsdebugmode\immediate\write\fwf@g{\p@urcent\space#1}\fi}
\ctr@ln@m\acc@louv \ctr@ln@m\acc@lfer
{\catcode`\[=1\catcode`\{=12\gdef\acc@louv[{}}
{\catcode`\]=2\catcode`\}=12\gdef\acc@lfer{}]]
\ctr@ld@f\def\PSdict@{\ifUse@llipse%
    \immediate\write\fwf@g{/ellipsedict 9 dict def ellipsedict /mtrx matrix put}%
    \immediate\write\fwf@g{/ellipse \acc@louv ellipsedict begin}%
    \immediate\write\fwf@g{ /endangle exch def /startangle exch def}%
    \immediate\write\fwf@g{ /yrad exch def /xrad exch def}%
    \immediate\write\fwf@g{ /rotangle exch def /y exch def /x exch def}%
    \immediate\write\fwf@g{ /savematrix mtrx currentmatrix def}%
    \immediate\write\fwf@g{ x y translate rotangle rotate xrad yrad scale}%
    \immediate\write\fwf@g{ 0 0 1 startangle endangle arc}%
    \immediate\write\fwf@g{ savematrix setmatrix end\acc@lfer def}%
    \fi\PShe@der{EndProlog}}
\ctr@ld@f\def\Pssetc@rve#1=#2|{\keln@mun#1|%
    \def\n@mref{r}\ifx\l@debut\n@mref\pssetroundness{#2}\else% roundness
    \immediate\write16{*** Unknown attribute: \BS@ psset curve(..., #1=...)}%
    \fi}
\ctr@ln@m\curv@roundness
\ctr@ld@f\def\pssetroundness#1{\edef\curv@roundness{#1}}
\ctr@ld@f\def\defaultroundness{0.2} % Valeur par defaut
\ctr@ln@m\pscurve
\ctr@ld@f\def\pscurveDD[#1]{{\ifcurr@ntPS\ifps@cri\PSc@mment{pscurveDD Points=#1}%
    \s@uvc@ntr@l\et@tpscurveDD%
    \iffillm@de\Psc@rveDD\curv@roundness[#1]%
    \f@gfill%
    \else\Psc@rveDD\curv@roundness[#1]\f@gstroke\fi%
    \PSc@mment{End pscurveDD}\resetc@ntr@l\et@tpscurveDD\fi\fi}}
\ctr@ld@f\def\pscurveTD[#1]{{\ifcurr@ntPS\ifps@cri%
    \PSc@mment{pscurveTD Points=#1}\s@uvc@ntr@l\et@tpscurveTD\let\c@lprojSP=\relax%
    \iffillm@de\Psc@rveTD\curv@roundness[#1]%
    \f@gfill%
    \else\Psc@rveTD\curv@roundness[#1]\f@gstroke\fi%
    \PSc@mment{End pscurveTD}\resetc@ntr@l\et@tpscurveTD\fi\fi}}
\ctr@ld@f\def\Psc@rveDD#1[#2]{%
    \def\list@num{#2}\extrairelepremi@r\Ak@\de\list@num%
    \extrairelepremi@r\Ai@\de\list@num\extrairelepremi@r\Aj@\de\list@num%
    \f@gnewpath\PSwrit@cmdS{\Ai@}{\c@mmoveto}{\fwf@g}{\X@un}{\Y@un}%
    \setc@ntr@l{2}\figvectPDD -1[\Ak@,\Aj@]%
    \@ecfor\Ak@:=\list@num\do{\figpttraDD-2:=\Ai@/#1,-1/\BdingB@xfalse%
       \PSwrit@cmdS{-2}{}{\fwf@g}{\X@de}{\Y@de}%
       \figvectPDD -1[\Ai@,\Ak@]\figpttraDD-2:=\Aj@/-#1,-1/%
       \PSwrit@cmdS{-2}{}{\fwf@g}{\X@tr}{\Y@tr}\BdingB@xtrue%
       \PSwrit@cmdS{\Aj@}{\c@mcurveto}{\fwf@g}{\X@qu}{\Y@qu}%
       \B@zierBB@x{1}{\Y@un}(\X@un,\X@de,\X@tr,\X@qu)%
       \B@zierBB@x{2}{\X@un}(\Y@un,\Y@de,\Y@tr,\Y@qu)%
       \edef\X@un{\X@qu}\edef\Y@un{\Y@qu}\edef\Ai@{\Aj@}\edef\Aj@{\Ak@}}}
\ctr@ld@f\def\Psc@rveTD#1[#2]{\ifnum\curr@ntproj<\tw@\Psc@rvePPTD#1[#2]\else\Psc@rveCPTD#1[#2]\fi}
\ctr@ld@f\def\Psc@rvePPTD#1[#2]{\setc@ntr@l{2}%
    \def\list@num{#2}\extrairelepremi@r\Ak@\de\list@num\Figptpr@j-5:/\Ak@/%
    \extrairelepremi@r\Ai@\de\list@num\Figptpr@j-3:/\Ai@/%
    \extrairelepremi@r\Aj@\de\list@num\Figptpr@j-4:/\Aj@/%
    \f@gnewpath\PSwrit@cmdS{-3}{\c@mmoveto}{\fwf@g}{\X@un}{\Y@un}%
    \figvectPDD -1[-5,-4]%
    \@ecfor\Ak@:=\list@num\do{\Figptpr@j-5:/\Ak@/\figpttraDD-2:=-3/#1,-1/%
       \BdingB@xfalse\PSwrit@cmdS{-2}{}{\fwf@g}{\X@de}{\Y@de}%
       \figvectPDD -1[-3,-5]\figpttraDD-2:=-4/-#1,-1/%
       \PSwrit@cmdS{-2}{}{\fwf@g}{\X@tr}{\Y@tr}\BdingB@xtrue%
       \PSwrit@cmdS{-4}{\c@mcurveto}{\fwf@g}{\X@qu}{\Y@qu}%
       \B@zierBB@x{1}{\Y@un}(\X@un,\X@de,\X@tr,\X@qu)%
       \B@zierBB@x{2}{\X@un}(\Y@un,\Y@de,\Y@tr,\Y@qu)%
       \edef\X@un{\X@qu}\edef\Y@un{\Y@qu}\figptcopyDD-3:/-4/\figptcopyDD-4:/-5/}}
\ctr@ld@f\def\Psc@rveCPTD#1[#2]{\setc@ntr@l{2}%
    \def\list@num{#2}\extrairelepremi@r\Ak@\de\list@num%
    \extrairelepremi@r\Ai@\de\list@num\extrairelepremi@r\Aj@\de\list@num%
    \Figptpr@j-7:/\Ai@/%
    \f@gnewpath\PSwrit@cmd{-7}{\c@mmoveto}{\fwf@g}%
    \figvectPTD -9[\Ak@,\Aj@]%
    \@ecfor\Ak@:=\list@num\do{\figpttraTD-10:=\Ai@/#1,-9/%
       \figvectPTD -9[\Ai@,\Ak@]\figpttraTD-11:=\Aj@/-#1,-9/%
       \subB@zierTD\NBz@rcs[\Ai@,-10,-11,\Aj@]\edef\Ai@{\Aj@}\edef\Aj@{\Ak@}}}
\ctr@ld@f\def\psendfig{\ifcurr@ntPS\ifps@cri\immediate\closeout\fwf@g%
    \immediate\openout\fwf@g=\PSfilen@me\relax%
    \ifPDFm@ke\PSBdingB@x\else%
    \immediate\write\fwf@g{\p@urcent\string!PS-Adobe-2.0 EPSF-2.0}%
    \PShe@der{Creator\string: TeX (fig4tex.tex)}%
    \PShe@der{Title\string: \PSfilen@me}%
    \PShe@der{CreationDate\string: \the\day/\the\month/\the\year}%
    \PSBdingB@x%
    \PShe@der{EndComments}\PSdict@\fi%
    \immediate\write\fwf@g{\c@mgsave}%
    \openin\frf@g=\auxfilen@me\c@pypsfile\fwf@g\frf@g\closein\frf@g%
    \immediate\write\fwf@g{\c@mgrestore}%
    \PSc@mment{End of file.}\immediate\closeout\fwf@g%
    \immediate\openout\fwf@g=\auxfilen@me\immediate\closeout\fwf@g%
    \immediate\write16{File \PSfilen@me\space created.}\fi\fi\curr@ntPSfalse\ps@critrue}
\ctr@ld@f\def\PShe@der#1{\immediate\write\fwf@g{\p@urcent\p@urcent#1}}
\ctr@ld@f\def\PSBdingB@x{{\v@lX=\ptT@ptps\c@@rdXmin\v@lY=\ptT@ptps\c@@rdYmin%
     \v@lXa=\ptT@ptps\c@@rdXmax\v@lYa=\ptT@ptps\c@@rdYmax%
     \PShe@der{BoundingBox\string: \repdecn@mb{\v@lX}\space\repdecn@mb{\v@lY}%
     \space\repdecn@mb{\v@lXa}\space\repdecn@mb{\v@lYa}}}}
\ctr@ld@f\def\psfcconnect[#1]{{\ifcurr@ntPS\ifps@cri\PSc@mment{psfcconnect Points=#1}%
    \pssetfillmode{no}\s@uvc@ntr@l\et@tpsfcconnect\resetc@ntr@l{2}%
    \fcc@nnect@[#1]\resetc@ntr@l\et@tpsfcconnect\PSc@mment{End psfcconnect}\fi\fi}}
\ctr@ld@f\def\fcc@nnect@[#1]{\let\N@rm=\n@rmeucDD\def\list@num{#1}%
    \extrairelepremi@r\Ai@\de\list@num\edef\pr@m{\Ai@}\v@leur=\z@\p@rtent=\@ne\c@llgtot%
    \ifcase\fclin@typ@\edef\list@num{[\pr@m,#1,\Ai@}\expandafter\pscurve\list@num]%
    \else\ifdim\fclin@r@d\p@>\z@\Pslin@conge[#1]\else\psline[#1]\fi\fi%
    \v@leur=\@rrowp@s\v@leur\edef\list@num{#1,\Ai@,0}%
    \extrairelepremi@r\Ai@\de\list@num\mili@u=\epsil@n\c@llgpart%
    \advance\mili@u-\epsil@n\advance\mili@u-\delt@\advance\v@leur-\mili@u%
    \ifcase\fclin@typ@\invers@\mili@u\delt@%
    \ifnum\@rrowr@fpt>\z@\advance\delt@-\v@leur\v@leur=\delt@\fi%
    \v@leur=\repdecn@mb\v@leur\mili@u\edef\v@lt{\repdecn@mb\v@leur}%
    \extrairelepremi@r\Ak@\de\list@num%
    \figvectPDD-1[\pr@m,\Aj@]\figpttraDD-6:=\Ai@/\curv@roundness,-1/%
    \figvectPDD-1[\Ak@,\Ai@]\figpttraDD-7:=\Aj@/\curv@roundness,-1/%
    \delt@=\@rrowheadlength\p@\delt@=\C@AHANG\delt@\edef\R@dius{\repdecn@mb{\delt@}}%
    \ifcase\@rrowr@fpt%
    \FigptintercircB@zDD-8::\v@lt,\R@dius[\Ai@,-6,-7,\Aj@]\psarrowheadDD[-5,-8]\else%
    \FigptintercircB@zDD-8::\v@lt,\R@dius[\Aj@,-7,-6,\Ai@]\psarrowheadDD[-8,-5]\fi%
    \else\advance\delt@-\v@leur%
    \p@rtentiere{\p@rtent}{\delt@}\edef\C@efun{\the\p@rtent}%
    \p@rtentiere{\p@rtent}{\v@leur}\edef\C@efde{\the\p@rtent}%
    \figptbaryDD-5:[\Ai@,\Aj@;\C@efun,\C@efde]\ifcase\@rrowr@fpt%
    \delt@=\@rrowheadlength\unit@\delt@=\C@AHANG\delt@\edef\t@ille{\repdecn@mb{\delt@}}%
    \figvectPDD-2[\Ai@,\Aj@]\vecunit@{-2}{-2}\figpttraDD-5:=-5/\t@ille,-2/\fi%
    \psarrowheadDD[\Ai@,-5]\fi}
\ctr@ld@f\def\c@llgtot{\@ecfor\Aj@:=\list@num\do{\figvectP-1[\Ai@,\Aj@]\N@rm\delt@{-1}%
    \advance\v@leur\delt@\advance\p@rtent\@ne\edef\Ai@{\Aj@}}}
\ctr@ld@f\def\c@llgpart{\extrairelepremi@r\Aj@\de\list@num\figvectP-1[\Ai@,\Aj@]\N@rm\delt@{-1}%
    \advance\mili@u\delt@\ifdim\mili@u<\v@leur\edef\pr@m{\Ai@}\edef\Ai@{\Aj@}\c@llgpart\fi}
\ctr@ld@f\def\Pslin@conge[#1]{\ifnum\p@rtent>\tw@{\def\list@num{#1}%
    \extrairelepremi@r\Ai@\de\list@num\extrairelepremi@r\Aj@\de\list@num%
    \figptcopy-6:/\Ai@/\figvectP-3[\Ai@,\Aj@]\vecunit@{-3}{-3}\v@lmax=\result@t%
    \@ecfor\Ak@:=\list@num\do{\figvectP-4[\Aj@,\Ak@]\vecunit@{-4}{-4}%
    \minim@m\v@lmin\v@lmax\result@t\v@lmax=\result@t%
    \det@rm\delt@[-3,-4]\maxim@m\mili@u{\delt@}{-\delt@}\ifdim\mili@u>\Cepsil@n%
    \ifdim\delt@>\z@\figgetangleDD\Angl@[\Aj@,\Ak@,\Ai@]\else%
    \figgetangleDD\Angl@[\Aj@,\Ai@,\Ak@]\fi%
    \v@leur=\PI@deg\advance\v@leur-\Angl@\p@\divide\v@leur\tw@%
    \edef\Angl@{\repdecn@mb\v@leur}\c@ssin{\C@}{\S@}{\Angl@}\v@leur=\fclin@r@d\unit@%
    \v@leur=\S@\v@leur\mili@u=\C@\p@\invers@\mili@u\mili@u%
    \v@leur=\repdecn@mb{\mili@u}\v@leur%
    \minim@m\v@leur\v@leur\v@lmin\edef\t@ille{\repdecn@mb{\v@leur}}%
    \figpttra-5:=\Aj@/-\t@ille,-3/\psline[-6,-5]\figpttra-6:=\Aj@/\t@ille,-4/%
    \figvectNVDD-3[-3]\figvectNVDD-8[-4]\inters@cDD-7:[-5,-3;-6,-8]%
    \ifdim\delt@>\z@\psarccircP-7;\fclin@r@d[-5,-6]\else\psarccircP-7;\fclin@r@d[-6,-5]\fi%
    \else\psline[-6,\Aj@]\figptcopy-6:/\Aj@/\fi% Points alignes
    \edef\Ai@{\Aj@}\edef\Aj@{\Ak@}\figptcopy-3:/-4/}\psline[-6,\Aj@]}\else\psline[#1]\fi}
\ctr@ld@f\def\psfcnode[#1]#2{{\ifcurr@ntPS\ifps@cri\PSc@mment{psfcnode Points=#1}%
    \s@uvc@ntr@l\et@tpsfcnode\resetc@ntr@l{2}%
    \def\t@xt@{#2}\ifx\t@xt@\empty\def\g@tt@xt{\setbox\Gb@x=\hbox{\Figg@tT{\p@int}}}%
    \else\def\g@tt@xt{\setbox\Gb@x=\hbox{#2}}\fi%
    \v@lmin=\h@rdfcXp@dd\advance\v@lmin\Xp@dd\unit@\multiply\v@lmin\tw@%
    \v@lmax=\h@rdfcYp@dd\advance\v@lmax\Yp@dd\unit@\multiply\v@lmax\tw@%
    \Figv@ctCreg-8(\unit@,-\unit@)\def\list@num{#1}%
    \delt@=\curr@ntwidth bp\divide\delt@\tw@%
    \fcn@de\PSc@mment{End psfcnode}\resetc@ntr@l\et@tpsfcnode\fi\fi}}
\ctr@ld@f\def\d@butn@de{\g@tt@xt\v@lX=\wd\Gb@x%
    \v@lY=\ht\Gb@x\advance\v@lY\dp\Gb@x\advance\v@lX\v@lmin\advance\v@lY\v@lmax}
\ctr@ld@f\def\fcn@deE{%
    \@ecfor\p@int:=\list@num\do{\d@butn@de\v@lX=\unssqrttw@\v@lX\v@lY=\unssqrttw@\v@lY%
    \ifdim\thickn@ss\p@>\z@% Shadow
    \v@lXa=\v@lX\advance\v@lXa\delt@\v@lXa=\ptT@unit@\v@lXa\edef\XR@d{\repdecn@mb\v@lXa}%
    \v@lYa=\v@lY\advance\v@lYa\delt@\v@lYa=\ptT@unit@\v@lYa\edef\YR@d{\repdecn@mb\v@lYa}%
    \arct@n\v@leur(\v@lXa,\v@lYa)\v@leur=\rdT@deg\v@leur\edef\@nglde{\repdecn@mb\v@leur}%
    {\c@lptellDD-2::\p@int;\XR@d,\YR@d(\@nglde)}% \v@lmin & \v@lmax modified in \c@lptellDD
    \advance\v@leur-\PI@deg\edef\@nglun{\repdecn@mb\v@leur}%
    {\c@lptellDD-3::\p@int;\XR@d,\YR@d(\@nglun)}%
    \figptstra-6=-3,-2,\p@int/\thickn@ss,-8/\pssetfillmode{yes}\us@secondC@lor%
    \psline[-2,-3,-6,-5]\psarcell-4;\XR@d,\YR@d(\@nglun,\@nglde,0)\fi% End shadow
    \v@lX=\ptT@unit@\v@lX\v@lY=\ptT@unit@\v@lY%
    \edef\XR@d{\repdecn@mb\v@lX}\edef\YR@d{\repdecn@mb\v@lY}%
    \pssetfillmode{yes}\us@thirdC@lor\psarcell\p@int;\XR@d,\YR@d(0,360,0)%
    \pssetfillmode{no}\us@primarC@lor\psarcell\p@int;\XR@d,\YR@d(0,360,0)}}
\ctr@ld@f\def\fcn@deL{\delt@=\ptT@unit@\delt@\edef\t@ille{\repdecn@mb\delt@}%
    \@ecfor\p@int:=\list@num\do{\Figg@tXYa{\p@int}\d@butn@de%
    \ifdim\v@lX>\v@lY\itis@Ktrue\else\itis@Kfalse\fi%
    \advance\v@lXa-\v@lX\Figp@intreg-1:(\v@lXa,\v@lYa)%
    \advance\v@lXa\v@lX\advance\v@lYa-\v@lY\Figp@intreg-2:(\v@lXa,\v@lYa)%
    \advance\v@lXa\v@lX\advance\v@lYa\v@lY\Figp@intreg-3:(\v@lXa,\v@lYa)%
    \advance\v@lXa-\v@lX\advance\v@lYa\v@lY\Figp@intreg-4:(\v@lXa,\v@lYa)%
    \ifdim\thickn@ss\p@>\z@\Figg@tXYa{\p@int}\pssetfillmode{yes}\us@secondC@lor% Shadow
    \c@lpt@xt{-1}{-4}\c@lpt@xt@\v@lXa\v@lYa\v@lX\v@lY\c@rre\delt@%
    \Figp@intregDD-9:(\v@lZ,\v@lYa)\Figp@intregDD-11:(\v@lZa,\v@lYa)%
    \c@lpt@xt{-4}{-3}\c@lpt@xt@\v@lYa\v@lXa\v@lY\v@lX\delt@\c@rre%
    \Figp@intregDD-12:(\v@lXa,\v@lZ)\Figp@intregDD-10:(\v@lXa,\v@lZa)%
    \ifitis@K\figptstra-7=-9,-10,-11/\thickn@ss,-8/\psline[-9,-11,-5,-6,-7]\else%
    \figptstra-7=-10,-11,-12/\thickn@ss,-8/\psline[-10,-12,-5,-6,-7]\fi\fi% End shadow
    \pssetfillmode{yes}\us@thirdC@lor\psline[-1,-2,-3,-4]%
    \pssetfillmode{no}\us@primarC@lor\psline[-1,-2,-3,-4,-1]}}
\ctr@ld@f\def\c@lpt@xt#1#2{\figvectN-7[#1,#2]\vecunit@{-7}{-7}\figpttra-5:=#1/\t@ille,-7/%
    \figvectP-7[#1,#2]\Figg@tXY{-7}\c@rre=\v@lX\delt@=\v@lY\Figg@tXY{-5}}
\ctr@ld@f\def\c@lpt@xt@#1#2#3#4#5#6{\v@lZ=#6\invers@{\v@lZ}{\v@lZ}\v@leur=\repdecn@mb{#5}\v@lZ%
    \v@lZ=#2\advance\v@lZ-#4\mili@u=\repdecn@mb{\v@leur}\v@lZ%
    \v@lZ=#3\advance\v@lZ\mili@u\v@lZa=-\v@lZ\advance\v@lZa\tw@#1}
\ctr@ld@f\def\fcn@deR{\@ecfor\p@int:=\list@num\do{\Figg@tXYa{\p@int}\d@butn@de%
    \advance\v@lXa-0.5\v@lX\advance\v@lYa-0.5\v@lY\Figp@intreg-1:(\v@lXa,\v@lYa)%
    \advance\v@lXa\v@lX\Figp@intreg-2:(\v@lXa,\v@lYa)%
    \advance\v@lYa\v@lY\Figp@intreg-3:(\v@lXa,\v@lYa)%
    \advance\v@lXa-\v@lX\Figp@intreg-4:(\v@lXa,\v@lYa)%
    \ifdim\thickn@ss\p@>\z@\pssetfillmode{yes}\us@secondC@lor% Shadow
    \Figv@ctCreg-5(-\delt@,-\delt@)\figpttra-9:=-1/1,-5/%
    \Figv@ctCreg-5(\delt@,-\delt@)\figpttra-10:=-2/1,-5/%
    \Figv@ctCreg-5(\delt@,\delt@)\figpttra-11:=-3/1,-5/%
    \figptstra-7=-9,-10,-11/\thickn@ss,-8/\psline[-9,-11,-5,-6,-7]\fi% End shadow
    \pssetfillmode{yes}\us@thirdC@lor\psline[-1,-2,-3,-4]%
    \pssetfillmode{no}\us@primarC@lor\psline[-1,-2,-3,-4,-1]}}
\ctr@ln@m\@rrowp@s
\ctr@ln@m\Xp@dd     \ctr@ln@m\Yp@dd
\ctr@ln@m\fclin@r@d \ctr@ln@m\thickn@ss
\ctr@ld@f\def\Pssetfl@wchart#1=#2|{\keln@mtr#1|%
    \def\n@mref{arr}\ifx\l@debut\n@mref\expandafter\keln@mtr\l@suite|%
     \def\n@mref{owp}\ifx\l@debut\n@mref\edef\@rrowp@s{#2}\else% arrowposition
     \def\n@mref{owr}\ifx\l@debut\n@mref\setfcr@fpt#2|\else% arrowrefpt
     \immediate\write16{*** Unknown attribute: \BS@ psset flowchart(..., #1=...)}%
     \fi\fi\else%
    \def\n@mref{lin}\ifx\l@debut\n@mref\setfccurv@#2|\else% line
    \def\n@mref{pad}\ifx\l@debut\n@mref\edef\Xp@dd{#2}\edef\Yp@dd{#2}\else% padding
    \def\n@mref{rad}\ifx\l@debut\n@mref\edef\fclin@r@d{#2}\else% connection radius
    \def\n@mref{sha}\ifx\l@debut\n@mref\setfcshap@#2|\else% shape
    \def\n@mref{thi}\ifx\l@debut\n@mref\edef\thickn@ss{#2}\else% thickness
    \def\n@mref{xpa}\ifx\l@debut\n@mref\edef\Xp@dd{#2}\else% xpadding
    \def\n@mref{ypa}\ifx\l@debut\n@mref\edef\Yp@dd{#2}\else% ypadding
    \immediate\write16{*** Unknown attribute: \BS@ psset flowchart(..., #1=...)}%
    \fi\fi\fi\fi\fi\fi\fi\fi}
\ctr@ln@m\@rrowr@fpt \ctr@ln@m\fclin@typ@
\ctr@ld@f\def\setfcr@fpt#1#2|{\if#1e\def\@rrowr@fpt{1}\else\def\@rrowr@fpt{0}\fi}
\ctr@ld@f\def\setfccurv@#1#2|{\if#1c\def\fclin@typ@{0}\else\def\fclin@typ@{1}\fi}
\ctr@ln@m\h@rdfcXp@dd \ctr@ln@m\h@rdfcYp@dd
\ctr@ln@m\fcn@de \ctr@ln@m\fcsh@pe
\ctr@ld@f\def\setfcshap@#1#2|{%
    \if#1e\let\fcn@de=\fcn@deE\def\h@rdfcXp@dd{4pt}\def\h@rdfcYp@dd{4pt}%
     \edef\fcsh@pe{ellipse}\else%
    \if#1l\let\fcn@de=\fcn@deL\def\h@rdfcXp@dd{4pt}\def\h@rdfcYp@dd{4pt}%
     \edef\fcsh@pe{lozenge}\else%
          \let\fcn@de=\fcn@deR\def\h@rdfcXp@dd{6pt}\def\h@rdfcYp@dd{6pt}%
     \edef\fcsh@pe{rectangle}\fi\fi}
\ctr@ld@f\def\psline[#1]{{\ifcurr@ntPS\ifps@cri\PSc@mment{psline Points=#1}%
    \let\pslign@=\Pslign@P\Pslin@{#1}\PSc@mment{End psline}\fi\fi}}
\ctr@ld@f\def\pslineF#1{{\ifcurr@ntPS\ifps@cri\PSc@mment{pslineF Filename=#1}%
    \let\pslign@=\Pslign@F\Pslin@{#1}\PSc@mment{End pslineF}\fi\fi}}
\ctr@ld@f\def\pslineC(#1){{\ifcurr@ntPS\ifps@cri\PSc@mment{pslineC}%
    \let\pslign@=\Pslign@C\Pslin@{#1}\PSc@mment{End pslineC}\fi\fi}}
\ctr@ld@f\def\Pslin@#1{\iffillm@de\pslign@{#1}%
    \f@gfill%
    \else\pslign@{#1}\ifx\derp@int\premp@int%
    \f@gclosestroke%
    \else\f@gstroke\fi\fi}
\ctr@ld@f\def\Pslign@P#1{\def\list@num{#1}\extrairelepremi@r\p@int\de\list@num%
    \edef\premp@int{\p@int}\f@gnewpath%
    \PSwrit@cmd{\p@int}{\c@mmoveto}{\fwf@g}%
    \@ecfor\p@int:=\list@num\do{\PSwrit@cmd{\p@int}{\c@mlineto}{\fwf@g}%
    \edef\derp@int{\p@int}}}
\ctr@ld@f\def\Pslign@F#1{\s@uvc@ntr@l\et@tPslign@F\setc@ntr@l{2}\openin\frf@g=#1\relax%
    \ifeof\frf@g\message{*** File #1 not found !}\end\else%
    \read\frf@g to\tr@c\edef\premp@int{\tr@c}\expandafter\extr@ctCF\tr@c:%
    \f@gnewpath\PSwrit@cmd{-1}{\c@mmoveto}{\fwf@g}%
    \loop\read\frf@g to\tr@c\ifeof\frf@g\mored@tafalse\else\mored@tatrue\fi%
    \ifmored@ta\expandafter\extr@ctCF\tr@c:\PSwrit@cmd{-1}{\c@mlineto}{\fwf@g}%
    \edef\derp@int{\tr@c}\repeat\fi\closein\frf@g\resetc@ntr@l\et@tPslign@F}
\ctr@ln@m\extr@ctCF
\ctr@ld@f\def\extr@ctCFDD#1 #2:{\v@lX=#1\unit@\v@lY=#2\unit@\Figp@intregDD-1:(\v@lX,\v@lY)}
\ctr@ld@f\def\extr@ctCFTD#1 #2 #3:{\v@lX=#1\unit@\v@lY=#2\unit@\v@lZ=#3\unit@%
    \Figp@intregTD-1:(\v@lX,\v@lY,\v@lZ)}
\ctr@ld@f\def\Pslign@C#1{\s@uvc@ntr@l\et@tPslign@C\setc@ntr@l{2}%
    \def\list@num{#1}\extrairelepremi@r\p@int\de\list@num%
    \edef\premp@int{\p@int}\f@gnewpath%
    \expandafter\Pslign@C@\p@int:\PSwrit@cmd{-1}{\c@mmoveto}{\fwf@g}%
    \@ecfor\p@int:=\list@num\do{\expandafter\Pslign@C@\p@int:%
    \PSwrit@cmd{-1}{\c@mlineto}{\fwf@g}\edef\derp@int{\p@int}}%
    \resetc@ntr@l\et@tPslign@C}
\ctr@ld@f\def\Pslign@C@#1 #2:{{\def\t@xt@{#1}\ifx\t@xt@\empty\Pslign@C@#2:% Discard leading spaces
    \else\extr@ctCF#1 #2:\fi}}
\ctr@ln@m\c@ntrolmesh
\ctr@ld@f\def\Pssetm@sh#1=#2|{\keln@mun#1|%
    \def\n@mref{d}\ifx\l@debut\n@mref\pssetmeshdiag{#2}\else% diag
    \immediate\write16{*** Unknown attribute: \BS@ psset mesh(..., #1=...)}%
    \fi}
\ctr@ld@f\def\pssetmeshdiag#1{\edef\c@ntrolmesh{#1}}
\ctr@ld@f\def\defaultmeshdiag{0}    % Valeur par defaut
\ctr@ld@f\def\psmesh#1,#2[#3,#4,#5,#6]{{\ifcurr@ntPS\ifps@cri%
    \PSc@mment{psmesh N1=#1, N2=#2, Quadrangle=[#3,#4,#5,#6]}%
    \s@uvc@ntr@l\et@tpsmesh\Pss@tsecondSt\setc@ntr@l{2}%
    \ifnum#1>\@ne\Psmeshp@rt#1[#3,#4,#5,#6]\fi%
    \ifnum#2>\@ne\Psmeshp@rt#2[#4,#5,#6,#3]\fi%
    \ifnum\c@ntrolmesh>\z@\Psmeshdi@g#1,#2[#3,#4,#5,#6]\fi%
    \ifnum\c@ntrolmesh<\z@\Psmeshdi@g#2,#1[#4,#5,#6,#3]\fi\Psrest@reSt%
    \psline[#3,#4,#5,#6,#3]\PSc@mment{End psmesh}\resetc@ntr@l\et@tpsmesh\fi\fi}}
\ctr@ld@f\def\Psmeshp@rt#1[#2,#3,#4,#5]{{\l@mbd@un=\@ne\l@mbd@de=#1\loop%
    \ifnum\l@mbd@un<#1\advance\l@mbd@de\m@ne\figptbary-1:[#2,#3;\l@mbd@de,\l@mbd@un]%
    \figptbary-2:[#5,#4;\l@mbd@de,\l@mbd@un]\psline[-1,-2]\advance\l@mbd@un\@ne\repeat}}
\ctr@ld@f\def\Psmeshdi@g#1,#2[#3,#4,#5,#6]{\figptcopy-2:/#3/\figptcopy-3:/#6/%
    \l@mbd@un=\z@\l@mbd@de=#1\loop\ifnum\l@mbd@un<#1%
    \advance\l@mbd@un\@ne\advance\l@mbd@de\m@ne\figptcopy-1:/-2/\figptcopy-4:/-3/%
    \figptbary-2:[#3,#4;\l@mbd@de,\l@mbd@un]%
    \figptbary-3:[#6,#5;\l@mbd@de,\l@mbd@un]\Psmeshdi@gp@rt#2[-1,-2,-3,-4]\repeat}
\ctr@ld@f\def\Psmeshdi@gp@rt#1[#2,#3,#4,#5]{{\l@mbd@un=\z@\l@mbd@de=#1\loop%
    \ifnum\l@mbd@un<#1\figptbary-5:[#2,#5;\l@mbd@de,\l@mbd@un]%
    \advance\l@mbd@de\m@ne\advance\l@mbd@un\@ne%
    \figptbary-6:[#3,#4;\l@mbd@de,\l@mbd@un]\psline[-5,-6]\repeat}}
\ctr@ln@m\psnormal
\ctr@ld@f\def\psnormalDD#1,#2[#3,#4]{{\ifcurr@ntPS\ifps@cri%
    \PSc@mment{psnormal Length=#1, Lambda=#2 [Pt1,Pt2]=[#3,#4]}%
    \s@uvc@ntr@l\et@tpsnormal\resetc@ntr@l{2}\figptendnormal-6::#1,#2[#3,#4]%
    \figptcopyDD-5:/-1/\psarrow[-5,-6]%
    \PSc@mment{End psnormal}\resetc@ntr@l\et@tpsnormal\fi\fi}}
\ctr@ld@f\def\psreset#1{\trtlis@rg{#1}{\Psreset@}}
\ctr@ld@f\def\Psreset@#1|{\keln@mde#1|%
    \def\n@mref{ar}\ifx\l@debut\n@mref\psresetarrowhead\else% arrowhead
    \def\n@mref{cu}\ifx\l@debut\n@mref\psset curve(roundness=\defaultroundness)\else% curve
    \def\n@mref{fi}\ifx\l@debut\n@mref\psset (color=\defaultcolor,dash=\defaultdash,%
         fill=\defaultfill,join=\defaultjoin,width=\defaultwidth)\else% primary settings
    \def\n@mref{fl}\ifx\l@debut\n@mref\psset flowchart(arrowp=\defaultfcarrowposition,%
	arrowr=\defaultfcarrowrefpt,line=\defaultfcline,xpadd=\defaultfcxpadding,%
	ypadd=\defaultfcypadding,radius=\defaultfcradius,shape=\defaultfcshape,%
	thick=\defaultfcthickness)\else% flow chart
    \def\n@mref{me}\ifx\l@debut\n@mref\psset mesh(diag=\defaultmeshdiag)\else% mesh
    \def\n@mref{se}\ifx\l@debut\n@mref\psresetsecondsettings\else% secondary
    \def\n@mref{th}\ifx\l@debut\n@mref\psset third(color=\defaultthirdcolor)\else% ternary
    \immediate\write16{*** Unknown keyword #1 (\BS@ psreset).}%
    \fi\fi\fi\fi\fi\fi\fi}
\ctr@ld@f\def\psset#1(#2){\def\t@xt@{#1}\ifx\t@xt@\empty\trtlis@rg{#2}{\Pssetf@rst}% primary settings
    \else\keln@mde#1|%
    \def\n@mref{ar}\ifx\l@debut\n@mref\trtlis@rg{#2}{\Psset@rrowhe@d}\else% arrow-head
    \def\n@mref{cu}\ifx\l@debut\n@mref\trtlis@rg{#2}{\Pssetc@rve}\else% curve
    \def\n@mref{fi}\ifx\l@debut\n@mref\trtlis@rg{#2}{\Pssetf@rst}\else% primary settings
    \def\n@mref{fl}\ifx\l@debut\n@mref\trtlis@rg{#2}{\Pssetfl@wchart}\else% flow chart
    \def\n@mref{me}\ifx\l@debut\n@mref\trtlis@rg{#2}{\Pssetm@sh}\else% mesh
    \def\n@mref{se}\ifx\l@debut\n@mref\trtlis@rg{#2}{\Pssets@cond}\else% secondary settings
    \def\n@mref{th}\ifx\l@debut\n@mref\trtlis@rg{#2}{\Pssetth@rd}\else% ternary settings
    \immediate\write16{*** Unknown keyword: \BS@ psset #1(...)}%
    \fi\fi\fi\fi\fi\fi\fi\fi}
\ctr@ld@f\def\pssetdefault#1(#2){\ifcurr@ntPS\immediate\write16{*** \BS@ pssetdefault is ignored
    inside a \BS@ psbeginfig-\BS@ psendfig block.}%
    \immediate\write16{*** It must be called before \BS@ psbeginfig.}\else%
    \def\t@xt@{#1}\ifx\t@xt@\empty\trtlis@rg{#2}{\Pssd@f@rst}\else\keln@mde#1|%
    \def\n@mref{ar}\ifx\l@debut\n@mref\trtlis@rg{#2}{\Pssd@@rrowhe@d}\else% arrow-head
    \def\n@mref{cu}\ifx\l@debut\n@mref\trtlis@rg{#2}{\Pssd@c@rve}\else% curve
    \def\n@mref{fi}\ifx\l@debut\n@mref\trtlis@rg{#2}{\Pssd@f@rst}\else% primary settings
    \def\n@mref{fl}\ifx\l@debut\n@mref\trtlis@rg{#2}{\Pssd@fl@wchart}\else% flow chart
    \def\n@mref{me}\ifx\l@debut\n@mref\trtlis@rg{#2}{\Pssd@m@sh}\else% mesh
    \def\n@mref{se}\ifx\l@debut\n@mref\trtlis@rg{#2}{\Pssd@s@cond}\else% secondary settings
    \def\n@mref{th}\ifx\l@debut\n@mref\trtlis@rg{#2}{\Pssd@th@rd}\else% ternary settings
    \immediate\write16{*** Unknown keyword: \BS@ pssetdefault #1(...)}%
    \fi\fi\fi\fi\fi\fi\fi\fi\initpss@ttings\fi}
\ctr@ld@f\def\Pssd@f@rst#1=#2|{\keln@mun#1|%
    \def\n@mref{c}\ifx\l@debut\n@mref\edef\defaultcolor{#2}\else% color
    \def\n@mref{d}\ifx\l@debut\n@mref\edef\defaultdash{#2}\else% dash
    \def\n@mref{f}\ifx\l@debut\n@mref\edef\defaultfill{#2}\else% fillmode
    \def\n@mref{j}\ifx\l@debut\n@mref\edef\defaultjoin{#2}\else% line join
    \def\n@mref{u}\ifx\l@debut\n@mref\edef\defaultupdate{#2}\pssetupdate{#2}\else% update
    \def\n@mref{w}\ifx\l@debut\n@mref\edef\defaultwidth{#2}\else% line width
    \immediate\write16{*** Unknown attribute: \BS@ pssetdefault (..., #1=...)}%
    \fi\fi\fi\fi\fi\fi}
\ctr@ld@f\def\Pssd@@rrowhe@d#1=#2|{\keln@mun#1|%
    \def\n@mref{a}\ifx\l@debut\n@mref\edef\defaultarrowheadangle{#2}\else% angle
    \def\n@mref{f}\ifx\l@debut\n@mref\edef\defaultarrowheadangle{#2}\else% fillmode
    \def\n@mref{l}\ifx\l@debut\n@mref\y@tiunit{#2}\ifunitpr@sent%
     \edef\defaulth@rdahlength{#2}\else\edef\defaulth@rdahlength{#2pt}%
     \message{*** \BS@ pssetdefault (..., #1=#2, ...) : unit is missing, pt is assumed.}%
     \fi\else% length
    \def\n@mref{o}\ifx\l@debut\n@mref\edef\defaultarrowheadout{#2}\else% out
    \def\n@mref{r}\ifx\l@debut\n@mref\edef\defaultarrowheadratio{#2}\else% ratio
    \immediate\write16{*** Unknown attribute: \BS@ pssetdefault arrowhead(..., #1=...)}%
    \fi\fi\fi\fi\fi}
\ctr@ld@f\def\Pssd@c@rve#1=#2|{\keln@mun#1|%
    \def\n@mref{r}\ifx\l@debut\n@mref\edef\defaultroundness{#2}\else%
    \immediate\write16{*** Unknown attribute: \BS@ pssetdefault curve(..., #1=...)}%
    \fi}
\ctr@ld@f\def\Pssd@fl@wchart#1=#2|{\keln@mtr#1|%
    \def\n@mref{arr}\ifx\l@debut\n@mref\expandafter\keln@mtr\l@suite|%
     \def\n@mref{owp}\ifx\l@debut\n@mref\edef\defaultfcarrowposition{#2}\else% arrowposition
     \def\n@mref{owr}\ifx\l@debut\n@mref\edef\defaultfcarrowrefpt{#2}\else% arrowrefpt
     \immediate\write16{*** Unknown attribute: \BS@ pssetdefault flowchart(..., #1=...)}%
     \fi\fi\else%
    \def\n@mref{lin}\ifx\l@debut\n@mref\edef\defaultfcline{#2}\else% line
    \def\n@mref{pad}\ifx\l@debut\n@mref\edef\defaultfcxpadding{#2}%
                    \edef\defaultfcypadding{#2}\else% padding
    \def\n@mref{rad}\ifx\l@debut\n@mref\edef\defaultfcradius{#2}\else% connection radius
    \def\n@mref{sha}\ifx\l@debut\n@mref\edef\defaultfcshape{#2}\else% shape
    \def\n@mref{thi}\ifx\l@debut\n@mref\edef\defaultfcthickness{#2}\else% thickness
    \def\n@mref{xpa}\ifx\l@debut\n@mref\edef\defaultfcxpadding{#2}\else% xpadding
    \def\n@mref{ypa}\ifx\l@debut\n@mref\edef\defaultfcypadding{#2}\else% ypadding
    \immediate\write16{*** Unknown attribute: \BS@ pssetdefault flowchart(..., #1=...)}%
    \fi\fi\fi\fi\fi\fi\fi\fi}
\ctr@ld@f\def\defaultfcarrowposition{0.5}%\ctr@ld@f\let\defaultfcarrowpos=\defaultfcarrowposition
\ctr@ld@f\def\defaultfcarrowrefpt{start}
\ctr@ld@f\def\defaultfcline{polygon}
\ctr@ld@f\def\defaultfcradius{0}
\ctr@ld@f\def\defaultfcshape{rectangle}
\ctr@ld@f\def\defaultfcthickness{0}%\ctr@ld@f\let\defaultfcthick=\defaultfcthickness
\ctr@ld@f\def\defaultfcxpadding{0}%\ctr@ld@f\let\defaultfcxpad=\defaultfcxpadding
\ctr@ld@f\def\defaultfcypadding{0}%\ctr@ld@f\let\defaultfcypad=\defaultfcypadding
\ctr@ld@f\def\Pssd@m@sh#1=#2|{\keln@mun#1|%
    \def\n@mref{d}\ifx\l@debut\n@mref\edef\defaultmeshdiag{#2}\else%
    \immediate\write16{*** Unknown attribute: \BS@ pssetdefault mesh(..., #1=...)}%
    \fi}
\ctr@ld@f\def\Pssd@s@cond#1=#2|{\keln@mun#1|%
    \def\n@mref{c}\ifx\l@debut\n@mref\edef\defaultsecondcolor{#2}\else%
    \def\n@mref{d}\ifx\l@debut\n@mref\edef\defaultseconddash{#2}\else%
    \def\n@mref{w}\ifx\l@debut\n@mref\edef\defaultsecondwidth{#2}\else%
    \immediate\write16{*** Unknown attribute: \BS@ pssetdefault second(..., #1=...)}%
    \fi\fi\fi}
\ctr@ld@f\def\Pssd@th@rd#1=#2|{\keln@mun#1|%
    \def\n@mref{c}\ifx\l@debut\n@mref\edef\defaultthirdcolor{#2}\else%
    \immediate\write16{*** Unknown attribute: \BS@ pssetdefault third(..., #1=...)}%
    \fi}
\ctr@ln@w{newif}\iffillm@de
\ctr@ld@f\def\pssetfillmode#1{\expandafter\setfillm@de#1:}
\ctr@ld@f\def\setfillm@de#1#2:{\if#1n\fillm@defalse\else\fillm@detrue\fi}
\ctr@ld@f\def\defaultfill{no}     % Valeur par defaut
\ctr@ln@w{newif}\ifpsupdatem@de
\ctr@ld@f\def\pssetupdate#1{\ifcurr@ntPS\immediate\write16{*** \BS@ pssetupdate is ignored inside a
     \BS@ psbeginfig-\BS@ psendfig block.}%
    \immediate\write16{*** It must be called before \BS@ psbeginfig.}%
    \else\expandafter\setupd@te#1:\fi}
\ctr@ld@f\def\setupd@te#1#2:{\if#1n\psupdatem@defalse\else\psupdatem@detrue\fi}
\ctr@ld@f\def\defaultupdate{no}     % Valeur par defaut
\ctr@ln@m\curr@ntcolor \ctr@ln@m\curr@ntcolorc@md
\ctr@ld@f\def\Pssetc@lor#1{\ifps@cri\result@tent=\@ne\expandafter\c@lnbV@l#1 :%
    \def\curr@ntcolor{}\def\curr@ntcolorc@md{}%
    \ifcase\result@tent\or\pssetgray{#1}\or\or\pssetrgb{#1}\or\pssetcmyk{#1}\fi\fi}
\ctr@ln@m\curr@ntcolorc@mdStroke
\ctr@ld@f\def\pssetcmyk#1{\ifps@cri\def\curr@ntcolor{#1}\def\curr@ntcolorc@md{\c@msetcmykcolor}%
    \def\curr@ntcolorc@mdStroke{\c@msetcmykcolorStroke}%
    \ifcurr@ntPS\PSc@mment{pssetcmyk Color=#1}\us@primarC@lor\fi\fi}
\ctr@ld@f\def\pssetrgb#1{\ifps@cri\def\curr@ntcolor{#1}\def\curr@ntcolorc@md{\c@msetrgbcolor}%
    \def\curr@ntcolorc@mdStroke{\c@msetrgbcolorStroke}%
    \ifcurr@ntPS\PSc@mment{pssetrgb Color=#1}\us@primarC@lor\fi\fi}
\ctr@ld@f\def\pssetgray#1{\ifps@cri\def\curr@ntcolor{#1}\def\curr@ntcolorc@md{\c@msetgray}%
    \def\curr@ntcolorc@mdStroke{\c@msetgrayStroke}%
    \ifcurr@ntPS\PSc@mment{pssetgray Gray level=#1}\us@primarC@lor\fi\fi}
\ctr@ln@m\fillc@md
\ctr@ld@f\def\us@primarC@lor{\immediate\write\fwf@g{\d@fprimarC@lor}%
    \let\fillc@md=\prfillc@md}
\ctr@ld@f\def\prfillc@md{\d@fprimarC@lor\space\c@mfill}
\ctr@ld@f\def\defaultcolor{0}       % Valeur par defaut
\ctr@ld@f\def\c@lnbV@l#1 #2:{\def\t@xt@{#1}\relax\ifx\t@xt@\empty\c@lnbV@l#2:% Discard leading spaces
    \else\c@lnbV@l@#1 #2:\fi}
\ctr@ld@f\def\c@lnbV@l@#1 #2:{\def\t@xt@{#2}\ifx\t@xt@\empty%
    \def\t@xt@{#1}\ifx\t@xt@\empty\advance\result@tent\m@ne\fi% Discard trailing spaces
    \else\advance\result@tent\@ne\c@lnbV@l@#2:\fi}
\ctr@ld@f\def\Blackcmyk{0 0 0 1}
\ctr@ld@f\def\Whitecmyk{0 0 0 0}
\ctr@ld@f\def\Cyancmyk{1 0 0 0}
\ctr@ld@f\def\Magentacmyk{0 1 0 0}
\ctr@ld@f\def\Yellowcmyk{0 0 1 0}
\ctr@ld@f\def\Redcmyk{0 1 1 0}
\ctr@ld@f\def\Greencmyk{1 0 1 0}
\ctr@ld@f\def\Bluecmyk{1 1 0 0}
\ctr@ld@f\def\Graycmyk{0 0 0 0.50}
\ctr@ld@f\def\BrickRedcmyk{0 0.89 0.94 0.28} % PANTONE 1805
\ctr@ld@f\def\Browncmyk{0 0.81 1 0.60} % PANTONE 1615
\ctr@ld@f\def\ForestGreencmyk{0.91 0 0.88 0.12} % PANTONE 349
\ctr@ld@f\def\Goldenrodcmyk{ 0 0.10 0.84 0} % PANTONE 109
\ctr@ld@f\def\Marooncmyk{0 0.87 0.68 0.32} % PANTONE 201
\ctr@ld@f\def\Orangecmyk{0 0.61 0.87 0} % PANTONE ORANGE-021
\ctr@ld@f\def\Purplecmyk{0.45 0.86 0 0} % PANTONE PURPLE
\ctr@ld@f\def\RoyalBluecmyk{1. 0.50 0 0} % No PANTONE match
\ctr@ld@f\def\Violetcmyk{0.79 0.88 0 0} % PANTONE VIOLET
\ctr@ld@f\def\Blackrgb{0 0 0}
\ctr@ld@f\def\Whitergb{1 1 1}
\ctr@ld@f\def\Redrgb{1 0 0}
\ctr@ld@f\def\Greenrgb{0 1 0}
\ctr@ld@f\def\Bluergb{0 0 1}
\ctr@ld@f\def\Cyanrgb{0 1 1}
\ctr@ld@f\def\Magentargb{1 0 1}
\ctr@ld@f\def\Yellowrgb{1 1 0}
\ctr@ld@f\def\Grayrgb{0.5 0.5 0.5}
\ctr@ld@f\def\Chocolatergb{0.824 0.412 0.118}
\ctr@ld@f\def\DarkGoldenrodrgb{0.722 0.525 0.043}
\ctr@ld@f\def\DarkOrangergb{1 0.549 0}
\ctr@ld@f\def\Firebrickrgb{0.698 0.133 0.133}
\ctr@ld@f\def\ForestGreenrgb{0.133 0.545 0.133}
\ctr@ld@f\def\Goldrgb{1 0.843 0}
\ctr@ld@f\def\HotPinkrgb{1 0.412 0.706}
\ctr@ld@f\def\Maroonrgb{0.690 0.188 0.376}
\ctr@ld@f\def\Pinkrgb{1 0.753 0.796}
\ctr@ld@f\def\RoyalBluergb{0.255 0.412 0.882}
\ctr@ld@f\def\Pssetf@rst#1=#2|{\keln@mun#1|%
    \def\n@mref{c}\ifx\l@debut\n@mref\Pssetc@lor{#2}\else% color
    \def\n@mref{d}\ifx\l@debut\n@mref\pssetdash{#2}\else% dash
    \def\n@mref{f}\ifx\l@debut\n@mref\pssetfillmode{#2}\else% fillmode
    \def\n@mref{j}\ifx\l@debut\n@mref\pssetjoin{#2}\else% line join
    \def\n@mref{u}\ifx\l@debut\n@mref\pssetupdate{#2}\else% update
    \def\n@mref{w}\ifx\l@debut\n@mref\pssetwidth{#2}\else% line width
    \immediate\write16{*** Unknown attribute: \BS@ psset (..., #1=...)}%
    \fi\fi\fi\fi\fi\fi}
\ctr@ln@m\curr@ntdash
\ctr@ld@f\def\s@uvdash#1{\edef#1{\curr@ntdash}}
\ctr@ld@f\def\defaultdash{1}        % Valeur par defaut (numero sans espace)
\ctr@ld@f\def\pssetdash#1{\ifps@cri\edef\curr@ntdash{#1}\ifcurr@ntPS\expandafter\Pssetd@sh#1 :\fi\fi}
\ctr@ld@f\def\Pssetd@shI#1{\PSc@mment{pssetdash Index=#1}\ifcase#1%
    \or\immediate\write\fwf@g{[] 0 \c@msetdash}%         Index=1
    \or\immediate\write\fwf@g{[6 2] 0 \c@msetdash}%      Index=2
    \or\immediate\write\fwf@g{[4 2] 0 \c@msetdash}%      Index=3
    \or\immediate\write\fwf@g{[2 2] 0 \c@msetdash}%      Index=4
    \or\immediate\write\fwf@g{[1 2] 0 \c@msetdash}%      Index=5
    \or\immediate\write\fwf@g{[2 4] 0 \c@msetdash}%      Index=6
    \or\immediate\write\fwf@g{[3 5] 0 \c@msetdash}%      Index=7
    \or\immediate\write\fwf@g{[3 3] 0 \c@msetdash}%      Index=8
    \or\immediate\write\fwf@g{[3 5 1 5] 0 \c@msetdash}%  Index=9
    \or\immediate\write\fwf@g{[6 4 2 4] 0 \c@msetdash}%  Index=10
    \fi}
\ctr@ld@f\def\Pssetd@sh#1 #2:{{\def\t@xt@{#1}\ifx\t@xt@\empty\Pssetd@sh#2:% Discard leading spaces
    \else\def\t@xt@{#2}\ifx\t@xt@\empty\Pssetd@shI{#1}\else\s@mme=\@ne\def\debutp@t{#1}%
    \an@lysd@sh#2:\ifodd\s@mme\edef\debutp@t{\debutp@t\space\finp@t}\def\finp@t{0}\fi%
    \PSc@mment{pssetdash Pattern=#1 #2}%
    \immediate\write\fwf@g{[\debutp@t] \finp@t\space\c@msetdash}\fi\fi}}
\ctr@ld@f\def\an@lysd@sh#1 #2:{\def\t@xt@{#2}\ifx\t@xt@\empty\def\finp@t{#1}\else%
    \edef\debutp@t{\debutp@t\space#1}\advance\s@mme\@ne\an@lysd@sh#2:\fi}
\ctr@ln@m\curr@ntwidth
\ctr@ld@f\def\s@uvwidth#1{\edef#1{\curr@ntwidth}}
\ctr@ld@f\def\defaultwidth{0.4}     % Valeur par defaut
\ctr@ld@f\def\pssetwidth#1{\ifps@cri\edef\curr@ntwidth{#1}\ifcurr@ntPS%
    \PSc@mment{pssetwidth Width=#1}\immediate\write\fwf@g{#1 \c@msetlinewidth}\fi\fi}
\ctr@ln@m\curr@ntjoin
\ctr@ld@f\def\pssetjoin#1{\ifps@cri\edef\curr@ntjoin{#1}\ifcurr@ntPS\expandafter\Pssetj@in#1:\fi\fi}
\ctr@ld@f\def\Pssetj@in#1#2:{\PSc@mment{pssetjoin join=#1}%
    \if#1r\def\t@xt@{1}\else\if#1b\def\t@xt@{2}\else\def\t@xt@{0}\fi\fi%
    \immediate\write\fwf@g{\t@xt@\space\c@msetlinejoin}}
\ctr@ld@f\def\defaultjoin{miter}   % Valeur par defaut
\ctr@ld@f\def\Pssets@cond#1=#2|{\keln@mun#1|%
    \def\n@mref{c}\ifx\l@debut\n@mref\Pssets@condcolor{#2}\else%
    \def\n@mref{d}\ifx\l@debut\n@mref\pssetseconddash{#2}\else%
    \def\n@mref{w}\ifx\l@debut\n@mref\pssetsecondwidth{#2}\else%
    \immediate\write16{*** Unknown attribute: \BS@ psset second(..., #1=...)}%
    \fi\fi\fi}
\ctr@ln@m\curr@ntseconddash
\ctr@ld@f\def\pssetseconddash#1{\edef\curr@ntseconddash{#1}}
\ctr@ld@f\def\defaultseconddash{4}  % Valeur par defaut (numero sans espace)
\ctr@ln@m\curr@ntsecondwidth
\ctr@ld@f\def\pssetsecondwidth#1{\edef\curr@ntsecondwidth{#1}}
\ctr@ld@f\edef\defaultsecondwidth{\defaultwidth} % Valeur par defaut
\ctr@ld@f\def\psresetsecondsettings{%
    \pssetseconddash{\defaultseconddash}\pssetsecondwidth{\defaultsecondwidth}%
    \Pssets@condcolor{\defaultsecondcolor}}
\ctr@ln@m\sec@ndcolor \ctr@ln@m\sec@ndcolorc@md
\ctr@ld@f\def\Pssets@condcolor#1{\ifps@cri\result@tent=\@ne\expandafter\c@lnbV@l#1 :%
    \def\sec@ndcolor{}\def\sec@ndcolorc@md{}%
    \ifcase\result@tent\or\pssetsecondgray{#1}\or\or\pssetsecondrgb{#1}%
    \or\pssetsecondcmyk{#1}\fi\fi}
\ctr@ln@m\sec@ndcolorc@mdStroke
\ctr@ld@f\def\pssetsecondcmyk#1{\def\sec@ndcolor{#1}\def\sec@ndcolorc@md{\c@msetcmykcolor}%
    \def\sec@ndcolorc@mdStroke{\c@msetcmykcolorStroke}}
\ctr@ld@f\def\pssetsecondrgb#1{\def\sec@ndcolor{#1}\def\sec@ndcolorc@md{\c@msetrgbcolor}%
    \def\sec@ndcolorc@mdStroke{\c@msetrgbcolorStroke}}
\ctr@ld@f\def\pssetsecondgray#1{\def\sec@ndcolor{#1}\def\sec@ndcolorc@md{\c@msetgray}%
    \def\sec@ndcolorc@mdStroke{\c@msetgrayStroke}}
\ctr@ld@f\def\us@secondC@lor{\immediate\write\fwf@g{\d@fsecondC@lor}%
    \let\fillc@md=\sdfillc@md}
\ctr@ld@f\def\sdfillc@md{\d@fsecondC@lor\space\c@mfill}
\ctr@ld@f\edef\defaultsecondcolor{\defaultcolor} % Valeur par defaut
\ctr@ld@f\def\Pss@tsecondSt{%
    \s@uvdash{\typ@dash}\pssetdash{\curr@ntseconddash}%
    \s@uvwidth{\typ@width}\pssetwidth{\curr@ntsecondwidth}\us@secondC@lor}
\ctr@ld@f\def\Psrest@reSt{\pssetwidth{\typ@width}\pssetdash{\typ@dash}\us@primarC@lor}
\ctr@ld@f\def\Pssetth@rd#1=#2|{\keln@mun#1|%
    \def\n@mref{c}\ifx\l@debut\n@mref\Pssetth@rdcolor{#2}\else%
    \immediate\write16{*** Unknown attribute: \BS@ psset third(..., #1=...)}%
    \fi}
\ctr@ln@m\th@rdcolor \ctr@ln@m\th@rdcolorc@md
\ctr@ld@f\def\Pssetth@rdcolor#1{\ifps@cri\result@tent=\@ne\expandafter\c@lnbV@l#1 :%
    \def\th@rdcolor{}\def\th@rdcolorc@md{}%
    \ifcase\result@tent\or\Pssetth@rdgray{#1}\or\or\Pssetth@rdrgb{#1}%
    \or\Pssetth@rdcmyk{#1}\fi\fi}
\ctr@ln@m\th@rdcolorc@mdStroke
\ctr@ld@f\def\Pssetth@rdcmyk#1{\def\th@rdcolor{#1}\def\th@rdcolorc@md{\c@msetcmykcolor}%
    \def\th@rdcolorc@mdStroke{\c@msetcmykcolorStroke}}
\ctr@ld@f\def\Pssetth@rdrgb#1{\def\th@rdcolor{#1}\def\th@rdcolorc@md{\c@msetrgbcolor}%
    \def\th@rdcolorc@mdStroke{\c@msetrgbcolorStroke}}
\ctr@ld@f\def\Pssetth@rdgray#1{\def\th@rdcolor{#1}\def\th@rdcolorc@md{\c@msetgray}%
    \def\th@rdcolorc@mdStroke{\c@msetgrayStroke}}
\ctr@ld@f\def\us@thirdC@lor{\immediate\write\fwf@g{\d@fthirdC@lor}%
    \let\fillc@md=\thfillc@md}
\ctr@ld@f\def\thfillc@md{\d@fthirdC@lor\space\c@mfill}
\ctr@ld@f\def\defaultthirdcolor{1}  % Valeur par defaut
\ctr@ld@f\def\pstrimesh#1[#2,#3,#4]{{\ifcurr@ntPS\ifps@cri%
    \PSc@mment{pstrimesh Type=#1, Triangle=[#2,#3,#4]}%
    \s@uvc@ntr@l\et@tpstrimesh\ifnum#1>\@ne\Pss@tsecondSt\setc@ntr@l{2}%
    \Pstrimeshp@rt#1[#2,#3,#4]\Pstrimeshp@rt#1[#3,#4,#2]%
    \Pstrimeshp@rt#1[#4,#2,#3]\Psrest@reSt\fi\psline[#2,#3,#4,#2]%
    \PSc@mment{End pstrimesh}\resetc@ntr@l\et@tpstrimesh\fi\fi}}
\ctr@ld@f\def\Pstrimeshp@rt#1[#2,#3,#4]{{\l@mbd@un=\@ne\l@mbd@de=#1\loop\ifnum\l@mbd@de>\@ne%
    \advance\l@mbd@de\m@ne\figptbary-1:[#2,#3;\l@mbd@de,\l@mbd@un]%
    \figptbary-2:[#2,#4;\l@mbd@de,\l@mbd@un]\psline[-1,-2]%
    \advance\l@mbd@un\@ne\repeat}}
\initpr@lim\initpss@ttings\initPDF@rDVI% Initialisation preliminaire
\ctr@ln@w{newbox}\figBoxA
\ctr@ln@w{newbox}\figBoxB
\ctr@ln@w{newbox}\figBoxC
\catcode`\@=12

\pssetdefault(update=yes)
\newbox\demo

\newtheorem{theorem}{Theorem}[section]
\newtheorem{lemma}[theorem]{Lemma}
\newtheorem{proposition}[theorem]{Proposition}
\newtheorem{corollary}[theorem]{Corollary}

\theoremstyle{definition}
\newtheorem{definition}[theorem]{Definition}
\newtheorem{assumption}[theorem]{Assumption}

\theoremstyle{remark}
\newtheorem{remark}[theorem]{Remark}
\newtheorem{example}[theorem]{Example}

\numberwithin{equation}{section}

\let\ssz\scriptscriptstyle
\let\sz\scriptstyle
\let\te\textstyle
\let\di\displaystyle
\let\ov\overline
\let\un\underline

\DeclareMathOperator{\supp}{supp}
\DeclareMathOperator{\dist}{dist}
\renewcommand{\Re}{\operatorname{Re}}
\renewcommand{\Im}{\operatorname{Im}}

\newcommand  {\C}{\mathbb{C}}
\newcommand{\Fbb}{\mathbb{F}}
\newcommand  {\N}{\mathbb{N}}
\renewcommand{\P}{\mathbb{P}}
\newcommand  {\Q}{\mathbb{Q}}
\newcommand  {\R}{\mathbb{R}}
\newcommand{\Sbb}{\mathbb{S}}
\newcommand  {\T}{\mathbb{T}}
\newcommand{\Ubb}{\mathbb{U}}
\newcommand  {\Z}{\mathbb{Z}}

\newcommand  {\bA}{\boldsymbol{\mathsf A}}
\newcommand  {\bB}{\boldsymbol{\mathsf B}}
\newcommand  {\bD}{\boldsymbol{\mathsf D}}
\newcommand  {\bE}{\boldsymbol{\mathsf E}}
\newcommand  {\bF}{\boldsymbol{\mathsf F}}
\newcommand  {\bG}{\boldsymbol{\mathsf G}}
\newcommand  {\bH}{\boldsymbol{\mathsf H}}
\newcommand  {\bJ}{\boldsymbol{\mathsf J}}
\newcommand  {\bK}{\boldsymbol{\mathsf K}}
\newcommand  {\bL}{\boldsymbol{\mathsf L}}
\newcommand  {\bM}{\boldsymbol{\mathsf M}}
\newcommand  {\bN}{\boldsymbol{\mathsf N}}
\newcommand  {\bP}{\boldsymbol{\mathsf P}}
\newcommand  {\bQ}{\boldsymbol{\mathsf Q}}
\newcommand  {\bR}{\boldsymbol{\mathsf R}}
\newcommand  {\bS}{\boldsymbol{\mathsf S}}
\newcommand  {\bT}{\boldsymbol{\mathsf T}}
\newcommand  {\bU}{\boldsymbol{\mathsf U}}
\newcommand  {\bV}{\boldsymbol{\mathsf V}}
\newcommand  {\bW}{\boldsymbol{\mathsf W}}
\newcommand  {\bX}{\boldsymbol{\mathsf X}}
\newcommand  {\bY}{\boldsymbol{\mathsf Y}}
\newcommand  {\bZ}{\boldsymbol{\mathsf Z}}

\newcommand  {\ba}{{\boldsymbol{\mathsf a}}}
\newcommand  {\bb}{{\boldsymbol{\mathsf b}}}
\newcommand  {\bc}{{\boldsymbol{\mathsf c}}}
\newcommand  {\bd}{{\boldsymbol{\mathsf d}}}
\newcommand  {\be}{{\boldsymbol{\mathsf e}}}
\newcommand {\bff}{{\boldsymbol{\mathsf f}}}
\newcommand  {\bg}{{\boldsymbol{\mathsf g}}}
\newcommand  {\bj}{{\boldsymbol{j}}}
\newcommand  {\bk}{{\boldsymbol{\mathsf k}}}
\newcommand  {\bh}{{\boldsymbol{\mathsf h}}}
\newcommand  {\bm}{{\boldsymbol{\mathsf m}}}
\newcommand  {\bn}{{\boldsymbol{\mathsf n}}}
\newcommand  {\bp}{{\boldsymbol{\mathsf p}}}
\newcommand  {\bq}{{\boldsymbol{\mathsf q}}}
\newcommand  {\br}{{\boldsymbol{\mathsf r}}}
\newcommand  {\bs}{{\boldsymbol{\mathsf s}}}
\newcommand  {\bt}{{\boldsymbol{\mathsf t}}}
\newcommand  {\bu}{{\boldsymbol{\mathsf u}}}
\newcommand  {\bv}{{\boldsymbol{\mathsf v}}}
\newcommand  {\bw}{{\boldsymbol{\mathsf w}}}
\newcommand  {\bx}{{\boldsymbol{\mathsf x}}}
\newcommand  {\by}{{\boldsymbol{\mathsf y}}}
\newcommand  {\bz}{{\boldsymbol{\mathsf z}}}

\newcommand  {\bfz}{{\boldsymbol0}}
\newcommand  {\betab}{{\hskip0.05em\underline{\hskip0.35em}\hskip-0.45em\beta}}
\newcommand  {\ub}{{\underline{b}}}
\newcommand  {\deltad}{{\underline{\delta}}}
\newcommand  {\etae}{{\boldsymbol{\eta}}}
\newcommand  {\gammag}{{\underline{\gamma}}}
\newcommand  {\omegaw}{{\boldsymbol{\omega}}}
\newcommand  {\taut}{{\boldsymbol{\tau}}}
\newcommand  {\chic}{{\boldsymbol{\chi}}}
\newcommand  {\xix}{{\boldsymbol{\xi}}}
\newcommand  {\varphif}{{\boldsymbol{\varphi}}}
\newcommand  {\psiy}{{\boldsymbol{\psi}}}
\newcommand  {\phif}{{\boldsymbol{\phi}}}
\newcommand  {\Phif}{{\boldsymbol{\Phi}}}
\newcommand  {\Psiy}{{\boldsymbol{\Psi}}}
\newcommand  {\Xix}{{\boldsymbol{\Xi}}}

\newcommand  {\cB}{\mathcal{B}}
\newcommand  {\cC}{\mathcal{C}}
\newcommand  {\cD}{\mathcal{D}}
\newcommand  {\cI}{\mathcal{I}}
\newcommand  {\cK}{\mathcal{K}}
\newcommand  {\cL}{\mathcal{L}}
\newcommand  {\cM}{\mathcal{M}}
\newcommand  {\cN}{\mathcal{N}}
\newcommand  {\cO}{\mathcal{O}}
\newcommand  {\cP}{\mathcal{P}}
\newcommand  {\cU}{\mathcal{U}}
\newcommand  {\cV}{\mathcal{V}}
\newcommand  {\cW}{\mathcal{W}}
\newcommand  {\cX}{\mathcal{X}}
\newcommand  {\cY}{\mathcal{Y}}
\newcommand  {\cZ}{\mathcal{Z}}

\newcommand  {\ff}{\boldsymbol{f}}
\newcommand  {\fp}{\boldsymbol{p}}
\newcommand  {\fq}{\boldsymbol{q}}
\newcommand  {\fu}{\boldsymbol{u}}
\newcommand  {\fv}{\boldsymbol{v}}
\newcommand  {\fw}{\boldsymbol{w}}
\newcommand  {\fW}{\boldsymbol{W}}

\newcommand  {\gA}{{\mathfrak A}}
\newcommand  {\gB}{{\mathfrak B}}
\newcommand  {\gC}{{\mathfrak C}}
\newcommand  {\gE}{{\mathfrak E}}
\newcommand  {\gG}{{\mathfrak G}}
\newcommand  {\gH}{{\mathfrak H}}
\newcommand  {\gK}{{\mathfrak K}}
\newcommand  {\gL}{{\mathfrak L}}
\newcommand  {\gM}{{\mathfrak M}}
\newcommand  {\gP}{{\mathfrak P}}
\newcommand  {\gQ}{{\mathfrak Q}}
\newcommand  {\gR}{{\mathfrak R}}
\newcommand  {\gS}{{\mathfrak S}}
\newcommand  {\gT}{{\mathfrak T}}
\newcommand  {\gU}{{\mathfrak U}}
\newcommand  {\gY}{{\mathfrak Y}}
\newcommand  {\gh}{{\mathfrak h}}
\newcommand  {\gs}{{\mathfrak s}}
\newcommand  {\gt}{{\mathfrak t}}
\newcommand  {\gbs}{{\boldsymbol{\mathfrak s}}}

\newcommand{\sA}{{\mathscr A}}
\newcommand{\sB}{{\mathscr B}}
\newcommand{\sC}{{\mathscr C}}
\newcommand{\sD}{{\mathscr D}}
\newcommand{\sE}{{\mathscr E}}
\newcommand{\sF}{{\mathscr F}}
\newcommand{\sG}{{\mathscr G}}
\newcommand{\sI}{{\mathscr I}}
\newcommand{\sK}{{\mathscr K}}
\newcommand{\sL}{{\mathscr L}}
\newcommand{\sM}{{\mathscr M}}
\newcommand{\sN}{{\mathscr N}}
\newcommand{\sO}{{\mathscr O}}
\newcommand{\sP}{{\mathscr P}}
\newcommand{\sR}{{\mathscr R}}
\newcommand{\sS}{{\mathscr S}}
\newcommand{\sT}{{\mathscr T}}
\newcommand{\sV}{{\mathscr V}}
\newcommand{\sX}{{\mathscr X}}

\newcommand  {\rA}{{\mathsf A}}
\newcommand  {\rB}{{\mathsf B}}
\newcommand  {\rD}{{\mathrm D}}
\newcommand  {\rE}{{\mathsf E}}
\newcommand  {\rG}{{\mathsf G}}
\newcommand  {\rH}{{\mathsf H}}
\newcommand  {\rJ}{{\mathsf J}}
\newcommand  {\rK}{{\mathsf K}}
\newcommand  {\rL}{{\mathsf L}}
\newcommand  {\rM}{{\mathsf M}}
\newcommand  {\rN}{{\mathsf N}}
\newcommand  {\rP}{{\mathsf P}}
\newcommand  {\rR}{{\mathsf R}}
\newcommand  {\rS}{{\mathsf S}}
\newcommand  {\rU}{{\mathsf U}}
\newcommand  {\rV}{{\mathsf V}}
\newcommand  {\rW}{{\mathsf W}}
\newcommand  {\rX}{{\mathsf X}}
\newcommand  {\rY}{{\mathsf Y}}
\newcommand  {\rZ}{{\mathsf Z}}

\newcommand  {\ra}{{\mathsf a}}
\newcommand  {\rb}{{\mathsf b}}
\newcommand  {\rd}{{\mathrm d}}
\newcommand  {\re}{{\mathrm e}}
\newcommand  {\rf}{{\mathsf f}}
\newcommand  {\rg}{{\mathsf g}}
\newcommand  {\rp}{{\mathsf p}}
\newcommand  {\rt}{{\mathsf t}}
\newcommand  {\ru}{{\mathsf u}}
\newcommand  {\rv}{{\mathsf v}}
\newcommand  {\rx}{{\mathsf x}}

\newcommand{\ppD}{{\ee\underline{\me D\me}\ee}}
\newcommand{\ppL}{{\ee\underline{\me L\me}\ee}}
\newcommand{\ppT}{{\ee\underline{\me T\me}\ee}}

\newcommand{\bule}{\mbox{$\star\ $}}
\newcommand{\iti}[1]{\par\vspace{1.5pt}\noindent\mbox{{\it (\romannumeral #1)\/}}}
\newcommand{\itj}[1]{\mbox{{\it (\romannumeral #1)\/}}}
\newcommand{\ite}{\item[\bule]}

\newcommand{\dd}[1]{_{\raise-0.25ex\hbox{$\scriptstyle #1$}}}
\newcommand{\ee}{\hskip 0.15ex}
\newcommand{\me}{\hskip-0.15ex}
\newcommand{\on}[1]{\big|_{#1}}
\newcommand{\loc}{{\mathsf{loc}}}
\newcommand{\Dir}{{\mathsf{Dir}}}
\newcommand{\Neu}{{\mathsf{Neu}}}

\newcommand{\cmpl}[1]{\ov{#1\me\me}^{\ee*}}

\newcommand{\vpha}{\left.\vphantom{T^{j_0}_{j_0}}\!\!\right.}
\newcommand {\Norm}[2]{\|#1\|\vpha_{#2}}           % Norm mit Index
\newcommand{\SNorm}[2]{|#1|\vpha_{#2}}             % Semi-Norm mit Index
\newcommand{\REF}[1]{\mbox{\rm \ref{#1}}}         % in Klammern

\newcommand{\Cinf}{{\mathscr C}^\infty}

\newcommand{\prm}{\hskip0.125ex'\hskip-0.8ex}

\hyphenation{wave-let Galer-kin}

\title{Analytic Regularity for Linear Elliptic Systems in Polygons and Polyhedra}

\date{\today}

\author{Martin Costabel, Monique Dauge and Serge Nicaise}

\address{IRMAR, Universit\'{e} de Rennes 1,
Campus de Beaulieu,
35042, Rennes Cedex, France}
\email{martin.costabel@univ-rennes1.fr}

\address{IRMAR, Universit\'{e} de Rennes 1,
Campus de Beaulieu,
35042, Rennes Cedex, France}
\email{monique.dauge@univ-rennes1.fr}

\address{LAMAV,
FR CNRS 2956, Universit\'e Lille Nord de France, UVHC\\ 59313
Valenciennes Cedex 9, France}
\email{snicaise@univ-valenciennes.fr}



\begin{document}

\begin{abstract}
We prove weighted anisotropic analytic estimates for solutions of second order elliptic boundary value problems in polyhedra. The weighted analytic classes which we use are the same as those introduced by Guo in 1993 in view of establishing exponential convergence for $hp$ finite element methods in polyhedra. We first give a simple proof of the known weighted analytic regularity in a polygon, relying on a new formulation of elliptic a priori estimates in smooth domains with analytic control of derivatives. The technique is based on dyadic partitions near the corners. This technique can successfully be extended to polyhedra, providing isotropic analytic regularity. This is not optimal, because it does not take advantage of the full regularity along the edges. We combine it with a nested open set technique to obtain the desired three-dimensional anisotropic analytic regularity result. Our proofs are global and do not require the analysis of singular functions.
\end{abstract}

\maketitle

\tableofcontents

%%%%%%%%%%%%%%%%%%%%%%%%%%%%%%%%%%%%%%%%%%%%%%%%%%%%%%%%%%%%%%%%%%%%%%%%%%%%%%%%%%%%%%%%
\section*{Introduction}
\label{sec0}
%%%%%%%%%%%%%%%%%%%%%%%%%%%%%%%%%%%%%%%%%%%%%%%%%%%%%%%%%%%%%%%%%%%%%%%%%%%%%%%%%%%%%%%%

\subsection*{Motivation}
Solutions of elliptic boundary value problems with analytic data are analytic. This classical result has played an important role in the analysis of harmonic functions since Cauchy's time and in the analysis of more general elliptic problems since Hilbert formulated it as his 19th problem. Hilbert's problem for second order nonlinear problems in variational form in two variables was solved by Bernstein in 1904 \cite{Bernstein04}. After this, many techniques were developed for proving analyticity, culminating in the 1957 paper \cite{MorreyNirenberg57} by Morrey and Nirenberg on linear problems, where Agmon's elliptic regularity estimates in nested open sets were refined to get Cauchy-type analytic estimates, both in the interior of a domain and near analytic parts of its boundary.

Analyticity means exponentially fast approximation by polynomials, and therefore it plays an important role in numerical analysis, too. Analytic estimates have gained a renewed interest through the development of the $p$ and $hp$ versions of the finite element method by Babu\v{s}ka and others. In this context, applications often involve boundaries that are not globally analytic, but only piecewise analytic due to the presence of corners and edges, and therefore global elliptic regularity results cannot be used directly. 

Elliptic boundary value problems in domains with corners and
edges have been investigated by many authors. Let us quote the
pioneering papers of  Kondrat'ev \cite{Kondratev67} and
of Maz'ya and  Plamenevskii \cite{MazyaPlamenevskii73, MazyaPlamenevskii77,MazyaPlamenevskii80b,MazyaPlamenevskii84b}. In these
works, the regularity of the solution and its singular behavior near edges and corners is described  in terms
of weighted Sobolev spaces.
Besides their own theoretical interest, these results are the basis for the convergence analysis of finite element approximations of the boundary value problems. 

But whereas these results on elliptic regularity of finite order allow to prove optimal convergence estimates for the $h$ version or the $p$ version of the finite element method, they are not sufficient for proving the (numerically observed) exponential convergence rate of the $hp$-version  of the finite element method.
Indeed, as has been shown for two-dimensional problems by Babu\v{s}ka and
Guo in \cite{BabuskaGuo88,BabuskaGuo89}, the convergence
analysis of the $hp$-FEM requires the introduction of weighted
spaces with analytic-type control of all derivatives, so-called  ``countably normed spaces''. Babu\v{s}ka and Guo proved corresponding weighted analytic
regularity results for several model problems
\cite{BabuskaGuo88,BabuskaGuo89,GuoBabuska93,GuoSchwab06}.

In three-dimensional domains, as soon as \emph{edges} are present, there is higher regularity in the direction along the edge, and in the $hp$-version one introduces anisotropic refinement,  performed only in the direction transverse to the edge \cite{BabuskaGuo96}. The corresponding weighted spaces have to take this anisotropy into account. In \cite{GuoBabuska97a, GuoBabuska97b} Babu\v{s}ka and Guo have started proving estimates in such spaces in a model situation.

For three-dimensional polyhedra (containing edges and corners) Guo
has introduced the corresponding relevant spaces in
1993 \cite{Guo95}: The anisotropy along edges has to be combined with the distance to corners in a specific way. Since that time, the proof that the regularity of solutions of elliptic boundary value problems with analytic right hand sides is described by these spaces has been an open problem, even for the simplest cases of the Laplace equation with Dirichlet or Neumann
boundary conditions. In the error analysis of $hp$-FEM, such regularity estimates have been taken as an assumption
\cite{Guo95, GuoStephan98,Schotzau09}.

In this paper, we first give a simple proof of the 2D weighted analytic regularity result on polygons, for
Dirichlet and Neumann conditions, using a dyadic partition technique.
Then, relying on a nested open set
technique, we prove anisotropic regularity along edges in the framework of the anisotropic weighted spaces introduced and used in
\cite{BuffaCoDa03, BuffaCoDa04}, but now with analytic-type estimates for all derivatives. Combining the previous two steps with a 3D dyadic partition technique at polyhedral corners, we obtain the desired analytic weighted regularity in a 3D polyhedron.

We use two types of weighted spaces of analytic functions. The first type is constructed from weighted Sobolev spaces of Kondrat'ev type. These spaces with ``homogeneous norms'' are suitable for the description of the regularity in the presence of Dirichlet boundary conditions. For Neumann conditions, a new class of weighted analytic function spaces, constructed from Maz'ya-Plamenevskii-type weighted Sobolev spaces with ``non-homogeneous norms'', has to be used.

It is important to notice that the above spaces naturally contain the singular parts of solutions, and give an accurate account of their generic regularity. Thus, in contrast with investigations such as \cite{HolmMaischakStephan08}, we do not need to address separately vertex, edge and edge-vertex singularities. Our estimates cover regular and singular parts at the same time.

Analytic regularity estimates consist of regularity estimates of arbitrary finite order in which the dependency of the constants on the order is controlled in a Cauchy-type manner. The results of this paper contain therefore, in particular, finite regularity estimates of any order in anisotropic weighted Sobolev spaces. For polyhedra, these finite regularity results are also new in this generality. In particular, our proof covers the statements formulated   in \cite[sec.\ 3]{BuffaCoDa03}, and our results generalize those of \cite{ApelNicaise98,BacutaNistor07}. 

Our proof of analytic regularity estimates is modular in the sense that it starts from low-regularity a-priori estimates on smooth domains and proceeds to singular points, edges, and finally polyhedral corners by employing the two techniques of dyadic partitions and nested open sets. In order to avoid drowning this clear structure in too many technical difficulties, we mainly restrict ourselves to the situation of homogeneous elliptic equations with constant coefficients. Generalizations to operators with lower order terms and variable coefficients will be briefly indicated. They will be discussed in more detail in our forthcoming book \cite{GLC}.

\subsection*{Organization of the paper}
The main results of the paper are Theorems \REF{5T2} and \REF{6T3} in sections 
\ref{sec5} and \ref{sec6}. The hypotheses of these theorems as well as the definitions \ref{5Dflag} and \ref{5DB} of the relevant function spaces are necessarily rather complicated. 
To facilitate their understanding, we begin in section \ref{secIllu} by illustrating the general results in the classical simple examples of the Dirichlet and Neumann boundary value problems for the Laplace equation. We also give simplified definitions for the function spaces in the case of a rectangular polyhedron where all edges are parallel to the coordinate axes. 

We start the analysis in section \ref{sec1} by quoting from \cite{GLCI} an elliptic a priori estimate on smooth domains with analytic control of derivatives. This estimate improves the readability and efficiency of classical proofs of analytic regularity in smooth domains as can be found in \cite{MorreyNirenberg57,Morrey66, LionsMagenes70}. We then prove a refinement of this estimate in view of tackling problems of Neumann type. 

In section \ref{sec2}, we make use of a dyadic partition technique to construct weighted analytic estimates in plane sectors. This technique has been used in a similar framework in \cite{BolleyCamusDauge85} for weighted Gevrey regularity.  It has been first employed for corner domains in \cite{Kondratev67}, then for domains with edges \cite{MazyaPlamenevskii80b}, and even for the Laplace operator on a polygon with non-linear boundary conditions \cite{Kawohl80}.
The technique of dyadic partitions is a powerful tool to prove what we call {\em natural regularity shift} results near corners.  This expression means that from two ingredients, namely \emph{basic} regularity, i.e.\ a certain weighted Sobolev regularity of low order, of the \emph{solution}, and \emph{improved} regularity, i.e.\  high order weighted Sobolev regularity or weighted analytic regularity, of the \emph{right hand side}, one deduces improved regularity \emph{with the same weight} of the solution. This type of regularity result requires very few hypotheses on the weight exponents, none at all in the class of spaces ($\bK$ and $\bA$) with homogeneous norms and only a bound from below in the class of spaces ($\bJ$ and $\bB$) with nonhomogeneous norms.
 
In section \ref{sec3}, we combine the local estimates to obtain the natural regularity shift for polygons in analytic weighted spaces.

In section \ref{sec4} we start the three-dimensional investigation with estimates along an edge. The fact that there is additional regularity along the edge has been known and studied for a long time (see \cite[Theorem 16.13]{Dauge88}, \cite[Satz 3.1]{MazyaRossmann88}, \cite[Theorem 4.4]{CostabelDauge93a}). We therefore introduce anisotropic weighted spaces in which derivatives along the direction of the edge are less singular than transversal derivatives. There are again two classes of spaces -- with homogeneous norm (spaces $\bM$) and with non homogeneous norm (spaces $\bN$).
Under the assumption of a certain local a priori estimate of low order in the neighborhood of an edge point, we prove local analytic anisotropic regularity shift along this edge, by combining dyadic partition technique and the classical (and delicate) tool of nested open sets. 

In section \ref{sec5}, we treat polyhedral corners. Relying on suitable definitions of various families of weighted spaces (similar to \cite{MazyaRossmann03}, but with anisotropy along edges), we are able to prove the analytic regularity shift for polyhedra by dyadic partitions around each corner of a polyhedron.

In section \ref{sec6}, we prove the main analytic regularity result for solutions of problems given in variational form. The proof combines our analytic regularity shift results with known results on basic regularity for the variational solutions. On polygons, we use for this purpose Kondrat'ev's classical regularity results in weighted Sobolev spaces, and on polyhedra, we use recent regularity results by Maz'ya and Rossmann \cite{MazyaRossmann03}. In this way, we finally obtain the weighted analytic regularity of variational solutions in the right functional classes of \cite{Guo95}. For polygons, we thus prove again in a different and simpler way results which were first established by Babu\v{s}ka and Guo \cite{BabuskaGuo88,GuoBabuska93}. For polyhedra, the results are new.

We conclude our paper in section \ref{sec8}  by discussing various generalizations. For our proofs, we choose in this paper the simplest possible framework of second order homogeneous systems with constant coefficients and zero boundary data on domains with piecewise straight or plane boundaries. In dimension 2, it is a mere technicality to generalize these proofs to the case of second order elliptic systems with analytic coefficients and non-zero boundary data. In dimension 3, the possible variation of coefficients along edges introduces more serious complications and would require to estimate commutators in a systematic way as in \cite[Lemmas 1.6.2 \& 2.6.2]{GLCI}. In comparison, the generalization to homogeneous transmission problems with constant coefficients on a polyhedral partition would be much less difficult. Whereas the Stokes system could be considered similarly, things are different for regularized harmonic Maxwell equations, for which it is necessary to detach the first singularity if one wants to obtain a valuable result, see \cite{CostabelDaugeSchwab05} in dimension two.

\medskip
We denote by $\rH^m(\Omega)$ the usual Hilbert Sobolev space of exponent $m$, by $\|{\,\cdot\,}\|_{m;\,\Omega}$ and $|{\,\cdot\,}|_{m;\,\Omega}$ its norm and semi-norm. The $\rL^2(\Omega)$-norm is denoted by $\|{\,\cdot\,}\|_{0;\,\Omega}$ or simply by $\|{\,\cdot\,}\|_{\Omega}\,$. Boldface letters like $\bH^m(\Omega)$ indicate spaces of vector functions.

%%%%%%%%%%%%%%%%%%%%%%%%%%%%%%%%%%%%%%%%%%%%%%%%%%%%%%%%%%%%%%%%%%%%%%%%% 
\section{Illustration}
\label{secIllu}
%%%%%%%%%%%%%%%%%%%%%%%%%%%%%%%%%%%%%%%%%%%%%%%%%%%%%%%%%%%%%%%%%%%%%%%%%

In this section, we explain the main definitions and results that culminate in Theorems \REF{5T2} and \REF{6T3}, for a simple class of geometrical configurations and for the simplest elliptic boundary value problems, namely the Dirichlet and Neumann problems for the Laplace equation.

\subsection{Dirichlet conditions}
Let us consider the Dirichlet problem for the Laplace operator on a domain $\Omega$
\begin{equation}
\label{0Ebvp}
   \left\{ \begin{array}{rclll}
   \Delta\,u &=& f \quad & \mbox{in}\ \ \Omega, \\[0.3ex]
   u &=& 0  & \mbox{on}\ \ \partial\Omega\,,
   \end{array}\right.
\end{equation}
for right hand side $f\in\rH^{-1}(\Omega)$. There exists a unique solution $u\in \rH^1(\Omega)$, owing to the Lax-Milgram lemma applied to a variational formulation that is strongly elliptic on $\rH_{0}^1(\Omega)$.

If $\Omega$ has a smooth boundary, there holds what can be called the {\em elliptic regularity shift}: For any natural number $n$, if $f\in\rH^{n}(\Omega)$, then $u\in\rH^{n+2}(\Omega)$. Moreover, if the boundary is analytic and $f$ belongs to the class $\rA(\Omega)$ of functions analytic up to the boundary of $\Omega$, then $u\in\rA(\Omega)$.

\subsubsection{Polygons}
If $\Omega$ has a {\em polygonal} boundary, the situation is quite different: If for instance $\Omega$ has a non-convex angle, the solution $u$ does not belong to $\rH^2(\Omega)$ in general when $f\in\rL^2(\Omega)$. Instead there hold expansions in regular and singular parts: If $f$ is smooth, then for any natural number $n$ we can write
\begin{equation}
\label{0E2}
   u = v_n + \sum_{\bc\in\sC} w_{\bc,n}, \quad v_n\in \rH^{n+2}(\Omega).
\end{equation}
Here $\sC$ is the set of the corners $\bc$ of $\Omega$. Let $\omega_\bc$ be the angle of $\Omega$ at the corner $\bc$. Each corner singular part has the form\footnote{When $\frac{k\pi}{\omega_\bc}$ is an integer, there is a logarithmic term instead of $r_\bc^{\frac{k\pi}{\omega_\bc}} \sin\Big(\frac{k\pi\theta_\bc}{\omega_\bc}\Big)$.} 
\begin{equation}
\label{0E3}
   w_{\bc,n} = \chi_\bc(r_\bc) \sum_{k\in\N,\ 0<\frac{k\pi}{\omega_\bc}\le n+1} d_{\bc,k}\,
   r_\bc^{\frac{k\pi}{\omega_\bc}} \sin\Big(\frac{k\pi\theta_\bc}{\omega_\bc}\Big)\,.
\end{equation}
The cut-off function $\chi_\bc$, the polar coordinates $(r_\bc,\theta_\bc)$ and the coefficients $d_{\bc,k}$ are related to $\bc$. The regularity implication
\begin{equation}
\label{0E4}
   u\in\rH^1(\Omega) \ \ \mbox{and}\ \  f\in\rH^{n}(\Omega) 
   \quad\Longrightarrow\quad u\in\rH^{n+2}(\Omega)
\end{equation}
holds only if $n+1<\frac\pi{\omega_\bc}$ for all $\bc$. This precludes any regularity in the analytic class.

It is known since Kondrat'ev that the use of weighted Sobolev spaces allows a better description of the regularity of solutions. We introduce now the Kondrat'ev spaces with a notation of our own --- which facilitates the definition of weighted analytic classes. With $r_\bc=r_\br(\bx)$ the distance function to the corner $\bc$ and $\betab=(\beta_\bc)\dd{\bc\in\sC}\in\R^{\#\sC}$ a weight multi-exponent we define the weighted semi-norms
\begin{equation}
\label{0E0K}
   \SNorm{u}{\rK;\,k,\betab\,;\,\Omega} =
   \Big\{ \sum_{|\alpha|=k}
   \Big\| \big(\prod_{\bc\in\sC} r_\bc^{\beta_\bc + |\alpha|}\big) \,\partial^\alpha_\bx u
   \Big\|_{0;\, \Omega}^2
   \Big\}^{\frac12},\quad k\in\N\,.
\end{equation}
The space $\rK^m_\betab(\Omega)$ is the space of distributions $u$ such that the norm
\[
   \Norm{u}{\rK^m_\betab(\Omega)} = 
   \Big\{ \sum_{k=0}^m \SNorm{u}{\rK;\,k,\betab\,;\,\Omega}^2\Big\}^{\frac12} \quad\mbox{is finite}.
\]
These norms are qualified as {\em homogeneous} because of the shift $+ |\alpha|$ for the weight exponent, which makes each term homogeneous with respect to dilations with center in the corresponding corner. 
In the case of homogeneous Dirichlet conditions as in problem \eqref{0Ebvp}, an angular Poincar\'e inequality allows to establish the estimate
\begin{equation}
\label{0Epoi}
   \Norm{\big(\prod_{\bc\in\sC} r_\bc^{-1} \big) u}{0;\Omega} 
   \le C\SNorm{u}{1;\Omega},\quad u\in\rH^1_0(\Omega),
\end{equation}
whence the embedding
\begin{equation}
\label{0EKm1}
   \rH^1_0(\Omega) \subset \rK^1_{-1}(\Omega).
\end{equation}
This is one of the reasons why the $\rK$ spaces are appropriate for describing the regularity of Dirichlet solutions.

Kondrat'ev's result for problem \eqref{0Ebvp} can be phrased as follows

\begin{theorem}\cite[section 5.4]{Kondratev67}
\label{0T1}
If the following condition holds for the polygon $\Omega$ and the exponents $\beta_\bc$
\begin{equation}
\label{0Eco1}
   0\le -\beta_\bc-1<\frac\pi{\omega_\bc}\quad\forall\bc\in\sC
\end{equation}
then for any natural number $n$ the solution of problem \eqref{0Ebvp} satisfies the regularity result:
\begin{equation}
\label{0E6}
   u\in\rH^1(\Omega) \ \ \mbox{and}\ \  f\in\rK^{n}_{\betab+2}(\Omega) 
   \quad\Longrightarrow\quad u\in\rK^{n+2}_{\betab}(\Omega).
\end{equation}
\end{theorem}

The analytic class that we associate with the family of semi-norms $\SNorm{u}{\rK;\,k,\betab\,;\,\Omega}$ is
\begin{equation}
\label{0E4b}
   \rA_\betab (\Omega) = \Big\{\  u \in \bigcap_{m\ge0}\rK^m_\betab(\Omega) \ : \
    \exists C>0, \forall m\in\N,\quad
   \SNorm{u}{\rK;\,m,\betab\,;\,\Omega} \le C^{m+1} m!   \Big\},
\end{equation}
and our regularity result is the following.

\begin{theorem}[see Thm \ref{6T1} for the general case]
\label{0T2}
If condition \eqref{0Eco1} holds for the polygon $\Omega$ and the exponents $\beta_\bc$ then the solution of problem \eqref{0Ebvp} satisfies the regularity result:
\begin{equation}
\label{0E8}
   u\in\rH^1(\Omega) \ \ \mbox{and}\ \  f\in\rA_{\betab+2}(\Omega) 
   \quad\Longrightarrow\quad u\in\rA_{\betab}(\Omega).
\end{equation}
\end{theorem}

Our proof consists of the combination of Theorem \ref{0T1} with the proof of what we call {\em natural regularity shift} which holds for any weight exponent without limitation

\begin{theorem}[see Thm\ \ref{3T1} for the general case]
\label{0T3}
For any multi-exponent $\betab$ the following regularity result holds for solutions of problem \eqref{0Ebvp} in the polygon $\Omega$ 
\begin{equation}
\label{0E8nat}
   u\in\rK^1_{\betab}(\Omega)\ \ \mbox{and}\ \  f\in\rA_{\betab+2}(\Omega) 
   \quad\Longrightarrow\quad u\in\rA_{\betab}(\Omega).
\end{equation}
\end{theorem}

This theorem is ``simply'' the analytic version of the well-known regularity shift result
\begin{equation}
\label{0E8b}
   u\in\rK^1_{\betab}(\Omega)\ \ \mbox{and}\ \  f\in\rK^n_{\betab+2}(\Omega) 
   \quad\Longrightarrow\quad u\in\rK^{n+2}_{\betab}(\Omega)
\end{equation}
valid for any $n\ge0$ and any $\betab$. In fact, the proof of Theorem \ref{0T3} consists in showing that the norm estimates corresponding to \eqref{0E8b} are uniform of Cauchy type in the order $n$, namely there exists a constant $C>0$ such that for all integer $k\ge2$ and all solutions of \eqref{0Ebvp}
\begin{equation}
\label{0E8c}
   \frac{1}{k!} \SNorm{u}{\rK;\,k,\betab\,;\,\Omega} \le C^{k+1}\Big \{
   \sum_{\ell=0}^{k-2} \frac{1}{\ell!} \SNorm{f}{\rK;\,\ell,\betab\,;\,\Omega} +
   \sum_{\ell=0}^1 \SNorm{u}{\rK;\,\ell,\betab\,;\,\Omega} \Big\}.
\end{equation}
The proof of this family of estimates uses local dyadic partitions and Cauchy type estimates for smooth domains with analytic boundary, see Theorem \ref{2T1}.

\smallskip
From this short introduction, we see that the expression ``regularity result'', which generally means the existence of estimates for derivatives of the solution, involves a triple of function spaces $(\Ubb_0,\Fbb,\Ubb)$ and states the implication
\begin{equation}
\label{0EUF}
   u\in\Ubb_0\ \ \mbox{and}\ \  f\in\Fbb
   \quad\Longrightarrow\quad u\in\Ubb.
\end{equation}
In principle $\Ubb$ is in a certain sense optimal with respect to $\Fbb$. Comparing the assumptions of Theorems \ref{0T1} and \ref{0T3}, we can see the important role of the different choices of the basic regularity $\Ubb_0$. In Theorem~\REF{0T3}, the natural regularity shift result is based on a space $\Ubb_0=\rK^1_{\betab}(\Omega)$ that has the same weight exponent $\betab$ as the space $\Ubb$ of improved regularity, which is $\rK^{n+2}_{\betab}(\Omega)$ in \eqref{0E8b} and $\rA_{\betab}(\Omega)$ in \eqref{0E8nat}.
In Theorem~\REF{0T1}, the regularity result for the solution of the variational problem is based on the choice of the energy space $\rH^1(\Omega)$ for the space of basic regularity $\Ubb_0$.

\subsubsection{Polyhedra}
If $\Omega$ has a {\em polyhedral} boundary, at first glance the situation is similar. Besides the set of corners, we have the set $\sE$ of the edges $\be$ and the distance functions $r_\be$ to each edge $\be$. For $\betab=\{\beta_\bc\}_{\bc\in \sC}\cup\{\beta_\be\}_{\be\in \sE}$ the weighted semi-norms are defined as
\begin{equation}
\label{0E0K3d}
   \SNorm{u}{\rK;\,k,\betab\,;\,\Omega} =
   \Big\{ \sum_{|\alpha|=k}
   \Big\| \big\{ \prod_{\bc\in\sC} r_\bc^{\beta_\bc + |\alpha|}\big\}
   \big\{ \prod_{\be\in\sE} \big( \frac{r_\be}{r_\sC} \big)^{\beta_\be + |\alpha|} \big\}
   \,\partial^\alpha_\bx u\,
   \Big\|_{0;\, \cV}^2
   \Big\}^{\frac12},\quad k\in\N.
\end{equation}
where $r_\sC$ denotes the distance function to the set $\sC$ of corners (note that $r_\sC$ is equivalent to the product $\prod_{\bc\in\sC} r_\bc$ on $\overline\Omega$). The space $\rK^m_\betab(\Omega)$ is the space of distributions $u$ such that the sum
$\sum_{k=0}^m \SNorm{u}{\rK;\,k,\betab\,;\,\Omega}$ is finite.
The conditions leading to a statement corresponding to Theorem \ref{0T1} depend on
\begin{itemize}
\item[i)] The opening $\omega_\be$ of the dihedral angle tangent to $\Omega$ along the edge $\be$,
\item[ii)] The Dirichlet limiting exponent $\lambda^\Dir_\bc$ at the corner $\bc$ defined as
\begin{equation}
\label{0ELamdir}
   \lambda^\Dir_\bc = -\frac12 + \sqrt{\mu^\Dir_{\bc,1}+\frac14}
\end{equation}
where $\mu^\Dir_{\bc,1}$ is the first eigenvalue of the Laplace-Beltrami operator with Dirichlet conditions on the spherical cap $G_\bc$ cut out by the cone $\Gamma_\bc$ tangent to $\Omega$ at $\bc$.
\end{itemize}

\begin{theorem} \cite[Thm 2]{MazyaRossmann91b}
\label{0T4}
If the following condition holds for the polyhedron $\Omega$ and the exponents $\betab$
\begin{equation}
\label{0Eco2}
   0\le -\beta_\be-1<\frac\pi{\omega_\be}\quad\forall\be\in\sE
   \quad\mbox{and}\quad
   -\tfrac12\le -\beta_\bc-\tfrac32<\lambda^\Dir_\bc \quad\forall\bc\in\sC
\end{equation}
then for any natural number $n$ the solution of problem \eqref{0Ebvp} satisfies the regularity result:
\begin{equation}
\label{0E6b}
   u\in\rH^1(\Omega) \ \ \mbox{and}\ \  f\in\rK^{n}_{\betab+2}(\Omega) 
   \quad\Longrightarrow\quad u\in\rK^{n+2}_{\betab}(\Omega).
\end{equation}
\end{theorem}

By some direct extension to the technique leading to Theorem \ref{0T2} we could prove the corresponding statement in the analytic classes associated with the seminorms $\SNorm{u}{\rK;\,k,\betab\,;\,\Omega}$. But such a result {\em would not be of any use} for proving exponential convergence of finite element approximations, which is one of the main motivations for requiring such analytic regularity estimates, cf.\ \cite{BabuskaGuo96}. To achieve exponential convergence, one has to use anisotropic mesh refinements along the edges, and for the proof one needs corresponding anisotropic norm estimates.

Fortunately, tangential regularity along edges holds, which allows (with a little additional effort) to work in {\em anisotropic weighted spaces}. To simplify this preliminary exposition, we assume that all edges are parallel to one of the Cartesian axes (this is the case if $\Omega$ is a cube, a thick L-shaped domain or a Fichera corner). Then for each edge $\be$ we particularize the derivatives in the directions transverse $\partial^{\alpha_{\be}^\perp}_\bx$ or parallel $\partial^{\alpha_{\be}^\parallel}\dd\bx$ to that edge so that
\[
   \partial^\alpha_\bx = \partial^{\alpha_{\be}^\perp}_\bx\, 
   \partial^{\alpha_{\be}^\parallel}\dd\bx\,.
\]
We define the anisotropic weighted semi-norms
\begin{equation}
\label{0E0M}
   \SNorm{u}{\rM;\,k,\betab\,;\,\Omega} =
   \Big\{ \sum_{|\alpha|=k}
   \Big\| \big\{ \prod_{\bc\in\sC} r_\bc^{\beta_\bc + |\alpha|}\big\}
   \big\{ \prod_{\be\in\sE} \big( \frac{r_\be}{r_\sC} \big)^{\beta_\be + |\alpha^\perp_\be|} \big\}
   \,\partial^\alpha_\bx u\,
   \Big\|_{0;\, \Omega}^2
   \Big\}^{\frac12},\quad k\in\N.
\end{equation}
Note that \eqref{0E0K3d} and \eqref{0E0M} differ by the exponent $\beta_\be + |\alpha|$ replaced with $\beta_\be + |\alpha^\perp_\be|$. For a general polyhedral domain, the definition is given in \eqref{5EnormM}. The space $\rM^m_\betab(\Omega)$ is the space of distributions $u$ such that the sum $\sum_{k=0}^m \SNorm{u}{\rM;\,k,\betab\,;\,\Omega}$ is finite and we denote by $\rA_\betab(\Omega)$ the corresponding analytic class :
\begin{equation}
\label{0EA}
   \rA_\betab (\Omega) = \Big\{\  u \in \bigcap_{m\ge0}\rM^m_\betab(\Omega) \ : \
    \exists C>0, \forall m\in\N,\quad
   \SNorm{u}{\rM;\,m,\betab\,;\,\Omega} \le C^{m+1} m!   \Big\}.
\end{equation}
The anisotropic weighted analytic regularity result is the following.

\begin{theorem}[see Cor \ref{6T4} for the general case]
\label{0T5}
If condition \eqref{0Eco2} holds for the polyhedron $\Omega$ and the exponents $\betab$, then the solution of problem \eqref{0Ebvp} satisfies the regularity result:
\begin{equation}
\label{0E83d}
   u\in\rH^1(\Omega) \ \ \mbox{and}\ \  f\in\rA_{\betab+2}(\Omega) 
   \quad\Longrightarrow\quad u\in\rA_{\betab}(\Omega).
\end{equation}
\end{theorem}

Once again our proof consists of the combination of Theorem \ref{0T4} with the proof of a {\em natural anisotropic regularity shift}:

\begin{theorem}[see Thm \ref{5T1} for the general case]
\label{0T6}
For any multi-exponent $\betab$ such that
\begin{equation}
\label{0Eco3}
   0\le -\beta_\be-1\quad\mbox{and}\quad 
   -\beta_\be-1\neq\frac{k\pi}{\omega_\be}\quad\forall k\in\N,\ \forall\be\in\sE
\end{equation}
the following regularity result holds for solutions of problem \eqref{0Ebvp} in the polyhedron $\Omega$ 
\begin{equation}
\label{0E8natani}
   u\in\rK^1_{\betab}(\Omega)\ \ \mbox{and}\ \  f\in\rA_{\betab+2}(\Omega) 
   \quad\Longrightarrow\quad u\in\rA_{\betab}(\Omega).
\end{equation}
\end{theorem}

Note that condition \eqref{0Eco3} --- which in the simple situation discussed here is a  sufficient condition for Assumption~\ref{5GA} to hold --- is far less restrictive than \eqref{0Eco2}. Nevertheless, we need this condition for polyhedra whereas no condition at all was required in the polygonal case (Theorem \REF{0E3}).

\subsection{Neumann conditions}
Let us consider now the Neumann problem for the Laplace operator on the domain $\Omega$
\begin{equation}
\label{0EbvpN}
   \left\{ \begin{array}{rclll}
   \Delta\,u &=& f \quad & \mbox{in}\ \ \Omega, \\[0.3ex]
   \partial_n u &=& 0  & \mbox{on}\ \ \partial\Omega\,,
   \end{array}\right.
\end{equation}
for an $\rL^2(\Omega)$ right hand side with zero mean value. There exists a solution 
$u\in \rH^1(\Omega)$, unique up to the addition of a constant. In a smooth domain, $u$ satisfies the elliptic regularity shift.

\subsubsection{Polygons}
The solutions have singular expansions like \eqref{0E2} with $\cos$ functions instead of $\sin$.  In particular, independent constants are present at each corner.  The scale of weighted spaces $\rK^n_\betab(\Omega)$ cannot yield optimal regularity results because   solutions of problem \eqref{0EbvpN}, even regular, do not belong to $\rK^1_\betab(\Omega)$ for $\beta_\bc=-1$ in general. 
%	for any exponent $\betab$ with $0\le-\beta_\bc-1$, i.e., $\beta_\bc\le-1$, 
%	the variational space $\rH^1(\Omega)$ {\em is not contained in } $\rK^1_\betab(\Omega)$. 
To overcome this difficulty, we consider weighted spaces $\rJ^n_\betab(\Omega)$ with {\em non-homogeneous norms} defined  as   %by their norm 
\begin{equation}
\label{0E0J}
   \Norm{u}{\rJ^n_\betab(\Omega)} =
   \Big\{ \sum_{|\alpha|\le n}
   \Big\| \big(\prod_{\bc\in\sC} r_\bc^{\beta_\bc + n}\big) \,\partial^\alpha_\bx u
   \Big\|_{0;\, \Omega}^2
   \Big\}^{\frac12},
\end{equation}
instead of
\[
   \Norm{u}{\rK^n_\betab(\Omega)} =
   \Big\{ \sum_{|\alpha|\le n}
   \Big\| \big(\prod_{\bc\in\sC} r_\bc^{\beta_\bc + |\alpha|}\big) \,\partial^\alpha_\bx u
   \Big\|_{0;\, \Omega}^2
   \Big\}^{\frac12},
\]
for $\rK^n_\betab(\Omega)$: the exponent in $\rJ$ spaces does not depend on the order of derivation. As a particular case we obtain the standard Sobolev spaces
\[
   \rH^n(\Omega) = \rJ^n_{-n}(\Omega).
\]
The advantage of this notation is a natural notion of the corresponding analytic class.
For a multi-exponent $\betab$ we first set
\begin{equation}
\label{0E5a}
   \framebox{$\vphantom{\Big(}\di\kappa_\betab:=\max_{\bc\in\sC}-\beta_\bc$}
\end{equation}
An important property of the spaces $\rJ^n_\betab(\Omega)$ is that if $m\ge\kappa_\betab$, then the norm $\Norm{u}{\rJ^n_\betab(\Omega)}$ is equivalent to the following ``step-weighted'' norm
\begin{equation}
\label{0E0J0}
   \Big\{ \sum_{|\alpha|\le n}
   \Big\| \big(\prod_{\bc\in\sC} r_\bc^{\max\{\beta_\bc + |\alpha|,0\}}\big) \,
   \partial^\alpha_\bx u
   \Big\|_{0;\, \Omega}^2
   \Big\}^{\frac12}.
\end{equation}
This implies that we have the continuous embedding of $\rJ^{m+1}_\betab(\Omega)$ into $\rJ^m_\betab(\Omega)$ when $m\ge\kappa_\betab$, which leads to the definition of the analytic class
\begin{equation}
\label{0E5b}
   \rB_\betab (\Omega) = \Big\{\  u \in \bigcap_{m\ge\kappa_\betab}\rJ^m_\betab(\Omega) \ : \
    \exists C>0, \forall m\ge\kappa_\betab \quad
   \Norm{u}{\rJ^m_\betab(\Omega)} \le C^{m+1} m!   \Big\}.
\end{equation}
If $\beta_\bc\not\in\Z$, near the corner $\bc$ the analytic classes $\rA_\betab(\Omega)$ and $\rB_\betab(\Omega)$ differ by polynomial functions of degree $\le [-\beta_\bc-1]$ in Cartesian variables: Such functions   are present in $\rB_\betab(\Omega)$  but not in $\rA_\betab(\Omega)$, see Remark \ref{3R1}. The counterparts of Theorems \ref{0T1} to \ref{0T3} are\footnote{The fact that the same condition \eqref{0Eco1} works for both Neumann and Dirichlet boundary conditions is a very particular case (Laplace operator in 2D).}

\begin{theorem}\cite[Thm 7.2.4 and section 7.3.4]{KozlovMazyaRossmann97b}
\label{0T1N}
If condition \eqref{0Eco1} holds, then for all integer $n$ such that $n+2\ge\kappa_\betab$ solutions of problem \eqref{0EbvpN} satisfy:
\begin{equation}
\label{0E6N}
   u\in\rH^1(\Omega) \ \ \mbox{and}\ \  f\in\rJ^{n}_{\betab+2}(\Omega) 
   \quad\Longrightarrow\quad u\in\rJ^{n+2}_{\betab}(\Omega).
\end{equation}
\end{theorem}

\begin{theorem}[see Thm \ref{6T2} for the general case]
\label{0T2N}
If condition \eqref{0Eco1} holds, then solutions of problem \eqref{0EbvpN} satisfy:
\begin{equation}
\label{0E8N}
   u\in\rH^1(\Omega) \ \ \mbox{and}\ \  f\in\rB_{\betab+2}(\Omega) 
   \quad\Longrightarrow\quad u\in\rB_{\betab}(\Omega).
\end{equation}
\end{theorem}

\begin{theorem}[see Thm \ref{3T1b} for the general case]
\label{0T3N}
For any multi-exponent $\betab$ and any integer $m\ge\kappa_\betab$ solutions of problem \eqref{0EbvpN} satisfy: 
\begin{equation}
\label{0E8natN}
   u\in\rJ^m_{\betab}(\Omega)\ \ \mbox{and}\ \  f\in\rB_{\betab+2}(\Omega) 
   \quad\Longrightarrow\quad u\in\rB_{\betab}(\Omega).
\end{equation}
\end{theorem}

\subsubsection{Polyhedra}
We have to introduce the isotropic and anisotropic weighted spaces with non-homogeneous norms $\rJ^n_\betab(\Omega)$ and $\rN^n_\betab(\Omega)$ corresponding to $\rK^n_\betab(\Omega)$ and $\rM^n_\betab(\Omega)$, respectively. They are defined by their norms as follows:
\begin{equation}
\label{0E0J3d}
   \Norm{u}{\rJ^n_\betab(\Omega)} =
   \Big\{ \sum_{|\alpha|\le n}
   \Big\| \big\{ \prod_{\bc\in\sC} r_\bc^{\beta_\bc + n}\big\}
   \big\{ \prod_{\be\in\sE} \big( \frac{r_\be}{r_\sC} \big)^{\beta_\be + n} \big\}
   \partial^\alpha_\bx u
   \Big\|_{0;\, \Omega}^2
   \Big\}^{\frac12},
\end{equation}
and, when the edges are parallel to the axes:
\begin{equation}
\label{0E0N}
   \Norm{u}{\rN^n_\betab(\Omega)} =
   \Big\{ \sum_{|\alpha|\le n}
   \Big\| \big\{ \prod_{\bc\in\sC} r_\bc^{\max\{\beta_\bc + |\alpha|,0\}}\big\}
   \big\{ \prod_{\be\in\sE} \big( \frac{r_\be}{r_\sC} \big)^{\max\{
   \beta_\be + |\alpha^\perp_\be|,0\}} \big\}
   \partial^\alpha_\bx u
   \Big\|_{0;\, \Omega}^2
   \Big\}^{\frac12},
\end{equation}
for any $n\ge\kappa_\betab$ where 
\begin{equation}
\label{0E5a3d}
   \framebox{$\di\kappa_\betab:=\max\Big\{ \max_{\bc\in\sC}-\beta_\bc, 
   \max_{\be\in\sE}-\beta_\be\Big\}$}
\end{equation}
Note that the definitions \eqref{0E0J3d} and \eqref{0E0N} are coherent since when $n\ge\kappa_\betab$, the norm of $\rJ^n_\betab(\Omega)$ is equivalent to the norm obtained by replacing $\beta_\bc + n$ by $\max\{\beta_\bc + |\alpha|,0\}$ and $\beta_\be + n$ by $\max\{\beta_\be + |\alpha|,0\}$.

The anisotropic weighted analytic classes are then defined as
\begin{equation}
\label{0E5b3d}
   \rB_\betab (\Omega) = \Big\{\  u \in \bigcap_{m\ge\kappa_\betab}\rN^m_\betab(\Omega) \ : \
    \exists C>0, \forall m\ge\kappa_\betab \quad
   \Norm{u}{\rN^m_\betab(\Omega)} \le C^{m+1} m!   \Big\}.
\end{equation}

On the same model as \eqref{0ELamdir} we define the 3D Neumann limiting exponent $\lambda^\Neu_\bc$ at the corner $\bc$ as
\begin{equation}
\label{0ELamneu}
   \lambda^\Neu_\bc = -\frac12 + \sqrt{\mu^\Neu_{\bc,2}+\frac14}
\end{equation}
where $\mu^\Neu_{\bc,2}$ is the {\em second} eigenvalue\footnote{The first Neumann eigenvalue $\mu^\Neu_{\bc,1}$ is zero.} of the Laplace-Beltrami operator with Neumann conditions on the spherical cap $G_\bc$. With this we can state the counterparts of Theorem \REF{0T4}--\REF{0T6}.

\begin{theorem}[{cf \cite[Thm 7.1]{MazyaRossmann03} and \cite[Chap 8]{Dauge88}}]
\label{0T10}
If the following condition holds for the polyhedron $\Omega$ and the exponents $\betab$
\begin{equation}
\label{0Eco2N}
   0\le -\beta_\be-1<\frac\pi{\omega_\be}\quad\forall\be\in\sE
   \quad\mbox{and}\quad
   -\tfrac12\le -\beta_\bc-\tfrac32<\min\{2,\lambda^\Neu_\bc\} \quad\forall\bc\in\sC
\end{equation}
then for any natural number $n$ such that $n+2\ge\kappa_\betab$ solutions of problem \eqref{0EbvpN} satisfy:
\begin{equation}
\label{0E6bN}
   u\in\rH^1(\Omega) \ \ \mbox{and}\ \  f\in\rJ^{n}_{\betab+2}(\Omega) 
   \quad\Longrightarrow\quad u\in\rJ^{n+2}_{\betab}(\Omega).
\end{equation}
\end{theorem}

The minimum of $\lambda_\bc$ with $2$ in \eqref{0Eco2N} comes from the conditions of injectivity modulo polynomials of \cite{Dauge88} which replaces the usual spectral conditions when polynomial right hand sides are involved,  see section \ref{sec62} \ref{7Simp}. Then the anisotropic weighted analytic regularity result in non-homogeneous norms is the following.

\begin{theorem}[see Thm \ref{6T5} for the general case]
\label{0T11}
If condition \eqref{0Eco2N} holds for the polyhedron $\Omega$ and the exponents $\betab$ then solutions of problem \eqref{0EbvpN} satisfy:
\begin{equation}
\label{0E83e}
   u\in\rH^1(\Omega) \ \ \mbox{and}\ \  f\in\rB_{\betab+2}(\Omega) 
   \quad\Longrightarrow\quad u\in\rB_{\betab}(\Omega).
\end{equation}
\end{theorem}

The corresponding natural regularity shift result is the following.
\begin{theorem}[see Thm \ref{5T2} and   Remark \ref{5R2}  for the general case]
\label{0T12}
For any multi-exponent $\betab$ such that condition \eqref{0Eco3} holds, we have
the following regularity result for solutions of problem \eqref{0EbvpN} in the polyhedron $\Omega$ 
\begin{equation}
\label{0E8nataniN}
   (u\in\rJ^m_{\betab}(\Omega)\ \ \mbox{with}\ \ m\ge\kappa_\betab)
   \ \ \mbox{and}\ \  (f\in\rB_{\betab+2}(\Omega) )
   \quad\Longrightarrow\quad u\in\rB_{\betab}(\Omega).
\end{equation}
\end{theorem}

\begin{remark}
In Theorem \REF{0T12}, condition \eqref{0Eco3} plays the role of a necessary and sufficient condition for the validity of Assumption \REF{5GB}. The fact that this is the same condition as the one in Theorem~\REF{0T6} is again a particularity of the Laplace operator in 2D, like we have seen for condition \eqref{0Eco1}.
More precisely, whereas in the Neumann case the lower bound $0\le -\beta_\be-1$ is necessary, in the Dirichlet case of Theorem~\REF{0T6}, the sharp lower bound would be $-\frac\pi{\omega_\be}<-\beta_\be-1$, if we allow stepping out of the variational space $\rH^1(\Omega)$,   see \S \ref{sec62}\ref{6Sapriori}.
\end{remark}

\begin{remark}
Theorem \REF{0T12} and, if we replace $\lambda^\Neu_\bc$ by $\lambda^\Dir_\bc$, also Theorems \REF{0T10} and \REF{0T11}, are true for the solution of the Dirichlet problem \eqref{0Ebvp}.
\end{remark}

\begin{example}
\iti1 For the unit cube $\Box=(0,1)^3$, 
\[
   \frac\pi{\omega_\be} = 2,\quad
   \lambda^\Dir_\bc = 3 \quad\mbox{and}\quad \lambda^\Neu_\bc = 2.
\]
\iti2 For the Fichera corner $F=(-1,1)^3\setminus (0,1)^3$, at its non-convex corner $\bc_0=(0,0,0)$ and its non-convex edges $\be_i$, $i=1,2,3$ we have
\[
   \frac\pi{\omega_{\be_i}} = \frac23,\quad
   \lambda^\Dir_{\bc_0} \simeq 0.45418 \quad\mbox{and}\quad \lambda^\Neu_{\bc_0} \simeq 0.84001.
\]
Here the 3D Dirichlet and Neumann limiting exponents $\lambda^\Dir_{\bc_0}$ and $\lambda^\Neu_{\bc_0}$ have been computed by Th. Apel using the method of \cite{ApelMehrmannWatkins02}; for the Dirichlet exponent see also \cite{SchmitzVolkWendland93}. 
\end{example}

\subsection{Bibliographical comments} On {\em polyhedra}, the first proof of a Fredholm theory in type $\rK$ weighted spaces is due to Maz'ya-Plamenevskii \cite{MazyaPlamenevskii77}. Optimal regularity and Fredholm results for coercive variational problems are then proved in \cite{Dauge88}  using unweighted Sobolev spaces $\rH^m=\rJ^m_{-m}$, in \cite{MazyaRossmann91b} using type $\rK$ weighted spaces (cf Theorem \REF{0T4}), and in \cite{MazyaRossmann03} using type $\rJ$ weighted spaces (cf Theorem \REF{0T10}). Let us mention that regularity results in $\rK$ weighted spaces are also proved in \cite{BacutaNistor05,MazzucatoNistor10} by an approach distinct from \cite{MazyaRossmann91b}. 

Concerning finite anisotropic regularity, the first partial result is due to Apel-Nicaise \cite{ApelNicaise98} for the Laplace-Dirichlet problem; more complete and general statements using spaces $\rM$ are announced in \cite{BuffaCoDa03} for the same problem; a proof of anisotropic regularity using distinct, but similar spaces, is provided in \cite{BacutaNistor07}. Finally, concerning analytic weighted regularity, prior to the present work, we only find two papers relating to three-dimensional domains: \cite{Guo95} where Guo describes the suitable weighted analytic classes of type $\rB$, and \cite{GuoBabuska97b} where estimates along edges are given for the Laplace-Dirichlet problem.

Recently higher order regularity for the Dirichlet problem on hypercubes was proved in a related, but different class of anisotropic weighted Sobolev spaces
\cite{DaugeStevenson09}. This was also motivated by questions of convergence of numerical approximations, but with a different aim, namely optimal approximations intended to overcome the curse of dimensionality.

%%%%%%%%%%%%%%%%%%%%%%%%%%%%%%%%%%%%%%%%%%%%%%%%%%%%%%%%%%%%%%%%%%%%%%%%%%%%%%%%%%%%%%%%
\section{Local Cauchy-type estimates in smooth domains}
\label{sec1}
%%%%%%%%%%%%%%%%%%%%%%%%%%%%%%%%%%%%%%%%%%%%%%%%%%%%%%%%%%%%%%%%%%%%%%%%%%%%%%%%%%%%%%%%

The starting and key point is a local Cauchy-type (``analytic'') estimate in
smooth domains that is proved by using nested open sets on model
problems, like in the Morrey-Nirenberg proof  \cite{MorreyNirenberg57}  of analytic regularity, 
and a Fa\`a di Bruno formula for local coordinate transformations, see
\cite[Theorem 2.7.1]{GLCI}. This proof, which is given in detail in \cite{GLCI}, will not be repeated here.  

\begin{proposition}
\label{1T1}
Let $\Omega$ be a bounded domain in $\R^n, n\geq 2$.  Let $\Gamma$
be an analytic part of the boundary of $\Omega$. Let $L$ be a
$N\times N$ elliptic system of second order operators with analytic
coefficients over $\Omega\cup\Gamma$.  Let $\{ T,D \}$ be a
  set of boundary operators  on
$\Gamma$ of order $1$ and $0$, respectively, with analytic coefficients, satisfying the Shapiro-Lopatinskii covering condition with respect to  $L$ on $\Gamma$.
Let two bounded subdomains $\widehat\Omega=\cU\cap\Omega$ and
$\widehat\Omega'=\cU'\cap\Omega$ be given with $\cU$ and $\cU'$ open
in $\R^n$ and $\ov\cU\subset \cU'$. We assume that
$\widehat\Gamma':=\partial\widehat\Omega'\cap\partial\Omega$ is
contained in $\Gamma$. Then there exists a constant $A$ such that
any $\bu\in\bH^2(\widehat\Omega)$ satisfies for all $k\in\N$,
$k\ge2$, the improved a priori estimates (``finite analytic
estimates'')
\begin{equation}
\label{1E1}
   \frac{1}{k!} \SNorm{\bu}{k;\,\widehat\Omega} \le A^{k+1} \Big\{
   \sum_{\ell=0}^{k-2} \frac{1}{\ell!}
   \Big(\SNorm{L\bu}{\ell; \,\widehat\Omega'} \!
   + \Norm{T\bu}{\ell+\frac12;\,\widehat\Gamma'} \!
   + \Norm{D\bu}{\ell+\frac32;\,\widehat\Gamma'}\Big)
   +
   \Norm{\bu}{1;\,\widehat\Omega'} \Big\}.
\end{equation}
\end{proposition}

We will know prove a refinement of this estimate where in the right-hand side of \eqref{1E1} the $\rH^1$-norm is replaced by the $\rH^1$-semi-norm.
This will be convenient for boundary value problems of Neumann type.
When $L$, $T$ and $D$ are homogeneous with constant
coefficients, this version is a consequence of the previous result, obtained by a simple argument based on the Bramble-Hilbert lemma. 

\begin{corollary}
\label{1C1}
We assume that the operators $L$, $T$ and $D$ are homogeneous with constant coefficients.
Let $m\ge1$. There exists a constant $A$ independent of $k$ such that there hold the following estimates for all $k\ge m$ and all  $\bu\in\bH^2(\hat\Omega')$   satisfying the zero boundary conditions $T\bu=0$ and $D\bu=0$ on $\hat\Gamma$:
\begin{equation}
\label{1E2}
   \frac{1}{k!} \SNorm{\bu}{k;\,\widehat\Omega} \le A^{k+1} \Big\{
   \sum_{\ell=m-1}^{k-2} \frac{1}{\ell!}
   \SNorm{L\bu}{\ell; \,\widehat\Omega'}
   + \SNorm{\bu}{m;\,\widehat\Omega'} \Big\}.
\end{equation}
\end{corollary}

\begin{proof}
We start with any $\bu\in\bH^k(\hat\Omega')$ and use estimate \eqref{1E1}. We split the right hand side of the inequality into two pieces according to:
\[
   \sum_{\ell=0}^{k-2} \frac{1}{\ell!}
   \Big(\SNorm{L\bu}{\ell; \,\widehat\Omega'} \!
   + \Norm{T\bu}{\ell+\frac12;\,\widehat\Gamma'} \!
   + \Norm{D\bu}{\ell+\frac32;\,\widehat\Gamma'}\Big)
   + \Norm{\bu}{1;\,\widehat\Omega'}= B^*(\bu) + B_*(\bu)
\]
with
\begin{align*}
B^*(\bu) &= \sum_{\ell=m-1}^{k-2} \frac{1}{\ell!}
   \Big(\SNorm{L\bu}{\ell; \,\widehat\Omega'}
   + \SNorm{T\bu}{\ell+\frac12;\,\widehat\Gamma'}
   + \!\!\sum_{j=m-1}^\ell\!  \SNorm{T\bu}{j;\,\widehat\Gamma'}
   + \SNorm{D\bu}{\ell+\frac32;\,\widehat\Gamma'}
   + \sum_{j=m}^{\ell+1} \SNorm{D\bu}{j;\,\widehat\Gamma'}
   \Big)\\
B_*(\bu) &=  \sum_{\ell=0}^{m-2} \frac{1}{\ell!}
   \Big(\SNorm{L\bu}{\ell; \,\widehat\Omega'}
   + \SNorm{T\bu}{\ell+\frac12;\,\widehat\Gamma'}
   + \SNorm{D\bu}{\ell+\frac32;\,\widehat\Gamma'}
   \Big) \\
   &+  \sum_{\ell=0}^{k-2} \frac{1}{\ell!}
   \Big( \sum_{j=0}^{\min\{\ell,m-2\}}\!  \SNorm{T\bu}{j;\,\widehat\Gamma'}
   + \sum_{j=0}^{{\min\{\ell+1,m-1\}}} \SNorm{D\bu}{j;\,\widehat\Gamma'}
   \Big) + \Norm{\bu}{1;\,\widehat\Omega'}
\end{align*}
Since the orders of $L$, $T$ and $D$ are $2$, $1$ and $0$ respectively, we obtain
\[
   B_*(\bu) \le C_m \Norm{\bu}{m;\,\widehat\Omega'}
\]
Since, moreover, the operators $L$, $T$ and $D$ are homogeneous, we have the invariance of $B^*(\bu)$ by subtraction of polynomials of degree less than $m-1$
\[
   B^*(\bu-\varphif) = B^*(\bu),\quad \forall\varphif\in\P^{m-1}(\hat\Omega').
\]
Altogether, using \eqref{1E1} for $\bu-\varphif$ we obtain for all $k\ge m$
\[
   \frac{1}{k!} \SNorm{\bu}{k;\,\widehat\Omega} \le A^{k+1}
   \big\{ B^*(\bu) + C_m  \Norm{\bu-\varphif}{m;\,\widehat\Omega'} \big\}
   ,\quad \forall\varphif\in\P^{m-1}(\hat\Omega').
\]
With the Bramble-Hilbert lemma \cite{BrambleHilbert70}, this gives
\[
   \frac{1}{k!} \SNorm{\bu}{k;\,\widehat\Omega} \le A^{k+1}
   \big\{ B^*(\bu) + C'_m  \SNorm{\bu}{m;\,\widehat\Omega'} \big\} .
\]
Applying this to functions $\bu$ satisfying zero boundary conditions, we obtain \eqref{1E2}.
\end{proof}

%%%%%%%%%%%%%%%%%%%%%%%%%%%%%%%%%%%%%%%%%%%%%%%%%%%%%%%%%%%%%%%%%%%%%%%%%%%%%%%%%%%%%%%%
\section{Weighted Cauchy-type estimates in plane sectors}
\label{sec2}
%%%%%%%%%%%%%%%%%%%%%%%%%%%%%%%%%%%%%%%%%%%%%%%%%%%%%%%%%%%%%%%%%%%%%%%%%%%%%%%%%%%%%%%%
The model singular domains in two dimensions are the infinite plane sectors. Let $\cK$ be an infinite sector with vertex at the coordinate origin $\bfz=(0,0)$. In polar coordinates $(r,\theta)$ such a sector can be described as
\begin{equation}
\label{2E0}
   \cK = \{\bx\in\R^2\ : \  \omega_1<\theta<\omega_2\},
\end{equation}
where $\omega_2=\omega_1+\omega$ with $\omega_1\in(-\pi,\pi)$, and $\omega\in(0,2\pi]$ is the opening of the sector $\cK$. For $i=1,2$, let $\Gamma_i$ be the side $\theta=\omega_i$ of $\cK$.

We consider an elliptic system  $L$ in $\cK$  and on each side $\Gamma_i$ a set of boundary operators $\{T_i,D_i\}$ satisfying the covering condition. We assume that the operators $L$, $T_i$ and $D_i$ are homogeneous of orders $2$, $1$ and $0$, respectively, with constant coefficients.   For any subdomain $\cW\ee'$ of $\cK$, we consider the system of local interior and boundary equations
\begin{equation}
\label{2Ebvp}
   \left\{ \begin{array}{rclll}
   L\,\bu &=& \bff \quad & \mbox{in}\ \ \cK\cap\cW\ee', \\[0.3ex]
   T_i\,\bu &=& 0  & \mbox{on}\ \ \Gamma_i\cap\ov\cW{}\ee', &i=1,2, \\[0.3ex]
   D_i\,\bu &=& 0  & \mbox{on}\ \ \Gamma_i\cap\ov\cW{}\ee', &i=1,2,
   \end{array}\right.
\end{equation}
which is the localization to $\cW\ee'$ of the elliptic boundary value problem $L\bu=\bff$ in $\cK$, with zero boundary conditions on $\Gamma_1$ and $\Gamma_2$.

%%%%%%%%%%%%%%%%%%%%%%%%%%%%%%%%%%%%%%%%%%%%%%%%%%%%%%%%%%%%%%%%%%%%%%%%%%%%%%%%%%%%%%%%
\subsection{Weighted spaces with homogeneous norms}
%%%%%%%%%%%%%%%%%%%%%%%%%%%%%%%%%%%%%%%%%%%%%%%%%%%%%%%%%%%%%%%%%%%%%%%%%%%%%%%%%%%%%%%%
These spaces coincide with those introduced by Kondrat'ev in his pioneering study of corner problems \cite{Kondratev67}. The weight depends on the order of the derivatives. We adopt a different convention than \cite{Kondratev67} in our notation in order to make the definition of corresponding analytic classes more natural (see \eqref{3E4} below).

\begin{definition}
\label{2DK}
Let $\beta$ be a real number called the \emph{weight exponent}, and let $m\ge0$ be an integer called the \emph{Sobolev exponent}. Let $\cW$ be a subdomain of $\cK$.
\begin{itemize}
\item[]
 The \emph{weighted space with homogeneous norm} $\rK^m_\beta(\cW)$ is defined,
 with the distance  $r=|\bx|$ to the vertex $\bfz$, by
\begin{equation}
\label{2EK1}
   \rK^m_\beta (\cW) = \big\{ u \in \rL_\loc^2(\cW) \ : \
   r^{\beta +  |\alpha|}  \partial^{\alpha}_\bx u  \in  \rL^2(\cW), \quad
   \forall\alpha, \; |\alpha| \leq m \big\}
\end{equation}
and endowed with semi-norm and norm respectively defined as
\begin{equation}
\label{2EK2}
   \SNorm{u}{\rK;\,m,\beta\,;\,\cW}^2 = \!\sum_{|\alpha|=m}\!
   \Norm{r^{\beta +  |\alpha|}  \partial^\alpha_\bx u}{0;\, \cW}^2
   \ \ \mbox{and } \ \
   \Norm{u}{\rK^m_\beta (\cW)}^2 = \sum_{k=0}^m
   \SNorm{u}{\rK;\,k,\beta\,;\,\cW}^2 .
\end{equation}
\end{itemize}
\end{definition}

In these spaces we have the following estimates.

\begin{theorem}
\label{2T1}
Let $\cW$ and $\cW\ee'$ be the intersections of $\cK$ with the balls centered at $\bfz$ of radii $1$ and $1+\delta$, respectively. Let $\beta\in\R$ and $n\in\N$.
Let $\bu\in\bH^2_{\loc}(\ov \cW{}\ee'\setminus\{\bfz\})$ be a solution of problem \eqref{2Ebvp}. Then the following implication holds
\begin{equation}
\label{2EK3}
   \bu\in\bK^1_{\beta}(\cW\ee') \ \ \mbox{and}\ \
   \bff\in\bK^{n}_{\beta+2}(\cW\ee')
   \quad\Longrightarrow\quad
   \bu\in\bK^{n+2}_{\beta}(\cW)
\end{equation}
and there
exists a constant $C\ge1$ independent of $\bu$ and $n$ such that for any integer $k$,
$0\le k\le n+2$, we have
\begin{multline}
\label{2EK4}
   \frac{1}{k!}\;
   \Big(\sum_{|\alpha|=k}
   \Norm{r^{\beta+|\alpha|}\partial^\alpha_\bx \bu}{0;\,\cW}^2\Big)^{\frac12} \le
    C^{k+1} \Big\{
   \sum_{\ell=0}^{k-2} \frac{1}{\ell!}\;
   \Big(\sum_{|\alpha|=\ell}\Norm{r^{\beta+2+|\alpha|}
   \partial^\alpha_\bx \bff}{0;\,\cW\ee'}^2\Big)^{\frac12}
   \\
   +
   \sum_{|\alpha|\le1} \Norm{r^{\beta+|\alpha|}
   \partial^\alpha_{\bx}\bu}{0;\,\cW\ee'}\Big\}.
\end{multline}
\end{theorem}

\begin{proof}
Let us assume that $\bu\in\bK^1_{\beta}(\cW\ee')$ and $L\bu=\bff\in\bK^{n}_{\beta+2}(\cW\ee')$.
Let us prove estimate \eqref{2EK4}.
By definition of the weighted spaces, the right-hand side of \eqref{2EK4} is bounded.
The proof of the estimate is based on a locally finite dyadic covering of $\cW$ and $\cW\ee'$. Let us introduce the reference annuli, see Fig.\ \ref{F1}
\begin{equation}
\label{2E4V}
   \widehat \cV = \{\bx\in \cK\ : \ \tfrac14< r(\bx) < 1\}
   \quad\mbox{and}\quad
   \widehat \cV\ee' = \{\bx\in \cK\ : \  \tfrac14-\delta< r(\bx) < 1+\delta\}.
\end{equation}
and for $\mu\in\N$ the scaled annuli:
\[
   \cV_\mu = 2^{-\mu}\widehat \cV \quad\mbox{and}\quad \cV\prm_\mu = 2^{-\mu}\widehat \cV\ee'.
\]
\begin{figure}
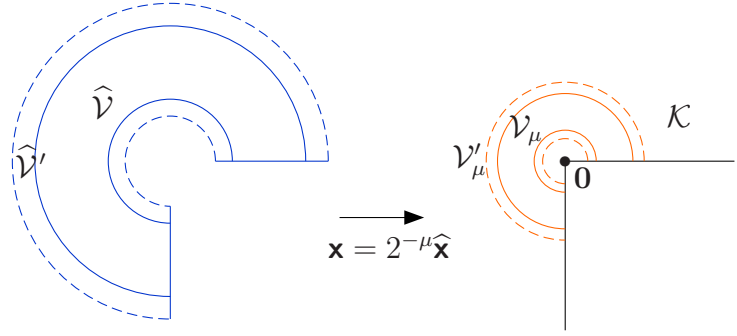

\begin{center}
% 1. Definition of characteristic points
    \figinit{0.75mm}
    \figpt 1:(  0, 0)
    \figpt 2:(28, 0)
    \figpt 3:( 0, -28)
    \figpt 32:(8, 0)
    \figpt 33:(0, -8)
    \figpt 4:( -20, 20)
    \figpt 8:( -17, 17)
    \figpt 6:(45,-10)
    \figpt 5:(30,-10)
    \figpt 12:(30,0)
    \figpt 13:(0,-30)
    \figvectC 100 ( 70, 0)
    \figptstra 21 = 1, 12, 13, 8 /1, 100/

% 2. Creation of the postscript file
    \def\MyPSfile{}
    \psbeginfig{\MyPSfile}
    \psset arrowhead (fillmode=yes)
    \psarrow[5,6]
    \psline[21,22]
    \psline[21,23]
    \psset (color=0. 0.2 0.8)
    \psline[32,2]
    \psline[33,3]
    \psset (color=1.0 0.4 0)
    \psarccirc 21; 5.5 (0,270)
    \psarccirc 21; 12 (0,270)
    \psset (color=0. 0.2 0.8)
    \psarccirc 1; 11 (0,270)
    \psarccirc 1; 24 (0,270)
    \psset (dash=3)
    \psset (color=1.0 0.4 0)
    \psarccirc 21; 4 (0,270)
    \psarccirc 21; 14 (0,270)
    \psset (color=0. 0.2 0.8)
    \psarccirc 1; 8 (0,270)
    \psarccirc 1; 28 (0,270)
    \psendfig

% 3. Writing text on the figure
    \figvisu{\demo}{}{%
    \figinsert{\MyPSfile}
    \figsetmark{\footnotesize\Bullet}
    \figwritegce 21:$\bf0$(1.5,-3)
    \figsetmark{}
    \figwritegce 21:$\cK$(18,8)
    \figwritegce 5:{${\bx=2^{-\mu}}{\widehat\bx}$}(-2,-5)
%    \figwritegce 5:{$(\mu=1)$}(-2,-12)
    \figwritegce 1:{{$\widehat\cV$}}(-14,10.5)
    \figwritegce 1:{{$\widehat\cV\ee'$}}(-27.5,0)
    \figwritegce 21:{{$\cV_\mu$}}(-10,5.5)
    \figwritegce 21:{{$\cV'_\mu$}}(-20,0)
   }
    \centerline{\box\demo}
    \caption{Reference and scaled annuli for a sector $\cK$ of opening $3\pi/2$}
\label{F1}
\end{center}
\end{figure}
We check immediately that
\[
   \cW = \bigcup_{\mu\in\N} \cV_\mu
   \quad\mbox{and}\quad
   \cW\ee' = \bigcup_{\mu\in\N} \cV\prm_\mu\,.
\]
{\sc Step 1.\ } We are going to apply Proposition \REF{1T1} in two regions which separate the two sides $\Gamma_1$ and $\Gamma_2$ of $\cK$ where the boundary conditions can be distinct.
We recall that the sector $\cK$ is defined by the angular inequalities $\omega_1<\theta<\omega_2$. Let $\omega_3 := \frac12(\omega_1+\omega_2)$. We define the sectors $\cK_1$ and $\cK_2$ by
\[
   \cK_1 = \{\bx\in\R^2\ : \  \omega_1<\theta<\omega_3\} \quad\mbox{and}\quad
   \cK_2 = \{\bx\in\R^2\ : \  \omega_3<\theta<\omega_2\} .
\]
Let  $0<\delta_0<\frac12(\omega_2-\omega_1)$. We define the larger sectors $\cK'_1$ and $\cK'_2$ by
\[
   \cK'_1 = \{\bx\in\R^2\ : \  \omega_1<\theta<\omega_3+\delta_0\} \quad\mbox{and}\quad
   \cK'_2 = \{\bx\in\R^2\ : \  \omega_3-\delta_0<\theta<\omega_2\} .
\]
Let $i\in\{1,2\}$. Since the system $L$ is elliptic and covered by its boundary
conditions $\{T_i,D_i\}$ on $\Gamma_i$, the reference domains $\widehat \cV\cap\cK_i$ and $\widehat
\cV\ee'\cap\cK'_i$ satisfy the assumptions of Proposition  \REF{1T1}, and
there exists a positive constant $A_i$ such that for all $k\in \N,
k\geq 2$, we have:
\begin{equation}
\label{2E4i}
   \frac{1}{k!}\; \SNorm{\widehat\bu}{k;\, \widehat \cV\cap\cK_i} \le A^{k+1}_i \Big\{
   \sum_{\ell=0}^{k-2} \frac{1}{\ell!}\;
   \SNorm{\widehat\bff}{\ell; \, \widehat \cV\ee'\cap\cK'_i} +
   \sum_{\ell=0}^{1} \SNorm{\widehat\bu}{\ell;\, \widehat \cV\ee'\cap\cK'_i} \Big\},
\end{equation}
for any function $\widehat\bu$ satisfying the boundary conditions of \eqref{2Ebvp} and $\widehat\bff:=L \widehat\bu$. From these estimates for $i=1,2$ we deduce immediately, with $A=\max\{A_1,A_2\}$
\begin{equation}
\label{2E4}
   \frac{1}{k!}\; \SNorm{\widehat\bu}{k;\, \widehat \cV} \le  2 A^{k+1} \Big\{
   \sum_{\ell=0}^{k-2} \frac{1}{\ell!}\;
   \SNorm{\widehat\bff}{\ell; \, \widehat \cV\ee'} +
   \sum_{\ell=0}^{1} \SNorm{\widehat\bu}{\ell;\, \widehat \cV\ee'} \Big\},
\end{equation}
{\sc Step 2.\ }  Since $r(\hat\bx)\simeq 1$ on $\widehat \cV\ee'$, we can insert weights in the reference estimate \eqref{2E4}: There exists
a positive constant $B$ such that for all $k\in \N, k\geq 2$
\begin{align}
\label{2E4b}
   \frac{1}{k!}\; \Big(\sum_{|\alpha|=k}
   \Norm{r(\hat\bx)^{\beta+|\alpha|}\partial^\alpha_{\hat\bx} \widehat\bu}{0;\, \widehat \cV}^2
   \Big)^{\frac12} \le B^{k+1} \Big\{
   \sum_{\ell=0}^{k-2} \frac{1}{\ell!}\;
   \Big(\sum_{|\alpha|=\ell}\Norm{r(\hat\bx)^{\beta+2+|\alpha|} \partial^\alpha_{\hat\bx}
   \widehat\bff}{0; \, \widehat \cV\ee'}^2 \Big)^{\frac12}
   \\ \nonumber
   + \sum_{|\alpha| \le 1}
   \Norm{r(\hat\bx)^{\beta+|\alpha|}
   \partial^\alpha_{\hat\bx} \widehat\bu}{0;\, \widehat \cV\ee'} \Big\}.
\end{align}
By the change of variables $\hat\bx\rightarrow\bx=2^{-\mu}\hat\bx$ that maps $\widehat \cV$ onto $\cV_{\mu}$ (resp. $\widehat \cV\ee'$ onto
$\cV\prm_{\mu}$) coupled with the change of functions
\[
   \widehat\bu(\widehat\bx) := \bu(\bx) \quad\mbox{and}\quad
   \widehat\bff(\widehat\bx) := L \widehat\bu
   \quad\mbox{which implies}\quad \widehat\bff(\widehat\bx) = 2^{-2\mu}\bff(\bx),
\]
we deduce from estimate \eqref{2E4b} that
\begin{align*}
   &\frac{1}{k!}\, 2^{\mu\beta-\mu}
   \Big(\sum_{|\alpha|=k}
   \Norm{r(\bx)^{\beta+|\alpha|}\partial^\alpha_\bx \bu}{0;\,\cV_{\mu}}^2\Big)^{\frac12}
   \le B^{k+1} \Big\{    \\
   &\ \ \sum_{\ell=0}^{k-2} \frac{1}{\ell!}\;
   2^{\mu (\beta+2)-\mu}\Big(\sum_{|\alpha|=\ell} 2^{-2\mu}
   \Norm{r(\bx)^{\beta+2+|\alpha|} \partial^\alpha_\bx \bff}{0;
   \,\cV\prm_{\mu}}^2\Big)^{\frac12} \!\!
   + 2^{\mu \beta-\mu}  \sum_{|\alpha\le 1}
   \Norm{r(\bx)^{\beta+|\alpha|} \partial^\alpha_\bx \bu}{0;\,\cV\prm_{\mu}}  \Big\}.
\end{align*}
Multiplying this identity by $2^{-\mu\beta+\mu}$, the above
estimate is equivalent to
\begin{align*}
   \frac{1}{k!}\;
   \Big(\sum_{|\alpha|=k}
   \Norm{r(\bx)^{\beta+|\alpha|}\partial^\alpha_\bx \bu}{0;\,\cV_{\mu}}^2\Big)^{\frac12}
   \le B^{k+1} \Big(
   \sum_{\ell=0}^{k-2} \frac{1}{\ell!}\;
   \Big(\sum_{|\alpha|=\ell}\Norm{r(\bx)^{\beta+2+|\alpha|} \partial^\alpha_\bx \bff}{0;
   \,\cV\prm_{\mu}}^2\Big)^{\frac12}
   \\
   + \sum_{|\alpha|\le 1}
   \Norm{r(\bx)^{\beta+|\alpha|} \partial^\alpha_\bx \bu}{0;\,\cV\prm_{\mu}} \Big).
\end{align*}
Summing up the square of this estimate over all $\mu$ and
considering that only a finite number of the $\cV\prm_{\mu}$ overlap,
we get the desired estimate \eqref{2EK4}.
\end{proof}

%%%%%%%%%%%%%%%%%%%%%%%%%%%%%%%%%%%%%%%%%%%%%%%%%%%%%%%%%%%%%%%%%%%%%%%%%%%%%%%%%%%%%%%%
\subsection{Weighted spaces with non-homogeneous norms}
%%%%%%%%%%%%%%%%%%%%%%%%%%%%%%%%%%%%%%%%%%%%%%%%%%%%%%%%%%%%%%%%%%%%%%%%%%%%%%%%%%%%%%%%
\label{ss3.2}
In these spaces the weight exponent does not depend on the order of derivatives. Standard unweighted Sobolev spaces are a special case. The weighted Sobolev spaces with nonhomogeneous norms allow an accurate description of the regularity of functions with non-trivial Taylor expansion at the corners. In particular, they are useful for studying variational problems of Neumann type, because the variational space $\rH^1$ does not fit properly into the scale $\rK^1_\beta$.

\begin{definition}
\label{2DJ}
 Let $\beta$ be a real number and $m\ge0$  an integer.
\begin{itemize}
\item[]  Let $\cW$ be an open subset of $\cK$. The \emph{weighted space with non-homogeneous norm}
$\rJ^m_\beta(\cW)$ is defined by
\begin{equation}
\label{2EJ1}
   \rJ^m_\beta (\cW) = \big\{ u \in \rL_\loc^2(\cW) \ : \
   r^{\beta+m}  \partial^{\alpha}_\bx u  \in  \rL^2(\cW), \quad
   \forall\alpha, \; |\alpha| \leq m \big\}
\end{equation}
with its norm
$$
   \Norm{u}{\rJ^m_\beta (\cW)}^2 = \sum_{|\alpha|\le m}
   \Norm{r^{\beta+m}  \partial^\alpha_\bx u}{0;\, \cW}^2 .
$$
\end{itemize}
\end{definition}

Note that the semi-norm of $\rJ^m_\beta (\cW)$ coincides with the semi-norm of
$\rK^m_\beta(\cW)$. They are both denoted by $\SNorm{\cdot}{\rK;\,m,\beta\,;\,\cW}$.
With this notation, we have
\begin{equation}
\label{2EJ2}
   \Norm{u}{\rJ^m_\beta(\cW)}^2 =
   \sum_{k=0}^m \SNorm{u}{\rK;\,k,\beta+m-k\,;\, \cW}^2 .
\end{equation}

\begin{remark}
\label{2R1}
If $\cW$ is a finite sector with vertex at the origin, there holds \cite{MazyaPlamenevskii84b, KozlovMazyaRossmann97b} (more details are given in \cite{CostabelDaugeNicaise09}
 and \cite[Ch.\,11]{GLC}):

If $\beta>-1$, then $\rJ^m_{\beta}(\cW)=\rK^m_{\beta}(\cW)$ for all $m\in\N$.

If $\beta\le-1$ and $m\le-\beta-1$, then, again, $\rJ^m_{\beta}(\cW)=\rK^m_{\beta}(\cW)$.

If $\beta\le-1$ and $m>-\beta-1$, then one has to distinguish two cases:
\begin{itemize}
\item the generic case $-\beta\not\in\N$, in which one has
$$
  \rJ^m_{\beta}(\cW)=\rK^m_{\beta}(\cW)\oplus
    \P^{[-\beta-1]}
$$
where $\P^{[-\beta-1]}$ is the space of polynomials of degree not exceeding $-\beta-1$;
\item the \emph{critical} case $-\beta\in\N$, in which  $\rJ^m_{\beta}(\cW)$ contains $\rK^m_{\beta}(\cW) \oplus \P^{-\beta-1}$ as a strict subspace.
\end{itemize}
\end{remark}

We deduce from \cite[Prop.\ 3.18]{CostabelDaugeNicaise09} the following ``step-weighted'' characterization of $\rJ^m_\beta$ in the case of two space dimensions:

\begin{proposition}
\label{2P1}
Let $\beta\in\R$ and $m\ge-\beta$ be a natural number. If $\cW$ is bounded, then the norm in the space
$\rJ^m_\beta(\cW)$ is equivalent to
\begin{equation}
\label{2E6}
   \Bigl( \sum_{|\alpha|\le m}
   \Norm{r^{\max\{ \beta+|\alpha|,\,0 \}} \partial^\alpha_\bx u}
   {0; \,\cW}^2 \Bigr)^{\frac12}.
\end{equation}
\end{proposition}

\begin{corollary}
\label{2C2}
Let $\beta\in\R$. Let $m\ge-\beta$ be a natural number\footnote{
For the sake of simplicity we did not quote \cite[Prop.\ 3.18]{CostabelDaugeNicaise09} in its full optimality. In fact the embedding $\rJ^{m+1}_\beta(\cW)\subset\rJ^m_\beta(\cW)$ holds as soon as $m>-\beta-1$.}. Then
$\rJ^{m+1}_\beta(\cW)\subset\rJ^m_\beta(\cW)$.
\end{corollary}

\begin{theorem}
\label{2T2}
Let $\cW$ and $\cW\ee'$ be the intersections of $\cK$ with the balls centered at $\bfz$ of radii $1$ and $1+\delta$, respectively. 
Let $\beta\in\R$. Let $m\ge1$ be an integer such that $m+1\ge-\beta$.
Let $n\ge m-1$ be another integer. 
Let $\bu\in\bH^2_{\loc}(\ov \cW{}\ee'\setminus\{\bfz\})$ be a solution of problem \eqref{2Ebvp}.
Then the following implication holds
\begin{equation}
\label{2EJ3}
   \bu\in\bJ^m_{\beta}(\cW\ee') \ \ \mbox{and}\ \
   \bff\in\bJ^{n}_{\beta+2}(\cW\ee')
   \quad\Longrightarrow\quad
   \bu\in\bJ^{n+2}_{\beta}(\cW)
\end{equation}
and there exists a constant $C\ge1$ independent of $\bu$ and $n$ such that for all integer $k$,
$m\le k\le n+2$, we have
\begin{multline}
\label{2EJ4}
   \frac{1}{k!}\;
   \Big(\sum_{|\alpha|=k}
   \Norm{r^{\beta+|\alpha|}\partial^\alpha_\bx \bu}{0;\,\cW}^2\Big)^{\frac12} \le
    C^{k+1} \Big\{
   \sum_{\ell=m-1}^{k-2} \frac{1}{\ell!}\;
   \Big(\sum_{|\alpha|=\ell}\Norm{r^{\beta+2+|\alpha|}
   \partial^\alpha_\bx \bff}{0;\,\cW\ee'}^2\Big)^{\frac12}
   \\
   +
   \sum_{|\alpha|=m} \Norm{r^{\beta+|\alpha|}
   \partial^\alpha_{\bx}\bu}{0;\,\cW\ee'}\Big\}.
\end{multline}
\end{theorem}

\begin{proof}
Let $n\ge m-1$ and assume that $\bu\in\bJ^m_{\beta}(\cW\ee')$ is such that
$\bff\in\bJ^{n}_{\beta+2}(\cW\ee')$. 
If $k=m$, estimate \eqref{2EJ4} is trivial.
So, let $k$  be such that $m+1\le k\le n+2$.
Let us prove estimate \eqref{2EJ4}. Since $m+1\ge-\beta$, we have
$2+|\alpha|\ge-\beta$ for all $\alpha$ with length $\ge m-1$.   
Therefore, as a consequence of Proposition \REF{2P1}, the right-hand
side of \eqref{2EJ4} is bounded.

Then, in a similar way as in the proof of Theorem \REF{2T1}, we
start from estimate \eqref{1E2} written for the reference domains
$\widehat \cV$ and $\widehat \cV\ee'$ and we apply the same dyadic
covering technique. We arrive directly at the estimate \eqref{2EJ4}.

It remains to prove that $\bu\in\bJ^{n+2}_{\beta}(\cW)$. Since $\cW$ is bounded, estimate \eqref{2EJ4} implies that $r^{\beta+n+2}\partial^\alpha_\bx\bu$ belongs to $\bL^2(\cW)$ for all $\alpha$, $m\le|\alpha|\le n+2$. Since $\bu\in\bJ^m_{\beta}(\cW\ee')$, we deduce that $r^{\beta+n+2}\partial^\alpha_\bx\bu$ also belongs to $\bL^2(\cW)$ when $|\alpha|<m$, which ends the proof.
\end{proof}

%%%%%%%%%%%%%%%%%%%%%%%%%%%%%%%%%%%%%%%%%%%%%%%%%%%%%%%%%%%%%%%%%%%%%%%%%%%%%%%%%%%%%%%%
\section{Natural weighted regularity shift in polygons}
\label{sec3}
%%%%%%%%%%%%%%%%%%%%%%%%%%%%%%%%%%%%%%%%%%%%%%%%%%%%%%%%%%%%%%%%%%%%%%%%%%%%%%%%%%%%%%%%
Let $\Omega$ be a polygonal domain. This means that the boundary of $\Omega$ is the union of a finite number of line segments (the sides $\Gamma_\bs$, for indices $\bs\in\sS$). We do not assume that $\Omega$ is a Lipschitz domain, that is we include the presence of cracks in our analysis. The vertices $\bc$ are the ends of the edges. Let us denote by $\sC$ the set of vertices and
\begin{equation}
\label{3E0}
   r_\bc(\bx) = \dist(\bx,\bc).
\end{equation}
There exists $\varepsilon>0$ such that, setting
\begin{subequations}
\begin{equation}
\label{3E1a}
   \Omega_\bc = \{\bx\in\Omega \ :\ r_\bc<\varepsilon\},
\end{equation}
we have
\begin{equation}
\label{3E1b}
   \ov\Omega_\bc\cap\ov\Omega_{\bc'} = \varnothing,\ \ \forall\bc\neq\bc'.
\end{equation}
Setting $\Omega^{(2)}_\bc = \{\bx\in\Omega : r_\bc<\frac\varepsilon2\}$, we define
\begin{equation}
\label{3E1c}
   \Omega_0 = \Omega\setminus \bigcup_{\bc\in\sC} \ov {\Omega^{(2)}\dd\bc}\,.
\end{equation}
We also define larger neighborhoods choosing $\varepsilon'>\varepsilon$ such that
\begin{equation}
\label{3E1d}
   \Omega'_\bc = \{\bx\in\Omega \ :\ r_\bc<\varepsilon'\},\quad
   \ov\Omega{}'_\bc\cap\ov\Omega{}'_{\bc'} = \varnothing,\ \ \forall\bc\neq\bc' \,,
\end{equation}
and we finally set
\begin{equation}
\label{3E1}
   \Omega'_0 = \Omega\setminus \bigcup_{\bc\in\sC}  \ov {\Omega^{(3)}\dd\bc}\,,
\end{equation}
where $\Omega^{(3)}_\bc = \{\bx\in\Omega : r_\bc<\frac\varepsilon3\}$.
For each corner there is a plane sector $\cK_\bc$ with vertex $\bfz$ such that the translation $\bx\mapsto\bx-\bc$ sends $\Omega_\bc$ onto $\cK_\bc\cap \cB(\bfz,\varepsilon)$.
\end{subequations}

Let $\betab=(\beta_\bc)\dd{\bc\in\sC}\in\R^{\#\sC}$ be a weight multi-exponent and $m\in\N$ a Sobolev exponent. By localization we define the weighted semi-norm on any domain $\cV\subset\Omega$:
\begin{equation}
\label{3E2}
   \SNorm{u}{m,\betab\,;\,\cV}^2 = \sum_{|\alpha|=m} \Big(
   \Norm{ \partial^\alpha_\bx u}{0;\, \cV\cap\Omega_0}^2
   + \sum_{\bc\in\sC}
   \Norm{r_\bc^{\beta_\bc + |\alpha|}  \partial^\alpha_\bx u}{0;\, \cV\cap\Omega_\bc}^2
   \Big),
\end{equation}
and the norms, {\em cf.}\  \eqref{2EK2} and \eqref{2EJ2}
\begin{equation}
\label{3E3}
   \Norm{u}{\rK^m_\betab (\cV)}^2 =
   \sum_{k=0}^m \SNorm{u}{\rK;\,k,\betab\,;\,\cV}^2
   \quad\mbox{and}\quad
   \Norm{u}{\rJ^m_\betab(\Omega)}^2 =
   \sum_{k=0}^m \SNorm{u}{\rK;\,k,\betab+m-k\,;\, \cV}^2 \,,
\end{equation}
defining the spaces $\rK^m_\betab(\cV)$ and $\rJ^m_\betab(\cV)$, respectively. If all weight exponents $\beta_\bc$ coincide with the same number $\beta$, these spaces are simply denoted by $\rK^m_\beta(\cV)$ and $\rJ^m_\beta(\cV)$, respectively. Boldface notations $\bK^m_\betab(\cV)$ and $\bJ^m_\betab(\cV)$ indicate vector-valued functions.

\begin{remark}
\label{3R00}
The semi-norm $\SNorm{u}{m,\betab\,;\,\Omega}$ is equivalent to the globally defined semi-norm
\begin{equation}
\label{3E00}
   \Big\{ \sum_{|\alpha|=m}
   \Big\| \big(\prod_{\bc\in\sC} r_\bc^{\beta_\bc + |\alpha|}\big) \,\partial^\alpha_\bx u
   \Big\|_{0;\, \Omega}^2
   \Big\}^{\frac12}.
\end{equation}
\end{remark}

We define on any domain $\cV\subset\Omega$ the corresponding weighted analytic classes.
\iti1 With homogeneous norm:
\begin{equation}
\label{3E4}
   \rA_\betab (\cV) = \Big\{\  u \in \bigcap_{m\ge0}\rK^m_\betab(\cV) \ : \
    \exists C>0, \forall m\in\N,\quad
   \SNorm{u}{m,\betab\,;\,\cV} \le C^{m+1} m!   \Big\}.
\end{equation}
\iti2 With non-homogeneous norm: For a multi-exponent $\betab$ let
\begin{equation}
\label{3E5a}
   \kappa_\betab:=\max_{\bc\in\sC}-\beta_\bc.
\end{equation}
As a consequence of Proposition \REF{2P1}, for all $m\ge\kappa_\betab$ we have the continuous embedding of $\rJ^{m+1}_\betab(\cV)$ into $\rJ^m_\betab(\cV)$.
We introduce
\begin{equation}
\label{3E5b}
   \rB_\betab (\cV) = \Big\{\  u \in \bigcap_{m\ge\kappa_\betab}\rJ^m_\betab(\cV) \ : \
    \exists C>0, \forall m\ge\kappa_\betab\quad
   \SNorm{u}{m,\betab\,;\,\cV} \le C^{m+1} m!   \Big\}.
\end{equation}

\begin{remark}
\label{3R0}
\iti1 The classes $\rA_\betab (\Omega)$ and $\rB_\betab(\Omega)$ can be equivalently defined replacing semi-norms $\SNorm{u}{m,\betab\,;\,\Omega}$ by the global semi-norms \eqref{3E00}.
\iti2 The classes $\rA_\betab (\Omega)$ can also be equivalently defined locally i.e.
\[
   \rA_\betab (\Omega) = \{u\in\rL^2_\loc(\Omega)\ : \
   u\on{\Omega_0}\in\rA(\Omega_0) \quad\mbox{and}\quad
   u\on{\Omega_\bc}\in\rA_{\beta_\bc}(\Omega_\bc)\ \forall\bc\in\sC\}.
\]
Here $\rA(\Omega_0)$ is the unweighted class of analytic functions on $\Omega_0$.
The spaces $\rB_\betab(\Omega)$ allow analogous local descriptions.
\end{remark}

\begin{remark}
\label{3R1}
\itj1
Our spaces $\rB_{\beta}(\Omega)$ coincide with the family of \emph{countably normed spaces} $B^\ell_\beta(\Omega)$, introduced by Babu\v{s}ka and Guo \cite{BabuskaGuo88}: The spaces  $B^\ell_\beta(\Omega)$ are defined for $\ell\in\N$ and $0<\beta<1$, and there holds
\begin{equation}
\label{3eBG=Bb}
  B^\ell_\beta(\Omega) = \rB_{\beta-\ell}(\Omega) \,.
\end{equation}
\iti2
The relation between the classes $\rA_\betab (\Omega)$ and $\rB_\betab (\Omega)$ follows from the relation between the weighted spaces with homogeneous and nonhomogeneous norms  $\rK^m_\betab (\Omega)$ and $\rJ^m_\betab(\Omega)$. As a consequence of Remark~\ref{2R1} it follows that for $\beta>-1$ there holds
$\rB_{\beta}(\Omega_\bc)=\rA_{\beta}(\Omega_\bc)$, whereas for
$\beta\le-1$ one has in the non-critical case $-\beta\not\in\N$:
\begin{equation}
\label{3EBAP}
  \rB_{\beta}(\Omega_\bc)=\rA_{\beta}(\Omega_\bc)\oplus
    \P^{[-\beta-1]}
\end{equation}
and in the \emph{critical} case $-\beta\in\N$: $\rB_{\beta}(\Omega_\bc)$ contains $\rA_{\beta}(\Omega_\bc) \oplus \P^{-\beta-1}$ as a strict subspace.
\end{remark}

We consider a ``mixed'' boundary value problem on the polygonal domain $\Omega$: We suppose that we are given an homogeneous second order elliptic system $L$ with constant coefficients and for each side $\bs$ a covering set of boundary operators $\{T_\bs,D_\bs\}$ of order $1$ and $0$, homogeneous with constant coefficients. The boundary value problem under consideration is:
\begin{equation}
\label{3Ebvp}
   \left\{ \begin{array}{rclll}
   L\,\bu &=& \bff \quad & \mbox{in}\ \ \Omega, \\[0.3ex]
   T_\bs\,\bu &=& 0  & \mbox{on}\ \ \Gamma_\bs, & \bs\in\sS\,, \\[0.3ex]
   D_\bs\,\bu &=& 0  & \mbox{on}\ \ \Gamma_\bs, & \bs\in\sS\,.
   \end{array}\right.
\end{equation}
Note that one of $T_\bs$ or $D_\bs$ may be the zero operator, in which case the corresponding boundary condition is empty.

We can now prove the following statements of natural regularity shift
in weighted analytic spaces with homogeneous or non-homogeneous
semi-norms, respectively:

\begin{theorem}
\label{3T1}
Let $\betab=(\beta_\bc)\dd{\bc\in\sC}$ be a weight multi-exponent.
Let $\bu\in\bH^2_{\loc}(\ov \Omega\setminus\sC)$ be a solution of problem \eqref{3Ebvp}.
For all integer $k\ge1$, there holds the global uniform estimate 
\begin{equation}
\label{3EK3est}
   \frac1{k!} \SNorm{\bu}{\rK;\,k,\betab\,;\,\Omega} \le C^{k+1} \Big(
   \sum_{\ell=0}^{k-2} \frac1{\ell!} \SNorm{\bff}{\rK;\,\ell,\betab+2\,;\,\Omega} +
   \Norm{\bu}{\bK^1_{\betab}(\Omega)} \Big).
\end{equation}
The following implications hold
\begin{subequations}
\begin{equation}
\label{3EK3}
   \bu\in\bK^1_{\betab}(\Omega) \ \ \mbox{and}\ \
   \bff\in\bK^{n}_{\betab+2}(\Omega)
   \quad\Longrightarrow\quad
   \bu\in\bK^{n+2}_{\betab}(\Omega) \ \ (n\in\N).
\end{equation}
and
\begin{equation}
\label{3EA3}
   \bu\in\bK^1_{\betab}(\Omega) \ \ \mbox{and}\ \
   \bff\in\bA_{\betab+2}(\Omega)
   \quad\Longrightarrow\quad
   \bu\in\bA_{\betab}(\Omega).
\end{equation}
\end{subequations}
\end{theorem}

\begin{proof}
The uniform estimate \eqref{2EK4} is valid between $\Omega_\bc$ and $\Omega'_\bc$ for all $\bc\in\sC$. The uniform estimate \eqref{1E1} of the smooth case is valid between $\Omega_0$ and $\Omega'_0$. Combining these estimates we obtain the global uniform estimate \eqref{3EK3est}. The finite regularity shift \eqref{3EK3} is then straightforward.
If $\bff\in\bA_{\betab+2}(\Omega)$, it satisfies $\SNorm{\bff}{\ell,\betab\,;\,\Omega} \le F^{\ell+1} \ell!$ for some constant $F>1$. Thus estimate \eqref{3EK3est} yields
\[
   \SNorm{\bu}{\rK;\,k,\betab\,;\,\Omega} \le k! \ C^{k+1} \Big(
   \sum_{\ell=0}^{k-2} F^{\ell+1} + \Norm{\bu}{\bK^1_{\betab}(\Omega)} \Big) =
   k! \ C^{k+1} \Big(\frac{F^{k}-F}{F-1}
   + \Norm{\bu}{\bK^1_{\betab}(\Omega)} \Big).
\]
Hence $\bu\in\bA_{\betab}(\Omega)$, which proves \eqref{3EA3}.
\end{proof}

\begin{theorem}
\label{3T1b}
Let $\betab=(\beta_\bc)\dd{\bc\in\sC}$ be a weight multi-exponent.
Let $\bu\in\bH^2_{\loc}(\ov \Omega\setminus\sC)$ be a solution of problem \eqref{3Ebvp}.
Let $m\ge1$ be an integer such that $m\ge-\beta_\bc$ for all $\bc\in\sC$.
For all integer $k\ge m$, there holds the global uniform estimate 
\begin{equation}
\label{3EJ3est}
   \frac1{k!} \SNorm{\bu}{\rK;\,k,\betab\,;\,\Omega} \le C^{k+1} \Big(
   \sum_{\ell=m-1}^{k-2} \frac1{\ell!} \SNorm{\bff}{\rK;\,\ell,\betab+2\,;\,\Omega} +
   \Norm{\bu}{\bJ^m_{\betab}(\Omega)} \Big).
\end{equation}
The following implications hold
\begin{subequations}
\begin{equation}
\label{3EJ3}
   \bu\in\bJ^m_{\betab}(\Omega) \ \ \mbox{and}\ \
   \bff\in\bJ^{n}_{\betab+2}(\Omega)
   \quad\Longrightarrow\quad
   \bu\in\bJ^{n+2}_{\betab}(\Omega) \ \ (n+2\ge m).
\end{equation}
and
\begin{equation}
\label{3EB3}
   \bu\in\bJ^m_{\betab}(\Omega) \ \ \mbox{and}\ \
   \bff\in\bB_{\betab+2}(\Omega)
   \quad\Longrightarrow\quad
   \bu\in\bB_{\betab}(\Omega).
\end{equation}
\end{subequations}
\end{theorem}

\begin{proof}
The finite regularity shift \eqref{3EJ3} is an obvious consequence of Theorem \REF{2T2}. The proof of \eqref{3EB3} is similar to that of \eqref{3EA3}, based on estimate \eqref{2EJ4}.
\end{proof}

%%%%%%%%%%%%%%%%%%%%%%%%%%%%%%%%%%%%%%%%%%%%%%%%%%%%%%%%%%%%%%%%%%%%%%%%%%%%%%%%%%%%%%%%
\section{Local anisotropic Cauchy-type estimates in dihedral domains}
\label{sec4}
%%%%%%%%%%%%%%%%%%%%%%%%%%%%%%%%%%%%%%%%%%%%%%%%%%%%%%%%%%%%%%%%%%%%%%%%%%%%%%%%%%%%%%%%
Infinite dihedral domains (or wedges) are the model domains for polyhedra which have the lowest level of complexity. In this section, we consider dihedral domains $\cD$ in a model configuration, that is there exists a plane sector $\cK$ with vertex $\bfz$ so that
\begin{equation}
\label{4E1}
   \cD = \cK\times\R
   \quad\mbox{and}\quad
   \bx = (x_1,x_2,x_3) = (\bx_\perp,x_3)\in \cD \ \ \Leftrightarrow\ \
   \bx_\perp\in \cK,\  x_3\in\R.
\end{equation}
The edge $\be$ of the dihedral domain $\cD$ is the line $x_1=x_2=0$.

Let $\cV$ be any subdomain of $\cD$. We consider the system of local interior and boundary equations
\begin{equation}
\label{4Ebvp}
   \left\{ \begin{array}{rclll}
   L\,\bu &=& \bff \quad & \mbox{in}\ \ \cD\cap\cV, \\[0.3ex]
   T_i\,\bu &=& 0  & \mbox{on}\ \ (\Gamma_i\times\R)\cap\ov\cV, &i=1,2, \\[0.3ex]
   D_i\,\bu &=& 0  & \mbox{on}\ \ (\Gamma_i\times\R)\cap\ov\cV, &i=1,2,
   \end{array}\right.
\end{equation}
where the operators $L$, $T_i$ and $D_i$ are homogeneous with constant coefficients and form an elliptic system.
The system \eqref{4Ebvp} is the localization to $\cV$ of the elliptic boundary value problem $L\bu=\bff$ in $\cD$, with zero boundary conditions on $\Gamma_1\times\R$ and $\Gamma_2\times\R$.

%%%%%%%%%%%%%%%%%%%%%%%%%%%%%%%%%%%%%%%%%%%%%%%%%%%%%%%%%%%%%%%%%%%%%%%%%%%%%%%%%%%%%%%%
\subsection{Isotropic estimates: natural regularity shift}
%%%%%%%%%%%%%%%%%%%%%%%%%%%%%%%%%%%%%%%%%%%%%%%%%%%%%%%%%%%%%%%%%%%%%%%%%%%%%%%%%%%%%%%%
The weighted spaces for the dihedron are defined by the same formulas as in the case of a plane sector:

\begin{definition}
\label{4DKJ}
Let $\beta$ be a real number and let $m\ge0$ be an integer. Let $\cW\subset\cD$.
\begin{itemize}
\item[]
 The \emph{isotropic weighted spaces} $\rK^m_\beta(\cW)$  and $\rJ^m_\beta(\cW)$ are defined,
 with the distance  $r:=|\bx_\perp|=\sqrt{x_1^2+x_2^2}$ to the edge $\be$, by
\begin{gather*}
%\label{4EK1}
   \rK^m_\beta (\cW) = \big\{ u \in \rL_\loc^2(\cW) \ : \
   r^{\beta +  |\alpha|}  \partial^{\alpha}_\bx u  \in  \rL^2(\cW), \quad
   \forall\alpha, \; |\alpha| \leq m \big\}\\
%\label{4EJ1}
   \rJ^m_\beta (\cW) = \big\{ u \in \rL_\loc^2(\cW) \ : \
   r^{\beta +  m}  \partial^{\alpha}_\bx u  \in  \rL^2(\cW), \quad
   \forall\alpha, \; |\alpha| \leq m \big\}
\end{gather*}
endowed with their natural semi-norms and norms.
Recall that $\partial^\alpha_\bx$ denotes the derivative with respect to the three variables $x_1$, $x_2$, $x_3$.
\end{itemize}
\end{definition}

We call these spaces {\em isotropic}, in opposition with the {\em anisotropic} spaces $\rM^m_\beta(\cW)$ and $\rN^n_\beta(\cW)$ which will be introduced in the next subsection.

We gather in one statement the results concerning the $\rK$ and the $\rJ$ spaces. Here we set
\begin{equation}
\label{4E2}
\begin{aligned}
   \cW & = \big( \cK\cap \cB(\bfz,1) \big) \times (-1,1)
  \\
   \cW_\varepsilon & =
   \big( \cK\cap \cB(\bfz,1+\varepsilon) \big) \times (-1-\varepsilon,1+\varepsilon),
   \quad\varepsilon>0.
\end{aligned}
\end{equation}

\begin{theorem}
\label{4T1}
Let $\beta\in\R$ and $n\in\N$.
Let $\bu\in\bH^2_{\loc}(\ov \cW{}_\varepsilon\setminus\be)$ be a solution of problem \eqref{4Ebvp} with $\cV=\cW_\varepsilon$.
\iti1 If $\bu\in\bK^1_{\beta}(\cW_\varepsilon)$ and $\bff\in\bK^{n}_{\beta+2}(\cW_\varepsilon)$ then $\bu\in\bK^{n+2}_{\beta}(\cW)$ and there exists a constant $C\ge1$ independent of $\bu$ and $n$ such that for any integer $k$, $0\le k\le n+2$, we have
\begin{multline}
\label{4EK4}
   \frac{1}{k!}\;
   \Big(\sum_{|\alpha|=k}
   \Norm{r^{\beta+|\alpha|}\partial^\alpha_\bx \bu}{0;\,\cW}^2\Big)^{\frac12} \le
    C^{k+1} \Big\{
   \sum_{\ell=0}^{k-2} \frac{1}{\ell!}\;
   \Big(\sum_{|\alpha|=\ell}\Norm{r^{\beta+2+|\alpha|}
   \partial^\alpha_\bx \bff}{0;\,\cW_\varepsilon}^2\Big)^{\frac12}
   \\
   +
   \sum_{|\alpha|\le1} \Norm{r^{\beta+|\alpha|}
   \partial^\alpha_{\bx}\bu}{0;\,\cW_\varepsilon}\Big\}.
\end{multline}
\iti2   Let $m\ge1$ be an integer such that $m+1\ge-\beta$. Let $n\ge m-1$ be another integer.   If $\bu\in\bJ^m_{\beta}(\cW_\varepsilon)$ and $\bff\in\bJ^{n}_{\beta+2}(\cW_\varepsilon)$, then $\bu\in\bJ^{n+2}_{\beta}(\cW)$
and there exists a constant $C\ge1$ independent of $\bu$ and $n$ such that for any integer $k$,
$m\le k\le n+2$, we have
\begin{multline}
\label{4EJ4}
   \frac{1}{k!}\;
   \Big(\sum_{|\alpha|=k}
   \Norm{r^{\beta+|\alpha|}\partial^\alpha_\bx \bu}{0;\,\cW}^2\Big)^{\frac12} \le
    C^{k+1} \Big\{
   \sum_{\ell=m-1}^{k-2} \frac{1}{\ell!}\;
   \Big(\sum_{|\alpha|=\ell}\Norm{r^{\beta+2+|\alpha|}
   \partial^\alpha_\bx \bff}{0;\,\cW_\varepsilon}^2\Big)^{\frac12}
   \\
   +
   \sum_{|\alpha|=m} \Norm{r^{\beta+|\alpha|}
   \partial^\alpha_{\bx}\bu}{0;\,\cW_\varepsilon}\Big\}.
\end{multline}
\end{theorem}

\begin{proof}
Like in the case of Theorems \REF{2T1} and \REF{2T2}, the proof relies on a
locally finite dyadic covering of $\cW$ and $\cW_\varepsilon$. The
reference domains are now
\[
\begin{aligned}
   \widehat \cV &= \{\bx_\perp\in \cK\ : \ \tfrac14< |\bx_\perp| < 1\}
   \times (-\tfrac12,\tfrac12)
   \\
   \widehat \cV\ee' &=
   \{\bx_\perp\in \cK\ : \ \tfrac14-\varepsilon< |\bx_\perp| < 1+\varepsilon\}
   \times (-\tfrac12-\varepsilon,\tfrac12+\varepsilon)
\end{aligned}
\]
and for $\mu\in\N$ and $\nu\in\Z$:
\[
   \cV_{\mu,\nu} = 2^{-\mu}\big(\widehat \cV + (0,0,\tfrac\nu2) \big)
   \quad\mbox{and}\quad \cV\prm_{\mu,\nu} = 2^{-\mu}\big(\widehat \cV\ee' + (0,0,\tfrac\nu2) \big).
\]
We check immediately that
\[
   \cW = \bigcup_{\mu\in\N} \bigcup_{|\nu|< 2^{\mu+1}} \cV_{\mu,\nu}
   \quad\mbox{and}\quad
   \cW_\varepsilon \supset \bigcup_{\mu\in\N} \bigcup_{|\nu|< 2^{\mu+1}} \cV\prm_{\mu,\nu}\,,
\]
and that these coverings are locally finite. An a priori estimate between $\cV_{\mu,\nu}$ and $\cV\prm_{\mu,\nu}$ is deduced from a reference a priori estimate between $\widehat \cV$ and $\widehat \cV\ee'$ by the change of variables $\hat\bx\rightarrow\bx=2^{-\mu}(\hat\bx+(0,0,\frac\nu2))$ that maps $\widehat\cV$ onto $\cV_{\mu,\nu}$ and $\cV\ee'$ onto $\widehat\cV\prm_{\mu,\nu}$.
Here we use the fact that the operators $L$, $T$ and $D$ are homogeneous with constant coefficients. Then the rest of the proof goes exactly as in the case of the plane sectors.
\end{proof}

%%%%%%%%%%%%%%%%%%%%%%%%%%%%%%%%%%%%%%%%%%%%%%%%%%%%%%%%%%%%%%%%%%%%%%%%%%%%%%%%%%%%%%%%
\subsection{Tangential regularity along the edge (homogeneous norms)}
%%%%%%%%%%%%%%%%%%%%%%%%%%%%%%%%%%%%%%%%%%%%%%%%%%%%%%%%%%%%%%%%%%%%%%%%%%%%%%%%%%%%%%%%
The result in the previous sections only rely on the ellipticity of the boundary value problem under consideration. Now we will require a stronger condition, which is a local Peetre-type a priori estimate in an edge neighborhood. From this condition we will derive analytic type estimates for all derivatives $\partial^{j}_{x_3}$ in the direction of the edge.

\begin{assumption}
\label{4GA}
Let $\beta\in\R$. Let $\cW$ and $\cW\ee'=\cW_\varepsilon$ be the domains defined in \eqref{4E2} for some $\varepsilon>0$. We assume that the following a priori estimate holds for problem \eqref{4Ebvp} on $\cV=\cW\ee'$:
There is a constant $C$ such that
any
\[
   \bu\in \bK^2_{\beta}(\cW)\,,
\]
solution of problem \eqref{4Ebvp} with $\bff\in\bK^0_{\beta+2}(\cW\ee')$,
satisfies:
\begin{equation}
\label{4GA3}
   \Norm{\bu}{\bK^{2}_{\beta}(\cW)} \leq
   C\Big(\Norm{ \bff}{\bK^0_{\beta+2}(\cW\ee')}
   + \Norm{ \bu}{\bK^1_{\beta+1}(\cW\ee')}\Big).
\end{equation}
\end{assumption}

\begin{remark}
\label{4R1}
\iti1 Assumption \REF{4GA} is independent of $\varepsilon$ (although the constant $C$ depends on it), and more generally independent of the choice of the domains $\cW$ and $\cW\ee'$, if they satisfy the following conditions: There exists a ball with center on the edge $\be$ contained in $\cW$, and $\cW\ee'$ contains $\ov\cW\cap\cD$.
\iti2 The inequality \eqref{4GA3} is a Peetre-type estimate, since
$\bK^2_{\beta}(\cW)$ is compactly embedded in $\bK^1_{\beta+1}(\cW)$.
\iti3 As a consequence of Theorem \REF{4T1}, it is equivalent to postulate the estimate
\[
   \Norm{\bu}{\bK^{1}_{\beta}(\cW)} \leq
   C\Big(\Norm{ \bff}{\bK^0_{\beta+2}(\cW\ee')}
   + \Norm{ \bu}{\bK^1_{\beta+1}(\cW\ee')}\Big)
\]
for all $\bu\in\bK^1_{\beta}(\cW\ee') \cap \bH^2_{\loc}(\ov \cW{}\ee'\setminus\be)$
\end{remark}

\begin{remark}
\label{4RGA}
Assumption \REF{4GA} can be characterized by a condition on the partial Fourier symbol of $L$ along the edge, as follows. If we write the system $L$ in the form $L(\partial_{\bx_\perp},\partial_{x_3})$, its partial Fourier symbol $\hat L(\xi)$ is defined on the sector $\cK$ for all $\xi\in\R$ by 
\begin{equation}
\label{4Esymb}
   \hat L(\xi)(\partial_{\bx_\perp}) = L(\partial_{\bx_\perp},i\xi),\quad \bx_\perp\in\cK.
\end{equation} 
We define $\hat T_i$ and $\hat D_i$ on the same way. We also need the weighted spaces on $\cK$
\begin{equation}
\label{4EE}
   \rE^m_\beta (\cK) = \big\{ u \in \rL_\loc^2(\cK) \ : \
   \max\{ r^{\beta +  |\alpha|},  r^{\beta + m}\}\, \partial^{\alpha}_\bx u  \in  \rL^2(\cK), \quad
   \forall\alpha, \; |\alpha| \leq m \big\}.
\end{equation}
Then Assumption \REF{4GA} holds if (and only if) the problem
\begin{equation}
\label{4Ebvpxi}
   \left\{ \begin{array}{rclll}
   \hat L(\xi)\,\bu &=& \bff \quad & \mbox{in}\ \ \cK \\[0.3ex]
   \hat T_i(\xi)\,\bu &=& 0  & \mbox{on}\ \ \Gamma_i, &i=1,2, \\[0.3ex]
   \hat D_i(\xi)\,\bu &=& 0  & \mbox{on}\ \ \Gamma_i, &i=1,2,
   \end{array}\right.
\end{equation}
defines a injective operator with closed range from $\bE^2_\beta(\cK)$ into $\bE^0_{\beta+2}(\cK)$ for $\xi=\pm1$.
In \cite{MazyaPlamenevskii80b}, Maz'ya and Plamenevskii introduced the spaces $\bE^2_\beta(\cK)$ and proved that isomorphism properties of the transversal problem \eqref{4Ebvpxi} are necessary and sufficient for Fredholm properties of the boundary value problem \eqref{4Ebvp} on the wedge. The same technique proves that left invertibility of the transversal problem \eqref{4Ebvpxi} implies the semi-Fredholm estimate \eqref{4GA3} of Assumption~\ref{4GA}.
\end{remark}

The first step for higher order estimates is the $\rho$-estimate for which we control the dependence of the constant $C$ in \eqref{4GA3} on the ``distance'' between $\cW$ and $\cW'$.

\begin{lemma}
\label{4L2} Under Assumption \REF{4GA}, let $R\in[0,\varepsilon)$
and $\rho\in(0,\varepsilon-R]$. Assume that $\bu\in
\bK^2_{\beta}(\cW_R)$ is a solution of problem \eqref{4Ebvp} with
$\bff\in\bK^0_{\beta+2}(\cV)$ for $\cV=\cW_{R+\rho}\,$. There
exists a constant $C$ independent of $\bu$, $R$ and $\rho$ such that
\begin{equation}
   \label{4E6a}
   \Norm{\bu}{\bK^{2}_{\beta}(\cW_R)} \leq
   C\Big(\Norm{ \bff}{\bK^0_{\beta+2}(\cW_{R+\rho})}
   + \rho^{-1} \Norm{ \bu}{\bK^1_{\beta+1}(\cW_{R+\rho})}
   + \rho^{-2} \Norm{ \bu}{\bK^0_{\beta+2}(\cW_{R+\rho})}\Big).
\end{equation}
\end{lemma}

\begin{proof}
We introduce a special family of cut-off functions $\chi_{\rho}$.
Let $\hat\chi\in\Cinf(\R)$ be such that $\hat\chi\equiv1$ on
$(-\infty,0]$ and $\hat\chi\equiv0$ on $[1,+\infty)$. Define $\hat\chi_{\rho}$ on $\R$ by:
\begin{equation}
\label{4Ecut1}
   \hat\chi_{\rho}(t) = \hat\chi\left(\frac{|t|-1-R}{\rho}\right).
\end{equation}
Thus $\hat\chi_{\rho}$ equals $1$ in $[-1-R,1+R]$ and $0$ outside $(-1-R-\rho,1+R+\rho)$. Then we set
\begin{equation}
\label{4Ecut3}
   \chi_{\rho}(\bx) = \hat\chi_{\rho}(|\bx_\perp|)\,\hat\chi_{\rho}(x_3).
\end{equation}
Thus by construction, {\em cf.}\ \eqref{4E2}
\[
   \chi_{\rho} \equiv 1 \ \ \mbox{on}\ \  \cW_R
   \quad\mbox{and}\quad
   \chi_{\rho} \equiv 0 \ \ \mbox{outside}\ \  \cW_{R+\rho}.
\]
We note the following important bound on the derivatives of $\chi_{\rho}$
\begin{equation}
\label{4Ecut}
   \exists D>0, \quad
   \forall\rho>0, \; \forall\alpha,\,|\alpha|\le2,\quad
   |\partial^\alpha_\bx\chi_{\rho}| \le D \rho^{-|\alpha|}.
\end{equation}
Then in order to prove \eqref{4E6a}, it suffices to apply estimate \eqref{4GA3} to $\chi_{\rho}\bu$ and to check that the commutator $[L,\chi_{\rho}]$ applied to $\bu$ satisfies
\begin{equation}
\label{4E6c}
   \Norm{[L,\chi_{\rho}]\bu}{\bK^0_{\beta+2}(\cW_{R+\rho})} \leq
   C\Big(\rho^{-1} \Norm{ \bu}{\bK^1_{\beta+1}(\cW_{R+\rho})}
   + \rho^{-2} \Norm{ \bu}{\bK^0_{\beta+2}(\cW_{R+\rho})}\Big).
\end{equation}
The latter estimate is an obvious consequence of \eqref{4Ecut} and the fact that
\[
   \Norm{\partial^\alpha_\bx\bu}{\bK^0_{\beta+2}(\cW_{R+\rho})} \le
   \Norm{ \bu}{\bK^{2-|\alpha|}_{\beta+|\alpha|}(\cW_{R+\rho})}
\]
for all $\alpha$, $|\alpha|\le1$.
\end{proof}

\begin{corollary}
\label{4L2coro}
Under the assumptions of Lemma \REF{4L2}, if
$\partial_{x_3} \bff\in \bK^0_{\beta+2}(\cW_{R+\rho})$, then
$\partial_{x_3}\bu\in\bK^{2}_{\beta}(\cW_R)$ and there
exists a constant $C\ge1$ independent of $R$, $\rho$ and $\bu$ such that
\begin{equation} \label{4E6}
   \Norm{\partial_{x_3} \bu}{\bK^{2}_{\beta}(\cW_R)} \leq
   C\Big(\Norm{ \partial_{x_3} \bff}{\bK^0_{\beta+2}(\cW_{R+\rho})}
   + \rho^{-1} \Norm{ \bu}{\bK^2_{\beta}(\cW_{R+\rho})}
   + \rho^{-2} \Norm{ \bu}{\bK^1_{\beta+1}(\cW_{R+\rho})}\Big). \!\!\!\!
\end{equation}
\end{corollary}

\begin{proof}
For any $h<\rho/2$, we apply \eqref{4E6a} in $\cW_{R+\rho/2}$ to
$\bv_h$ defined by
\[
\bv_h:\bx\to h^{-1}(\bu(\bx+h e_3)-\bu(\bx)),
\]
where $e_3=(0,0,1)$. This yields
\begin{equation}
\label{4E(1)}
\begin{aligned}
\Norm{\bv_h}{\bK^{2}_{\beta}(\cW_R)} \leq
   4C&\Big(\Norm{ L\bv_h}{\bK^0_{\beta+2}(\cW_{R+\rho/2})} \\
   &+ \rho^{-1} \Norm{ \bv_h}{\bK^1_{\beta+1}(\cW_{R+\rho/2})}
   + \rho^{-2} \Norm{\bv_h}{\bK^0_{\beta+2}(\cW_{R+\rho/2})}\Big),\quad
\end{aligned}
\end{equation}
where $C$ is the positive constant from Lemma \REF{4L2}. By
noticing that
\[
\bv_h= h^{-1}\int_0^h \partial_{x_3}\bu(\bx+t e_3)\, \rd t,
\]
we check that for all $h<\rho/2$
\begin{eqnarray*} \Norm{
L\bv_h}{\bK^0_{\beta+2}(\cW_{R+\rho/2})}&\leq& \Norm{
\partial_{x_3} L\bu}{\bK^0_{\beta+2}(\cW_{R+\rho})},\\
 \Norm{ \bv_h}{\bK^1_{\beta+1}(\cW_{R+\rho/2})} &\leq& \Norm{
\partial_{x_3} \bu}{\bK^1_{\beta+1}(\cW_{R+\rho})}
\ \le\
   \Norm{ \bu}{\bK^2_{\beta}(\cW_{R+\rho})},
   \\
 \Norm{ \bv_h}{\bK^0_{\beta+2}(\cW_{R+\rho/2})} &\leq& \Norm{
\partial_{x_3} \bu}{\bK^0_{\beta+2}(\cW_{R+\rho})}
\ \le\
   \Norm{ \bu}{\bK^1_{\beta+1}(\cW_{R+\rho})}.
\end{eqnarray*}
This shows that the right-hand side of \eqref{4E(1)} is bounded uniformly in
$h$. Therefore passing to the limit in \eqref{4E(1)}, we find that
$\partial_{x_3}\bu$ belongs to $\bK^{2}_{\beta}(\cW_R)$ and that
\eqref{4E6} holds.
\end{proof}

\begin{corollary}  \label{4C2}
Under Assumption \REF{4GA}, let $\bu\in \bK^2_{\beta}(\cW_\varepsilon)$ be a solution of \eqref{4Ebvp}.
Let $R\in[0,\varepsilon/2]$ and $R'\ge \varepsilon/2$ with $R+R'\le\varepsilon$.
Then there exists a constant $C$ independent of $R$, $R'$ and $\bu$ such
that for all $\ell\in \N$, we have
\begin{equation}
   \label{4E7new}
   \frac{1}{ \ell!} \Norm{
   \partial_{x_3}^\ell \bu}{\bK^2_{\beta}(\cW_R)}
   \leq C^{\ell+1} \Big\{ \sum_{j=0}^{\ell}\frac{1}{j!}
   \Norm{\partial_{x_3}^j L\bu}{\bK^0_{\beta+2}(\cW_{R+R'})}
   + \Norm{\bu}{\bK^1_{\beta+1}(\cW_{R+R'})} \Big\}.
\end{equation}
\end{corollary}

\begin{proof}
If $\ell=0$, this is a consequence of estimate \eqref{4GA3}. For $\ell\ge1$ the proof is divided into two steps. To keep notations simpler we take $R=0$.

\smallskip
\iti1 We first prove by induction on $\ell$ that if $\rho\leq \varepsilon/(2\ell-1)$, then
\begin{multline}
   \label{4E7}
    \Norm{\partial_{x_3}^\ell \bu}{\bK^2_{\beta}(\cW)}
   \leq (2C)^{\ell} \Big\{ \sum_{j=1}^{\ell}\rho^{-(\ell-j)} \Norm{\partial_{x_3}^j
   L\bu}{\bK^0_{\beta+2}(\cW_{(2\ell-j)\rho})} \\
   +\rho^{-\ell} \Norm{\bu}{\bK^2_{\beta}(\cW_{(2\ell-1)\rho})}
   +\rho^{-\ell-1} \Norm{\bu}{\bK^1_{\beta+1}(\cW_{(2\ell-1)\rho})}
   \Big\},
\end{multline}
where $C\ge1$ is the constant from Corollary  \REF{4L2coro}.
\\
$\bullet$ \ If $\ell=1$, the estimate \eqref{4E7} is nothing else than
\eqref{4E6}. Hence it suffices to show that if \eqref{4E7} holds
for $\ell$, it holds for $\ell+1$.
\\
$\bullet$ \  For that purpose, we first apply
\eqref{4E7} to $\bv_h$ defined as before by
\[
\bv_h:\bx\to h^{-1}(\bu(\bx+h e_3)-\bu(\bx)),
\]
and passing to the limit in $h$, we get
\begin{multline*}
    \Norm{\partial_{x_3}^{\ell+1} \bu}{\bK^2_{\beta}(\cW)}
   \leq (2C)^{\ell} \Big\{ \sum_{j=1}^{\ell}\rho^{-(\ell-j)}
   \Norm{\partial_{x_3}^{j+1}
   L\bu}{\bK^0_{\beta+2}(\cW_{(2\ell+1-j)\rho})} \\
   +\rho^{-\ell} \Norm{\partial_{x_3}\bu}{\bK^2_{\beta}(\cW_{2\ell\rho})}
   +\rho^{-\ell-1} \Norm{\partial_{x_3}\bu}{\bK^1_{\beta+1}(\cW_{2\ell\rho})}
   \Big\}.
\end{multline*}
For the second term of this right-hand side, we apply \eqref{4E6}
to $\bu$ but between $\cW_{2\ell\rho}$ and $\cW_{(2\ell+1)\rho}$, while
for the third term we  use the fact that $
\Norm{\partial_{x_3}\bu}{\bK^1_{\beta+1}(\cW_{2\ell\rho})}\leq
 \Norm{\bu}{\bK^2_{\beta}(\cW_{2\ell\rho})}$. This leads to
 \begin{multline*}
    \Norm{\partial_{x_3}^{\ell+1} \bu}{\bK^2_{\beta}(\cW)}
   \leq (2C)^{\ell} \sum_{j=1}^{\ell}\rho^{-(\ell-j)}
   \Norm{\partial_{x_3}^{j+1}
   L\bu}{\bK^0_{\beta+2}(\cW_{(2\ell+1-j)\rho})} \\
   +(2C)^{\ell} C \rho^{-\ell}
   \Big(\Norm{ \partial_{x_3} L\bu}{\bK^0_{\beta+2}(\cW_{(2\ell+1)\rho})}
   + \rho^{-1} \Norm{ \bu}{\bK^2_{\beta}(\cW_{(2\ell+1)\rho})}
   + \rho^{-2} \Norm{ \bu}{\bK^1_{\beta+1}(\cW_{(2\ell+1)\rho})}\Big)
 \\
   +(2C)^{\ell} \rho^{-\ell-1} \Norm{\bu}{\bK^2_{\beta}(\cW_{2\ell\rho})}.
\end{multline*}
By the change of index $j'=j+1$ in the sum on $j$, we finally get
(since $(2C)^{\ell}\leq  2^{\ell} C^{\ell+1}$)
\begin{multline*}
    \Norm{\partial_{x_3}^{\ell+1} \bu}{\bK^2_{\beta}(\cW)}
   \leq   2^{\ell} C^{\ell+1}   \sum_{j=1}^{\ell+1}\rho^{-(\ell+1-j)}
   \Norm{\partial_{x_3}^{j}
   L\bu}{\bK^0_{\beta+2}(\cW_{(2(\ell+1)-j)\rho})} \\
   +(2C)^{\ell} (C+1) \rho^{-\ell-1}
    \Norm{ \bu}{\bK^2_{\beta}(\cW_{(2\ell+1)\rho})}
   +(2C)^{\ell} C  \rho^{-\ell-2} \Norm{ \bu}{\bK^1_{\beta+1}(\cW_{(2\ell+1)\rho})}.
\end{multline*}
Since $C\ge1$, $C+1\le2C$, and this proves that \eqref{4E7} holds for $\ell+1$.

\smallskip
\iti2 Now we choose $\rho$ such that
\[
\cW_{(2\ell-1)\rho}\subset \cW_{\varepsilon'} \quad\mbox{with}\quad \varepsilon'=\varepsilon/2.
\]
This holds if we take
\[
   \rho = \frac\gamma\ell \quad\mbox{with}\quad
   \gamma = \min\{\frac{\varepsilon}{4},1\} \,.
\]
Hence applying \eqref{4E7} with this choice of $\rho$, we obtain for all $\ell\ge1$
\begin{equation}
\label{4E(4)}
\begin{aligned}
    \Norm{\partial_{x_3}^\ell \bu}{\bK^2_{\beta}(\cW)}
   \leq & (2C)^{\ell} \Big\{ \sum_{j=1}^{\ell}(\gamma^{-1})^{(\ell-j)} \ell^{\ell-j}
   \Norm{\partial_{x_3}^j
   L\bu}{\bK^0_{\beta+2}(\cW_{\varepsilon'})} \\
   &+\ (\gamma^{-1})^{\ell} \ell^{\ell} \Norm{\bu}{\bK^2_{\beta}(\cW_{{\varepsilon'}})}
   +\ (\gamma^{-1})^{\ell+1} \ell^{\ell+1} \Norm{\bu}{\bK^1_{\beta+1}(\cW_{{\varepsilon'}})}
   \Big\}.
\end{aligned}
\end{equation}
Since $\gamma\leq 1$,
$(\gamma^{-1})^{(\ell-j)}\leq (\gamma^{-1})^{\ell}$. Moreover by
Stirling's formula, one has
\[
\ell^{\ell}\leq S^{\ell} \ell\,!
\]
for some $S> 1$. We find
\[
   \frac{\ell^{\ell-j} j!}{\ell!} = \frac{\ell^\ell j! }{\ell! \ell^j}
   \leq  \frac{S^\ell  j!}{ \ell^j} \leq S^\ell\,,
\]
since  $j!\leq j^j\leq \ell^j$.
Inserting this into \eqref{4E(4)} gives, with $C_1=2C\gamma^{-1}S$,
\[
 \Norm{
   \partial_{x_3}^\ell \bu}{\bK^2_{\beta}(\cW)}
   \leq C_1^\ell \Big\{ \sum_{j=1}^{\ell} \frac{\ell!}{j!}\,
   \Norm{\partial_{x_3}^j L\bu}{\bK^0_{\beta+2}(\cW_{\varepsilon'})}
   + \ell!\, \Norm{\bu}{\bK^2_{\beta}(\cW_{\varepsilon'})}
   +\gamma^{-1} \ell\,\ell!\,\Norm{\bu}{\bK^1_{\beta+1}(\cW_{{\varepsilon'}})} \Big\}.
   \]
Using the trivial inequality $\ell\le 2^\ell$, we arrive at
\[
 \Norm{
   \partial_{x_3}^\ell \bu}{\bK^2_{\beta}(\cW)}
   \leq C_2^\ell \Big\{ \sum_{j=1}^{\ell} \frac{\ell!}{j!}\,
   \Norm{\partial_{x_3}^j L\bu}{\bK^0_{\beta+2}(\cW_{\varepsilon'})}
   + \ell!\, \Norm{\bu}{\bK^2_{\beta}(\cW_{\varepsilon'})}
   +\ell!\,\Norm{\bu}{\bK^1_{\beta+1}(\cW_{{\varepsilon'}})} \Big\},
   \]
which, combined with \eqref{4E6a} between $\cW_{\varepsilon'}$ and $\cW_\varepsilon\,$, yields the requested estimate.
\end{proof}

%%%%%%%%%%%%%%%%%%%%%%%%%%%%%%%%%%%%%%%%%%%%%%%%%%%%%%%%%%%%%%%%%%%%%%%%%%%%%%%%%%%%%%%%
\subsection{Anisotropic estimates in dihedral domains (homogeneous norms)}
%%%%%%%%%%%%%%%%%%%%%%%%%%%%%%%%%%%%%%%%%%%%%%%%%%%%%%%%%%%%%%%%%%%%%%%%%%%%%%%%%%%%%%%%
We are now ready to prove the main results of this section, namely
the weighted anisotropic regularity of solutions of our local boundary value problem \eqref{4Ebvp}. For this we introduce the following new class of weighted spaces:

\begin{definition}
\label{4DM}
Let $\beta$ be a real number and let $m\ge0$ be an integer.
\begin{itemize}
\item[] Let $\cW$ be a subdomain of the dihedral domain $\cD$.
We recall that $r=|\bx_\perp|$ denotes the distance to the edge $\be\equiv \{\bx_\perp=0\}$.
The \emph{anisotropic weighted space with homogeneous norm} $\rM^m_\beta(\cW)$ is defined by
\begin{equation}
\label{4EM1}
   \rM^m_\beta (\cW) = \big\{ u \in \rL_\loc^2(\cW) \ : \
   r^{\beta +  |\alpha_\perp|}  \partial^{\alpha}_\bx u  \in  \rL^2(\cW), \quad
   \forall\alpha, \; |\alpha| \leq m \big\}
\end{equation}
where  for $\alpha=(\alpha_1,\alpha_2,\alpha_3)$,
$\alpha_\perp=(\alpha_1,\alpha_2)$ is the component of $\alpha$ in
the direction perpendicular to the edge $\be$.
The norm of this space is defined as
\begin{equation}
\label{4EM2}
   \Norm{u}{\rM^m_\beta (\cW)}^2 = \sum_{k=0}^m  \!\sum_{|\alpha|=k}\!
   \Norm{r^{\beta +  |\alpha_\perp|}  \partial^\alpha_\bx u}{0;\, \cW}^2 .
\end{equation}
\end{itemize}
\end{definition}

\begin{theorem}
\label{4T3}
 Let $\beta\in\R$.   Under Assumption \REF{4GA}, let $\bu\in \bK^1_{\beta}(\cW_\varepsilon)$ be a solution of problem \eqref{4Ebvp}.
If $\bff\in \bM^{n}_{\beta+2}(\cW_\varepsilon)$, then
$\bu\in\bM^{n}_{\beta}(\cW)$, and there exists a positive constant $C$ independent of $\bu$ and $n$ such that for all integer $k$, $0\le k\le n$ we have
\begin{multline}
\label{4E7c}
   \frac{1}{k!}\;
   \Big(\sum_{|\alpha|=k}\Norm{r^{\beta+|\alpha_\perp|}
   \partial^\alpha_\bx \bu}{0;\,\cW}^2\Big)^{\frac12} \le
    C^{k+1} \Big\{
   \sum_{\ell=0}^{k} \frac{1}{\ell!}\;
   \Big(\sum_{|\alpha|=\ell}\Norm{r^{\beta+2+|\alpha_\perp|} \partial^\alpha_\bx \bff}{0;
   \,\cW_\varepsilon}^2\Big)^{\frac12}
   \\
   + \Norm{\bu}{\bK^1_{\beta+1}(\cW_\varepsilon)}\Big\}.
\end{multline}
\end{theorem}

\begin{proof}
\itj1 We first  apply the isotropic estimate \eqref{4EK4} between $\cW$ and $\cW_{\varepsilon/4}$, and combine with \eqref{4GA3} between $\cW_{\varepsilon/4}$ and $\cW_{\varepsilon/2}$ ({\em cf.}\ Remark \REF{4R1} \itj1). This yields the estimate for all $k$, $0\le k\le n$, and with $\varepsilon'=\varepsilon/2$
\begin{multline*}
   \frac{1}{k!}\;
   \Big(\sum_{|\alpha|=k}
   \Norm{r^{\beta+|\alpha|}\partial^\alpha_\bx \bu}{0;\,\cW}^2\Big)^{\frac12} \le
    C^{k+1} \Big\{
   \sum_{\ell=0}^{k-2} \frac{1}{\ell!}\;
   \Big(\sum_{|\alpha|=\ell}\Norm{r^{\beta+2+|\alpha|}
   \partial^\alpha_\bx \bff}{0;\,\cW_{\varepsilon'}}^2\Big)^{\frac12}
   \\
   + \Norm{r^{\beta+2}\bff}{0;\,\cW_{\varepsilon'}} +
   \Norm{\bu}{\bK^1_{\beta+1}(\cW_{\varepsilon'})}\Big\}.
\end{multline*}
In order to absorb the term $\Norm{r^{\beta+2}\bff}{0;\,\cW_{\varepsilon'}}$ in the sum on the right-hand side (including when $k=0$ or $1$), we write the previous inequality in the slightly weaker form
\begin{multline*}
   \frac{1}{k!}\;
   \Big(\sum_{|\alpha|=k}
   \Norm{r^{\beta+|\alpha|}\partial^\alpha_\bx \bu}{0;\,\cW}^2\Big)^{\frac12} \le
    C_1^{k+1} \Big\{
   \sum_{\ell=0}^{k} \frac{1}{\ell!}\;
   \Big(\sum_{|\alpha|=\ell}\Norm{r^{\beta+2+|\alpha|}
   \partial^\alpha_\bx \bff}{0;\,\cW_{\varepsilon'}}^2\Big)^{\frac12}
   \\
   +
   \Norm{\bu}{\bK^1_{\beta+1}(\cW_{\varepsilon'})}\Big\}.
\end{multline*}
We reduce the left-hand side to any $\alpha=(\alpha_\perp,0)$ of length
$q\geq 0$, and bound $r^{\beta+2+|\alpha|}$ by $r^{\beta+2+|\alpha_\perp|}$ in the right-hand side (recall that $r$ is bounded in $\cW_\varepsilon$) to obtain for all $q$, $0\le q\le n$
\begin{equation}
\label{4EK5}
\begin{aligned}
   \frac{1}{q!}\,
   \Big(\sum_{|\alpha_\perp|=q}\Norm{r^{\beta+|\alpha_\perp|}
   & \partial^{\alpha_\perp}_{\bx_\perp} \bu}{0;\,\cW}^2\Big)^{\frac12} \le
    C_2^{q+1}  \Big\{ \\
   \sum_{\ell=0}^{q} & \ \frac{1}{\ell!}\,
   \Big(\sum_{|\alpha|=\ell}\Norm{r^{\beta+2+|\alpha_\perp|} \partial^\alpha_\bx \bff}{0;
   \,\cW_{\varepsilon'}}^2\Big)^{\frac12}
   + \Norm{\bu}{\bK^1_{\beta+1}(\cW_{\varepsilon'})} \Big\}.
\end{aligned}
\end{equation}

\iti2 We now prove that for all $\mu=0,\ldots,n$ and for all $q=0,\ldots,n-\mu$
one has the following estimates with $k:=q+\mu$ and a constant $C_3$
independent of $\bu$, $q$ and $\mu$
\begin{equation}
\label{4EK6}
\begin{aligned}
   \frac{1}{k!}\;
   \Big(\sum_{|\alpha_\perp|=q}
   \Norm{r^{\beta+|\alpha_\perp|}
   & \partial^{\alpha_\perp}_{\bx_\perp}\partial_{x_3}^\mu \bu}{0;\,\cW}^2\Big)^{\frac12}
   \le C_3^{k+1} \Big\{ \\
   \sum_{\ell=0}^{k} &\ \frac{1}{\ell!}\;
   \Big(\sum_{|\alpha|=\ell}\Norm{r^{\beta+2+|\alpha_\perp|} \partial^\alpha_\bx \bff}{0;
   \,\cW_\varepsilon}^2\Big)^{\frac12}
   + \Norm{\bu}{\bK^1_{\beta+1}(\cW_\varepsilon)}\Big\}.
\end{aligned}
\end{equation}

\noindent
1. If $\mu=0$, this estimate is a consequence of \eqref{4EK5} since
 $\cW_{\varepsilon'}\subset \cW_\varepsilon$.

\noindent
2. If $\mu>0$ (or equivalently $q<k$), we apply \eqref{4EK5} to
$\partial_{x_3}^\mu \bu$ to obtain
\begin{equation}
\label{4E(3)}
 \begin{aligned}
   \frac{1}{q!}\,
   \Big(\sum_{|\alpha_\perp|=q}\Norm{r^{\beta+|\alpha_\perp|}
   & \partial^{\alpha_\perp}_{\bx_\perp} \partial_{x_3}^\mu \bu}{0;\,\cW}^2\Big)^{\frac12}
   \le C_2^{q+1}  \Big\{ \\
   \sum_{\ell=0}^{q} \frac{1}{\ell!}\,
   \Big(&\sum_{|\alpha|=\ell}\Norm{r^{\beta+2+|\alpha_\perp|}
   \partial^\alpha_\bx\partial_{x_3}^\mu  \bff}{0;
   \,\cW_{\varepsilon'}}^2\Big)^{\frac12}
   + \Norm{\partial_{x_3}^\mu \bu}{\bK^1_{\beta+1}(\cW_{\varepsilon'})} \Big\}.
\end{aligned}
\end{equation}
The last term of this right-hand side is now estimated with the help
of Corollary \REF{4C2}. Using that
\[
   \Norm{ \partial^\mu_{x_3} \bu}{\bK^1_{\beta+1}(\cW_{\varepsilon'})} \le
   \Norm{ \partial^{\mu-1}_{x_3}\bu}{\bK^2_{\beta}(\cW_{\varepsilon'})} ,
\]
and applying \eqref{4E7new} between
$\cW_{\varepsilon'}$ and $\cW_{\varepsilon}$ with $\ell=\mu-1$, we
obtain
\[
   \Norm{ \partial^\mu_{x_3} \bu}{\bK^1_{\beta+1}(\cW_{\varepsilon'})}
   \leq C_4^\mu (\mu-1)! \,\Big( \sum_{j=0}^{\mu-1} \frac{1}{j!}
   \Norm{\partial_{x_3}^j \bff}{\bK^0_{\beta+2}(\cW_\varepsilon)}
   + \Norm{\bu}{\bK^1_{\beta+1}(\cW_\varepsilon)} \Big).
\]
Using this estimate in \eqref{4E(3)} we obtain that
\begin{align*}
   \frac{1}{q!}\,
   \Big(\sum_{|\alpha_\perp|=q}\Norm{r^{\beta+|\alpha_\perp|}
   \partial^{\alpha_\perp}_{\bx_\perp}\partial_{x_3}^\mu \bu}{0;\,\cW}^2\Big)^{\frac12} \le
    C_2^{q+1}
   \sum_{\ell=0}^{q} \frac{1}{\ell!}\,
   \Big(\sum_{|\alpha|=\ell}\Norm{r^{\beta+2+|\alpha_\perp|}
   \partial_{x_3}^\mu \partial^\alpha_\bx  \bff}{0;\,\cW_\varepsilon}^2\Big)^{\frac12}
   \\
   +\ C_2^{q+1} C_4^\mu (\mu-1)! \, \Big( \sum_{j=0}^{\mu-1}\frac{1}{j!}
   \Norm{\partial_{x_3}^j \bff}{\bK^0_{\beta+2}(\cW_\varepsilon)}
   + \Norm{\bu}{\bK^1_{\beta+1}(\cW_\varepsilon)}\Big).
\end{align*}
Multiplying this estimate by ${q!}({k!})^{-1}$, we find (since $q!(\mu-1)!({k!})^{-1}\le1$)
\begin{align*}
   \frac{1}{k!}\,
   \Big(\sum_{|\alpha_\perp|=q}\Norm{r^{\beta+|\alpha_\perp|}
   \partial^{\alpha_\perp}_{\bx_\perp}\partial_{x_3}^\mu \bu}{0;\,\cW}^2\Big)^{\frac12} \le
    C_2^{q+1}
   \sum_{\ell=0}^{q} \frac{q!}{\ell!k!}\,
   \Big(\sum_{|\alpha|=\ell}\Norm{r^{\beta+2+|\alpha_\perp|}
   \partial_{x_3}^\mu \partial^\alpha_\bx  \bff}{0;\,\cW_\varepsilon}^2\Big)^{\frac12}
   \\
   +\ C_2^{q+1}C_4^\mu \Big(  \sum_{j=0}^{\mu-1}\frac{1}{j!}
   \Norm{\partial_{x_3}^j \bff}{\bK^0_{\beta+2}(\cW_\varepsilon)}
   +
   \Norm{\bu}{\bK^1_{\beta+1}(\cW_\varepsilon)}\Big).
\end{align*}
For the first term of this right-hand side we finally notice that $
\partial_{x_3}^\mu \partial^\alpha=
\partial^{\alpha+(0,0,\mu)}$
and that $|\alpha+(0,0,\mu)|=\ell+\mu$. Hence we have to check that
\[
\frac{q!}{\ell!k!}\leq \frac{1}{(\ell+\mu)!},\] which is equivalent to
\[
\frac{(\ell+\mu)!q!}{\ell!k!}\leq 1,\] and holds since $\ell+\mu\leq k$
and $q\leq k$.

Altogether we have proved that \eqref{4EK6} holds for all $\mu\in \N$ such
that $q+\mu=k$.

\smallskip
\iti3 Summing the square of this estimate \eqref{4EK6} on $q=0,\ldots, k$ and
$\mu=0,\ldots, k-q$, we arrive at
\[
\begin{aligned}
   \frac{1}{k!}\;
   \Big(\sum_{|\alpha|=k}
   \Norm{r^{\beta+|\alpha_\perp|}\partial^\alpha\bu}{0;\,\cW}^2\Big)^{\frac12} \le
    k^2C_3^{k+1} \Big(
   \sum_{\ell=0}^{k} \frac{1}{\ell!}\;
   \Big(\sum_{|\alpha|=\ell}\Norm{r^{\beta+2+|\alpha_\perp|}
   \partial^\alpha    \bff}{0;
   \,\cW_\varepsilon}^2\Big)^{\frac12}
   \\
   + \Norm{\bu}{\bK^1_{\beta+1}(\cW_\varepsilon)}\Big).
  \end{aligned}
\]
This proves the theorem.
\end{proof}

%%%%%%%%%%%%%%%%%%%%%%%%%%%%%%%%%%%%%%%%%%%%%%%%%%%%%%%%%%%%%%%%%%%%%%%%%%%%%%%%%%%%%%%%
\subsection{Anisotropic estimates in dihedral domains (non-homogeneous norms)}
%%%%%%%%%%%%%%%%%%%%%%%%%%%%%%%%%%%%%%%%%%%%%%%%%%%%%%%%%%%%%%%%%%%%%%%%%%%%%%%%%%%%%%%%
In this last part of section \ref{sec4} devoted to local estimates in dihedral domains, we investigate the situation where the a priori estimate holds in the $\rJ$-weighted scale instead the $\rK$ scale. We set:

\begin{assumption}
\label{4GB}
Let $\beta\in\R$. Let $m\ge1$ be an integer such that $m+1\ge-\beta$.
We assume that the following a priori estimate holds for problem \eqref{4Ebvp}:
There is a constant $C$ such that
any
\[
   \bu\in \bJ^{m+1}_{\beta}(\cW)\,,
\]
solution of problem \eqref{4Ebvp} in $\cV=\cW\ee'$ with $\bff\in\bJ^{m-1}_{\beta+2}(\cW\ee')$,
satisfies:
\begin{equation}
\label{4GB3}
   \Norm{\bu}{\bJ^{m+1}_{\beta}(\cW)} \leq
   C\Big(\Norm{ \bff}{\bJ^{m-1}_{\beta+2}(\cW\ee')}
   + \Norm{ \bu}{\bJ^m_{\beta+1}(\cW\ee')}\Big).
\end{equation}
\end{assumption}

\begin{remark}
\label{4R2}
Using the analogue of Proposition \REF{2P1} for dihedral domains, we obtain that in the situation of Assumption \REF{4GB} the norm in the space
$\rJ^{m+1}_\beta(\cW)$ is equivalent to
\begin{equation}
\label{4EJmax}
   \Bigl( \sum_{|\alpha|\le m+1}
   \Norm{r^{\max\{ \beta+|\alpha|,\,0 \}} \partial^\alpha_\bx u}
   {0; \,\cW}^2 \Bigr)^{\frac12}.
\end{equation}
\end{remark}

\begin{remark}
\label{4RGB}
Along the same lines as Remark \REF{4RGA}, we have a characterization of Assumption \REF{4GB} by the partial Fourier symbol of $L$: Assumption \REF{4GB} holds if (and only if) problem \eqref{4Ebvpxi} defines a injective operator with closed range from $\bJ^{m+1}_\beta(\cK)$ into $\bJ^{m-1}_{\beta+2}(\cK)$ for $\xi=\pm1$.
\end{remark}

The non-homogeneous anisotropic weighted spaces are defined as follows on the model of the homogeneous ones (Definition \REF{4DM}):

\begin{definition}
\label{4DN}
Let $\beta\in\R$. Let $n\ge1$ be a natural number such that $n\ge-\beta$.
\begin{itemize}
 \item[] Let $\cW$ be a subdomain of the dihedral domain $\cD$. 
The \emph{anisotropic weighted space with non-homogeneous norm} $\rN^n_\beta(\cW)$ is defined by
\begin{equation}
\label{4EN1}
   \rN^n_\beta (\cW) = \big\{ u \in \rL_\loc^2(\cW) \ : \
   r^{\max\{\beta+|\alpha_\perp|,\,0\}}  \partial^{\alpha}_\bx u  \in  \rL^2(\cW), \quad
   \forall\alpha, \; |\alpha| \leq n \big\}
\end{equation}
endowed with its natural norm.
\end{itemize}
\end{definition}

Our aim is to prove the ``non-homogeneous'' analogue of Theorem \REF{4T3}:

\begin{theorem}
\label{4T4}
Let $\beta\in\R$. Let $m\ge1$ be an integer such that $m+1\ge-\beta$.  
Under Assumption \REF{4GB}, let $\bu\in \bJ^m_{\beta}(\cW_\varepsilon)$ be a solution of problem \eqref{4Ebvp}.
If $\bff\in \bN^{n}_{\beta+2}(\cW_\varepsilon)$ for an integer $n>m$, then
$\bu\in\bN^{n}_{\beta}(\cW)$, and there exists a positive constant $C$ independent of $\bu$ and $n$ such that for all integer $k$, $0\le k\le n$ we have
\begin{multline}
\label{4E8}
   \frac{1}{k!}\;
   \Big(\sum_{|\alpha|=k}\Norm{r^{\max\{\beta+|\alpha_\perp|,\,0\}}
   \partial^\alpha_\bx \bu}{0;\,\cW}^2\Big)^{\frac12} \le
    C^{k+1} \Big\{ \\
   \sum_{\ell=0}^{k} \frac{1}{\ell!}\;
   \Big(\sum_{|\alpha|=\ell}
   \Norm{r^{\max\{\beta+2+|\alpha_\perp|,\,0\}} \partial^\alpha_\bx \bff}{0;
   \,\cW_\varepsilon}^2\Big)^{\frac12}
   + \Norm{\bu}{\bJ^m_{\beta+1}(\cW_\varepsilon)}\Big\}.
\end{multline}
\end{theorem}

\begin{proof}
We first notice that Theorem \REF{4T1} 
yields $\bu\in \bJ^{n+2}_{\beta}(\cW_{\varepsilon'})$ for any $\varepsilon'\in(0,\varepsilon)$.
As $n\geq m-1$, we have obtained the basic regularity $\bu\in \bJ^{m+1}_{\beta}(\cW_{\varepsilon'})$.
We review now the sequence of steps leading to Theorem \REF{4T3} and adapt them to non-homogeneous norms.  

\iti1 Applying \eqref{4GB3} to $\chi_\rho\bu$ with the function $\chi_\rho$ introduced in \eqref{4Ecut3}, we obtain, -- compare with \eqref{4E6a},
\[
   \Norm{\bu}{\bJ^{m+1}_{\beta}(\cW_R)} \leq
   C\Big(\Norm{ \bff}{\bJ^{m-1}_{\beta+2}(\cW_{R+\rho})}
   + \sum_{\lambda=0}^m \rho^{-1-\lambda}
   \Norm{ \bu}{\bJ^{m-\lambda}_{\beta+1+\lambda}(\cW_{R+\rho})}
   \Big).
\]

\iti2 By the differential quotients technique we deduce, -- compare with \eqref{4E6},
\[
   \Norm{\partial_{x_3}\bu}{\bJ^{m+1}_{\beta}(\cW_R)} \leq
   C\Big(\Norm{ \partial_{x_3}\bff}{\bJ^{m-1}_{\beta+2}(\cW_{R+\rho})}
   + \sum_{\lambda=0}^m \rho^{-1-\lambda}
   \Norm{ \bu}{\bJ^{m+1-\lambda}_{\beta+\lambda}(\cW_{R+\rho})}
   \Big),
\]
since $\Norm{\partial_{x_3}\bu}{\bJ^{m-\lambda}_{\beta+1+\lambda}(\cW_{R+\rho})}$ is bounded by $\Norm{\bu}{\bJ^{m+1-\lambda}_{\beta+\lambda}(\cW_{R+\rho})}$.

\iti3 Iterating this on the model of \eqref{4E7} we find for $\ell\ge1$
\begin{align*}
    \Norm{\partial_{x_3}^\ell \bu}{\bJ^{m+1}_{\beta}(\cW)}
   \leq (2C)^{\ell} \Big\{ &\sum_{j=1}^{\ell}\rho^{-(\ell-j)} \Norm{\partial_{x_3}^j
   L\bu}{\bJ^{m-1}_{\beta+2}(\cW_{(2\ell-j)\rho})} \\
   &+ \sum_{\lambda=0}^m  \rho^{-\ell-\lambda}
   \Norm{ \bu}{\bJ^{m+1-\lambda}_{\beta+\lambda}(\cW_{(2\ell-1)\rho})}
   \Big\},
\end{align*}
leading to the analytic type estimate,  -- compare with
\eqref{4E7new},
\begin{equation}
   \label{4E8new}
   \frac{1}{ \ell!} \Norm{
   \partial_{x_3}^\ell \bu}{\bJ^{m+1}_{\beta}(\cW_R)}
   \leq C^{\ell+1} \Big\{ \sum_{j=0}^{\ell}\frac{1}{j!}
   \Norm{\partial_{x_3}^j L\bu}{\bJ^{m-1}_{\beta+2}(\cW_{R+R'})}
   + \Norm{\bu}{\bJ^m_{\beta+1}(\cW_{R+R'})} \Big\}.
\end{equation}

\iti4 To prove \eqref{4E8}, we start with the proof of, -- compare with \eqref{4EK5},
\begin{equation}
\label{4EJ5}
\begin{aligned}
   \frac{1}{q!}\,
   \Big(\sum_{|\alpha_\perp|=q}
   &\Norm{r^{\max\{\beta+|\alpha_\perp|,\,0\}}
   \partial^{\alpha_\perp}_{\bx_\perp} \bu}{0;\,\cW}^2\Big)^{\frac12} \le
    C_2^{q+1}  \Big\{ \\
   &\sum_{\ell=0}^{q}  \ \frac{1}{\ell!}\,
   \Big(\sum_{|\alpha|=\ell}\Norm{r^{\max\{\beta+2+|\alpha_\perp|,\,0\}}
   \partial^\alpha_\bx \bff}{0;
   \,\cW_{\varepsilon'}}^2\Big)^{\frac12}
   + \Norm{\bu}{\bJ^m_{\beta+1}(\cW_{\varepsilon'})} \Big\}.
\end{aligned}
\end{equation}
$\bullet$ \ For $q=0,\ldots,m$, we rely on the estimate \eqref{4GB3} combined with the use of the norm \eqref{4EJmax} for $\bJ^{m+1}_{\beta}(\cW)$: If we restrict the left-hand side to the derivatives of the form $\partial^{\alpha_\perp}_{\bx_\perp}$ and replace the weight $r^{\max\{\beta+2+|\alpha|,\,0\}}$ by $r^{\max\{\beta+2+|\alpha_\perp|,\,0\}}$ in the right-hand side, we obtain \eqref{4EJ5}.

\noindent
$\bullet$ \ For $q\ge m+1$, we combine the estimate \eqref{4GB3} with the isotropic non-homogeneous estimate \eqref{4EJ4} and making the same restriction to $\partial^{\alpha_\perp}_{\bx_\perp}$ in the left-hand side and the same change of weights in the right-hand side.

\smallskip
\iti5 We continue with the proof that for all $\mu=0,\ldots,n$ and for all $q=0,\ldots,n-\mu$
one has the following estimates with $k:=q+\mu$ and a constant $C_3$
independent of $\bu$, $q$ and $\mu$
\begin{equation}
\label{4EJ6}
\begin{aligned}
   \frac{1}{k!}\;
   \Big(\sum_{|\alpha_\perp|=q}
   &\Norm{r^{\max\{\beta+|\alpha_\perp|,\,0\}}
   \partial^{\alpha_\perp}_{\bx_\perp}\partial_{x_3}^\mu \bu}{0;\,\cW}^2\Big)^{\frac12}
   \le C_3^{k+1} \Big\{ \\
   &\sum_{\ell=0}^{k}\ \frac{1}{\ell!}\;
   \Big(\sum_{|\alpha|=\ell}
   \Norm{r^{\max\{\beta+2+|\alpha_\perp|,\,0\}} \partial^\alpha_\bx \bff}{0;
   \,\cW_\varepsilon}^2\Big)^{\frac12}
   + \Norm{\bu}{\bJ^m_{\beta+1}(\cW_\varepsilon)}\Big\}.
\end{aligned}
\end{equation}

\noindent
1. If $\mu=0$, this estimate is a consequence of \eqref{4EJ5} since
 $\cW_{\varepsilon'}\subset \cW_\varepsilon$.

\noindent
2. If $\mu>0$ (or equivalently $q<k$), we apply \eqref{4EJ5} to
$\partial_{x_3}^\mu \bu$ to obtain
\begin{equation}
\label{4EJ(3)}
 \begin{aligned}
   \frac{1}{q!}\,
   \Big(\sum_{|\alpha_\perp|=q} & \Norm{r^{\max\{\beta+|\alpha_\perp|,\,0\}}
   \partial^{\alpha_\perp}_{\bx_\perp} \partial_{x_3}^\mu \bu}{0;\,\cW}^2\Big)^{\frac12}
   \le C_2^{q+1}  \Big\{ \\
   &\sum_{\ell=0}^{q} \frac{1}{\ell!}\,
   \Big(\sum_{|\alpha|=\ell}\Norm{r^{\max\{\beta+2+|\alpha_\perp|,\,0\}}
   \partial^\alpha_\bx\partial_{x_3}^\mu  \bff}{0;
   \,\cW_{\varepsilon'}}^2\Big)^{\frac12}
   + \Norm{\partial_{x_3}^\mu \bu}{\bJ^m_{\beta+1}(\cW_{\varepsilon'})} \Big\}.
\end{aligned}
\end{equation}
The last term of this right-hand side is now estimated with the help
of \eqref{4E8new} with $\ell=\mu-1$
\[
   \Norm{ \partial^\mu_{x_3} \bu}{\bJ^m_{\beta+1}(\cW_{\varepsilon'})}
   \!\leq \Norm{ \partial^{\mu-1}_{x_3} \bu}{\bJ^{m+1}_{\beta}(\cW_{\varepsilon'})}
   \!\leq C_4^\mu (\mu-1)! \Big( \sum_{j=0}^{\mu-1} \frac{1}{j!}
   \Norm{\partial_{x_3}^j \bff}{\bJ^{m-1}_{\beta+2}(\cW_\varepsilon)}
   + \Norm{\bu}{\bJ^{m}_{\beta+1}(\cW_\varepsilon)} \Big).
\]
Using this estimate in \eqref{4EJ(3)} we obtain that
\begin{align*}
   \frac{1}{q!}\,
   \Big(\sum_{|\alpha_\perp|=q}\Norm{r^{\max\{\beta+|\alpha_\perp|,\,0\}}
   &\partial^{\alpha_\perp}_{\bx_\perp}\partial_{x_3}^\mu \bu}{0;\,\cW}^2\Big)^{\frac12} \le
   \\
   & C_2^{q+1}
   \sum_{\ell=0}^{q} \frac{1}{\ell!}\,
   \Big(\sum_{|\alpha|=\ell}\Norm{r^{\max\{\beta+2+|\alpha_\perp|,\,0\}}
   \partial^\alpha_\bx \partial_{x_3}^\mu \bff}{0;\,\cW_\varepsilon}^2\Big)^{\frac12}
   \\
   +\ & C_2^{q+1} C_4^\mu (\mu-1)! \, \Big( \sum_{j=0}^{\mu-1}\frac{1}{j!}
   \Norm{\partial_{x_3}^j \bff}{\bJ^{m-1}_{\beta+2}(\cW_\varepsilon)}
   + \Norm{\bu}{\bJ^m_{\beta+1}(\cW_\varepsilon)}\Big).
\end{align*}
We note that the norm in the space $\bJ^{m-1}_{\beta+2}(\cW_\varepsilon)$ is equivalent to ({\em cf.}\ \eqref{4EJmax})
\[
   \Bigl( \sum_{|\alpha|\le m-1}
   \Norm{r^{\max\{ \beta+2+|\alpha|,\,0 \}} \partial^\alpha_\bx u}
   {0; \,\cW}^2 \Bigr)^{\frac12}.
\]
Thus dividing the latter estimate by $\mu!$ and recalling that $k=q+\mu$ we deduce
\begin{align*}
   \frac{1}{k!}\,
   \Big(\sum_{|\alpha_\perp|=q}\Norm{ & r^{\max\{\beta+|\alpha_\perp|,\,0\}}
   \partial^{\alpha_\perp}_{\bx_\perp}\partial_{x_3}^\mu \bu}{0;\,\cW}^2\Big)^{\frac12} \le
   \\
   & C_5^{k+1}
   \sum_{\ell=0}^{q} \frac{1}{\ell!\mu!}\,
   \Big(\sum_{|\alpha|=\ell}\Norm{r^{\max\{\beta+2+|\alpha_\perp|,\,0\}}
   \partial^\alpha_\bx \partial_{x_3}^\mu \bff}{0;\,\cW_\varepsilon}^2\Big)^{\frac12}
   \\
   +\ & C_5^{k+1} \Big( \sum_{j=0}^{\mu-1}\frac{1}{j!}
   \sum_{|\alpha|\le m-1}
   \Norm{r^{\max\{ \beta+2+|\alpha|,\,0 \}} \partial^\alpha_\bx \partial_{x_3}^j \bff}
   {0; \,\cW_\varepsilon}
    + \Norm{\bu}{\bJ^m_{\beta+1}(\cW_\varepsilon)}\Big).
\end{align*}
From this we deduce \eqref{4EJ6}. The final way to \eqref{4E8} is very similar to the conclusion of the proof of Theorem \REF{4T3}. This ends the proof of Theorem \REF{4T4}.
\end{proof}

\begin{remark}
\label{4R3}
We note some similarities between our estimates and those obtained in \cite{GuoBabuska97b} for the Laplace operator. Our argument based on the dyadic partition technique clearly improves the structure of the whole proof.
\end{remark}

%%%%%%%%%%%%%%%%%%%%%%%%%%%%%%%%%%%%%%%%%%%%%%%%%%%%%%%%%%%%%%%%%%%%%%%%%%%%%%%%%%%%%%%%
\section{Natural anisotropic weighted regularity shift in polyhedra}
\label{sec5}
%%%%%%%%%%%%%%%%%%%%%%%%%%%%%%%%%%%%%%%%%%%%%%%%%%%%%%%%%%%%%%%%%%%%%%%%%%%%%%%%%%%%%%%%
\subsection{Edge and corner neighborhoods}
%%%%%%%%%%%%%%%%%%%%%%%%%%%%%%%%%%%%%%%%%%%%%%%%%%%%%%%%%%%%%%%%%%%%%%%%%%%%%%%%%%%%%%%%
Let $\Omega$ be a polyhedron in $\R^3$, that is a domain whose
boundary is a finite union of plane domains (the faces $\Gamma_\bs$,
$\bs\in\sS$). The faces are polygonal, the segments forming their
boundaries are the edges $\be$ of $\Omega$, and the ends of the
edges are the corners $\bc$ of $\Omega$. We denote the set of edges
by $\sE$ and the set of corners by $\sC$. Edge openings may be equal to $2\pi$,
allowing domains with crack surfaces.

In order to prove global regularity results in suitable weighted
Sobolev spaces, we introduce corner,
edge and  edge-vertex neighborhoods of $\Omega$.
For a fixed corner $\bc\in \sC$, we denote by $\sE_\bc$
the set of edges that have $\bc$ as extremities. Similarly for a
fixed edge $\be\in \sE$, we denote by $\sC_\be$ the set of corners
that are  extremities of  $\be$. Now we introduce the following
distances:
\begin{equation}
\label{5E0}
r_\bc(\bx)=\dist(\bx, \bc), \quad r_\be(\bx)=\dist(\bx, \be),
\quad
\rho_{\bc\be}(\bx)=\frac{r_\be(\bx)}{r_\bc(\bx)} \,.
\end{equation}
There exists $\varepsilon>0$ small enough such that if we set
\begin{subequations}
\begin{eqnarray}
\label{5Ea}
   \Omega_\be&=&\{\bx\in \Omega\, :\, r_\be(\bx)<\varepsilon\ \hbox{ and }\
   r_{\bc}(\bx)>\varepsilon/2 \quad\forall \bc\in \sC_\be\},
   \nonumber\\
   \Omega_\bc&=&\{\bx\in \Omega\, :\, r_\bc(\bx)<\varepsilon \ \hbox{ and } \
   \rho_{\bc\be}(\bx)>\varepsilon/2 \quad\forall \be\in \sE_\bc\},\\
  \Omega_{\bc\be}&=&\{\bx\in \Omega\, :\, r_\bc(\bx)<\varepsilon \ \hbox{
   and } \ \rho_{\bc\be}(\bx)<\varepsilon\},
   \nonumber
\end{eqnarray}
we have the following properties:
\begin{equation}
\label{5E1}
   \left\{ \begin{array}{rcll}
   \ov\Omega_\be\cap\ov\Omega_{\be'}&=&\varnothing, &\forall\be'\neq\be,
   \\[0.4ex]
   \ov\cB(\bc,\varepsilon)\cap\ov\cB(\bc',\varepsilon) &=&\varnothing,&\forall\bc'\neq\bc,
   \\[0.4ex]
   \ov\Omega_{\bc\be}\cap\ov\Omega_{\bc\be'}&=&\varnothing, &\forall\be'\neq\be.
   \end{array}\right.
\end{equation}
We also define the larger neighborhoods with $\varepsilon''<\varepsilon<\varepsilon'$
\begin{eqnarray}
\label{5Ec}
   \Omega'_\be&=&\{\bx\in \Omega\, :\, r_\be(\bx)<\varepsilon'\ \hbox{ and }\
   r_{\bc}(\bx)>\varepsilon''/2 \quad\forall \bc\in \sC_\be\},
   \nonumber\\
   \Omega'_\bc&=&\{\bx\in \Omega\, :\, r_\bc(\bx)<\varepsilon' \ \hbox{ and } \
   \rho_{\bc\be}(\bx)>\varepsilon''/2 \quad\forall \be\in \sE_\bc\},\\
  \Omega'_{\bc\be}&=&\{\bx\in \Omega\, :\, r_\bc(\bx)<\varepsilon' \ \hbox{
   and } \ \rho_{\bc\be}(\bx)<\varepsilon'\},
   \nonumber
\end{eqnarray}
assuming the $\varepsilon'$ and $\varepsilon''$ are sufficiently close to $\varepsilon$ for the above properties \eqref{5E1} to hold for $\Omega'_\be$, $\Omega'_\bc$, and $\Omega'_{\bc\be}$. We finally introduce the smaller neighborhoods  $\Omega''_\be$, $\Omega''_\bc$, and $\Omega''_{\bc\be}$ by inverting the roles of $\varepsilon'$ and $\varepsilon''$ and set,
\begin{equation}
\label{5Ed}
\Omega_{\sC}=\bigcup_{\bc\in \sC} \Omega''_\bc, \quad
\Omega_{\sE}=\bigcup_{\be\in \sE} \Omega''_\be, \quad
\Omega_{\sC\sE}=\bigcup_{\bc\in \sC}\bigcup_{\be\in \sE_\bc}
\Omega''_{\bc\be}.
\end{equation}
We define $\Omega_0$ as the remainder:
\begin{equation}
\label{5Esmo}
\Omega_0=\Omega\setminus\ov{\Omega_{\sC}\cup \Omega_{\sE} \cup
\Omega_{\sC\sE}}.
\end{equation}
Note that $\Omega_0$ is far from the singular points of $\Omega$. 
We finally choose a larger ``smooth'' neighborhood $\Omega'_0\subset\Omega$ such that 
$\Omega\cap\ov\Omega_0\subset\Omega'_0$ and $\ov\Omega{}'_0\cap(\sE\cup\sC)=\varnothing$.
\end{subequations}

Let $\cV$ be any subdomain of $\Omega$. We consider the system of local interior and boundary equations
\begin{equation}
\label{5Ebvp}
   \left\{ \begin{array}{rclll}
   L\,\bu &=& \bff \quad & \mbox{in}\ \ \Omega\cap\cV, \\[0.3ex]
   T_\bs\,\bu &=& 0  & \mbox{on}\ \ \Gamma_\bs\cap\ov\cV, &\bs\in\sS, \\[0.3ex]
   D_\bs\,\bu &=& 0  & \mbox{on}\ \ \Gamma_\bs\cap\ov\cV, &\bs\in\sS,
   \end{array}\right.
\end{equation}
where the operators $L$, $T_\bs$ and $D_\bs$ are homogeneous with constant coefficients and form an elliptic system. The choice $\cV=\Omega$ gives back the global boundary value problem on the polyhedron $\Omega$.

\begin{definition}
\label{5DK}
On $\cV\subset \Omega$, for $m\in \N$ and
$\betab=\{\beta_\bc\}_{\bc\in \sC}\cup\{\beta_\be\}_{\be\in \sE}$,
the weighted space with homogeneous norm $\rK^m_\betab(\cV)$ is
defined as follows, {\em cf.}\
\cite{MazyaRossmann91b,MazyaRossmann03,BuffaCoDa03, BuffaCoDa04}
\begin{align}
\label{5E1wsK}
    \rK^m_{\betab}(\cV) = \Big\{ & u\in \rL^2_{\rm loc}(\cV)\ : \
    \forall\alpha, \; |\alpha| \leq m,
    \quad    \partial^{\alpha}_\bx u\in \rL^2(\cV\cap \Omega_{0})\quad\mbox{and}\quad \\ \nonumber
    & r_\bc(\bx)^{\beta_\bc+|\alpha|} \,
    \partial^{\alpha}_\bx u\in \rL^2(\cV\cap \Omega_\bc) \quad \forall
    \bc\in \sC,\\ \nonumber
    & r_\be(\bx)^{\beta_\be+|\alpha|} \,
    \partial^{\alpha}_\bx u\in \rL^2(\cV\cap \Omega_\be) \quad \forall
    \be\in \sE,\\ \nonumber
    &r_\bc(\bx)^{\beta_\bc+|\alpha|} \,
    \rho_{\bc\be}(\bx)^{\beta_\be+|\alpha|} \,
    \partial^{\alpha}_\bx u\in \rL^2(\cV\cap \Omega_{\bc\be}) \quad \forall
    \bc\in \sC, \ \forall\be\in \sE_\bc \nonumber
    \Big\},
\end{align}
and endowed with its natural semi-norms and norm.
\end{definition}

Note that the condition in the edge-vertex neighborhood $\Omega_{\bc\be}$ can be equivalently written as
\[
    r_\bc(\bx)^{\beta_\bc-\beta_\be} \,
    r_{\be}(\bx)^{\beta_\be+|\alpha|} \,
    \partial^{\alpha}_\bx u\in \rL^2(\cV\cap \Omega_{\bc\be}).
\]

\begin{remark}
\label{5R0}
The semi-norms issued from \eqref{5E1wsK} are equivalent to the globally defined semi-norms
\begin{equation}
\label{5E00}
   \Big\{ \sum_{|\alpha|=k}
   \Big\| \big\{ \prod_{\bc\in\sC} r_\bc^{\beta_\bc + |\alpha|}\big\}
   \big\{ \prod_{\be\in\sE} \big( \frac{r_\be}{r_\sC} \big)^{\beta_\be + |\alpha|} \big\}
   \,\partial^\alpha_\bx u\,
   \Big\|_{0;\, \cV}^2
   \Big\}^{\frac12},\quad k=0,\ldots,m.
\end{equation}
Here $r_\sC$ denotes the distance function to the set $\sC$ of
corners. With this expression, the relations between our spaces
$\rK^m_\betab(\Omega)$ and the spaces
$V^{m,p}_{\vec\beta,\vec\delta}(\Omega)$ defined in \cite[\S1.2]{MazyaRossmann91b} or
\cite[\S7.3]{MazyaRossmann03} become obvious:
\begin{equation}
\label{5E01}
   \rK^m_\betab(\Omega) = V^{m,p}_{\vec\beta,\vec\delta}(\Omega)\quad\mbox{if}\quad
   p=2,\ \
   \vec\beta = \big\{\beta_\bc+m\big\}_{\bc\in\sC},\ \
   \vec\delta = \big\{\beta_\be+m\big\}_{\be\in\sE}\,.
\end{equation}
\end{remark}

%%%%%%%%%%%%%%%%%%%%%%%%%%%%%%%%%%%%%%%%%%%%%%%%%%%%%%%%%%%%%%%%%%%%%%%%%%%%%%%%%%%%%%%%
\subsection{Anisotropic weighted spaces with homogeneous norms}
%%%%%%%%%%%%%%%%%%%%%%%%%%%%%%%%%%%%%%%%%%%%%%%%%%%%%%%%%%%%%%%%%%%%%%%%%%%%%%%%%%%%%%%%
Unlike in the conical case, the weighted spaces $\rK^m_\betab$ are in a certain sense too large to describe accurately the regularity of solutions of the elliptic problem \eqref{5Ebvp} along the directions of edges. Mimicking the definition of the spaces $\rM^m_\beta$ in the pure edge case, {\em cf.}\ \eqref{4EM1}, we particularize for each edge $\be\in\sE$, the derivatives in the directions transverse or parallel to that edge by the notations
\begin{equation}
\label{5Eder}
   \partial^{\alpha_\perp}_\bx \ \mbox{(transverse)}\quad\mbox{and}\quad
   \partial^{\alpha_\parallel}\dd\bx \ \mbox{(parallel)},\quad (\be\in\sE),
\end{equation}
so that
\[
   \partial^\alpha_\bx = \partial^{\alpha_\perp}_\bx\, \partial^{\alpha_\parallel}\dd\bx\,.
\]
Of course these directions are edge dependent. They are well-defined in each of the domains $\Omega_\be$ and $\Omega_{\bc\be}$ determined by the edge $\be$.

The following spaces were introduced in \cite{BuffaCoDa03, BuffaCoDa04} for similar purposes:

\begin{definition}
\label{5DM}
On $\cV\subset \Omega$, for $m\in \N$ and
$\betab=\{\beta_\bc\}_{\bc\in \sC}\cup\{\beta_\be\}_{\be\in \sE}$,
we define
\begin{align}
\label{5E1wsM}
    \rM^m_{\betab}(\cV) = \Big\{ & u\in \rL^2_{\rm loc}(\cV)\  : \
    \forall\alpha, \; |\alpha| \leq m,
    \quad    \partial^{\alpha}_\bx u\in \rL^2(\cV\cap \Omega_{0})\quad\mbox{and}\quad \\ \nonumber
    & r_\bc(\bx)^{\beta_\bc+|\alpha|} \,
    \partial^{\alpha}_\bx u\in \rL^2(\cV\cap \Omega_\bc) \quad \forall
    \bc\in \sC,\\ \nonumber
    & r_\be(\bx)^{\beta_\be+|\alpha_\perp|} \,
    \partial^{\alpha}_\bx u\in \rL^2(\cV\cap \Omega_\be) \quad \forall
    \be\in \sE,\\ \nonumber
    &r_\bc(\bx)^{\beta_\bc+|\alpha|} \,
    \rho_{\bc\be}(\bx)^{\beta_\be+|\alpha_\perp|} \,
    \partial^{\alpha}_\bx u\in \rL^2(\cV\cap \Omega_{\bc\be}) \quad \forall
    \bc\in \sC, \ \forall\be\in \sE_\bc
    \Big\},
\end{align}
\begin{itemize}
\item[]
We denote by $\Norm{\cdot}{\rM;\,m, \betab;\,\cV}$ and $\SNorm{\cdot}{\rM;\,m, \betab;\,\cV}$
its norm and semi-norm, namely
\[
   \Norm{\cdot}{\rM;\,m, \betab;\,\cV}^2 = \sum_{\ell=0}^m
   \SNorm{\cdot}{\rM;\,\ell, \betab;\,\cV}^2
\]
with
\begin{align}
\label{5EnormM}
    \SNorm{u}{\rM;\,\ell, \betab;\,\cV}^2=
    &\sum_{|\alpha|= \ell}
    \Big(\Norm{\partial^{\alpha}_\bx u}{0;\, \cV\cap \Omega_0}^2
    + \sum_{\bc\in \sC}\Norm{r_\bc^{\beta_\bc+|\alpha|}
    \partial^{\alpha}_\bx u}{0;\, \cV\cap \Omega_\bc}^2\\
    \nonumber
    &+\sum_{\be\in \sE}\Norm{r_\be^{\beta_\be+|\alpha_\perp|}
    \partial^{\alpha}_\bx u}{0;\, \cV\cap \Omega_\be}^2
    +\sum_{\bc\in \sC}\sum_{\be\in\sE_\bc}
    \Norm{r_\bc^{\beta_\bc+|\alpha|}
    \rho_{\bc\be}^{\beta_\be+|\alpha_\perp|}
    \partial^{\alpha}_\bx u}{0;\, \cV\cap \Omega_{\bc\be}}^2\Big).
\end{align}
\end{itemize}
\end{definition}

\noindent
Note that the condition in the edge-vertex neighborhood $\Omega_{\bc\be}$ can be written equivalently as
\[
    r_\bc(\bx)^{\beta_\bc-\beta_\be+\alpha_\parallel} \,
    r_{\be}(\bx)^{\beta_\be+|\alpha|} \,
    \partial^{\alpha}_\bx u\in \rL^2(\cV\cap \Omega_{\bc\be}).
\]
We can then define the corresponding analytic class as
follows:

\begin{definition}
\label{5DA}
We say that $u\in \rA_{\betab}(\Omega)$ if $u\in
\rM^k_{\betab}(\Omega)$ for all $k\geq 0$ and there exists a positive
constant $C$ such that
\[
   \SNorm{u}{\rM;\,k, \betab;\,\Omega}\leq C^{k+1} k! \quad \forall k\geq 0.
\]
\end{definition}

We rephrase Assumption \REF{4GA} for the dihedral neighborhood $\Omega_\be$:
\begin{assumption}
\label{5GA}
Let $\be\in\sE$ and $\beta_\be\in\R$. We assume the following a priori estimate:
There is a constant $C$ such that
any
\[
   \bu\in \bK^2_{\beta_\be}(\Omega_\be)\,,
\]
solution of problem \eqref{5Ebvp} in $\cV=\Omega'_\be$ with $\bff\in\bK^0_{\beta_\be+2}(\Omega'_\be)$,
satisfies:
\begin{equation}
\label{5GA3}
   \Norm{\bu}{\bK^{2}_{\beta_\be}(\Omega_\be)} \leq
   C\Big(\Norm{ \bff}{\bK^0_{\beta_\be+2}(\Omega'_\be)}
   + \Norm{ \bu}{\bK^1_{\beta_\be+1}(\Omega'_\be)}\Big).
\end{equation}
\end{assumption}

\medskip
We can apply Theorem \REF{4T3} to the edge neighborhood $\Omega_\be$. We obtain that under Assumption \REF{5GA}, any solution $\bu\in \bK^1_{\beta_\be}(\Omega'_\be)$ of problem \eqref{5Ebvp} with  $\bff\in \bM^{n}_{\beta_\be+2}(\Omega'_\be)$ satisfies the uniform estimates for $0\le k\le n$
\begin{multline}
\label{5E7e}
   \frac{1}{k!}\;
   \Big(\sum_{|\alpha|=k}\Norm{r_\be^{\beta_\be+|\alpha_\perp|}
   \partial^\alpha_\bx \bu}{0;\,\Omega_\be}^2\Big)^{\frac12} \le
    C^{k+1} \Big\{
   \sum_{\ell=0}^{k} \frac{1}{\ell!}\;
   \Big(\sum_{|\alpha|=\ell}
   \Norm{r_\be^{\beta_\be+2+|\alpha_\perp|} \partial^\alpha_\bx \bff}{0;
   \,\Omega'_\be}^2\Big)^{\frac12}
   \\
   + \Norm{\bu}{\bK^1_{\beta_\be+1}(\Omega'_\be)}\Big\}.
\end{multline}

Now we consider the edge-vertex domain $\Omega_{\bc\be}$.

\begin{proposition}
\label{5P1}
Let $\bc\in\sC$ and $\be\in\sE_\bc$. Let $\betab=\{\beta_\bc,\beta_\be\}$.
Under Assumption \REF{5GA}, any solution $\bu\in \bK^1_{\betab}(\Omega'_{\bc\be})$ of problem \eqref{5Ebvp} with  $\bff\in \bM^{n}_{\betab+2}(\Omega'_{\bc\be})$ belongs to $\bM^{n}_{\betab}(\Omega_{\bc\be})$ and satisfies the uniform estimates for $0\le k\le n$
\begin{multline}
\label{5E7ce}
   \frac{1}{k!}\;
   \Big(\sum_{|\alpha|=k}\Norm{r_\bc^{\beta_\bc+|\alpha|}
   \rho_{\bc\be}^{\beta_\be+|\alpha_\perp|}
   \partial^\alpha_\bx \bu}{0;\,\Omega_{\bc\be}}^2\Big)^{\frac12} \le
    C^{k+1} \Big\{ \\
   \sum_{\ell=0}^{k} \frac{1}{\ell!}\;
   \Big(\sum_{|\alpha|=\ell}
   \Norm{r_\bc^{\beta_\bc+2+|\alpha|}
   \rho_{\bc\be}^{\beta_\be+2+|\alpha_\perp|} \partial^\alpha_\bx \bff}{0;
   \,\Omega'_{\bc\be}}^2\Big)^{\frac12}
   + \Norm{\bu}{\bK^1_{\betab+1}(\Omega'_{\bc\be})}\Big\}.
\end{multline}
\end{proposition}

\begin{proof}
We mimic the proof of Theorem \REF{2T1}.
The proof of estimate \eqref{5E7ce} is based on a locally finite dyadic covering of $\Omega_{\bc\be}$ and $\Omega'_{\bc\be}$. Define, compare with \eqref{5Ea}-\eqref{5Ec},
\begin{align*}
   \widehat \cV &= \{\bx\in \Omega\ : \
   \tfrac\varepsilon4< r_\bc(\bx) < \varepsilon \ \ \mbox{and}\ \
   \rho_{\bc\be}<\varepsilon\} \\
   \widehat \cV\ee' &= \{\bx\in \Omega\ : \
   \tfrac{\varepsilon^2}{4\varepsilon'}< r_\bc(\bx) < \varepsilon' \ \ \mbox{and}\ \
   \rho_{\bc\be}<\varepsilon'\},
\end{align*}
and for $\mu\in\N$:
\[
   \cV_\mu = 2^{-\mu}\widehat \cV \quad\mbox{and}\quad \cV\prm_\mu = 2^{-\mu}\widehat \cV\ee'.
\]
We check:
\[
   \Omega_{\bc\be} = \bigcup_{\mu\in\N} \cV_\mu
   \quad\mbox{and}\quad
   \Omega'_{\bc\be} = \bigcup_{\mu\in\N} \cV\prm_\mu\,.
\]
The estimate \eqref{5E7e} between $\Omega_\be$ and $\Omega'_\be$
also holds in the configuration of $\widehat \cV$ and $\widehat
\cV\ee'$ which is similar:  $\widehat \cV$ and $\widehat
\cV\ee'$ are {\em nested edge neighborhoods} which do not touch any corner, see Fig.~\ref{F2}.
\begin{figure}
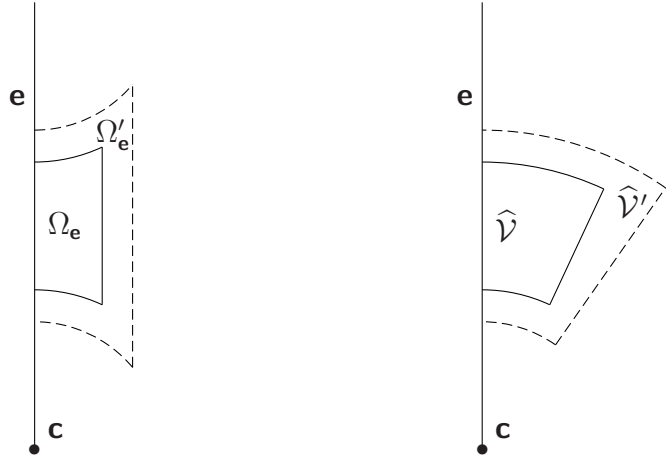

\begin{center}
% 1. Definition of characteristic points
    \figinit{0.85mm}
    \figpt 1:( 0, -5)
    \figpt 2:( 0,65)
    \figpt 3:( 0,15)
    \figpt 4:( 0,45)
    \figpt 5:( 0,20)
    \figpt 6:( 0,40)
    \figvectC 100 ( 1, 0)
    \figptrot 13: = 3 /1, -50/
    \figptrot 14: = 4 /2, 50/
    \figptrot 15: = 5 /1, -25/
    \figptrot 16: = 6 /2, 25/
    \figptstra 21 = 1, 2, 3, 4, 5, 6 /70, 100/
    \figptrot 33: = 23 /21, -35/
    \figptrot 34: = 24 /21, -35/
    \figptrot 35: = 25 /21, -25/
    \figptrot 36: = 26 /21, -25/

% 2. Creation of the postscript file
    \def\MyPSfile{}
    \psbeginfig{\MyPSfile}
    \psline[1,2]
    \psline[21,22]
    \psline[15,16]
    \psline[35,36]
    \psarccircP 1 ; 25 [15,5]
    \psarccircP 2 ; 25 [6,16]
    \psarccircP 21 ; 25 [35,25]
    \psarccircP 21 ; 45 [36,26]
    \psset (dash=3)
    \psarccircP 1 ; 20 [13,3]
    \psarccircP 2 ; 20 [4,14]
    \psarccircP 21 ; 20 [33,23]
    \psarccircP 21 ; 50 [34,24]
    \psline[13,14]
    \psline[33,34]
    \psendfig

% 3. Writing text on the figure
    \figvisu{\demo}{}{%
    \figinsert{\MyPSfile}
    \figsetmark{\footnotesize\Bullet}
    \figwritegce 1:$\bc$(2,3)
    \figwritegce 21:$\bc$(2,3)
    \figsetmark{}
    \figwritegce 4:$\be$(-4,5)
    \figwritegce 24:$\be$(-4,5)
    \figwritegce 5:$\Omega_\be$(2,10)
    \figwritegce 16:$\Omega'_\be$(-1,2)
    \figwritegce 25:$\widehat\cV$(2,10)
    \figwritegce 36:$\widehat\cV\ee'$(2,-2)
   }
    \centerline{\box\demo}
    \caption{Nested edge neighborhoods (section determined by azimuthal angle $\theta_\be=$ constant)}
\label{F2}
\end{center}
\end{figure}

Since $r_\bc$ is bounded from above and from below by strictly positive constants, the distance $r_\be$ is equivalent to $\rho_{\bc\be}$ on the reference domains: We have
\begin{multline*}
   \frac{1}{k!}
   \Big(\sum_{|\alpha|=k}\Norm{\rho_{\bc\be}(\widehat\bx)^{\beta_\be+|\alpha_\perp|}
   \partial^\alpha_\bx \widehat\bu}{0;\,\widehat \cV}^2\Big)^{\frac12} \le
    C^{k+1} \Big\{
   \sum_{\ell=0}^{k} \frac{1}{\ell!}
   \Big(\sum_{|\alpha|=\ell}
   \Norm{\rho_{\bc\be}(\widehat\bx)^{\beta_\be+2+|\alpha_\perp|} \partial^\alpha_\bx \widehat\bff}{0;
   \,\widehat \cV\ee'}^2\Big)^{\frac12}
   \\
   + \sum_{|\alpha|\le1} \Norm{\rho_{\bc\be}(\widehat\bx)^{\beta_\be+|\alpha|}
   \partial^\alpha_\bx \widehat\bu}{0;\,\widehat \cV\ee'}\Big\}.
\end{multline*}
for any reference function $\widehat\bu$ satisfying the boundary conditions of \eqref{5Ebvp} and $\widehat\bff:=L \widehat\bu$.

For the same reason, we can insert powers of $r_\bc$ in the above estimate, to obtain our new reference estimate
\begin{align}
\label{5E7f}
   \frac{1}{k!}\;
   \Big(\sum_{|\alpha|=k}\Norm{r_\bc(\widehat\bx)^{\beta_\bc+|\alpha|} &
   \rho_{\bc\be}(\widehat\bx)^{\beta_\be+|\alpha_\perp|}
   \partial^\alpha_\bx \widehat\bu}{0;\,\widehat \cV}^2\Big)^{\frac12} \le
    C^{k+1} \Big\{ \\ \nonumber
   &\sum_{\ell=0}^{k} \frac{1}{\ell!}\;
   \Big(\sum_{|\alpha|=\ell} \Norm{r_\bc(\widehat\bx)^{\beta_\bc+2+|\alpha|}
  \rho_{\bc\be}(\widehat\bx)^{\beta_\be+2+|\alpha_\perp|} \partial^\alpha_\bx \widehat\bff}{0;
   \,\widehat \cV\ee'}^2\Big)^{\frac12}
   \\ \nonumber
   &+ \sum_{|\alpha|\le1} \Norm{
   r_\bc(\widehat\bx)^{\beta_\bc+|\alpha|}\rho_{\bc\be}(\widehat\bx)^{\beta_\be+|\alpha|}
   \partial^\alpha_\bx \widehat\bu}{0;\,\widehat \cV\ee'} \Big\}.
\end{align}
The change of variables $\hat\bx\rightarrow\bx=2^{-\mu}\hat\bx$ maps $\widehat \cV$ to $\cV_{\mu}$ (resp. $\widehat \cV\ee'$ to $\cV\prm_{\mu}$). We note that
\[
   \rho_{\bc\be}(\widehat\bx) = \rho_{\bc\be}(\bx) \quad\mbox{and}\quad
   r_\bc(\widehat\bx) = 2^\mu r_\bc(\bx).
\]
With the change of functions
\[
   \widehat\bu(\widehat\bx) := \bu(\bx) \quad\mbox{and}\quad
   \widehat\bff(\widehat\bx) := L \widehat\bu\,,
   \quad\mbox{which implies}\quad \widehat\bff(\widehat\bx) = 2^{-2\mu}\bff(\bx),
\]
we deduce from estimate \eqref{5E7f} that
\begin{align*}
   \frac{1}{k!}\; 2^{\mu\beta_\bc}
   \Big(\sum_{|\alpha|=k}\Norm{r_\bc(\bx & )^{\beta+|\alpha|}
   \rho_{\bc\be}(\bx)^{\beta_\be+|\alpha_\perp|}
   \partial^\alpha_\bx \bu}{0;\, \cV}^2\Big)^{\frac12} \le
    C^{k+1} \Big\{ \\ \nonumber
   &\sum_{\ell=0}^{k} \frac{1}{\ell!}\; 2^{\mu(\beta_\bc+2)}
   \Big(\sum_{|\alpha|=\ell} 2^{-2\mu} \Norm{r_\bc(\bx)^{\beta+2+|\alpha|}
  \rho_{\bc\be}(\bx)^{\beta_\be+2+|\alpha_\perp|} \partial^\alpha_\bx \bff}{0;
   \, \cV\ee'}^2\Big)^{\frac12}
   \\ \nonumber
   &+ \Big( 2^{\mu\beta_\bc} \sum_{|\alpha|\le1} \Norm{
   r_\bc(\bx)^{\beta_\bc+|\alpha|}\rho_{\bc\be}(\bx)^{\beta_\be+|\alpha|}
   \partial^\alpha_\bx \bu}{0;\, \cV\ee'}^2\Big)^{\frac12}\Big\}.
\end{align*}
Multiplying this identity by $2^{-\mu \beta}$, taking squares, and summing up over all $\mu$, we get the requested estimate \eqref{5E7ce}.
\end{proof}

The estimates in pure vertex domains $\Omega_{\bc}$ (i.e., close to corners but ``relatively far'' from the edges) are similar to those in obtained in Theorem \REF{2T1} for plane sectors:

\begin{proposition}
\label{5P2}
Let $\bc\in\sC$ and $\betab=\{\beta_\bc\}$.
Any solution $\bu\in \bK^1_{\betab}(\Omega'_{\bc})$ of problem \eqref{5Ebvp} with  $\bff\in \bM^{n-2}_{\betab+2}(\Omega'_{\bc})$ belongs to $\bM^{n}_{\betab}(\Omega_{\bc})$ and satisfies the uniform estimates for $0\le k\le n$
\begin{multline}
\label{5E7c}
   \frac{1}{k!}\;
   \Big(\sum_{|\alpha|=k}\Norm{r_\bc^{\beta_\bc+|\alpha|}
   \partial^\alpha_\bx \bu}{0;\,\Omega_{\bc}}^2\Big)^{\frac12} \le
    C^{k+1} \Big\{ \\
   \sum_{\ell=0}^{k-2} \frac{1}{\ell!}\;
   \Big(\sum_{|\alpha|=\ell}
   \Norm{r_\bc^{\beta_\bc+2+|\alpha|}
   \partial^\alpha_\bx \bff}{0;
   \,\Omega'_{\bc}}^2\Big)^{\frac12}
   + \Norm{\bu}{\bK^1_{\betab+1}(\Omega'_{\bc})}\Big\}.
\end{multline}
\end{proposition}

\begin{proof}
The proof is again based on the argument of dyadic partitions with reference domains defined as
\[
   \widehat \cV = \{\bx\in \Omega_{\bc},\quad
   \tfrac\varepsilon4< r_\bc(\bx) < \varepsilon\}
   \quad\mbox{and}\quad
   \widehat \cV\ee' = \{\bx\in \Omega'_\bc,\quad
   \tfrac{\varepsilon^2}{4\varepsilon'}< r_\bc(\bx) < \varepsilon'\}.
\]
and for $\mu\in\N$:
\[
   \cV_\mu = 2^{-\mu}\widehat \cV \quad\mbox{and}\quad \cV\prm_\mu = 2^{-\mu}\widehat \cV\ee'.
\]
We check:
\[
   \Omega_{\bc} = \bigcup_{\mu\in\N} \cV_\mu
   \quad\mbox{and}\quad
   \Omega'_{\bc} = \bigcup_{\mu\in\N} \cV\prm_\mu\,.
\]
We can apply the a priori estimates of the smooth case between $\widehat \cV$ and $\widehat \cV\ee'$, cf.\ \eqref{2E4} and deduce \eqref{5E7c} in the same way.
\end{proof}

We obtain now the anisotropic regularity shift in homogeneous weighted spaces on polyhedra :

\begin{theorem}
\label{5T1}
Let $\Omega$ be a polyhedron and $\betab=\{\beta_\bc,\beta_\be\}$ be a weight multi-exponent.
Let Assumption \REF{5GA} be satisfied for all edges $\be\in\sE$.
Let $\bu\in\bH^2_{\loc}(\ov \Omega\setminus\sE)$ be a solution of problem \eqref{5Ebvp} in $\cV=\Omega$.
Then the following implications hold
\begin{subequations}
\begin{gather}
\label{5EM}
   \bu\in\bK^1_{\betab}(\Omega) \ \ \mbox{and}\ \
   \bff\in\bM^m_{\betab+2}(\Omega)
   \quad\Longrightarrow\quad
   \bu\in\bM^m_{\betab}(\Omega)\quad (m\in\N),
\\
\label{5EA}
   \bu\in\bK^1_{\betab}(\Omega) \ \ \mbox{and}\ \
   \bff\in\bA_{\betab+2}(\Omega)
   \quad\Longrightarrow\quad
   \bu\in\bA_{\betab}(\Omega).
\end{gather}
\end{subequations}
\end{theorem}

\begin{proof}
The proof is a consequence of
\begin{itemize}
\item[\iti1]  elliptic estimates in the smooth case applied between $\Omega_0$ and $\Omega'_0$,
\item[\iti2]  pure corner estimates \eqref{5E7c},
\item[\iti3]  edge estimates \eqref{5E7e} between the pure edge domains $\Omega_\be$ and $\Omega'_\be$,
\item[\iti4]  edge-vertex estimates \eqref{5E7ce}.
\end{itemize}

\end{proof}

%%%%%%%%%%%%%%%%%%%%%%%%%%%%%%%%%%%%%%%%%%%%%%%%%%%%%%%%%%%%%%%%%%%%%%%%%%%%%%%%%%%%%%%%
\subsection{Anisotropic weighted spaces with non-homogeneous norms}
%%%%%%%%%%%%%%%%%%%%%%%%%%%%%%%%%%%%%%%%%%%%%%%%%%%%%%%%%%%%%%%%%%%%%%%%%%%%%%%%%%%%%%%%
For the same reason as in the two-dimensional case, it is valuable to have alternative statements to \eqref{5EM} and \eqref{5EA} in which the a priori condition $\bu\in\bK^1_{\betab}(\Omega)$ can be replaced by the weaker condition $\bu\in\bJ^1_{\betab}(\Omega)$.

\begin{definition}
\label{5DJN}
For $\betab=\{\beta_\bc,\beta_\be\}$ and $n\in\N$, let us introduce the isotropic weighted space
\begin{align}
\label{5E1wsJ}
    \rJ^n_{\betab}(\cV) = \Big\{ & u\in \rL^2_{\rm loc}(\cV)\ : \
    \forall\alpha, \; |\alpha| \leq  n,
    \quad    \partial^{\alpha}_\bx u\in \rL^2(\cV\cap \Omega_{0})
    \quad\mbox{and}\quad \\ \nonumber
    & r_\bc(\bx)^{\beta_\bc+n} \,
    \partial^{\alpha}_\bx u\in \rL^2(\cV\cap \Omega_\bc) \quad \forall
    \bc\in \sC,\\ \nonumber
    & r_\be(\bx)^{\beta_\be+n} \,
    \partial^{\alpha}_\bx u\in \rL^2(\cV\cap \Omega_\be) \quad \forall
    \be\in \sE,\\ \nonumber
    &r_\bc(\bx)^{\beta_\bc+n} \,
    \rho_{\bc\be}(\bx)^ {\beta_\be+n} \,
    \partial^{\alpha}_\bx u\in \rL^2(\cV\cap \Omega_{\bc\be}) \quad \forall
    \bc\in \sC, \ \forall\be\in \sE_\bc
    \Big\},
\end{align}
and its anisotropic companion, for
$n\ge -\min\{\min_{\bc\in\sC}\beta_\bc, \min_{\be\in\sE}\beta_\be\}$,
{\em cf.}\ \eqref{4EN1}
\begin{align}
\label{5E1wsN}
    \rN^n_{\betab}&(\cV) = \Big\{ u\in \rL^2_{\rm loc}(\cV)\ : \
    \forall\alpha, \; |\alpha| \leq  n,
    \quad    \partial^{\alpha}_\bx u\in \rL^2(\cV\cap \Omega_{0})
    \quad\mbox{and}\quad \\ \nonumber
    & r_\bc(\bx)^{\max\{\beta_\bc+|\alpha|,0\}} \,
    \partial^{\alpha}_\bx u\in \rL^2(\cV\cap \Omega_\bc) \quad \forall
    \bc\in \sC,\\ \nonumber
    & r_\be(\bx)^{\max\{\beta_\be+|\alpha_\perp|,0\}} \,
    \partial^{\alpha}_\bx u\in \rL^2(\cV\cap \Omega_\be) \quad \forall
    \be\in \sE,\\ \nonumber
    &r_\bc(\bx)^{\max\{\beta_\bc+|\alpha|,0\}} \,
    \rho_{\bc\be}(\bx)^ {\max\{\beta_\be+|\alpha_\perp|,0\}} \,
    \partial^{\alpha}_\bx u\in \rL^2(\cV\cap \Omega_{\bc\be}) \quad \forall
    \bc\in \sC, \ \forall\be\in \sE_\bc
    \Big\}.\hskip-2em
\end{align}
\end{definition}

We note that, like in the case of $\rK$-weighted spaces, the semi-norms issued from \eqref{5E1wsJ} are equivalent to the globally defined semi-norms, compare with \eqref{5E00}
\begin{equation}
\label{5E10}
   \Big\{ \sum_{|\alpha|=k}
   \Big\| \big\{\prod_{\bc\in\sC} r_\bc^{\beta_\bc + n}\big\}
   \big\{ \prod_{\be\in\sE} \big( \frac{r_\be}{r_\sC}\big)^ {\beta_\be+n}\big\}
   \,\partial^\alpha_\bx u\,
   \Big\|_{0;\, \cV}^2
   \Big\}^{\frac12},\quad k=0,\ldots,n.
\end{equation}
It is useful to introduce, in the same spirit as in \cite{MazyaRossmann03}, a full range of intermediate spaces between $\rK^n_\betab(\Omega)$ and $\rJ^n_\betab(\Omega)$.

\begin{definition}
\label{5Dflag}
Let us flag a subset $\sC_0$ of corners and a subset $\sE_0$ of edges, and define $\rJ^n_{\betab}(\cV;\sC_0,\sE_0)$ as the space of functions such that all semi-norms
\begin{equation}
\label{5E10b}
   \Big\| \big\{\prod_{\bc\in\sC_0} r_\bc^{\beta_\bc + |\alpha|}\big\}
   \big\{\!\!\prod_{\bc\in\sC\setminus\sC_0} \! r_\bc^{\beta_\bc + n}\big\}
   \big\{\prod_{\be\in\sE_0} \big( \frac{r_\be}{r_\sC}\big)^ {\beta_\be+|\alpha|}\big\}
   \big\{\!\!\prod_{\be\in\sE\setminus\sE_0} \!
   \big( \frac{r_\be}{r_\sC}\big)^ {\beta_\be+n}\big\}
   \,\partial^\alpha_\bx u\,
   \Big\|_{0;\, \cV}
\end{equation}
are finite for $|\alpha|\le n$. Anisotropic spaces $\rN^n_{\betab}(\cV;\sC_0,\sE_0)$ are defined similarly, replacing in \eqref{5E1wsN} the weight $r_\bc^{\max\{\beta_\bc+|\alpha|,0\}}$ by $r_\bc^{\beta_\bc+|\alpha|}$ when $\bc\in\sC_0$, and $\{r_\be,\rho_{\bc\be}\}^{\max\{\beta_\be+|\alpha|,0\}}$ by $\{r_\be,\rho_{\bc\be}\}^{\beta_\be+|\alpha|}$ when $\be\in\sE_0$. The sum of the squares of these contributions for $|\alpha|=n$ defines the squared semi-norm
\[
   \SNorm{u}{\rN^n_{\betab}(\cV;\sC_0,\sE_0)}^2.
\]
\end{definition}

Note that with $\sC_0=\sE_0=\varnothing$, we obtain the maximal spaces  already introduced in \eqref{5E1wsJ} and \eqref{5E1wsN}:
\begin{equation}
\label{5EJN}
 \rJ^n_\betab(\cV) = \rJ^n_\betab(\cV;\varnothing,\varnothing)\,;\qquad
  \rN^n_\betab(\cV) = \rN^n_\betab(\cV;\varnothing,\varnothing)\;.
\end{equation}

The corresponding analytic class is defined as usual:

\begin{definition}
\label{5DB}
We say that $u\in \rB_{\betab}(\Omega;\sC_0,\sE_0)$ if $u\in\rN^k_{\betab}(\Omega;\sC_0,\sE_0)$ for all $k>k_\betab:=-\min\{\min_{\bc\in\sC}\beta_\bc, \min_{\be\in\sE}\beta_\be\}$ and there exists a positive
constant $C$ such that
\[
   \SNorm{u}{\rN^k_{\betab}(\Omega;\sC_0,\sE_0)}\leq C^{k+1} k! \quad \forall k>k_\betab.
\]
In accordance with \eqref{5EJN}, we write $\rB_\betab(\Omega)$ for $\rB_\betab(\Omega;\varnothing,\varnothing)$.
\end{definition}

\begin{remark}
\label{5R1}
\iti1 Choosing $\sC_0=\sC$ and $\sE_0=\sE$, we find that the spaces $\rJ^n_{\betab}(\Omega;\sC,\sE)$, $\rN^n_{\betab}(\Omega;\sC,\sE)$ and $\rB_{\betab}(\Omega;\sC,\sE)$ coincide with the homogeneous spaces $\rK^n_{\betab}(\Omega)$, $\rM^n_{\betab}(\Omega)$ and $\rA_{\betab}(\Omega)$, respectively.

\smallskip
\iti2 The following relations hold between our spaces $\rJ^m_\betab(\Omega;\sC_0,\sE_0)$ and the spaces $W^{m,p}_{\vec\beta,\vec\delta}(\Omega)$ of Maz'ya and Rossmann \cite{MazyaRossmann03}:
\begin{equation}
\label{5E11}
   \rJ^m_\betab(\Omega;\sC,\varnothing) = W^{m,p}_{\vec\beta,\vec\delta}(\Omega)
   \quad\mbox{if}\ \
   p=2,\ \
   \vec\beta = \big\{\beta_\bc+m\big\}_{\bc\in\sC},\ \
   \vec\delta = \big\{\beta_\be+m\big\}_{\be\in\sE}\,.
\end{equation}
In these spaces, the non-homogeneity is only related to {\em edges}. Under the same condition as in \eqref{5E11}, the intermediate spaces $W^{m,p}_{\vec\beta,\vec\delta}(\Omega;\tilde J)$ of \cite[\S\,7.3]{MazyaRossmann03} coincide with our spaces $\rJ^m_\betab(\Omega;\sC,\sE_0)$ if $\sE_0$ is chosen as the same set of edges as $\tilde J$.

\smallskip
\iti3 Our analytic class $\rB_\betab(\Omega)$ coincides with the so-called countably normed spaces $B^\ell_\beta(\Omega)$ introduced by Guo in \cite{Guo95}: If Guo's edge and corner exponents $\beta_{ij}\in(0,1)$ and $\beta_m\in(0,\frac12)$ satisfy $\beta_{ij}=\beta_\be+\ell$ and $\beta_m=\beta_\bc+\ell$, respectively, then $B^\ell_\beta(\Omega)=\rB_\betab(\Omega)$.
\end{remark}

We state the assumption for $\rJ$-weighted spaces corresponding to Assumption \REF{4GB} for the dihedral neighborhood $\Omega_\be$:

\begin{assumption}
\label{5GB}
Let $\be\in\sE$.
Let $\beta_\be\in\R$. Let $m\ge1$ be an integer such that $m+1\ge-\beta_\be$.
We assume the following a priori estimate:
There is a constant $C$ such that
any
\[
   \bu\in \bJ^{m+1}_{\beta_\be}(\Omega_\be)\,,
\]
solution of problem \eqref{5Ebvp} in $\cV=\Omega'_\be$ with $\bff\in\bJ^{m-1}_{\beta_\be+2}(\Omega'_\be)$,
satisfies:
\begin{equation}
\label{5GB3}
   \Norm{\bu}{\bJ^{m+1}_{\beta_\be} (\Omega_\be)} \leq
   C\Big(\Norm{ \bff}{\bJ^{m-1}_{\beta_\be+2} (\Omega'_\be)}
   + \Norm{ \bu}{\bJ^m_{\beta_\be+1} (\Omega'_\be)}\Big).
\end{equation}
\end{assumption}

We then have the following anisotropic regularity shift result in the {\em non-homogeneous} weighted spaces $\bN^n_{\betab}(\Omega;\sC,\varnothing)$ and $\bB_{\betab}(\Omega;\sC,\varnothing)$:

\begin{theorem}
\label{5T2}
Let $\Omega$ be a polyhedron and $\betab=\{\beta_\bc,\beta_\be\}$ be a weight multi-exponent.
Let $m\ge1$ be an integer such that $m+1\ge-\beta_\be$ for all edges.
Let Assumption \REF{5GB} be satisfied for all $\be\in\sE$.
Let $\bu\in\bH^2_{\loc}(\ov \Omega\setminus\sE)$ be a solution of problem \eqref{5Ebvp} in $\cV=\Omega$.
Then the following implications hold
\begin{equation}
\begin{gathered}
\label{5ENB}
   \bu\in\bJ^m_{\betab}(\Omega;\sC,\varnothing) \ \ \mbox{and}\ \
   \bff\in\bN^n_{\betab+2}(\Omega;\sC,\varnothing)
   \quad\Longrightarrow\quad
   \bu\in\bN^n_{\betab}(\Omega;\sC,\varnothing) \quad(n>m),\hskip-1em
\\
   \bu\in\bJ^m_{\betab}(\Omega;\sC,\varnothing) \ \ \mbox{and}\ \
   \bff\in\bB_{\betab+2}(\Omega;\sC,\varnothing)
   \quad\Longrightarrow\quad
   \bu\in\bB_{\betab}(\Omega;\sC,\varnothing).
\end{gathered}
\end{equation}
\end{theorem}

\begin{proof}
The proof is a consequence of suitable a priori estimates with analytic control in the four types of regions in the polyhedron:

\iti1 Elliptic estimates in the smooth case can be applied between $\Omega_0$ and $\Omega'_0$.

\smallskip
\iti2 Pure corner estimates \eqref{5E7c} are valid here: We note that in the pure corner region $\Omega_\bc$ the norms in $\rK$ and $\rJ$ spaces, or in $\rM$ and $\rN$ spaces, are the same.

\smallskip
\iti3 The edge estimates \eqref{4E8} are valid between the pure edge domains $\Omega_\be$ and $\Omega'_\be$.

\smallskip
\iti4 Finally, edge-vertex estimates are proved by the dyadic partition argument starting from the same reference domains $\widehat\cV$ and $\widehat\cV\ee'$ as in the proof of Proposition \REF{5P1}. The reference estimate can be written as
\begin{multline}
\label{5E20}
   \frac{1}{k!}\;
   \Big(\sum_{|\alpha|=k}\Norm{r_\be^{\max\{\beta_\be+|\alpha_\perp|,\,0\}}
   \partial^\alpha_\bx \widehat\bu}{0;\,\widehat\cV}^2\Big)^{\frac12} \le
    C^{k+1} \Big\{ \\[-1.5ex]
   \sum_{\ell=0}^{k} \frac{1}{\ell!}\;
   \Big(\sum_{|\alpha|=\ell}
   \Norm{r_\be^{\max\{\beta_\be+2+|\alpha_\perp|,\,0\}} \partial^\alpha_\bx \widehat\bff}{0;
   \,\widehat\cV\ee'}^2\Big)^{\frac12}\\[-1.5ex]
   +  \Big(\sum_{|\alpha|\le m}\Norm{r_\be^{\max\{\beta_\be+|\alpha|,\,0\}}
   \partial^\alpha_\bx \widehat\bu}{0;\,\widehat\cV\ee'}^2\Big)^{\frac12}
   \Big\}.%\hskip-1em
\end{multline}
Since $r_\bc$ and $(r_\bc)^{-1}$ are bounded on the reference domains, we can
\begin{itemize}
\item replace $r_\be$ by $\rho_{\bc\be}$
\item insert powers of $r_\bc$
\end{itemize}
in the previous estimate, thus obtaining
\begin{multline*}
   \frac{1}{k!}\;
   \Big(\sum_{|\alpha|=k}\Norm{r_\bc^{\beta_\bc+|\alpha|}\,
   \rho_{\bc\be}^{\max\{\beta_\be+|\alpha_\perp|,\,0\}}
   \partial^\alpha_\bx \widehat\bu}{0;\,\widehat\cV}^2\Big)^{\frac12} \le
    C^{k+1} \Big\{ \\[-1.5ex]
   \sum_{\ell=0}^{k} \frac{1}{\ell!}\;
   \Big(\sum_{|\alpha|=\ell}
   \Norm{r_\bc^{\beta_\bc+2+|\alpha|}\,
   \rho_{\bc\be}^{\max\{\beta_\be+2+|\alpha_\perp|,\,0\}} \partial^\alpha_\bx \widehat\bff}{0;
   \,\widehat\cV\ee'}^2\Big)^{\frac12} \\[-1.5ex]
   +  \Big(\sum_{|\alpha|\le m}\Norm{r_\bc^{\beta_\bc+|\alpha|}\,
   \rho_{\bc\be}^{\max\{\beta_\be+|\alpha|,\,0\}}
   \partial^\alpha_\bx \widehat\bu}{0;\,\widehat\cV\ee'}^2\Big)^{\frac12}
   \Big\}.
\end{multline*}
Owing to the homogeneity of the weights with respect to $r_\bc$, the dyadic partition argument yields the desired edge-vertex estimate, which allows to conclude the proof of the theorem.
\end{proof}

\begin{remark}
\label{5R2}
\iti1 If we replace Assumption \REF{5GB} by Assumption \REF{5GA} for edges $\be$ in the flagged subset $\sE_0$, we can prove, instead of \eqref{5ENB}, the implications
\begin{equation}
\begin{gathered}
\label{5ENBb}
   \bu\in\bJ^m_{\betab}(\Omega;\sC,\sE_0) \ \ \mbox{and}\ \
   \bff\in\bN^n_{\betab+2}(\Omega;\sC,\sE_0)
   \quad\Longrightarrow\quad
   \bu\in\bN^n_{\betab}(\Omega;\sC,\sE_0),
\\
   \bu\in\bJ^m_{\betab}(\Omega;\sC,\sE_0) \ \ \mbox{and}\ \
   \bff\in\bB_{\betab+2}(\Omega;\sC,\sE_0)
   \quad\Longrightarrow\quad
   \bu\in\bB_{\betab}(\Omega;\sC,\sE_0).
\end{gathered}
\end{equation}
\iti2 Under Assumption \REF{5GB}, the implications in the maximal non-homogeneous spaces, i.e., with $\sC_0=\sE_0=\varnothing$, are also true:
\begin{equation}
\begin{gathered}
\label{5ENBc}
   \bu\in\bJ^m_{\betab}(\Omega) \ \ \mbox{and}\ \
   \bff\in\bN^n_{\betab+2}(\Omega)
   \quad\Longrightarrow\quad
   \bu\in\bN^n_{\betab}(\Omega),
\\
   \bu\in\bJ^m_{\betab}(\Omega) \ \ \mbox{and}\ \
   \bff\in\bB_{\betab+2}(\Omega)
   \quad\Longrightarrow\quad
   \bu\in\bB_{\betab}(\Omega).
\end{gathered}
\end{equation}
If $\beta_\bc>-\frac32$ for any corner $\bc$, the statements \eqref{5ENB} and \eqref{5ENBc} coincide, since in this case the spaces $\bJ^m_{\betab}(\Omega;\sC,\varnothing)$ and $\bJ^m_{\betab}(\Omega)$ are the same (consequence of Hardy's inequality). In the general case \eqref{5ENBc} can be proved by two different methods:
\begin{itemize}
\item Deduced from \eqref{5ENB} by an argument of corner asymptotics (at each corner, the asymptotics modulo $\bJ^m_{\betab}(\Omega;\sC,\varnothing)$ contains only polynomials): For instance when $m=1$, if $\beta_\bc\in(-\frac52,-\frac32)$ for all corners $\bc$, any element of $\bu\in\bJ^m_{\betab}(\Omega)$ splits as
\[
   \bu = \bu_\bc + \bw_\bc \quad\mbox{in}\quad\Omega_\bc,\quad\mbox{with}\quad
   \bu_\bc\in\bJ^m_{\betab}(\Omega;\sC;\varnothing),\quad
   \bw_\bc\in\C^N,
\]
and we can apply \eqref{5ENB} locally near each corner, to each function $\bu_\bc$.
\item Directly proved by the same method as for Theorem \REF{5T2}, starting with the reference estimate for $k\ge m$
\begin{multline*}
   \quad\frac{1}{k!}\;
   \Big(\sum_{|\alpha|=k}\Norm{r_\be^{\max\{\beta_\be+|\alpha_\perp|,\,0\}}
   \partial^\alpha_\bx \widehat\bu}{0;\,\widehat\cV}^2\Big)^{\frac12} \le
    C^{k+1} \Big\{ \\[-1.5ex]
   \quad\sum_{\ell=m-1}^{k} \frac{1}{\ell!}\;
   \Big(\sum_{|\alpha|=\ell}
   \Norm{r_\be^{\max\{\beta_\be+2+|\alpha_\perp|,\,0\}} \partial^\alpha_\bx \widehat\bff}{0;
   \,\widehat\cV\ee'}^2\Big)^{\frac12}\\[-1.5ex]
   +  \Big(\sum_{|\alpha|= m}\Norm{r_\be^{\max\{\beta_\be+|\alpha|,\,0\}}
   \partial^\alpha_\bx \widehat\bu}{0;\,\widehat\cV\ee'}^2\Big)^{\frac12}
   \Big\},
\end{multline*}
instead of \eqref{5E20}: The $\rJ^m_{\beta_\be}$ norm present in \eqref{5E20} is replaced here by the corresponding semi-norm, {\em cf.}\ Corollary \REF{1C1}.
\end{itemize}
\end{remark}

%%%%%%%%%%%%%%%%%%%%%%%%%%%%%%%%%%%%%%%%%%%%%%%%%%%%%%%%%%%%%%%%%%%%%%%%%%%%%%%%%%%%%%%%
\section{Analytic weighted regularity for solutions of coercive problems}
\label{sec6}
%%%%%%%%%%%%%%%%%%%%%%%%%%%%%%%%%%%%%%%%%%%%%%%%%%%%%%%%%%%%%%%%%%%%%%%%%%%%%%%%%%%%%%%%
In this section, we show how Theorems \REF{3T1} and \REF{3T1b} in the polygonal case, or Theorems \REF{5T1} and \REF{5T2} in the polyhedral case, apply to solutions of variational problems. 
For second order boundary boundary value problems that allow a coercive variational formulation, one knows basic regularity in weighted Sobolev spaces in a form that fits the hypotheses of our natural regularity shift results. For polygons, this is obtained by means of Kondrat'ev's classical theory, and for polyhedra, such results were proved by Maz'ya and Rossmann \cite{MazyaRossmann03}. As a consequence, we obtain analytic regularity for solutions of variational problems.

Let $\Omega$ be a polygon or a polyhedron. In coherence with the previous sections, we consider a sesquilinear form $\ra$, homogeneous of order $1$ and with constant coefficients acting on vector-valued functions with $N$ components
\begin{equation}
\label{6E1}
   \ra(\bu,\bv) = \sum_{i=1}^N \sum_{j=1}^N \sum_{|\alpha|=1} \sum_{|\gamma|=1}
   \int_\Omega a^{\alpha\gamma}_{ij} \,
   \partial^\alpha_\bx u_j(\bx)\, \partial^\gamma_\bx \ov v_i(\bx)\,\rd\bx,
\end{equation}
and a subspace $\bV$ of $\rH^1(\Omega)^N=:\bH^1(\Omega)$ defined by essential boundary conditions on the sides $\Gamma_\bs$ of $\Omega$
\begin{equation}
\label{6E4}
   \bV = \{ \bu\in\bH^1(\Omega):\  D_\bs\bu = 0
   \;\mbox{ on }\; \Gamma_\bs,\ \bs\in\sS\}.
\end{equation}
We assume that the form $\ra$ is \emph{coercive} on $\bV$:
\[
   \exists c,\,C>0,\quad\forall\bu\in\bV,\quad
   \Re\ra(\bu,\bu) \ge c\Norm{\bu}{1;\ee\Omega}^2 - C \Norm{\bu}{0;\ee\Omega}^2.
\]
Standard examples of such sesquilinear forms are
the gradient form for scalar functions
\[
   \ra_\nabla(u,v) =
   \int_\Omega \nabla u(\bx) \cdot \nabla \ov v(\bx)\,\rd\bx
\]
and the stress-strain sesquilinear forms in linear elasticity:
\[
   \ra_{\sf ela} = \int_\Omega \sigma (\bu)(\bx) : \ov{\varepsilon (\bv)}(\bx)\,\rd\bx,
\]
where $\varepsilon$ is the symmetrized gradient tensor and $\sigma=A\varepsilon$, where $A$ is a material tensor with the usual symmetry and positivity properties.
Variational spaces $\rV$ on which $\ra_\nabla$ is coercive can be defined by any subset $\sS_D$ of the set of sides $\sS$:
\[
   \rV = \{u\in\rH^1(\Omega)\ :\ u\on{\Gamma_\bs}=0\quad\forall\bs\in\sS_D\}.
\]
As for $\ra_{\sf ela}$ we can take for $\bV$ any space of the type
\begin{multline}
\label{6Etan}
   \bV = \{\bu\in\bH^1(\Omega)\ :\ \bu\on{\Gamma_\bs}=0 \quad \forall\bs\in\sS_D,\quad
   \bu\cdot\bn\on{\Gamma_\bs}=0 \quad \forall\bs\in\sS_T\\
   \mbox{and}\quad
   \bu\times\bn\on{\Gamma_\bs}=0\quad \forall\bs\in\sS_N\},
\end{multline}
where $\bn$ is the outward unit normal vector to $\Gamma_\bs$, and $\sS_D$, $\sS_T$, and $\sS_N$ are disjoint subsets of $\sS$. As a consequence of Korn's inequality, $\ra_{\sf ela}$ is coercive on such spaces $\bV$.

We consider the variational problem
\begin{equation}
\label{6Evar}
   \mbox{\sl Find }\quad \bu\in\bV\quad\mbox{\sl such that} \quad
   \forall\bv\in\bV,\quad
   \ra(\bu,\bv) = \int_{\Omega} \bff\,\ov\bv \ \rd\bx\,.
\end{equation}
With $L$, $T_\bs$ and $D_\bs$ defined in an obvious way, solutions of \eqref{6Evar} satisfy
\begin{equation}
\label{6Ebvp}
   \left\{ \begin{array}{rclll}
   L\,\bu &=& \bff \quad & \mbox{in}\ \ \Omega \\[0.3ex]
   T_\bs\,\bu &=& 0  & \mbox{on}\ \ \Gamma_\bs, &\bs\in\sS, \\[0.3ex]
   D_\bs\,\bu &=& 0  & \mbox{on}\ \ \Gamma_\bs &\bs\in\sS.
   \end{array}\right.
\end{equation}

Having the natural analytic regularity shift results of Theorems \REF{3T1}, \REF{3T1b}, \REF{5T1} and \REF{5T2} at hand, the issue is to find suitable exponents $\betab$ so that
\begin{enumerate}
\item $\bA_\betab(\Omega)$ or $\bB_\betab(\Omega)$ are compactly embedded in $\bH^1(\Omega)$, --- in order to be useful in error analysis for example.
\item Variational solutions  $\bu$ with sufficiently smooth right hand sides belong to $\bK^1_\betab(\Omega)$ or $\bJ^1_\betab(\Omega)$.
\end{enumerate}

Condition (1) of compact embedding is satisfied on two- and three-dimensional domains for all $\betab<-1$ (this means that all components $\beta_\bc$ and $\beta_\be$ are $<-1$). 
% This is the reason why we exhibit weights of the form $\betab=-\ub-1$ 
% with $\ub>0$ in the statements below.

Condition (2) of initial regularity is the main question discussed in the rest of this section.

%%%%%%%%%%%%%%%%%%%%%%%%%%%%%%%%%%%%%%%%%%%%%%%%%%%%%%%%%%%%%%%%%%%%%%%%%%%%%%%%%%%%%%%%
\subsection{Regularity of variational solutions in polygons}
%%%%%%%%%%%%%%%%%%%%%%%%%%%%%%%%%%%%%%%%%%%%%%%%%%%%%%%%%%%%%%%%%%%%%%%%%%%%%%%%%%%%%%%%
Let $\Omega$ be a polygon with vertices $\bc\in\sC$. We recall that $\Omega_\bc$ denotes a  neighborhood of $\bc$ satisfying \eqref{3E1a}-\eqref{3E1b}, $\cK_\bc$ is the infinite sector which coincides with $\Omega_\bc$ near $\bc$, and $(r_\bc,\theta_\bc)$ are polar coordinates centered at $\bc$. Finally let $G_\bc$ denote the set of corresponding angles $\theta_\bc$. Denoting by $\sS_\bc$ the set of face indices such that $\bc$ belongs to the closure of $\Gamma_\bs$, the localized version of problem \eqref{6Ebvp} near the corner $\bc$ is
\begin{equation}
\label{6EbvpKc}
   \left\{ \begin{array}{rclll}
   L\,\bu &=& \bff \quad & \mbox{in}\ \ \cK_\bc \\[0.3ex]
   T_\bs\,\bu &=& 0  & \mbox{on}\ \ \Gamma_\bs, &\bs\in\sS_\bc, \\[0.3ex]
   D_\bs\,\bu &=& 0  & \mbox{on}\ \ \Gamma_\bs &\bs\in\sS_\bc.
   \end{array}\right.
\end{equation}

The standard Sobolev space $\rH^1(\Omega)$ coincides with $\rJ^1_{-1}(\Omega)$, see \eqref{3E3}. From Remark \REF{2R1}, we know that for the comparison of $\rJ^1_{-1}(\Omega)$ with $\rK^1_{-1}(\Omega)$ we are in a critical case, namely a function $u\in\rH^1(\Omega)$ neither has point values at corners nor satisfies  $r^{-1}_\bc u\in\rL^2(\Omega)$   in general (see \cite{KozlovMazyaRossmann97b}). There holds
$$\rK^1_{-1}(\Omega)\subset \rJ^1_{-1}(\Omega) \subset \rK^1_{-1+\varepsilon}(\Omega),\quad\forall\varepsilon>0.$$

Taking the essential boundary conditions into account that define the variational space $\bV\subset \bH^{1}(\Omega)$, one will sometimes find that $\bV$ is embedded in $\bK^1_{-1}(\Omega)$. This happens in particular if each corner lies on at least one side on which Dirichlet conditions are imposed. In the general case, one will just have
$\bV\subset\bK^1_{-1+\varepsilon}(\Omega)$ for all $\varepsilon>0$. Necessary and sufficient conditions for the embedding $\bV\subset\bK^1_{-1}(\Omega)$ are discussed in \cite[Ch.\,14]{GLC}.

\subsubsection{Case $\bV\subset\bK^1_{-1}(\Omega)$  (homogeneous norms)} In this case the analytic regularity shift \eqref{3EA3} in classes $\bA_\betab(\Omega)$ can be applied to variational solutions with well chosen weight exponents $\beta_\bc<-1$ as we explain now. For each corner $\bc$, the optimal condition on $\beta_\bc$ is related to the spectrum $\sigma(\gA_\bc)$ of the ``Mellin symbol'' $\gA_\bc$ of the system $(L,T_\bs,D_\bs)$\footnote{$\gA_\bc$ is also called ``operator pencil'' generated by the system $(L,T_\bs,D_\bs)$.} at $\bc$ (see \cite{Kondratev67, KozlovMazyaRossmann97b}): 
\begin{multline}
\label{6Esymb}
   \sigma(\gA_\bc) = \big\{\lambda\in\C,\quad\exists\varphif\in\bH^1(G_\bc),\ \varphif\neq0,
   \ \mbox{such that} \\
   \bu:=r_\bc^\lambda\varphif(\theta_\bc) \ 
   \mbox{solves problem \eqref{6EbvpKc} with $\bff=0$ on $\cK_\bc$}\big\}.
\end{multline}
Then we define $b_\bc(\Omega,\ra,\bV)$ as the supremum of the numbers $b>0$ such that
\begin{equation}
\label{6Ebc2}
    \{\lambda\in\C: 0<\Re\lambda< b\} \cap \sigma(\gA_\bc) = \varnothing.
\end{equation}
As a consequence of the coercivity of the form $\ra$ on $\bV$, the number $b_\bc(\Omega,\ra,\bV)$ is positive. 

\begin{theorem}
\label{6T1}
Let $\Omega$ be a polygon. We assume that the form $\ra$ is coercive on $\bV$, and that $\bV\subset\bK^1_{-1}(\Omega)$. If the following condition holds for the exponents $\beta_\bc$
\begin{equation}
\label{6Eco1}
   0\le -\beta_\bc-1 < b_\bc(\Omega,\ra,\bV)\quad\forall\bc\in\sC
\end{equation}
then any solution $\bu\in\bV$ of the variational problem \eqref{6Evar} satisfies the regularity result:
\begin{equation}
\label{6EA}
   \bff\in\bA_{\betab+2}(\Omega) 
   \quad\Longrightarrow\quad \bu\in\bA_{\betab}(\Omega).
\end{equation}
\end{theorem}

\begin{proof}
Invoking the general theory of corner problems in the variational setting, we know that if \eqref{6Eco1} holds, then
\begin{equation}
\label{6EK}
   \bff\in \bK^0_{\betab+2}(\Omega)
   \ \ \Longrightarrow\ \
   \bu\in\bK^2_{\betab}(\Omega).
\end{equation}
The proof of this essentially goes back to Kondrat'ev \cite{Kondratev67}, see also \cite[Ch.\,10]{GLC} for more details on the application of Kondrat'ev's technique to variational problems. Then \eqref{6EA} is a consequence of \eqref{3EA3} and \eqref{6EK}.
\end{proof}

\begin{example}
\label{6X1} Let us consider the gradient form $\ra=\ra_\nabla$ on
scalar functions. The associated operator is the Laplacian $\Delta$.
Let $\omega_\bc$ be the opening of $\Omega$ near the vertex $\bc$
and denote by $\Gamma^i_\bc$, $i=1,2$, the two sides of $\Omega$
containing $\bc$.
\iti1 For the Dirichlet problem, we have $\bV\subset\bK^1_{-1}(\Omega)$ and
\[
   b_\bc(\Omega,\ra_\nabla,\rH^1_0) = \frac\pi{\omega_\bc}\,,\quad \bc\in\sC.
\]
\iti2 In the mixed Neumann-Dirichlet case, if at all corners Dirichlet conditions are imposed on at least one side containing $\bc$, we still have $\bV\subset\bK^1_{-1}(\Omega)$ and
\[
   b_\bc(\Omega,\ra_\nabla,\rV) =  
   \frac\pi{\omega_\bc}\ \ \ \mbox{if}\ \  \ \bc\in\sC_D \quad\mbox{and}\quad
   b_\bc(\Omega,\ra_\nabla,\rV) = \frac\pi{2\omega_\bc}\ \ \ \mbox{if}\ \  \ \bc\in\sC_M,
\]
where $\sC_D$ is the set of Dirichlet corners $\bc$ (Dirichlet conditions on both sides $\Gamma^i_\bc$) and $\sC_M$ the set of ``Mixed'' corners $\bc$ (Dirichlet conditions on only one side $\Gamma^i_\bc$).
\end{example}

\subsubsection{General case  (non-homogeneous norms)} If $\bV\not\subset\bK^1_{-1}(\Omega)$ or for more general data with a nonzero Taylor expansion at corners, it is advantageous to use the analytic regularity shift  \eqref{3EB3} in classes $\bB_\betab(\Omega)$. Let us recall from formula \eqref{3EBAP} that for $-\beta_\bc-1\in(k,k+1)$   (with a natural number $k$): 
\[
   \bB_{\betab}(\Omega_\bc) = \bA_{\betab}(\Omega_\bc) \oplus (\P^{k})^N.
\]

\begin{theorem}
\label{6T2}
Let $\Omega$ be a polygon. We assume that the form $\ra$ is coercive on $\bV$. If condition \eqref{6Eco1} holds for the exponents $\beta_\bc$,
then any solution $\bu\in\bV$ of the variational problem \eqref{6Evar} satisfies the regularity result:
\begin{equation}
\label{6EB}
   \bff\in\bB_{\betab+2}(\Omega)
   \ \ \Longrightarrow\ \
   \bu\in\bB_{\betab}(\Omega).
\end{equation}
\end{theorem}
\begin{proof}
The proof relies on regularity results in spaces with non-homogeneous norms: By a modification of Kondrat'ev's method, see \cite{MazyaPlamenevskii84b, KozlovMazyaRossmann97b} and \cite{Dauge88}, one can prove that if \eqref{6Eco1} holds, for any $m\ge\max\{-\beta_\bc\}$ we have the implication
\begin{equation}
\label{6EJ}
   \bff\in\bJ^{m-2}_{\betab+2}(\Omega)
   \ \ \Longrightarrow\ \
   \bu\in\bJ^m_{\betab}(\Omega)
\end{equation}
for variational solutions: In addition to the standard theory, polynomial right-hand sides of degree $[-\beta_\bc-1]-2$ at each corner $\bc$ have to be taken into account. In dimension two of space, the condition that   the problem \eqref{6EbvpKc}   with a polynomial $\bff$ of degree $[-\beta_\bc-1]-2$ has a polynomial solution $\bu$ on the infinite cone $\cK_\bc$ is a consequence of the condition $0\le -\beta_\bc-1 < b_\bc(\Omega,\ra,\bV)$.
A complete proof in this framework is presented in \cite[Ch.\,13 \& 14]{GLC}. 
Then \eqref{6EB} is a consequence of \eqref{3EB3}, and \eqref{6EJ}.
\end{proof}

\begin{example}
\label{6X2}
Let us come back to the gradient form $\ra=\ra_\nabla$ on scalar functions.
For {\em any} mixed Neumann-Dirichlet problem, including the pure Neumann problem, Theorem \REF{6T2} is valid and we find
\[
   b_\bc(\Omega,\ra_\nabla,\rV) =  
   \frac\pi{\omega_\bc}\ \ \ \mbox{if}\ \  \ \bc\in\sC_D\cup\sC_N \quad\mbox{and}\quad
   b_\bc(\Omega,\ra_\nabla,\rV) = \frac\pi{2\omega_\bc}\ \ \ \mbox{if}\ \  \ \bc\in\sC_M,
\]
where $\sC_D$ is the set of Dirichlet corners,  $\sC_N$ is the set of Neumann corners, and $\sC_M$ the set of ``Mixed'' corners $\bc$.
Thus $b_\bc(\Omega,\ra_\nabla,\rV)$ will always be greater than $\frac14$. For the pure Dirichlet or pure Neumann problem on a convex polygon, it will be greater than $1$, and for some triangles even greater than $2$, but never greater than $3$.
\end{example}

\begin{remark}
\label{6R2}
Theorem \REF{6T2} has to be compared with earlier results by Babu\v{s}ka and Guo\footnote{When $-\beta\in(1,2)$, our space $\rB_{\beta}(\Omega)$ coincides with their space $B^2_{\beta+2}(\Omega)$.}: The Laplace operator with non-homogeneous mixed boundary conditions is considered in \cite{BabuskaGuo88,BabuskaGuo89}; more general scalar second order operators with analytic coefficients are addressed in \cite{BabuskaGuo88b}, and finally the Lam\'e system of linear elasticity with non-homogeneous mixed Dirichlet-Neumann boundary conditions is investigated in
\cite{GuoBabuska93}. These results are at the same time more general than Theorem \REF{6T2} since they address non-homogeneous boundary conditions and variable coefficients, but more restrictive since they do not include a full class of coercive second order systems with a unified approach. 
\end{remark}

%%%%%%%%%%%%%%%%%%%%%%%%%%%%%%%%%%%%%%%%%%%%%%%%%%%%%%%%%%%%%%%%%%%%%%%%%%%%%%%%%%%%%%%%
\subsection{Regularity of variational solutions in polyhedra}
\label{sec62}
%%%%%%%%%%%%%%%%%%%%%%%%%%%%%%%%%%%%%%%%%%%%%%%%%%%%%%%%%%%%%%%%%%%%%%%%%%%%%%%%%%%%%%%%
Let $\Omega$ be a polyhedron with edges $\be\in\sE$ and corners $\bc\in\sC$. We recall from \eqref{5Ea}-\eqref{5E1} the edge neighborhoods $\Omega_\be$ and the corner neighborhoods $\Omega\cap\cB(\bc,\varepsilon)$. Then $\cK_\bc$ is the infinite cone which coincides with $\Omega$ in $\cB(\bc,\varepsilon)$, and $G_\bc$ denotes the set of corresponding solid angles $\theta_\bc=(\bx-\bc)r^{-1}_\bc$. For any edge $\be$, let $\cW_\be$ be the wedge coinciding with $\Omega$ in $\Omega_\be$ and $\cK_\be$ be the plane sector such that such that $\cW_\be \cong \cK_\be\times\R$.

The comparison between the variational space $\bV$ and weighted spaces $\rK^1_\betab(\Omega)$ and $\rJ^1_\betab(\Omega)$, {\em cf.}\ \eqref{5E1wsK} and \eqref{5E1wsJ}, still involves the multi-exponent $\beta_\bc=\beta_\be=-1$ and essential boundary conditions: We have
\[
   \rJ^1_{-1}(\Omega) = \rH^1(\Omega)
\]
and, in the Dirichlet case
\[
   \rH^1_0(\Omega) \subset \rK^1_{-1}(\Omega).
\]
Moreover, the intermediate space
\[
   \rJ^1_{-1}(\Omega;\sC,\varnothing) = \big\{u\in\rH^1(\Omega)\ :\
   r_\bc^{-1}u\in \rL^2(\Omega) \quad
   \forall\bc\in \sC\big\},
\]
also coincides with $\rH^1(\Omega)$ by virtue of Hardy's inequality in three-dimensional cones.

If we want to establish that weighted analytic regularity results hold in polyhedra,
we have two tasks: 
\begin{enumerate}
\item  Verify Assumptions \REF{5GA} or \REF{5GB}, which are closed range properties along the edges,
\item  Give conditions for variational solutions to belong to spaces $\bK^1_\betab(\Omega)$ or $\bJ^1_\betab(\Omega)$.
\end{enumerate}

As a matter of fact, the condition which ensures the regularity of
variational solutions {\em implies} Assumptions \REF{5GA} or
\REF{5GB}. Hence we focus on conditions for the regularity. There
are not so many results on regularity for elliptic boundary value
problems in polyhedra. Let us quote
\cite{MazyaPlamenevskii73,MazyaPlamenevskii77} for early results in
general $n$-dimensional polyhedral domains in spaces of $\rK$ type,
\cite{Dauge88} in $n$-dimensional polyhedral domains in standard
Sobolev spaces, and more recently \cite{MazyaRossmann03} in
$3$-dimensional polyhedral domains in spaces
$\rJ^n_\betab(\sC,\sE_0)$, {\em cf.}\ Remark \REF{5R1} \itj2.

The latter results, especially \cite[Thms.\,7.1\,\&\,7.2]{MazyaRossmann03}, fit exactly our requirements, namely in the form \eqref{5ENBb}. For this reason we formulate Theorem \REF{6T3}   in the somewhat restricted framework of \cite{MazyaRossmann03}:   that is mixed Dirichlet-Neumann boundary conditions for second order systems.

The regularity conditions depend on the position of the spectra $\sigma(\gA_\bc)$ and $\sigma(\gA_\be)$ of the Mellin symbols $\gA_\bc$ and $\gA_\be$ of the system $(L,T_\bs,D_\bs)$ at the corners $\bc$ and the edges $\be$, respectively. The set $\sigma(\gA_\bc)$ is defined by \eqref{6Esymb} on the three-dimensional cone $\cK_\bc$. The set $\sigma(\gA_\be)$ is defined similarly by problem \eqref{4Ebvpxi} for $\xi=0$ posed on the plane sector $\cK_\be$, namely as the set of exponents $\lambda\in\C$ such that the totally homogeneous problem on $\cK_\be$ has a nontrivial solution homogeneous of degree $\lambda$.  

\begin{definition}
\label{6D1}
Let $\sigma(\gA_\bc)$ and $\sigma(\gA_\be)$ denote the spectrum of the Mellin symbol $\gA_\bc$ and $\gA_\be$ of the system $(L,T_\bs,D_\bs)$ at the corner $\bc$ and the edge $\be$, respectively.
Then for any edge $\be$ we define $b_\be(\Omega,\ra,\bV)$ as the supremum of the numbers $b>0$ such that
\[
   \{\lambda\in\C\ :\ 0<\Re\lambda< b\} \cap \sigma(\gA_\be) = \varnothing
\]
and, for any corner $\bc$, $b_\bc(\Omega,\ra,\bV)$ as the supremum of the numbers $b$ such that
\[
   \{\lambda\in\C\ :\
   -\tfrac12<\Re\lambda< b\} \cap \sigma(\gA_\bc) = \varnothing.
\]
\end{definition}

\begin{remark}
\label{6R4}
It is a consequence of the coercivity of the form $\ra$ that the numbers
$b_\be(\Omega,\ra,\bV)$ are positive and $b_\bc(\Omega,\ra,\bV)>-\frac12$.
\end{remark}

With these notations at hand, we can state

\begin{theorem}
\label{6T3} 
We consider a mixed Dirichlet-Neumann problem \eqref{6Evar}, which means that the variational space is of the form
\[
   \bV = \{\bu\in\bH^1(\Omega): \bu\on{\Gamma_\bs}=0,\ \bs\in \sS_D \}.
\]
Let $\sE_0$ be the set of edges $\be$ which are the sides of faces $\Gamma_\bs$ with $\bs\in\sS_D$. We assume that the form $\ra$ \eqref{6E1} is coercive on $\bV$.
If the following condition holds for the exponents $\beta_\be$ and $\beta_\bc$
\begin{equation}
\label{6Eco2}
\begin{cases}
   0\le -\beta_\be-1 < b_\be(\Omega,\ra,\bV) \quad&\forall\be\in\sE
   \\
   -\tfrac12\le -\beta_\bc-\tfrac32< b_\bc(\Omega,\ra,\bV) \quad&\forall\bc\in\sC
\end{cases}
\end{equation}
then any solution of the variational problem \eqref{6Evar} satisfies the regularity result:
\begin{equation}
\label{6EC} 
   \bff\in\bB_{\betab+2}(\Omega;\sC,\sE_0)
   \ \ \Longrightarrow\ \
   \bu\in\bB_{\betab}(\Omega;\sC,\sE_0).
\end{equation}
\end{theorem}

\begin{proof}
First, the Fredholm Theorem 7.2 of \cite{MazyaRossmann03} guarantees
that the Assumptions \REF{5GA} (if $\be\in\sE_0$) and \REF{5GB} (if $\be\in\sE\setminus\sE_0$) are  satisfied 
for any $\beta_\be$ satisfying \eqref{6Eco2}. Second,
the regularity Theorem 7.1 of \cite{MazyaRossmann03} shows   for any $m\ge\max\{-\beta_\be\}$   the regularity $\bu\in\bJ^m_{\betab}(\Omega;\sC,\sE_0)$ with $\betab$ satisfying \eqref{6Eco2}.
Hence the conclusion follows from Theorem \REF{5T2} extended by Remark
\REF{5R2}  --- in particular, implication \eqref{5ENBb}.
\end{proof}

\subsubsection{Dirichlet case (homogeneous norms)} 
As a consequence of the fact that $\bB_{\betab}(\Omega;\sC,\sE)=\bA_\betab(\Omega)$ (Remark \REF{5R1} \itj1) we immediately obtain a regularity result in the scale $\bA_\betab$ for the Dirichlet problem.

\begin{corollary}
\label{6T4}
When $\bV=\bH^1_0(\Omega)$ (Dirichlet problem), assuming that the form $\ra$ \eqref{6E1} is coercive on $\bV$, if condition \eqref{6Eco2} holds, then any solution of the variational problem \eqref{6Evar} satisfies the regularity result:
\begin{equation}
\label{6EC4} 
   \bff\in\bA_{\betab+2}(\Omega)
   \ \ \Longrightarrow\ \
   \bu\in\bA_{\betab}(\Omega).
\end{equation}
\end{corollary}

\begin{example}
\label{6X3}
For the gradient form $\ra_\nabla$ (for which $L$ is the Laplace operator) on $\rH^1_0(\Omega)$, the spectrum of the edge Mellin symbol $\gA_\be$ is 
\begin{equation}
\label{6Espde}
   \sigma(\gA_\be) = \Big\{ \frac{\ell\pi}{\omega_\be}\,,\quad \ell\in\Z\setminus\{0\} 
   \Big\}
\end{equation}
and the spectrum of the corner Mellin symbol $\gA_\bc$ is 
\begin{equation}
\label{6Espdc}
   \sigma(\gA_\bc) = \Big\{ -\frac12 \pm \sqrt{\mu^\Dir_{\bc,n}+\frac14}\,,\quad 
   n\in\N 
   \Big\}
\end{equation}
where $\mu^\Dir_{\bc,n}$ (for $n\ge1$) is the $n$-th eigenvalue of the Laplace-Beltrami operator with Dirichlet conditions on the spherical cap $G_\bc$. Hence
\[
   b_\be(\Omega,\ra_\nabla,\rH^1_0(\Omega)) = \frac\pi{\omega_\be} 
   \quad\mbox{and}\quad
   b_\bc(\Omega,\ra_\nabla,\rH^1_0(\Omega)) = 
   -\frac12 + \sqrt{\mu^\Dir_{\bc,1}+\frac14}\,.
\]
\end{example}

\subsubsection{Neumann case (non-homogeneous norms)} 
\label{7Simp}
For the Neumann problem, it is adequate to use the full spaces $\bB_{\betab}(\Omega)$ instead of $\bB_{\betab}(\Omega;\sC,\varnothing)$ as in Theorem \REF{6T3}, --- in the Neumann case $\sE_0$ is empty. These two families of spaces differ by the non-zero Taylor expansions at corners for the elements of $\bB_{\betab}(\Omega)$. 

For each corner $\bc$ the optimal condition on $\beta_\bc$ relating to spaces $\bJ^m_\betab$ and $\bB_\betab$ relies on the condition of {\em injectivity modulo polynomials} \cite{Dauge88,GLC}: The spectrum $\sigma(\gA_\bc)$ has to be replaced by the set $\sigma_\star(\gA_\bc)$ of complex $\lambda$'s for which the condition of injectivity modulo polynomials does not hold. This means that there exists a non-polynomial function 
\[
   \bu = \sum_{q=0}^Q r_\bc^\lambda \log^qr_\bc \varphif_q(\theta_\bc),\quad \varphif_q\in\bH^1(G_\bc)
\] 
solution of the problem \eqref{6EbvpKc} with a polynomial right hand side $\bff$ on the infinite three-dimensional cone $\cK_\bc$.
Note that this condition may differ from the condition in  \eqref{6Esymb} only for integer $\lambda$:
\[
   \sigma(\gA_\bc)\setminus\N = \sigma_\star(\gA_\bc)\setminus\N\,.
\]
Then $b^\star_\bc(\Omega,\ra,\bV)$ is defined as the supremum of the numbers $b$ such that
\[
   \{\lambda\in\C: -\tfrac12<\Re\lambda< b\} \cap \sigma_\star(\gA_\bc) = \varnothing.
\]

\begin{theorem}
\label{6T5}
When $\bV=\bH^1(\Omega)$ (Neumann problem), assuming that the form $\ra$ \eqref{6E1} is coercive on $\bV$, if the following condition holds for the exponents $\betab$
\begin{equation}
\label{6Eco5}
\begin{cases}
   0\le -\beta_\be-1 < b_\be(\Omega,\ra,\bV) \quad&\forall\be\in\sE
   \\
   -\tfrac12\le -\beta_\bc-\tfrac32< b^\star_\bc(\Omega,\ra,\bV) \quad&\forall\bc\in\sC
\end{cases}
\end{equation}
 then any solution of the variational problem \eqref{6Evar} satisfies the regularity result:
\begin{equation}
\label{6EC5} 
   \bff\in\bB_{\betab+2}(\Omega)
   \ \ \Longrightarrow\ \
   \bu\in\bB_{\betab}(\Omega).
\end{equation}
\end{theorem}

\begin{proof}
The proof follows the same steps as for Theorem \REF{6T4}, but with one difference: Instead of relying directly on \cite[Thm 7.1]{MazyaRossmann03}, by a modification of this statement we prove with the corner Mellin transform that condition \eqref{6Eco5} implies for any $m\ge\max\{-\beta_\be,-\beta_\bc\}$ the regularity $\bu\in\bJ^m_{\betab}(\Omega)$.
Once more, the conclusion then follows from Theorem \REF{5T2} extended by Remark
\REF{5R2} --- now, implication \eqref{5ENBc}.
\end{proof}

\begin{example}
\label{6X4}
For the gradient form $\ra_\nabla$ on $\rH^1(\Omega)$, the spectrum of the edge Mellin symbol $\gA_\be$ is 
\begin{equation}
\label{6Espne}
   \sigma(\gA_\be) = \Big\{ \frac{\ell\pi}{\omega_\be}\,,\quad \ell\in\Z 
   \Big\}
\end{equation}
and the spectrum of the corner Mellin symbol $\gA_\bc$ is 
\begin{equation}
\label{6Espnc}
   \sigma(\gA_\bc) = \Big\{ -\frac12 \pm \sqrt{\mu^\Neu_{\bc,n}+\frac14}\,,\quad 
   n\in\N 
   \Big\}
\end{equation}
where $\mu^\Neu_{\bc,n}$ (for $n\ge1$) is the $n$-th eigenvalue of the Laplace-Beltrami operator with Neumann conditions on the spherical cap $G_\bc$. Since $\mu^\Neu_{\bc,1}=0$, the set $\sigma(\gA_\bc)$ contains $0$. But one can show that the condition of injectivity modulo polynomials is satisfied in $\lambda=0$, and that it is also satisfied in $\lambda=1$ if $1\not\in\sigma(\gA_\bc)$.
Hence we deduce
\[
   b_\be(\Omega,\ra_\nabla,\rH^1(\Omega)) = \frac\pi{\omega_\be} 
   \quad\mbox{and}\quad
   b^\star_\bc(\Omega,\ra_\nabla,\rH^1(\Omega)) \ge 
   \min\Big\{2,-\frac12 + \sqrt{\mu^\Neu_{\bc,2}+\frac14}\Big\}\,.
\]
\end{example}

\subsubsection{A priori estimates along edges}
\label{6Sapriori}
We conclude this section by considerations about the nature of necessary and sufficient conditions ensuring the closed range properties along edges required by Assumptions \REF{5GA} or \REF{5GB}. As mentioned in Remarks \REF{4RGA} and \REF{4RGB}, a minimal condition for these assumptions to hold at a chosen edge $\be\in\sE$ is an injectivity and closed range condition for the Fourier symbol $(\hat L_\be(\xi),\hat T_{\be,\bs}(\xi), \hat D_{\be,\bs}(\xi))$ of the system $(L,T_\bs,D_\bs)$ on the plane sector $\cK_\be$:
\begin{equation}
\label{6Ebvpxi}
   \left\{ \begin{array}{rclll}
   \hat L_\be(\xi)\,\bu &=& \bff \quad & \mbox{in}\ \ \cK_\be \\[0.3ex]
   \hat T_{\be,\bs}(\xi)\,\bu &=& 0  & \mbox{on}\ \ \Gamma_\bs, &\bs\in\sS_\be, \\[0.3ex]
   \hat D_{\be,\bs}(\xi)\,\bu &=& 0  & \mbox{on}\ \ \Gamma_\bs, &\bs\in\sS_\be,
   \end{array}\right.
\end{equation}
Here $L_\be$ is the operator $L$ written in local Cartesian coordinates $(\bx^\perp_\be,x^\parallel_\be)\in\cK_\be\times\R$. The set $\sS_\be$ is the set of the two faces such that $\be\subset\ov\bs$, and for $\bs\in\sS_\be$, the boundary operators $T_{\be,\bs}$ and $D_{\be,\bs}$ are the local forms of $T_\bs$ and $D_\bs$, respectively.

\paragraph{Homogeneous norms} The necessary and sufficient conditions for Assumptions \REF{5GA} to hold is that $(\hat L_\be(\xi),\hat T_{\be,\bs}(\xi), \hat D_{\be,\bs}(\xi))$ defines an operator with trivial kernel and closed range from $\bE^2_{\beta_\be}(\cK_\be)$ into $\bE^0_{\beta_\be+2}(\cK_\be)$ for $\xi=\pm1$. The closed range condition is satisfied if and only if, cf \cite{MazyaPlamenevskii80b},
\begin{equation}
\label{6EclE}
   -\beta_\be-1\not\in \Re\big(\sigma(\gA_\be)\big) := 
   \{\eta\in\R:\quad\exists\lambda\in \sigma(\gA_\be)
   \ \ \mbox{with}\ \ \Re\lambda=\eta\}.
\end{equation}
In the variational case when the space $\bV$ is contained in $\bK^1_{-1}(\Omega)$, the trivial kernel condition is satisfied as soon as $-\beta_\be-1\ge0$, and even further, for all $\beta_\be$ such that $-\beta_\be-1> -b^-_\be(\Omega,\ra,\bV)$ where $b^-_\be(\Omega,\ra,\bV)$ is the supremum of the numbers $b>0$ such that
\[
   \{\lambda\in\C\ :\ -b<\Re\lambda< 0\} \cap \sigma(\gA_\be) = \varnothing.
\]
In the Laplace Dirichlet case, the conjunction of the two conditions is
\[
   -\beta_\be-1>-\frac{\pi}{\omega_\be} \quad\mbox{and}\quad
   -\beta_\be-1\not=\frac{\ell\pi}{\omega_\be} ,\ \ \ell\in\N (\ell\ge1).
\]

\paragraph{Non-homogeneous norms and Neumann case} Then the necessary and sufficient conditions for Assumptions \REF{5GB} to hold is that $(\hat L_\be(\xi),\hat T_{\be,\bs}(\xi), \hat D_{\be,\bs}(\xi))$ defines an operator with trivial kernel and closed range from $\bJ^2_{\beta_\be}(\cK_\be)$ into $\bJ^0_{\beta_\be+2}(\cK_\be)$ for $\xi=\pm1$. The closed range condition is \emph{implied} by \eqref{6EclE}. The optimal one has to be defined with the injectivity modulo polynomials. For the Neumann case, this makes a difference because $0$ belongs to the spectrum $\sigma(\gA_\be)$ and not to the star spectrum $\sigma^\star(\gA_\be)$ defined by the injectivity modulo polynomials. The optimal trivial kernel condition in the Neumann case is $-\beta_\be-1\ge0$.

More details in the forthcoming work \cite[Part III]{GLC}.

%%%%%%%%%%%%%%%%%%%%%%%%%%%%%%%%%%%%%%%%%%%%%%%%%%%%%%%%%%%%%%%%%%%%%%%%%%%%%%%%%%%%%%%%
\section{Extensions and generalizations}
\label{sec8}
%%%%%%%%%%%%%%%%%%%%%%%%%%%%%%%%%%%%%%%%%%%%%%%%%%%%%%%%%%%%%%%%%%%%%%%%%%%%%%%%%%%%%%%%
In this final section, we describe possible extensions and generalizations of our results. More or less straightforward extensions concern non-zero boundary conditions, non-con\-stant (analytic) coefficients in the two-dimensional polygonal case, and general boundary conditions in the polyhedral case. These situations can be handled with the techniques presented in the previous sections, and they were omitted here mainly for the sake of brevity. Generalizations that could be handled with similar methods, but would need further technical work, concern analytic coefficients in the polyhedral case, transmission problems, and higher order elliptic systems or systems elliptic in a more general sense. 

\subsection{Inhomogeneous boundary conditions and variable coefficients}
\label{ssInh&Var}
The fundamental estimate of Proposition~\REF{1T1} in the smooth case is already formulated in \eqref{1E1} for the situation of non-homogeneous boundary data. It is also  available for variable coefficients.
One has to introduce the trace spaces on the boundary that correspond to our function spaces on the domain. It is well known how to do this, and it is covered in the references given in section~\ref{sec1}. On the technical side, it is also known how to extend the method of dyadic partitions to the case of variable coefficients. Therefore our analytic regularity results in sections~\ref{sec3} and \ref{sec6} can be extended to cover the situation of elliptic systems with analytic coefficients on polygonal domains. Such results have been published by Babu\v{s}ka and Guo for a more restricted class of elliptic equations, see Remark~\ref{6R2}. 
Thus, on a polygonal domain $\Omega\subset\R^{2}$, we can consider a general boundary value problem
\begin{equation}
\label{7Ebvp}
   \left\{ \begin{array}{rclll}
   L\,\bu &=& \bff \quad & \mbox{in}\ \ \Omega, \\[0.3ex]
   T_\bs\,\bu &=& \bg_\bs  & \mbox{on}\ \ \Gamma_\bs, & \bs\in\sS\,, \\[0.3ex]
   D_\bs\,\bu &=& \bh_\bs  & \mbox{on}\ \ \Gamma_\bs, & \bs\in\sS\,,
   \end{array}\right.
\end{equation}
where $L=L(\bx;\partial_\bx)$ is a second order elliptic system, $T_\bs=T_\bs(\bx;\partial_\bx)$ and $D_\bs=D_\bs(\bx)$ are boundary operators of order $1$ and $0$, respectively, and the operators have analytic coefficients and may have lower order terms. 

The analogs of Theorems~\ref{3T1} and \ref{3T1b} are then true, if we augment the regularity assumptions on the right hand side $\bff$ by the appropriate regularity assumptions on $\bg_\bs$ and $\bh_\bs$.

Since Kondrat'ev's results \cite{Kondratev67} apply to general operators with variable coefficients, the basic regularity results for variational solutions are also available, and therefore Theorem~\REF{6T1} can be extended to cover coercive problems with analytic coefficients on polygons.

In the three-dimensional case, the extension to non-zero boundary data is as straightforward as in the two-dimensional case, but the handling of variable coefficients is of a different level of difficulty, due to the anisotropy along the edges. For this, the techniques of analytic estimates have to be resumed at a more basic level, involving commutator estimates and norms of Sobolev-Morrey type \cite[Lemmas 1.6.2 \& 2.6.2]{GLC}. This will be presented in detail elsewhere.

\subsection{General boundary conditions}
\label{ssGBC}
In our theorem on analytic regularity for coercive variational problems on polyhedra, Theorem~\ref{6T3}, we had to restrict the admissible boundary conditions to Dirichlet and Neumann conditions. A boundary condition concerning the tangential components, for example, such as described in \eqref{6Etan}, is not covered by Theorem~\ref{6T3}, although the corresponding result is undoubtedly true.
The restriction is not due to the tools developed in this paper --- the natural analytic regularity shift results in Theorems~\REF{5T1} and~\REF{5T2} are proved for solutions of problem \eqref{5Ebvp} without this restriction --- but due to the availability of the basic regularity results that we are quoting from \cite{MazyaRossmann03}, see the proof of Theorem~\ref{6T3}. If one wants to lift this restriction, one therefore has to prove basic regularity in the appropriate weighted Sobolev spaces for solutions of the boundary value problem \eqref{5Ebvp}. This is outside of the scope of the present paper, but it will be treated in \cite{GLC}.

\subsection{More general elliptic problems}
\label{ssGEP}
First we may easily extend the results of this paper to transmission
problems, namely problem like \eqref{5Ebvp} where $L$ has piecewise
constant coefficients (hence some transmission conditions have to be
imposed at the common boundary of the sub-domains). Indeed an
estimate like \eqref{1E1} holds for such problems and is proved in
\cite[Theorem 5.2.2]{GLC}.  Second, higher order differential
operators like $\Delta^2$ may be treated in a similar manner.
Finally, our method may be used for the Stokes system (see
\cite{GuoSchwab06} for two-dimensional results).

The interesting case of boundary value problems for the Maxwell equations is more delicate for several reasons. Whereas the Maxwell equations may be formulated more or less equivalently as a second order elliptic system, the boundary conditions will be of a more general type than the one treated here. More importantly, the energy space where variational solutions are to be found is, in the case of non-convex polygons or polyhedra, not contained in $\bH^1$, and therefore the basic regularity results will be of a different nature.

%\bibliographystyle{mnachrn}
%\bibliography{../Book/BookSing,../Book/serge}
%\bibliography{/Users/sergenicaise/Documents/Serge/GLC/GLCnouveau/BookSing,/Users/sergenicaise/Documents/Serge/GLC/GLCnouveau/serge}

\newcommand{\noopsort}[1]{}\def\cprime{$'$} \def\cprime{$'$}
  \def\polhk#1{\setbox0=\hbox{#1}{\ooalign{\hidewidth
  \lower1.5ex\hbox{`}\hidewidth\crcr\unhbox0}}}

\bigskip
\centerline{\sc Addresses}

\end{document}